\documentclass[11pt]{amsart}
\usepackage{rotating}
\usepackage{mathcomp,amscd}
\usepackage[all]{xy}
\usepackage{amssymb}
\usepackage{times}
\usepackage{hyperref}
\usepackage{enumerate}
\usepackage{verbatim}
\usepackage{mathabx,epsfig} %

\usepackage{color}

\newcommand{\Sha}{\mathrm{Sha}}

\newcommand{\modd}{/\!\!/}

\newcommand{\Frob}{\mathrm{Frob}}

\newcommand{\Ho}{\mathrm{Ho}}
\newcommand{\Maps}{\mathrm{Hom}}
\newcommand{\Sing}{\mathrm{Sing}}
\newcommand{\colim}{\mathrm{colim}}

\newcommand{\crys}{\mathrm{crys}}

\newcommand{\Lie}{\mathrm{Lie}} 

\newcommand{\GL}{\mathrm{GL}}
\newcommand{\dotimes}{\underline{\otimes}}
\newcommand{\adele}{\mathbb{A}}
\newcommand{\ur}{\mathrm{ur}}
\newcommand{\loc}{\mathrm{loc}}

\newcommand{\SL}{\mathrm{SL}}

\newcommand{\End}{\mathrm{End}}
 \newcommand{\Map}{\mathrm{Hom}}

\newcommand{\Sets}{\mathrm{Sets}}
\newcommand{\Der}{\mathrm{D}}
\newcommand{\Hom}{\mathrm{Hom}}
\newcommand{\Vinf}{\mathsf{V}}

 \newcommand{\fram}{\Box}
\renewcommand{\dim}{\mathrm{dim}}

\newcommand{\smallC}{\mathrm{Art}}
\newcommand{\SCR}{\mathrm{SCR}}

\newcommand{\Tor}{\mathrm{Tor}}
\newcommand{\Ker}{\mathrm{Ker}}
\newcommand{\Art}{\mathrm{Art}}

\newcommand{\Ext}{\mathrm{Ext}}
\newcommand{\hocolim}{\mathrm{hocolim}}
\newcommand{\holim}{\mathrm{holim}}

\newcommand{\Spec}{\mathrm{Spec}}

\newcommand{\Ad}{\mathrm{Ad}}

\newcommand{\Def}{\mathrm{Def}}

\newcommand{\hs}{\mathrm{hs}}

\newcommand{\Cotor}{\mathrm{Cotor}}
\newcommand{\Exinf}{\mathrm{Ex}^\infty}

\newcommand{\hofib}{\mathrm{hofib}}

\usepackage[T1]{fontenc}

\newcommand{\C}{{\Bbb C}}
\newcommand{\Z}{{\Bbb Z}}
 \newcommand{\R}{{\Bbb R}}

\newcommand{\PGL}{\mathrm{PGL}}

\newcommand{\DK}{\text{Dold--Kan}}
\newcommand{\Ch}{\mathrm{Ch}}

\newcommand{\Simp}{\mathrm{Simp}}
\newcommand{\F}{{\mathbb F}}
\newcommand{\N}{{\Bbb N}}

\newcommand{\Q}{\mathbb{Q}}
 
\newcommand{\Gal}{\mbox{Gal}}

\newcommand{\G}{\mathbf{G}}

\newcommand{\GG}{\mathbf{G}}

\newcommand{\bigtangent}{\mathfrak{t}}
\newcommand{\tangent}{\mathfrak{t}  }

\newcommand{\rhobar}{\overline{\rho}}

\numberwithin{equation}{section}
\numberwithin{table}{section}
\numberwithin{figure}{section}
\newtheorem{theorem}{Theorem}[section]
\newtheorem{Proposition}[theorem]{Proposition}
\newtheorem{Corollary}[theorem]{Corollary}
\newtheorem{Lemma}[theorem]{Lemma}
\newtheorem{Remark}[theorem]{Remark}
\newtheorem{Definition}[theorem]{Definition}
\newtheorem{Conjecture}[theorem]{Conjecture}
\newtheorem{Example}[theorem]{Example}
\allowdisplaybreaks
 \makeindex
\begin{document}

 \title{Derived Galois deformation rings}
 \author{S. Galatius, A. Venkatesh}

 \begin{abstract}
 We define a derived version of Mazur's Galois deformation ring. It is a pro-simplicial ring $\mathcal{R}$ classifying deformations
 of a fixed Galois representation 
to simplicial coefficient rings; its zeroth homotopy group $\pi_0 \mathcal{R}$ recovers Mazur's deformation ring. 

 We give evidence that these rings $\mathcal{R}$ occur in the wild: For suitable Galois representations, the Langlands program predicts
 that $\pi_0 \mathcal{R}$ should act on the homology of an arithmetic group. 
  We explain how the Taylor--Wiles method can be used to upgrade  such an action   to a graded action of $\pi_* \mathcal{R}$ on the homology.

 \end{abstract}
    \maketitle 
 
 {\small     
    \tableofcontents }

 \section{Introduction} 
 
 \subsection{} \label{introintro}

The  Langlands program posits a bijection between automorphic forms and Galois representations.
The study of $p$-adic congruences leads to the definition of a ``$p$-adic moduli space of automorphic forms,''
namely the spectrum $\mathrm{Spec}(\mathrm{T})$ of a suitable Hecke algebra,
and a ``$p$-adic moduli space of (geometric, e.g.\ crystalline) Galois representations,'' the spectrum $\mathrm{Spec}(\mathrm{R})$
of Mazur's Galois deformation ring.   Proving that the natural map
\begin{equation} \label{RT} \mathrm{R} \stackrel{\sim}{\rightarrow } \mathrm{T} \end{equation} is   an isomorphism is the basis
of modern
proofs of modularity, after the work of Wiles and Taylor--Wiles \cite{W, TW}.

 The purpose of this paper is to construct  a derived version $\mathcal{R}$ of $\mathrm{R}$, a pro-finite simplicial ring; it represents 
 Galois deformations with coefficients in simplicial commutative rings.  It is unsurprising that a derived version should play a role: speaking informally,  
 the  locus of crystalline Galois representations of $\Gal(\overline{\Q}/\Q)$ is obtained by intersecting the  space of Galois representations of $\Gal(\overline{\Q}/\Q)$ with the space of local geometric representations of $\Gal(\overline{\Q}_p/\Q_p)$ and the intersection is not, in general, transverse.

  The existence of $\mathcal{R}$  follows by applying the derived Schlessinger criterion  
 of Lurie \cite{LurieThesis}. The first part of this paper  is a leisurely exposition of this criterion and  also supplies various basic results that are useful when applying it.
  The goal of the second part is to explain how $\mathcal{R}$  arises naturally in the Taylor--Wiles method, or rather the obstructed version of this method developed  by Calegari and Geraghty: \cite{CG, H, KT}.

Our initial motivation for this construction was the numerology of the Betti numbers for an arithmetic group $\Gamma$. 
 For instance, if 
$\Gamma \leqslant \SL_n(\Z)$,  
it appears that the exterior algebra of a vector space  of dimension
  $\delta:=\left[ \frac{n-1}{2} \right]$ acts on $H^*(\Gamma, \overline{\Q}_p)_{\chi}$, 
  where the subscript $\chi$ means that we take the $\chi$-eigenspace for a tempered character $\chi$ of the Hecke algebra, and  there is a similar story for any arithmetic group (see \cite{VTakagi} for elaboration and discussion).  Now
 let  $\mathcal{O}$  be the ring of integers of $\overline{\Q}_p$  and  $\rho: \Gal(\overline{\Q}/\Q) \rightarrow \GL_n(\mathcal{O})$  the Galois representation attached to $\chi$ and $\rhobar$ its mod $p$ reduction.
Suppose the standard conjecture that $\rho$ does not have characteristic zero  crystalline deformations.
Then $\pi_* \mathcal{R}_{\rhobar} \otimes_{\pi_0 \mathcal{R}_{\rhobar}}   \overline{\Q}_p$ is   isomorphic to an exterior algebra on $\delta$ generators,
where the map $\pi_0 \mathcal{R}_{\rhobar} \rightarrow \overline{\Q}_p$ is the one associated to $\rho$.

 This  numerical coincidence naturally suggests that $\pi_* \mathcal{R}_{\rhobar}$ might in fact freely act on the integral homology $H_*(\Gamma,   \mathcal{O})_{\rhobar}$,  where the subscript 
 $\rhobar$ on homology means that we localize at the corresponding ideal of the Hecke algebra.
This should generalize 
 the way that the usual deformation ring acts on the  homology of a modular curve, and this is indeed what we establish under
 suitable hypotheses.  (Note that the precise hypotheses  are rather involved, but we give references to where they are specified;
 also, in the main text, we work with a specific finite field $k$ and its Witt vectors, rather than $\overline{\mathbb{F}}_p$ and $\mathcal{O}$ as above.)
 \begin{quote} {\em  Part of the main theorem:} 
 Assume the existence of Galois representations attached to  (possibly torsion) homology classes (Conjecture \ref{GaloisRepConjecture}); this
 amounts to giving an action of $\pi_0 \mathcal{R}_{\rhobar}$ on $H_*(\Gamma,\mathcal{O})_{\rhobar}$.
 Then,  under   conditions on $\rhobar$ that include enforcing ``minimal level,'' (\S \ref{minimalevel}, \S \ref{TWnotationsetup}), 
 the homology $H_*(\Gamma, \mathcal{O})_{\rhobar}$ admits a free action of $\pi_* \mathcal{R}_{\rhobar}$,
extending the  action of $\pi_0 \mathcal{R}_{\rhobar}$.  \end{quote}

 For more, see  \S \ref{proofs} and   the exact result is 
Theorem \ref{TW main theorem}.   Our simplifying assumptions on $\rhobar$ 
are certainly restrictive;   nonetheless we expect the statement above to be valid 
without any local assumptions, at least as long as one deals with crystalline representations in the Fontaine--Laffaille range. 
See also Remark  \ref{minimal level}.

   The proof amounts to using the Taylor--Wiles method   
  to upgrade  the action of the usual deformation ring $\pi_0 \mathcal{R}$ on  $H_*(\Gamma, \mathcal{O})_{\rhobar}$ 
 to a graded action of $\pi_* \mathcal{R}$.    Very roughly speaking, the Taylor-Wiles method shows that one can lift to a situation
 with unobstructed deformation theory, i.e.\ where the analogue $\mathcal{R}'$ of $\mathcal{R}$  satisfies $\mathcal{R}' \simeq \pi_0 \mathcal{R}'$; and then one descends. This method of defining the action of $\pi_* \mathcal{R}$ is very indirect,
 and we hope that it can be replaced by a better, more direct construction.
 
 In more detail: the method of Calegari and Geraghty has already shown that the homology  $H_*(\Gamma, \mathcal{O})_{\rhobar}$  has the structure of a free module
under a certain $\Tor$-algebra that arises in the Taylor-Wiles limit process.  The meaning of this $\Tor$-algebra is {\em a priori} obscure, because it depends
on all the choices made in that limit process.  So what we really do is to identify
this $\Tor$-algebra with a more intrinsic object, $\pi_* \mathcal{R}$. 

What would be even better  would be to refine this result by giving a chain level action of the simplicial ring $\mathcal{R}$
on the chain complex of $\Gamma$.

After a brief review   (\S \ref{sec:intro to scr}) of simplicial rings,
we give a quick overview of the definitions    (\S \ref{overview}) and  explain
what we prove about it (\S \ref{proofs}).

\subsection{Simplicial commutative rings} \label{sec:intro to scr} In  this paper we shall write $\SCR$ for the category of simplicial  commutative rings.
Since it is not standard in the number-theory literature, let us
briefly recall this concept.  
This is intended informally, and not as a substitute for 
a rigorous introduction; for that see \cite{GS, May}.    Also, note that we are only ever interested in simplicial {\em commutative} rings,
and occasionally will drop the word ``commutative.'' 

  For  a  topological   commutative ring $\mathbf{R}$, the homotopy groups $\pi_* (\mathbf{R},0)$  with basepoint $0 \in \mathbf{R}$
carry the structure of a graded ring, whose addition is defined by pointwise addition of continuous maps into $\mathbf{R}$.  To define the multiplication we represent
classes in $\pi_j (\mathbf{R},0)$ by maps $[0,1]^j \rightarrow \mathbf{R}$ sending the boundary to $0$,
and then we use the maps $[0,1]^j \times [0,1]^k \rightarrow [0,1]^{j+k}$.

One can extract from $\mathbf{R}$ a purely algebraic object -- a model example of a simplicial commutative ring --  which carries enough information
to recover  the graded ring $\pi_* \mathbf{R}$:
Let $\mathcal{R}_n =
\Sing_n(\mathbf{R})$ be the set of continuous maps from the $n$-simplex
$\Delta^n$ into $\mathbf{R}$.  The inclusion of the $n$-simplex as the
$i$th face of the $(n+1)$-simplex gives a map $d_i: \mathcal{R}_{n+1}
\rightarrow \mathcal{R}_n$; similarly, collapsing an
$(n+1)$-simplex to  its $i$th face gives maps $s_i: \mathcal{R}_n
\rightarrow \mathcal{R}_{n+1}$.  These maps $d_i: R_n \to R_{n-1}$ and
$s_i: R_n \to R_{n+1}$, $i =0, \dots, n$, satisfy various axioms, such
as $d_i d_j = d_{j-1} d_i$ when $i < j$.   %

It is possible to recover the homotopy groups $\pi_* \mathbf{R}$ from the collection of $\mathcal{R}_n$ and the maps $d_i$ and $s_i$: one builds a CW-complex $|\mathcal{R}|$,  the ``geometric realization,'' by gluing one copy of $\Delta^n$ for each element of $\mathcal{R}_n$,
and then there is a canonical map $|\mathcal{R}| \rightarrow \mathbf{R}$
which induces isomorphisms on all homotopy groups.

A simplicial commutative ring is a
collection of (commutative) rings $\mathcal{R}_n \ (n \geq 0)$ modelled on this
situation: i.e.\ equipped with maps $d_i: \mathcal{R}_n \rightarrow
\mathcal{R}_{n-1}$ and $s_i: \mathcal{R}_n \to \mathcal{R}_{n+1}$
satisfying the same axioms as the model case above. 
To such a simplicial $\mathcal{R}$ we may associate   homotopy groups $\pi_p \mathcal{R} $
in a way that will recover $\pi_* \mathbf{R}$ in the model case above (just take $\pi_p$ of the geometric realization $|\mathcal{R}_n|$). 
For example, $\pi_0 \mathcal{R}$ is the quotient of $\mathcal{R}_0$ by
the ideal $(d_0 - d_1)(\mathcal{R}_1)$.  The direct sum
\begin{equation} \label{pi_j_simplicial} \pi_*\mathcal{R} = \oplus_p \pi_p(\mathcal{R})\end{equation}  inherits the structure
of a graded ring.

A map $f: \mathcal{R} \to \mathcal{R}'$ of simplicial commutative rings is a
\emph{weak equivalence} if the induced map of homotopy groups is an
isomorphism.   This is equivalent to the map of
geometric realizations being a homotopy equivalence. 

Any (usual, non-simplicial) commutative ring $R$ gives rise to a
simplicial ring, in which we take all $\mathcal{R}_n$'s to equal $R$, and all $d_i$ and $s_i$ are the identity map of
$R$.  By a slight abuse of notation, the resulting simplicial
commutative ring shall often be denoted by the same letter $R$.
Objects of $\SCR$ arising in this way are called \emph{constant}, and
this construction gives a full and faithful embedding of the category of
commutative rings into $\SCR$.  This embedding is right adjoint to the
functor $\pi_0$ from $\SCR$ to commutative rings.  An object
$\mathcal{R} \in \SCR$ is weakly equivalent to a constant object if
and only if the natural map $\mathcal{R} \to \pi_0
\mathcal{R}$ is a weak equivalence, which in turn happens if and only
if $\pi_i\mathcal{R} = 0$ for all $i > 0$.  Let us say that
$\mathcal{R}$ is \emph{homotopy discrete}  when that
happens.

Finally, for $k$ a field,  we will be interested in ``Artinian''
simplicial rings over $k$: simplicial commutative rings $\mathcal{R}$
equipped with a homomorphism $\pi_0 \mathcal{R} \twoheadrightarrow k$ and with the
properties that $\pi_0\mathcal{R}$ is Artin local in the usual sense,
$\pi_i \mathcal{R}$ vanishes for all large enough $i$, and finally each 
 $ \pi_i \mathcal{R}$ is a finitely generated $\pi_0
\mathcal{R}$-module.   We denote by $\Art_k$ the category with objects  the Artinian simplicial rings over $k$,  and the morphisms  are   those  morphisms of simplicial commutative rings that commute with the map to $k$. 

It shall be important later on that these categories come with \emph{simplicial enrichments}: for objects $A$ and $B$ of $\SCR$ the set of morphisms $A \to B$ is the set of 0-simplices in a simplicial set $\Hom_{\SCR}(A,B)$, and similarly for the category $\smallC_k$.  In particular any representable functor naturally takes values in simplicial sets.

\subsection{Overview of the derived deformation ring}  \label{overview} 
 
For simplicity in this section, we   gloss over several technical
points -- in particular, the role of categories of pro-objects, and we talk only
about $\GL_n$ instead of a general algebraic group.

For a finite set of primes $S$, let
\begin{equation*}
  \Gamma_S = \pi_1^\mathrm{et}(\Spec \  \Z[1/S])) = \Gal(\Q^{(S)}/\Q),
\end{equation*}
the Galois extension of the largest extension $\Q^{(S)}$ unramified outside $S$, 
and  we fix $k$ a finite field whose characteristic divides $S$.   We fix
a representation
$$\rhobar: \Gamma_S \rightarrow \GL_n(k)$$
which we suppose to have centralizer consisting only of scalar matrices. 
Then, informally, Mazur's (non-simplicial) deformation ring $\mathrm{R}$ is the completed local ring of the
representation variety of the  group $\Gamma_S$ at  the representation $\rhobar: \Gamma_S \rightarrow \GL_n(k)$;
more precisely  $\mathrm{R}$ represents the functor
$\Def_{\rhobar}$ which sends an Artin ring $A \rightarrow k$ augmented
over $k$ to the set
\begin{equation*}
  \Def_{\rhobar}(A) =  \big(\mbox{ lifts $\tilde{\rho}: \Gamma_S
    \rightarrow \GL_n(A)$ }\big)/ \big(\mbox{ kernel of $\PGL_n(A) \rightarrow
    \PGL_n(k)$}  \big).
\end{equation*}
   Explicitly, then, there is a natural
bijection of  set-valued functors \begin{equation} \label{RMazur} \Hom(\mathrm{R}, -) \rightarrow \Def_{\rhobar}(-).\end{equation}
(where, on the left hand side, we consider only homomorphisms that commute with the augmentations to $k$).

The {\em derived} version $\mathcal{R}$ represents the functor which
sends a {\em simplicial} Artin ring $A$, augmented over $k$, to a {\em
  simplicial set} $\Def_{\rhobar}(A)$, defined similarly to the
 above but replacing
\index{$\Def_{\rhobar}$} the set of lifts $\tilde{\rho}$ by a suitable
simplicial set.  For this, one needs to generalize the objects
appearing in the above definition. 

 To give the reader some sense of the subtleties involved, let us briefly outline one  
way of carrying out this generalization, although we proceed differently in the text:
 For a simplicial commutative ring
$A$, we may define $\GL_n(A)$ as the simplicial monoid consisting of
those path components of $M_n(A) = A^{n^2}$ which map to
$\GL_n(\pi_0 A) \subset M_n(\pi_0 A)$.  (There is a more naive definition where one applies $\GL_n$ level-wise to $A$.  This would define a functor into simplicial groups but would be no good for our purposes: it would carry some morphisms $A \to A'$ in $\SCR$ which are homotopy equivalences into maps which are not homotopy equivalences.)
Next,  homomorphisms
$\Gamma_S \to \GL_n(A)$ must be defined in a derived sense. 
 At least if $\Gamma_S$ is a usual group and not a pro-group,  one
definition of the space of derived homomorphisms
$\Gamma_S \to \GL_n(A)$ proceeds by choosing a ``cofibrant
replacement'' of $\Gamma_S$, which roughly speaking means a simplicial
monoid $\Gamma_\bullet$ equipped with a homotopy equivalence
$|\Gamma_{\bullet}| \to \Gamma_S$ such that $\Gamma_p$ is a free
associative monoid on a (possibly infinite) set for all $p \geq 0$.
Then a zero-simplex of $R\Hom(\Gamma_S,\GL_n(A))$ is a map of
simplicial monoids $\Gamma_\bullet \to \GL_n(A)$. More generally, a $p$-simplex
of $R\Hom(\Gamma_S,\GL_n(A))$ would be a map of simplicial monoids $\Gamma_{\bullet} 
  \to \GL_n(A^{\Delta[p]})$, where $A^{\Delta[p]}$ is the mapping space $\Map(\Delta[p], A)$, 
  which itself has the structure of a simplicial commutative ring. 
     Having said this,    in our later presentation,
we will use a  different approach that  works better for general groups (not just $\GL_n$).

 That the resulting functor is
representable by some (pro-) simplicial ring $\mathcal{R}$ follows
from the derived Schlessinger criterion of Lurie.  What this means
(ignoring the ``pro'' subtlety for now) is that there is a natural
transformation
\begin{equation} \label{Rderived} \Maps(\mathcal{R}, -) \rightarrow \Def_{\rhobar}(-)\end{equation}
of functors valued in simplicial sets, and this natural
transformation induces a weak equivalence of simplicial sets for each
input $A \rightarrow k$.  (Recall that the mapping spaces between two simplicial
rings itself has the structure of a simplicial set).

  In the usual setting, the bijection \eqref{RMazur} determines $\mathrm{R}$ up to unique isomorphism. 
This is no longer true \emph{strictly} for~\eqref{Rderived}, only up to homotopy: any two $\mathcal{R}$'s will be homotopy equivalent as pro-simplicial rings, and in a suitable sense there will be a \emph{contractible} space of comparison maps.
  Thus we should,
strictly speaking, 
speak only of ``a representing ring.'' However,  the image of $\mathcal{R}$ in a suitable homotopy category
is still defined up to a unique isomorphism (see \S \ref{sec:hocat} for discussion) and therefore   associated invariants such as the graded ring $\pi_* \mathcal{R}$ are again determined up to unique isomorphism. 

For formal reasons (see Lemma \ref{MazurPi0})  $\pi_0\mathcal{R}$ will be isomorphic to Mazur's
deformation ring $\mathrm{R}$, but in fact, we should heuristically expect $\mathcal{R}$
to usually be homotopy discrete.    More precisely there
are always natural maps
\begin{equation} \label{iso} \mathcal{R} \longrightarrow \pi_0
  \mathcal{R} \stackrel{\cong}{\longrightarrow} \mathrm{R}.
\end{equation}
{\em and we heuristically expect the first map to be a weak equivalence;}
this is substantially equivalent (see Lemma \ref{uvd}) to the folklore
conjecture:
\begin{quote} Conjecture: The usual (unrestricted) deformation ring
  $\mathrm{R}$ is a complete intersection ring of expected
  dimension.
\end{quote}
This conjecture is seemingly quite difficult, but fortunately it is
entirely irrelevant for us; indeed, one of the advantages of the
derived deformation ring $\mathcal{R}$ is that this conjecture may
sometimes be circumvented.

  The derivedness of the rings we consider, then,   most likely  arises only at the next step, when we impose local conditions: 
  
What is important in number theory is not just the bare Galois deformation ring, but the ring
 which classifies only lifts $\Gamma_S \rightarrow \GL_n(A)$ that are (suitably defined) ``crystalline'' (or similar; see \S \ref{sec:localconditions}).
 In our setting, this means that we represent not the functor $A \mapsto \Def_{\rhobar}(A)$
 but rather the homotopy fiber product  
 \begin{equation} \label{obstructed} \Def_{\rhobar} \times^h_{\Def_{\rhobar_p}}  \mathrm{Def}^{\mathrm{crys}}_{\rhobar_p}, \end{equation} where 
 $\Def_{\rhobar_p}$ is the similarly defined deformation functor for the restriction $\rhobar_p$ of $\rhobar$  to the Galois group of $\Q_p$, and $\Def_{\rhobar_p}^{\mathrm{crys}}$ represents
 the crystalline deformations for $\rhobar_p$. Recall 
 (see Example \ref{homotopy pullback square Example}) 
 that this homotopy fiber product fits into the diagram  
     \begin{equation}  \label{TWdiagINTRO} 
 \xymatrix{
 \Def_{\rhobar} \times^h_{\Def_{\rhobar_p}}  \mathrm{Def}^{\mathrm{crys}}_{\rhobar_p}    \ar[r]\ar[d]  &   \mathrm{Def}^{\mathrm{crys}}_{\rhobar_p}     \ar[d] \\ 
\Def_{\rhobar} \ar[r] &    \Def_{\rhobar_p}}   
 \end{equation}
  and, explicitly  it assigns to $A \in \Art_k$ a simplicial set whose (e.g) vertices correspond to a vertex of $\Def_{\rhobar}(A)$, a vertex of $\mathrm{Def}^{\mathrm{crys}}_{\rhobar_p}(A)$
 and finally a path (i.e.\ a 1-simplex) between their images inside $ \Def_{\rhobar_p} (A)$.

For the moment, denote by $\mathcal{R}^{\crys}$ a ring that represents the functor \eqref{obstructed}. 
This is the derived deformation ring that is of primary interest to us. 

In the most classically studied cases,
for example, deformations of Galois representations for the modular curve, {\em the ring $\mathcal{R}^{\crys}$
will again be, under mild assumptions,  homotopy discrete:} that follows by Lemma \ref{uvd} and the fact that the usual crystalline deformation
ring is known to be a complete intersection.  In other words, the derived ring carries no extra information at all.

But for the general case (when one studies modular forms on $\SL_n$ for $n \geq 3$, for example -- or for that matter if we consider {\em even} 2-dimensional Galois representations), 
$\mathcal{R}^{\crys}$ should not be expected to be homotopy discrete  - it has higher homotopy groups. In the next section, we explain how it  should 
relate to the homology of arithmetic groups.

   We have already mentioned that $\pi_0 \mathcal{R} \cong  \mathrm{R}$. %
 A second basic property of the deformation ring $\mathcal{R}$ is that we can identify its  Andr{\'e}-Quillen cohomology with the cohomology of $\Gamma_S$, valued in the adjoint representation:
 $$ \left(\mbox{Andr\'{e}-Quillen cohomology of $W(k) \rightarrow \mathcal{R}$ in degree $i$}\right) \cong  H^{i+1}(\Gamma_S, \Ad \ \rhobar)$$
 There are similar results for the ring $\mathcal{R}^{\crys}$ where we impose crystalline conditions. 
 
 To get a sense of the extra information in $\mathcal{R}$, let us discuss  briefly $\pi_1 \mathcal{R}$, which is
 a module over $\pi_0 \mathcal{R}$. 
  Roughly speaking, $\pi_1 \mathcal{R}$ is made of two pieces:
 the first is related to the failure of $\mathrm{R}$ to be complete intersection; and the second  is related to the discrepancy
 between $\dim H^2(\Gamma_S, \Ad \rho)$ and the number of relations in a minimal presentation of $\mathrm{R}$.  The equation \eqref{pi0RRcompare} gives a precise formulation of the prior sentence.

{\em Comparison with  naive construction using $\mathrm{Tor}$-groups:} 
 It is well-known that ``derived'' rings  arise naturally  when one takes non-transversal intersections:
informally, when one intersects the subvarieties $X, Y \subset Z$, the structure sheaf
of $X \cap Y$ should be considered as a (homotopy) sheaf of simplicial commutative rings whose homotopy groups are given by   $\mathrm{Tor}^{\mathcal{O}_Z}_*(\mathcal{O}_X, \mathcal{O}_Y)$.  
 In the story above,
derivedness has come from the non-transversal intersection between $$X = \mbox{the moduli space of 
geometric representations of $\Gal(\overline{\Q}_p/\Q_p)$}$$
$$Y = \mbox{the moduli space of $\Gamma_S$-representations }$$
within 
$$ Z  = \mbox{ the  moduli space of all   representations of $\Gal(\overline{\Q}_p/\Q_p)$.}$$
Correspondingly, an easier-to-define version of the derived representation ring could be made by 
\begin{equation} \label{Tordefinition}  \mathrm{Tor}_*^{\mathrm{R}_p} (\mathrm{R}, \mathrm{R}_p^{\crys}), \end{equation} 
where $\mathrm{R}, \mathrm{R}_p, \mathrm{R}_p^{\crys}$ are respectively the usual deformation ring,
deformation ring at $p$, and the  crystalline deformation ring at $p$. If we accept the Conjecture 
discussed after \eqref{iso}, this direct construction actually does give the homotopy groups $\pi_* \mathcal{R}^{\crys}$. 
Unsurprisingly it is nonetheless helpful to work with the more refined version of the construction rather than this more concrete version; for example,  we do not need to worry about the validity of the Conjecture (if the Conjecture is false, 
\eqref{Tordefinition} will not have good formal properties).

\subsection{Arithmetic results about the derived deformation ring} \label{proofs}

  As mentioned in the abstract, the point of the number theory section is to describe how
 the ring $\mathcal{R}^{\crys}$ arises naturally in the context of the Langlands program.  More precisely,
 the Langlands program predicts  (approximately speaking) that $\pi_0 \mathcal{R}^{\crys}$
 acts on the homology of arithmetic groups, and we show that the Taylor--Wiles method can be
 used to upgrade this action to a graded action of $\pi_* \mathcal{R}^{\crys}$.  
To carry out this upgrading, we establish an isomorphism 
between $\pi_* \mathcal{R}^{\crys}$ and a limit ring that arises in the Taylor-Wiles method (in the form  \cite{CG} of Calegari-Geraghty).

{\em  Our results require some assumptions (Conjecture \ref{GaloisRepConjecture},  similar to the conjectures assumed in \cite{CG}) on the existence of Galois representations and local-global compatibility, which are not known in general at present;
we assume they are known in the discussion that follows.}

We need to briefly summarize the setup of the Taylor--Wiles method.
In the paragraphs that follows, we will be talking of {\em usual} (rather than derived) deformation rings, and will denote the usual rings by roman face, e.g. $\mathrm{R}$.
  
 In the Taylor-Wiles method,  we wish to study  the deformation ring $\mathrm{R}$ of a given residual representation $\rhobar: \Gal(\overline{\Q}/\Q) \rightarrow G(k)$, with $G$ an algebraic group and $k$ a finite field of characteristic $p$.   This representation $\rhobar$ arises, by means of the Langlands correspondence, 
 from a Hecke eigenclass in the homology of some arithmetic manifold $Y$.  The situation and our precise local and global assumptions  are detailed more carefully in \S  \ref{TWnotationsetup}.

 In our discussion we will always assume that $\rhobar$ is crystalline at $p$. 
 All our deformation rings will be those with crystalline conditions imposed at $p$, and we
 will therefore omit the superscript ``$\crys$'' from our notation. 
 
 Let $W=W(k)$ be the ring of Witt vectors of $k$. 
 
To proceed one considers the deformation ring $\mathrm{R}_n$ after enlarging the set $S$, replacing it by $S \coprod Q_n$
for a suitable set of auxiliary primes.     
 
Now there is a natural map $\mathrm{R}_n \rightarrow \mathrm{R}$.  Moreover,  by studying the action of inertia groups in $S$
we produce a map $(\Z/p^n)^s \rightarrow \mathrm{R}_n^{\times}$, for suitable $s$; thus 
$\mathrm{R}_n$ is an algebra over $\mathrm{S}_n = W [[ (\Z/p^n)^s ]]$, the group algebra of $(W/p^n)^s$. Finally one can recover $\mathrm{R}$ by descent:
\begin{equation} \label{Descent} \mathrm{R} = \mathrm{R}_n \otimes_{\mathrm{S}_n} W.\end{equation}

By a suitable limit process one can choose the rings $\mathrm{R}_n,  \mathrm{S}_n$
to approximate (i.e.\ be quotients of) limit rings $\mathrm{R}_{\infty}, \mathrm{S}_{\infty}$ that are of a very simple shape:
{\em in the best situation (which we will suppose from now on),} they are both formal power series rings over $W$, i.e.
$$\mathrm{R}_{\infty} \simeq W[[x_1, \dots, x_n]], \ \ \mathrm{S}_{\infty} \simeq W[[y_1, \dots, y_{n+\delta}]]$$
for some integers $n, \delta$; 
 and  by taking the $n \rightarrow \infty$ limit of \eqref{Descent}, we get 
\begin{equation} \label{tw2}  \mathrm{R} = \mathrm{R}_{\infty} \otimes_{\mathrm{S}_{\infty}} W.\end{equation} 
Because of the very simple shape of the limit rings, the formula \eqref{tw2} gives a lot of information about $\mathrm{R}$.

Let us say a little bit more about how this argument works, since we will need some of its internal notation. When we say $\mathrm{R}_n$ 
approximates $\mathrm{R}_{\infty}$, what we actually mean  is that a certain Artinian quotient  $\mathrm{R}_n \twoheadrightarrow  \overline{\mathrm{R}}_n$ is isomorphic to a  quotient of $\mathrm{R}_{\infty}$,
where both quotients become deeper and deeper as $n \rightarrow \infty$.    There is a similar result to \eqref{Descent} using $ \overline{\mathrm{R}}_n$ and $\overline{\mathrm{S}_n} := \mathrm{S}_n/p^n$  instead of $\mathrm{R}_n$, and just recovering   a certain quotient of $\mathrm{R}$ (this quotient will become closer and closer to all of $\mathrm{R}$ as $n$ increases).  One then obtains \eqref{tw2} by taking the $n \rightarrow \infty$ limit of this modified version of \eqref{Descent}.

It was discovered by Calegari and Geraghty that not just the tensor product, but also  the higher Tor groups,  play an important role in the theory of modular forms.
Their analysis implies the following:

\begin{quote} (Calegari--Geraghty): The homology $H_*(Y)$ of the associated arithmetic manifold $Y$, localized at the   ideal $\mathfrak{m}$ of the Hecke algebra that corresponds to $\rhobar$, 
is free   over the graded ring $\mathrm{Tor}^{\mathrm{S}_{\infty}}_*(\mathrm{R}_{\infty} , W)$.
\end{quote}

Our main result is that, in fact, these Tor-groups are captured by the
derived deformation ring:

 \begin{quote} 
{\em Main theorem  with assumptions as above:}
Notations as above,  with $\mathcal{R}$ the derived deformation ring of $\rhobar$ with crystalline conditions imposed at $p$, 	 there is an isomorphism  of graded rings $$ \pi_* \mathcal{R} \simeq  \Tor_*^{\mathrm{S}_{\infty}}(\mathrm{R}_{\infty} , W). $$
From the result of Calegari--Geraghty above, it follows that 
\begin{quote} (i)   the localized homology  $H_*(Y)_{\mathfrak{m}}$ carries the structure of a free graded module over   $ \pi_* \mathcal{R} $, and therefore\end{quote}
\begin{quote} (ii)   $\pi_j \mathcal{R}$ is nonvanishing precisely for $0 \leq j \leq \delta$. \end{quote} \end{quote} 

See Theorem \ref{TW main theorem} for the precise statement. Note that  there are various choices made in the proof of Theorem \ref{TW main theorem},
e.g., choices of subsequences to make compactness arguments, and {\em a  priori} the module structure from (i)  might depend on all these choices.
In fact it does not -- this 
  independence is proven only in the final section \S \ref{DHAcompare}.

  We now  describe the proof, in outline.

The first step of the proof is to construct a map from  $\mathcal{R}$ to the derived tensor product  $\mathrm{R}_{\infty} \otimes_{\mathrm{S}_{\infty}} W$. 
To do this, we use a compactness argument to extract a limit of the following maps:
\begin{equation} \label{patching_baby} \mathcal{R}  \simeq \mathcal{R}_n \dotimes_{\mathcal{S}_n} W \rightarrow \pi_0 \mathcal{R}_n \dotimes_{\pi_0 \mathcal{S}_n} W
\ \rightarrow  \overline{\mathrm{R}}_n \dotimes_{\overline{\mathrm{S}}_n} W/p^n \end{equation}
where $\mathcal{R}_n, \mathcal{S}_n$ are derived versions of $\mathrm{R}_n$ and $\mathrm{S}_n$ and $\dotimes$ is a derived version of tensor. 
The first isomorphism is a formality: it  exhibits that ``a representation of level $S \coprod Q_n$ unramified at $Q_n$ is actually of level $S$.''
The second is functoriality of the (derived) tensor product.   The last map arises from $\pi_0 \mathcal{R}_n  = \mathrm{R}_n \rightarrow \overline{\mathrm{R}}_n$
and similar for $\mathcal{S}$.

Now, passing to the limit over $n$, we produce a map $\mathcal{R} \rightarrow \mathrm{R}_{\infty} \dotimes_{\mathrm{S}_{\infty}} W$.  We  check it is an isomorphism by 
checking the induced map on Andr\'e--Quillen cohomology. Both sides have tangent complexes supported only in two dimensions, and it's clear it's an isomorphism in one degree;
to verify in the other degree, we check the induced map is surjective and then compare Euler characteristics.

 Although the implications (i) and (ii) mentioned in the statement of the Theorem don't seem to be accessible without proving the full result,
 other implications of the Theorem are obvious or can be obtained more directly. 
Let us talk through these to give some orientation of where the content is.
 \begin{itemize}
 \item[(a)]  The theorem implies that $\mathcal{R}$
 (at least at the level of homotopy groups)  is obtained by taking $n$ free generators and imposing (in the derived sense) $n+\delta$ relations. 
This can be deduced directly from  general facts of deformation theory (cf.\   Corollary \ref{cor:Corollary hurewicz analogue}),
 where $n$ is the dimension of a suitable tangent space. 
 \item[(b)] The theorem implies that, when $\delta = 0$, $\mathcal{R}$ is homotopy discrete, i.e.\  the natural map $\mathcal{R} \rightarrow \pi_0 \mathcal{R}$ is an isomorphism. This can be deduced directly  (Lemma \ref{uvd}) if one knows that $\pi_0 \mathcal{R}$ is a complete intersection of the expected size.  When $\delta = 0$ one can 
 usually get this from the usual
  Taylor-Wiles method \cite{W, TW};
 one does not really need to go through the main theorem.  
 \item[(c)]
The theorem also allows to compute the homotopy groups in characteristic zero: a geometric lift $\tilde{\rho}$ of the original Galois representation
to  $W$  
gives a homomorphism $\pi_0 \mathcal{R} \rightarrow W$; writing $E$ for the quotient field of $W(k)$,  the associated ring
$\pi_* \mathcal{R} \otimes_{\pi_0 \mathcal{R}} E$ 
  is isomorphic to the exterior algebra of a $\delta$-dimensional vector space over $E$.  Again  this can be deduced directly without much trouble (at least, assuming
 vanishing of $H^1_f(\Ad \widetilde{ \rho})$, which is a consequence of standard conjectures).%
 \end{itemize}

On the other hand, neither of the noted implications (i) or (ii) are obvious on general grounds. In particular 
the statement about the free action on homology seems to  be the key point.

To conclude,  we study in \S \ref{DHAcompare} the relationship between $\mathcal{R}$ and the  derived Hecke algebra introduced in \cite{DHA}. 
The derived Hecke algebra and the derived representation ring seem to be of different natures;  the Hecke algebra acts on cohomology, increasing cohomological degree, and
 $\mathcal{R}$ on homology, increasing homological degree.   The relationship is as follows:  one looks like the exterior algebra on a vector space $\mathsf{V}$
 and the other looks like the exterior algebra on $\mathsf{V}^{*}$.   The eventual result  (Theorem \ref{refit}) shows, in particular,  that {\em the action of $\pi_* \mathcal{R}$ is in fact independent of the choice of sequence of Taylor--Wiles primes used. }
 
 We also mention the related paper \cite{HT} of Hansen and Thorne. There an action of an exterior algebra on homology is produced
 by (roughly speaking) adding level at $p$ (rather than Taylor--Wiles level) and descending.  It would be useful to identify this exterior algebra also with $\pi_* \mathcal{R}$. 
  
 \begin{Remark} \label{minimal level} 
 Our  various assumptions on $\rhobar$   from \S \ref{TWnotationsetup} 
 --  in particular, excluding congruences with other forms -- have the effect of also forcing $\pi_* \mathcal{R}$
 to be an {\em integral} exterior algebra. As stated,  the results of \S \ref{DHAcompare} would not be true if this were not the case.
 If we dropped these  simplifying assumptions,   we don't expect $\pi_* \mathcal{R}$ to be an integral exterior algebra,
but nonetheless we would expect that the comparison results  of \S \ref{DHAcompare} should remain true after tensoring with $\Q$.

By contrast, the statement that the  localized homology $H_*(Y)_{\mathfrak{m}}$ carries the structure of a free $\pi_* \mathcal{R}$-module
doesn't use the full strength of the assumptions on $\rhobar$ -- in particular, it only uses ``minimal level.''  We expect that it continues to be valid, even integrally, 
even without that assumption. 

\end{Remark}

 \subsection{Overview of paper and suggestions for reading}
 
 The first part of the paper is generalities that are not specific to  number theory:
It uses freely the language of model categories. An introduction to this language may be found in \cite{GS}.
 We also use freely the language of homotopy limits and colimits, because we need to work with {\em pro}-representable functors.   A brief review of this language is given in Appendix \ref{sec:homotopy-theory}.
  
 \S \ref{sec:functors-artin-rings} is an overview of basic results about Artin local
 simplicial commutative rings and simplicial functors on this category;
 we define representability and prove a general result concerning the approximation
 of a functor by a representable functor.

\S \ref{sec:tang-compl-funct}  reviews the tangent complex of 
a functor and
 Lurie's derived Schlessinger criterion, and develops some tools for manipulating and understanding (pro-)representable functors and their representing objects.

 \S \ref{sec:gener-semis-algebr} defines the functors that we are
 interested in representing, namely, the derived space of homomorphisms from a
 profinite group $\Gamma$ to $G(A)$, where $G$ is an algebraic group
 and $A$ a simplicial ring.  The definitions here require some care. 

The remainder of the paper studies the specific case of deformations of a Galois representation: 

\S \ref{numtheorynotn}  and \S \ref{numbertheorynotn2} collects the notation to be used in the number theory sections
(\S \ref{numtheorynotn} collects notation about Galois representations and cohomology, and \S \ref{numbertheorynotn2}
summarizes notations about derived deformation rings). 

\S \ref{sec:TWprimes} examines what happens to the deformation ring when we add a single prime to the ramification set.
(The result is intuitively obvious.) The case of main interest is when this prime is a ``Taylor--Wiles prime,'' thus the title of the section. 

\S \ref{sec:localconditions} discusses how to impose ``local conditions'' on the derived deformation ring. The point of main interest to us is  imposing 
a crystalline condition.  

\S \ref{main} gets down to business:  The Taylor-Wiles method involves adding a set $Q$ of ramified primes with carefully chosen cohomological properties.
In \S \ref{main} we show that these cohomological properties give a tight control on the derived deformation ring after allowing ramification at $Q$. 

\S \ref{sec:patching} gives an abstract discussion of how to extract limits of maps like \eqref{patching_baby}.

\S \ref{derivedTW} summarizes the obstructed Taylor--Wiles method, as developed by Calegari--Geraghty; we use the formulation of Khare and Thorne.

\S \ref{piRSid} proves the main theorem, and  
\S \ref{DHAcompare} gives the comparison between the action of the derived deformation ring and the derived Hecke algebra.

\subsection{Acknowledgements}
The second-named author (A.V.) would like to thank Frank Calegari, Michael Harris and Shekhar Khare for helpful discussions (and encouragement)  about this paper.  
He was supported by a grant from the Packard foundation and by an NSF grant.   The first-named author (S.G.) was supported by NSF grant DMS-1405001 and the European Research Council (ERC) under the European Union's Horizon 2020 research and innovation programme (grant agreement No 682922).

\section{Functors of simplicial Artin rings}
\label{sec:functors-artin-rings}

We shall assume the reader is familiar with standard properties of the
category of simplicial sets, which we shall denote $s\Sets$, including
the notion of \emph{Kan complexes} and the fibrant replacement
$X \to \mathrm{Ex}^\infty(X)$, as well as the self-enrichment of\index{$\mathrm{Ex}^{\infty}$}
$s\Sets$: the set of morphisms $X \to Y$ between two simplicial sets
$X$ and $Y$ forms the 0-simplices in a simplicial set $s\Sets(X,Y)$.
This has the expected behavior when $Y$ is Kan (e.g.\ homotopy
equivalences $X' \to X$ and $Y \to Y'$ induce a homotopy equivalence
$s\Sets(X,Y) \to s\Sets(X',Y')$ when $Y$ and $Y'$ are Kan).
When the distinction between the set of maps and the simplicial set of
maps is important, we shall write $s\Sets_0(X,Y)$ for the set of maps.

As usual, we denote by $|X|$ the geometric realization of the
simplicial set $X$, and by $\Sing(Y)$ the simplicial set of simplices
associated to a topological space $Y$.  We shall write $\Delta^q$ or
$\Delta[q]$ for the usual simplex $[p] \mapsto \Delta([p],[q])$ and
$\partial \Delta^q = \partial \Delta[q]$ for its boundary.  If
$X = (X,x_0)$ is a pointed simplicial set, we denote by $\Omega X$ the
simplicial loop space, i.e.\ the simplicial set
$s\Sets((\Delta^1,\partial \Delta^1), (X,x_0))$ of pointed maps from
the simplicial circle to $X$.

\subsection{Simplicial commutative rings}
\label{sec:simplicial-rings}

Let us write $\SCR$ for the category of simplicial commutative rings,
i.e.\ functors from $\Delta^\mathrm{op}$ to commutative rings.  Recall
from \cite[II, \S 4]{QuillenHomotopicalAlgebra} that $\SCR$ comes with
subcategories of weak equivalences, cofibrations and fibrations
satisfying the axioms of a \emph{simplicial model category}.  We refer
there for more details, but briefly recall some of the key
definitions.

For an object $R \in \SCR$ we write $R^{\Delta[p]} =
s\Sets(\Delta^p,R)$, which is naturally an object of $\SCR$.  If $R'
\in \SCR$ is another object, the set of morphisms $R' \to
R^{\Delta[p]}$ is the $p$-simplices of a simplicial set $\SCR(R',R)$
of maps, and in this way $\SCR$ is enriched over $s\Sets$.

For an object $R\in \SCR$, the homotopy groups
$\pi_*(R) = \oplus_n \pi_n(|R|,0)$ form a graded commutative ring and
a morphism $R \to R'$ in $\SCR$ is a weak equivalence if it is a weak
equivalence of underlying simplicial sets, i.e.\ induces an
isomorphism on all homotopy groups.  A morphism $R' \to R$ in $\SCR$
is a \emph{fibration} the underlying morphisms of simplicial sets is a
Kan fibration, and we recall that this is automatic when the map of
$p$-simplices $R'_p \to R_p$ is surjective for all $p$ (in fact
happens if and only if the restriction to the components containing 0
is surjective in each simplicial degree).

Finally, \emph{cofibrations} $R' \to R$ are defined by a lifting
property, but we shall recall a particular source of cofibrations
which play an important role in this paper.  If $X$ is a set we shall
write $\Z[X]$ for the free commutative algebra on $X$ and if $X$ is a
simplicial set we shall use the same notation $\Z[X]$ for the simplicial commutative ring arising by applying
this construction in each simplicial degree.  Then if $R \in \SCR$ and 
$e: \partial \Delta^p \to R$ is a morphism of simplicial sets, there
is a unique extension $\Z[\partial \Delta^p] \to R$ to a morphism in
$\SCR$, and we obtain a morphism
\begin{equation}\label{eq:4}
  R \to R' = R \otimes_{\Z[\partial \Delta^p]} \Z[\Delta^p]
\end{equation}
in $\SCR$, whose target depends on the \emph{attaching map}
$e: \partial \Delta^p \to R$ even though we omitted it in the
notation. (In the above equation, the tensor product is taken level-wise.) 
\begin{Definition}\label{defn:attach-cell}
  The simplicial commutative ring obtained from $R$ by attaching a
  \emph{cell} along $e: \partial \Delta^p \to R$ is the simplicial
  commutative ring $R'$ defined by~\eqref{eq:4}.
\end{Definition}
For $p = 0$ this just amounts to adjoining a single polynomial
generator in each simplicial degree.  For any $p$ and $e$ the
resulting map $R \to R'$ is a cofibration, as are finite or
transfinite compositions of maps of this form.

Then an arbitrary morphism $R'' \to R$ may be factored as
$R'' \to R' \to R$ where $R'' \to R'$ is a cofibration and $R' \to R$
is both a weak equivalence and a fibration.  The existence of such a
factorization is part of the
axioms of a model category, but a particular proof of its existence
constructs $R'$ as a (possibly transfinite) composition of cell
attachments.  In the special case where $R'' = \Z$ is the initial
object we obtain a \emph{cofibrant approximation} $R' \to R$ where
$R'$ is obtained from $\Z$ by iterated cell attachments, similar to
``CW approximations'' of topological spaces.  If we
don't attempt to control the cardinality of the set of cell
attachments it is possible to construct $R' \to R$ as a functor of
$R$, a \emph{functorial cofibrant approximation}.  We shall pick one
such and denote it $c(R) \to R$ or sometimes $R^c \to R$ for
typographical convenience.

Later in the paper we shall discuss the analogue of ``minimal CW
approximations'' of topological spaces: roughly speaking we may for a
given $R$ ask for a weak equivalence $R' \to R$ where $R'$ is built
using the minimal possible number of cells of each dimension.  Such a
minimal $R' \to R$ will not be functorial in $R$, but turns out to
exist, at least when $\SCR$ is replaced by a modified category
``pro-$\smallC_k$'' which we shall also define later.

\subsection{Simplicial Artin rings}
\label{sec:simpl-artin-rings}

Recall that $k$ is a fixed (usually finite) field, which we regard as
an object of $\SCR$ and write $\SCR_{/k}$ for the over category.  Following the setup of \cite{LurieThesis}, we shall be especially interested in
a certain full subcategory $\smallC_k \subset \SCR_{/k}$ of ``Artin''
objects, defined in the next subsection.   In itself it is too small to
be a model category, since it does not have enough limits and colimits
(e.g.\ it has no initial object) --  but we shall study its homotopy theory using the 
forgetful functors $\smallC_k \to \SCR_{/k} \to \SCR$.  By a mild
abuse of language we shall say e.g.\ ``$R \in \smallC_k$ is
cofibrant'' to mean that ``$R \in \smallC_k$ is has cofibrant image
under the inclusion functor $\smallC_k \to \SCR_{/k}$'', etc.

\begin{Definition}
  An object $A \in \SCR$ is \emph{Artin local} if $\pi_0 A$ is Artin
  local in the usual sense and $\pi_* A = \oplus_n \pi_n A$  is finitely generated as a module over $\pi_0 A$.  For a (usually finite) field
  $k$, the category $\smallC_k \subset \SCR_{/k}$ is the full
  subcategory whose objects are the $\epsilon: A \to k$ with $A$ Artin
  local and $\epsilon: \pi_0(A) \to k$ is surjective.  (In other
  words, $\epsilon$ is a specified isomorphism from the residue field
  of $\pi_0(A)$ to $k$.)
\end{Definition}

For typographical reasons we shall often denote the object $(\epsilon:
A \to k) \in \smallC_k$ by simply $A$, but we emphasize that the
map to $k$ is part of the data and that morphisms in $\smallC_k$ are
required to commute with the maps to $k$.

\begin{Lemma}
  If $B \to D \leftarrow C$ is a diagram in $\smallC_k$ such that
  either $B \to D$ or $C \to D$ is surjective in each simplicial
  degree, then the fiber product $A = B \times_D C$ is also an object of 
  $\smallC_k$.
\end{Lemma}

\begin{proof}
  If one of the maps is degreewise surjective then it is a fibration
  and the map to the homotopy fiber product $B \times_D C \rightarrow B \times_D^h C$ is a weak equivalence.   (See
Example \ref{homotopy pullback square Example} for definition of the homotopy fiber product.)
Therefore, the homotopy groups of $A$ fit into a Mayer--Vietoris short exact
  sequence with those of $B$, $C$, and $D$.  This is a sequence of
  modules over the ring $\pi_0(A)$, from which it is easily deduced
  that $\pi_0(A)$ is Artin local in the usual sense.  The
  finite-length condition also follows from the Mayer--Vietoris sequence.
\end{proof}

\begin{Definition}
  \begin{enumerate}[(i)]
  \item\label{item:8} If  $V$ is a
    simplicial $k$-module, the object $k \oplus V \in \SCR_{/k}$ is
    defined by square-zero extension in each simplicial degree.  This
    is an object of $\smallC_k$ if and only if
    $\dim_k(\pi_*(V)) < \infty$.
  \item For $n \geq 0$ write $S^n = \Delta^n/\partial \Delta^n$ for
    pointed simplicial set obtained by collapsing the boundary of the
    simplex to a point.  Then write $k[n]$\index{$k[n]$} for the free
    simplicial $k$-module generated by $S^n$ (i.e.\ with $p$-simplices
    the free $k$-module on the $p$-simplices of $S^n$ modulo the span
    of the basepoint).  Write $k \oplus k[n]$ for the corresponding
    square-zero extension.
  \item More generally for a $k$-module $V$, write
    $V[n] = V \otimes_k k[n]$ for the simplicial $k$-module obtained
    as the tensor product in each simplicial degree, and
    $k \oplus V[n]$ for the square-zero extension.
  \item\label{item:9} Let 
    $\widetilde{k[n]} = \{0\} \times^h_{k[n]} k[n]$ be the homotopy
    fiber product of the diagram $0 \to k[n] \leftarrow k[n]$,
    \index{$\widetilde{k[n]}$} which is again a simplicial $k$-module,
    and $k \oplus \widetilde{k[n]}$ for the corresponding square-zero
    extension.  Similarly for $\widetilde{V[n]}$ and $k \oplus
    \widetilde{V[n]}$.
  \end{enumerate}
\end{Definition}

The notation $V[n]$ is inspired by the corresponding notation from
chain complexes: the homotopy groups of $V[n]$ are concentrated in
degree $n$ and $\pi_n(V[n],0)$ is canonically isomorphic to $V$.

The homotopy fiber product in~(\ref{item:9}) is contractible and has
the property that $\widetilde{k[n]} \to k[n]$ is a Kan fibration.
(In fact any simplicial $k$-module with these properties would work
just as well as $\widetilde{k[n]}$ for what follows.)  Similarly,
$k \oplus \widetilde{k[n]} \to k \oplus k[n]$ can be viewed as just a
particular way of replacing the unique morphism $k \to k\oplus k[n]$
in $\smallC_k$ by a fibration.

The objects $k \oplus k[n] \in \smallC_k$ play a special role, due to
the following special case of the pullback construction.
\begin{Example}\label{example:attach-homotopy-group}
  Let $h: A \to k \oplus k[n]$ be any morphism in $\smallC_k$, and
  define $A' \to A$ by the pullback diagram (pullback in each
  simplicial degree)
  \begin{equation} \label{build_ring_diagram}
    \begin{aligned}
      \xymatrix{
        A' \ar[d]\ar[r] & k \oplus \widetilde{k[n]} \ar[d]\\
        A \ar[r]_-{h} & k \oplus k[n],  }
    \end{aligned}
  \end{equation}
  Then,  for any $f: B \to A$, the space of null
  homotopies of the composition $B \to A \to k \oplus k[n]$ is
  isomorphic to the space of lifts of $f$ to $B \to A'$.
\end{Example}

The
word ``null homotopies'' in the above example should be interpreted as
``paths to the trivial homomorphism through ring homomorphisms'', in
the following way.  The composition $h \circ f$ is a 0-simplex in
$\smallC_k(B,k \oplus k[n])$ as is the composition
$B \to k \to k \oplus k[n]$ of the augmentation map and the unique
morphism $k \to k \oplus k[n]$.  These two compositions assemble to
one map $\partial \Delta^1 \to \smallC_k(B,k \oplus k[n])$ and by the
space of null homotopies we mean the simplicial set of extensions to
maps $\Delta^1 \to \smallC_k(B,k \oplus k[n])$.  By ``space of lifts''
we mean the simplicial subset $s\Sets(B,A')$ consisting of maps making
the triangle commute.  These are isomorphic by the definition of the
homotopy fiber product used in defining $\widetilde{k[n]}$.

One reason for the importance of the objects $k \oplus k[n]$ is that
any object or morphism in $\smallC_k$ is weakly equivalent to one
built by finitely many iterations of the process explained in
Example~\ref{example:attach-homotopy-group} above.  Let us establish a
few preliminary properties of the objects $k \oplus V[n]$.

\begin{Lemma} Let $n \geq 0$ and let $V$ and $W$ be finite-dimensional
  $k$-vector spaces.
    Let
    $R$ be an
    object of $\smallC_k$ and
    $\phi: \pi_*(R) \to \pi_*(k \oplus V[n])$ an isomorphism of graded
    rings.  Then $\phi$ is induced by a zig-zag
    $R \leftarrow R' \to k \oplus V[n]$ of weak equivalences in
    $\smallC_k$.
\end{Lemma}

\begin{proof}
  First use that $\Z \to k$ is complete intersection, so there exists
  a cofibrant model $R''$ of $k$ as a $\Z$-algebra built using
  generators in degree 0 and 1 only: indeed, the ring $k$ can be
  obtained from $\Z[x]$ by killing the regular sequence $(p,f(x))$,
  where $f$ is any integral lift of an irreducible polynomial in
  $\F_p[x]$ of appropriate degree, and we may build $R''$ by
  attaching two 1-cells to $\Z[x]$ along $p$ and $f(x)$.  Then there
  are no obstructions to finding a map $R'' \to R$ inducing
  isomorphisms in $\pi_k$ for $0 \leq k < n$.  Then pick a generating
  set $v_1, \dots, v_d \in V$ and write $R''[v_1,\dots, v_d]$ for the
  simplicial commutative ring obtained from $R''$ by attaching one
  $n$-cell along the constant map $0: \partial \Delta^n \to R''$ for
  each basis element, in the sense of
  Definition~\ref{defn:attach-cell}.  The morphism $R'' \to R$ then
  extends to a morphism $R''[v_1, \dots, v_d] \to R$ inducing a
  bijection in $\pi_k$ for $0 \leq k < n$ and a surjection for
  $k = n$.  If we then define $R''[v_1, \dots, v_d] \to R'$ by
  attaching cells of dimension $n+1$ and higher to kill generators for
  the kernel in $\pi_n$ and all higher homotopy groups, there is no
  obstruction to extend to a morphism $R' \to R$ which will then be a
  weak equivalence.  The same argument also applies to give a weak
  equivalence $R' \to k \oplus V[n]$, giving the desired zig-zag.
\end{proof}

\begin{Remark}
  It
  is
  not quite true that morphisms $k \oplus V[n] \to k \oplus W[n]$ in
  $\smallC_k$ are ``the same thing'' as $k$-linear maps $V \to W$,
  even up to homotopy.  For example for $n=1$, $V=0$ and $W = k$, the
  space $\smallC_k(k,k\oplus k[1])$ is not path connected, and in fact
  has $\pi_0 = \pi_{-1} \bigtangent k \cong k$.  Morally, the reason
  for this is that $\smallC_k(A^c, B) = \SCR_{/k}(A^c,B)$ is the space
  of derived morphisms of $\Z$-algebras; when $A$ and $B$ happen to be
  represented by simplicial $k$-algebras there is also a notion of a
  derived space of $k$-algebra maps, but it will have a different
  homotopy type.
\end{Remark}

Let us also briefly discuss \emph{tensor products}  of
simplicial commutative rings.  Given a diagram
$R' \leftarrow R \to R''$ we may form the tensor product
$R' \otimes_R R''$ levelwise, and this inherits the usual universal
property from commutative rings: a homomorphism out of it is the same
as a pair of homomorphisms out of $R'$ and $R''$ restricting to the
same homomorphism out of $R$.  In order to obtain a homotopy invariant
version of this construction, it should only be applied when either
$R'$ or $R''$ is cofibrant as an $R$-algebra (e.g.\ it is built from
$R$ using a finite or transfinite iteration of cell attachments).  In
that situation there is a spectral sequence (\cite[Theorem 6, \S 6,
II]{QuillenHomotopicalAlgebra})
\begin{equation*}
  E^2 = \mathrm{Tor}_{\pi_*(R)}(\pi_*(R'),\pi_*(R'')) \Rightarrow
  \pi_*(R' \otimes_R R'').
\end{equation*}
In particular suppose $R \to R'$ is a cofibration in $\smallC_k$ and
$\pi_m(R',R) = 0$ for all $m \leq n-1$.  In that case the spectral
sequence implies the isomorphism
$\pi_n(R' \otimes_R k) = \pi_n(R',R) \otimes_{\pi_0 R} k$.

\begin{Lemma}\label{lem:build-Artin-rings-Postnikov}
  \begin{enumerate}[(i)]
  \item Let $R \to k$ be an object of $\SCR_{/k}$.  Then $R \to k$ is an  
    object of $\smallC_k$ if and only if there exists a sequence of
    cofibrant 
    simplicial commutative rings $A_0$, $A_1$, \dots, $A_m$ with $A_0
    = k$ and a weak equivalence $R \to A_m$, where $A_i$ is a cofibrant replacement of a simplicial commutative ring $A_{i-1}'$ obtained
    from $A_{i-1}$ by a pullback along $h_i: A_{i-1} \to k \oplus
    k[n_i]$ as above, with all $n_i \geq 1$.
  \item Any morphism $R \to A$ in $\smallC_k$ (with $R$ cofibrant) may
    be factored as $R \to A' \to A$, where the first map $R \to A'$ is
    a weak equivalence and the second map $A' \to A$ is obtained by
    composing a finite sequence of pullback diagrams as
    in~\eqref{build_ring_diagram} (and cofibrant replacement) with varying $n \geq 0$.
  \end{enumerate}
\end{Lemma}
\begin{proof}[Proof sketch] For (i) see also \cite[Lemma 6.2.6]{LurieThesis}.

  We prove (ii).  We will use Postnikov truncations, for which we refer the reader ahead to  
\S \ref{sec:postnikov}.
  Without loss of generality we may assume that $R \to
  A$ is a cofibration.  Suppose in addition that it is
  $(n-1)$-connected, i.e.\ that the relative homotopy groups
  $\pi_*(A,R)$ vanish in degrees strictly below $n-1$, for 
  some $n \geq 1$, and consider the map from $A$ to the tensor product
  $A \otimes_R k$.  Then
  $\pi_n(A  \otimes_R k) = \pi_n(A,R) \otimes_{\pi_0 R} k =
  V$ for some $k$-vector space $V$.  Then the truncation 
  $\tau_{\leq n}(A   \otimes_R k)$ has the same homotopy
  groups as $k \oplus V[n]$ and hence for $R$ and $A$ cofibrant there is a
  commutative diagram
  \begin{equation*}
    \xymatrix{
      R \ar[r] \ar[d]&  k \oplus \widetilde{V[n]} \ar[d]\\
      A \ar[r] & k \oplus V[n]
    }
  \end{equation*}
  such that the induced map $\pi_m(A,R) \to \pi_m(k \oplus V[n],k)$ is
  an isomorphism for $m < n$ and surjective for $m = n$.  If we write
  $R \to A'$ for the induced homomorphism from $R$ to the 
  pullback of the rest of the diagram, then we conclude that this map
  remains $(n-1)$-connected and that the group $\pi_n(A',R)$ is identified
  with the kernel of the quotient map
  $\pi_n(A,R) \to \pi_n(A,R) \otimes_{\pi_0 R} k \cong \pi_n(A,A') =
  V$.  Since $\pi_n(A,R)$  has finite length as a module over $\pi_0 R$ we may repeat
  this process finitely many times to achieve $\pi_n(A',R) = 0$ so
  that $R \to A'$ is $n$-connected.  Then continue by induction.  This
  finishes (ii) in the case where $\pi_0(R) \to \pi_0(A)$ is
  surjective, and in particular proves (i).

  If $\pi_0(R) \to \pi_0(A)$ is not surjective there is an easy preliminary step first, using $n=0$.
\end{proof}

The first part of the above lemma has the following analogy in pointed
spaces.  A pointed connected space is $p$-finite if it has only
finitely many non-zero homotopy groups, all of which are finite
$p$-groups.  Then a pointed connected space is $p$-finite if and only
if it is weakly equivalent to one obtained from a point by finite
iterations replacing a pointed space $X$ by the homotopy fiber of a
map $h: X \to K(\F_p,n)$, with varying $n \geq 2$.  The second part of
the above lemma has a similar analogy for maps between $p$-finite
spaces.  (We find the analogy to $p$-finite spaces instructive and
shall return to it later in the paper.)

\subsection{Functors and natural transformations}
\label{sec:natur-weak-equiv}

We shall study functors $\mathcal{F}: \smallC_k \to s\Sets$ and
natural transformations between them.
\begin{Definition}
  A \emph{natural simplicial homotopy} between two natural
  transformations $S,T: \mathcal{F} \to \mathcal{G}$ is a natural
  transformation $\Delta[1] \times \mathcal{F}(-) \to \mathcal{G}$
  giving a homotopy between $S,T: \mathcal{F}(A) \to \mathcal{G}(A)$
  for all $A$.
\end{Definition}
From this definition of simplicial homotopy one could define a notion
of ``natural simplicial homotopy equivalence'' in terms of functors in
both directions and homotopies from the two compositions to the two
identity natural tranformations, but that would be a very strict
notion rarely satisfied in practice.  The following weaker notion is
much more useful.
\begin{Definition} \label{def:natural-weak-equivalence}
  Let $\mathcal{F}, \mathcal{G}: \smallC_k \to s\Sets$ be two
  functors.  A \emph{natural weak equivalence} is a natural
  transformation $T: \mathcal{F} \to \mathcal{G}$ inducing a weak
  equivalence $\mathcal{F}(A) \to \mathcal{G}(A)$ for all $A \in
  \smallC_k$.  Two functors are naturally weakly equivalent if there
  exists a (finite) zig-zag of natural weak equivalences between them.
\end{Definition}

For example, any functor is naturally weakly equivalent to one taking
values in Kan simplicial sets.  Indeed, either of the natural transformations
$\mathcal{F}(A) \to \mathrm{Sin}|\mathcal{F}(A)|$ or $\mathcal{F}(A) \to \mathrm{Ex}^\infty(A)$ is a natural weak
equivalence, where $\mathrm{Ex}^\infty$ is Kan's infinite iteration of the adjoint subdivision.  We shall be mostly concerned with functors having the
following extra property.
\begin{Definition}
  A functor $\mathcal{F}: \smallC_k \to s\Sets$ is \emph{homotopy
    invariant} if $\mathcal{F}(\phi): \mathcal{F}(A) \to
  \mathcal{F}(B)$ is a weak equivalence for any weak equivalence
  $\phi: A \to B$ in $\smallC_k$.
\end{Definition}

\begin{Definition}
  A \emph{simplicial enrichment} of a functor $\mathcal{F}: \smallC_k
  \to s\Sets$ is the specification of maps of simplicial sets
  $\smallC_k(A,B) \to s\Sets(\mathcal{F}A,\mathcal{F}B)$ for all
  objects $A, B$, agreeing with the functoriality on 0-simplices and
  compatible with the compositions in the $s\Sets$-enriched categories
  $\smallC_k$ and $s\Sets$.

  A simplicially enriched functor is a functor together with a
  simplicial enrichment.
\end{Definition}
\begin{Lemma}\label{Lemma:2.10}
  Any homotopy invariant functor $\mathcal{F}: \smallC_k \to s\Sets$ is
  naturally weakly equivalent to a simplicially enriched one.  In
  fact, there is a simplicially enriched functor $\mathcal{F}'$ with
  Kan values and a natural weak equivalence $\mathcal{F} \to
  \mathcal{F}'$. \end{Lemma}
\begin{proof}[Proof sketch.] We briefly give the construction: 
for an object $A \in \smallC_k$ we write $A^{\Delta[p]} =
\Map_{s\Sets}(\Delta[p],A)$.  For fixed $[p]$, this simplicial set is
canonically a simplicial ring.    Since $A$ is automatically
fibrant as a simplicial set, the canonical map $A \to A^{\Delta[p]}$
is a weak equivalence of simplicial sets and in particular
$A^{\Delta[p]}$ is again an object of $\smallC_k$.
Let $F'(A)$ be the  diagonal simplicial set
of the bisimplicial set $[p] \mapsto F(A^{\Delta[p]})$, 
  i.e.\ the
$p$-simplices of $F'(A)$ are the $p$-simplices of $F(A^{\Delta[p]})$.
The canonical map of simplicial sets $F(A) \to F'(A)$ is a weak equivalence
and it is possible to simplicially enrich $F'$ in a natural way. 
For further details  see \cite[Corollary 6.5]{RezkSchwedeShipley}.
 \end{proof}

Thus, when proving a statement of the form that a certain homotopy
invariant functor $\mathcal{F}$ is naturally weakly equivalent to
another functor with certain properties, we may without loss of
generality assume that $\mathcal{F}$ takes values in Kan complexes and
is simplicially enriched.

\subsection{Representable functors}
\label{sec:simpl-funct-simpl}

Following Schlessinger and Lurie, we shall be interested in functors
which are \emph{representable} and functors which are
\emph{pro-representable}.
\begin{Definition}\label{defn:representability}
  A functor $\mathcal{F}: \smallC_k \to s\Sets$ is
  representable if it is naturally weakly equivalent to
  $\Hom(R,-)$ for some cofibrant object $R \in \smallC_k$.
\end{Definition}

When $\mathcal{F}$ is simplicially enriched, there is a ``simplicial
Yoneda-lemma'': sending a natural transformation $T: \Hom(R,-) \to
\mathcal{F}$ of functors $\smallC_k \to s\Sets$ to the zero-simplex
$T(\mathrm{id}) \in \mathcal{F}(R)$ gives a bijection between such
natural transformations and zero-simplices of $\mathcal{F}(R)$.  In
the same way, natural simplicial homotopies $\Delta[1] \times
\Hom(R,-) \to \mathcal{F}$ correspond to 1-simplices of
$\mathcal{F}(R)$, etc.
\begin{Lemma}  \label{simpilcially enriched representable Lemma}
  A simplicially enriched functor $\mathcal{F}: \smallC_k \to s\Sets$ is
  representable if and only if there exists a cofibrant $R \in
  \smallC_k$ and $\iota \in \mathcal{F}(R)$ such that the induced
  $\Hom(R,-) \to \mathcal{F}$ is a natural weak equivalence.
\end{Lemma}
A general (possibly unenriched) representable functor $\mathcal{F}$ is
automatically homotopy invariant because $\Hom(R,-)$ is homotopy
invariant for any cofibrant $R \in \SCR$. (This follows, e.g., from \cite[Lemma 1.1.12]{Hovey}, together with the fact that all simplicial
rings are fibrant.)  Therefore, by Lemma~\ref{Lemma:2.10}, we can
always find a zig-zag
$\mathcal{F} \to \mathcal{F}' \leftarrow \Hom(R,-)$ of length two.
\begin{proof}
  It suffices to prove that if $S: \mathcal{F} \to \Hom(R,-)$ is a
  natural weak equivalence, then there exists a natural transformation
  $T: \Hom(R,-) \to \mathcal{F}$ such that $ST$ is naturally
  simplicially homotopic to the identity natural transformation.

  To construct such a $T$, pick $\iota \in \mathcal{F}(R)$ such that
  $S(\iota) \in \Hom(R,R)$ is in the same path component as the
  identity map and let $T$ be the corresponding natural
  transformation.  Then $ST$ is the natural transformation
  corresponding to $S(\iota) \in \Hom(R,R)$.  Since this is a Kan
  complex, there is a 1-simplex $\lambda: \Delta[1] \to \Hom(R,R)$
  connecting the identity with $S(\iota)$; the adjoint $\lambda \in
  \Hom(R,R^{\Delta[1]})$ classifies a natural transformation
  $\Delta[1] \times \Hom(R,-) \to \Hom(R,-)$ from the identity to
  $ST$.
\end{proof}

\subsection{Approximation by representable functors}

Representable
functors play a central role in this paper, and we shall need
criteria for which functors are representable.  Let us first discuss
some more simple-minded methods for ``approximating'' an
\emph{arbitrary} homotopy invariant, simplicially enriched, Kan valued
functor $F: \smallC_k \to s\Sets$ by representable ones.  Our method
shall be a homotopy version of the following strategy.  Suppose $C$ is
an essentially small category which has all small limits and
$F: C \to \Sets$ is a functor.  Then there is a category whose object
are pairs $(c,\phi)$ where $c$ is an object of $C$ and
$\phi: C(c,-) \to F$ is a natural transformation; the morphisms
$(c,\phi) \to (c',\phi')$ are pairs of a morphism $c' \to c$ and a
natural isomorphism between the two resulting functors
$\Hom(c',-) \to F$.  Then we may to each such functor $F$ canonically
associate an object $c_F = \lim_{(c,\phi)} c$, and natural
tranformations
\begin{equation*}
  C(c_0,-) \leftarrow \colim_{(c,f)} C(c,-) \to F,
\end{equation*}
which are natural isomorphisms if $F$ happens to be representable.  In
particular this formula extracted a representing object $c_0$ directly
in terms of the functor $F$.

We shall follow a similar procedure to extract a representing object
$R \in \smallC_k$ from a representable functor
$F: \smallC_k \to s\Sets$.  First, we must replace all limits and
colimits by homotopy limits and homotopy colimits.  Second, the
category $\smallC_k$ is not essentially small in the strict sense, but
it is ``homotopy small'' (there is a set of objects representing all
homotopy classes).  Third, the category of pairs $(c,\phi)$ above
should be replaced by a \emph{simplicial category} and the homotopy
(co)limits should take this into account.  In
appendix~\ref{sec:homotopy-theory} we discuss only homotopy (co)limits
indexed by ordinary categories, so we shall not use this terminology
here; instead we write explicit formulas.  (For example the simplicial
functor defined in
Definition~\ref{defn:approximating-functor-by-secret-hocolim} below is
the simplicial version of the formula
$\colim_{(S,\phi)} C(S,A) \to F(A)$ for functors into sets.)

\begin{Definition}\label{defn:approximating-functor-by-secret-hocolim}
  Pick, ``once and for all'' a \emph{set} of cofibrant objects of
  $\smallC_k$, such that any object $R \in \smallC_k$ admits a weak
  equivalence from an object in the set, and write
  $\smallC_k^\mathrm{skel} \subset \smallC_k$ for the full subcategory
  with these objects.

  Let $F: \smallC_k \to s\Sets$ be homotopy invariant, simplicially
  enriched, and Kan valued.  For $[p] \in \Delta$, let
  $M^F_p: \smallC_k \to s\Sets$ be given by
  \begin{align*}
    &M^F_p(A) = \\
    &\coprod_{\substack{(S_0, \dots, S_p) \\ \in (\smallC_k^\mathrm{skel})^{p+1}}} F(S_0) \times
    \smallC_k(S_0,S_1) \times \dots \times \smallC_k(S_{p-1},S_p)
    \times \smallC_k(S_p,A)
  \end{align*}
  Composition of morphisms and functoriality of $F$ define face maps
  $M^F_p(A) \to M^F_{p-1}(A)$ and insertion of identities define
  degeneracy maps.  The diagonal of the resulting bisimplicial set
  shall be denoted also by $M^F: \smallC_k \to s\Sets$.
\end{Definition}

More conceptually, there is a category
$(\smallC_k^\mathrm{skel} \wr F)$ whose objects are pairs $(S,x)$,
$S \in \smallC_k^\mathrm{skel}$ and $x \in F(S)$, and whose morphisms
are $\phi: S_0 \to S_1$ send $x_0 \to x_1$.  The category is  the zero simplices of a 
\emph{simplicial category} (i.e.\ simplicial object in the category of
small categories) whose $p$-simplices is the category with objects
pairs $(S,x)$ with $x: \Delta[p] \to F(S)$ and morphisms
$(S,x_0) \to (S,x_1)$ are the $\phi: \Delta[p] \to \smallC_k(S_0,S_1)$
with the property that the composition
\begin{equation*}
  \Delta[p] \xrightarrow{(x,\phi)} F(S_0)
  \times \smallC_k(S_0, S_1) \to F(S_1)
\end{equation*}
is $x_1$.  There is a simplicial functor
$(\smallC_k^\mathrm{skel} \wr F)^\mathrm{op} \to s\Sets$ given on
objects by $(S,x) \mapsto \smallC_k(S,A)$, and the above construction
of $M^F(A)$ can be regarded as the homotopy colimit of that functor,
following a natural generalization of the Bousfield--Kan formula to
the case where the indexing category is a simplicial category.

\begin{Lemma}
  The evaluation map  
  $M^F_0(A) = \coprod_{S \in A} F(S) \times \smallC_k(S,A) \to F(A)$
  gives rise to a map
  \begin{equation*}
    M^F(A) \to F(A)
  \end{equation*}
  which is a natural weak equivalence.
\end{Lemma}
\begin{proof}[Proof sketch]
  Since $F$ and $M^F$ are both homotopy invariant, it suffices to 
  consider $A \in \smallC_k^\mathrm{skel}$.  If we write
  $M^F_{-1}(A) = F(A)$ the resulting augmented simplicial space has an
  ``extra degeneracy'' $s_{-1}: M^F_{p-1}(A) \to M^F_p(A)$ for each
  fixed $A$, defined by letting $S_{p}= A$ and inserting the identity
  on $A$.  Hence the augmentation map is a weak equivalence (cf.\
  e.g.\ \cite[Lemma III.5.1]{GoerssJardine}).
\end{proof}

The universal strictly representable functor which admits a map from
$M^F$ is then the functor represented by the ``homotopy limit'' of the  
 functor $(\smallC_k^\mathrm{skel} \wr F) \to s\Sets$ given on object
as $(S,x) \mapsto S$, as we shall now explain.   
(This is entirely analogous to the discussion of homotopy limits in Appendix \ref{sec:homotopy-theory}, except now the indexing category $\smallC_k^\mathrm{skel} \wr F$
is a simplicial category, and so one needs to take this into account.) 
There is a functor
$(\smallC_k^\mathrm{skel} \wr F) \to s\Sets$ described on objects as
\begin{equation*}
  (S,x) \mapsto N((S,x) \downarrow (\smallC_k^\mathrm{skel} \wr F)^\mathrm{op}).  
\end{equation*}
This is simplicial functor and we shall write it as
$N(-\downarrow (\smallC_k^\mathrm{skel} \wr F)^\mathrm{op})$.  We
shall write
\begin{equation*}
  R^F \subset \prod_S s\Sets(N(S \downarrow (\smallC_k^\mathrm{skel}
  \wr F)^\mathrm{op}),S)
\end{equation*}
for the simplicial subset consisting of natural transformations of
simplicial functors $(\smallC_k^\mathrm{skel} \wr F) \to s\Sets$ from
the simplicial functor $N(- \downarrow (\smallC_k^\mathrm{skel}))$ to
the functor $(S,x) \mapsto S$.  This simplicial set $R^F$ inherits a
ring structure from the $S$: indeed $R^F$ is the ``homotopy limit'' of
a simplicial diagram in $\SCR_{/k}$ and in fact defines an object of
$\SCR_{/k}$.  In general we should not expect to have
$R^F \in \smallC_k$ and $R^F$ can be quite large without further
assumptions on $F$.  However, if $F$ is representable we will have
$R^F \in \smallC_k$ and in fact it will be a representing object for
$F$.  More precisely, there is a canonical natural transformation
\begin{equation*}
  M^F(A) \to \Hom(R^F,A),
\end{equation*}
and the precise ``functorial approximation'' of representable functors
$F: \smallC_k \to s\Sets$ is the following result.
\begin{Proposition}\label{prop:approximation}
  The associations $F \mapsto M^F$ and $F \mapsto R^F$ are
  functorial.  If $F \to F'$ is a natural weak equivalence then so is
  $M^F \to M^{F'}$, and $R^{F'} \to R^F$ is a weak equivalence.  The maps
  \begin{equation*}
    F(A) \leftarrow M^F(A) \rightarrow \Hom(R^F,A) \to \Hom(c(R^F),A)
  \end{equation*}
  defined above are natural in both $A\in \smallC_k$ and $F$, where  $c: \SCR_{/k} \to \SCR_{/k}$ denotes a functorial cofibrant    approximation.  The functors $M^F$ and $\Hom(c(R^F),-)$ are homotopy
  invariant.  Moreover the left hand map is a natural weak equivalence for all
  $F$ (simplicially enriched, homotopy invariant, and Kan valued) and
  the composition $M^F(A) \to \Hom(c(R^F),A)$ is a weak equivalence
  for representable $F$.
\end{Proposition}

\begin{proof}[Proof sketch]
  It remains to see that $M^F(A) \to \Hom(c(R^F),A)$ is a weak
  equivalence for all $A$ when $F$ is representable.  It suffices to
  consider $A \in \smallC_k^\mathrm{skel}$ and $F = \Hom(R,-)$ for
  $R \in \smallC_k^\mathrm{skel}$.  In this case the simplicial
  category $(\smallC_k^\mathrm{skel} \wr F)$ has a terminal object
  given by $(R,\mathrm{id})$ and hence the homotopy limit projects by
  a weak equivalence to the value at that object, giving a weak
  equivalence $R^F \to R$ and hence a natural weak equivalence
  $F = \Hom(R,-) \to \Hom(R^F),-)$.  It remains to check that this
  natural weak equivalence commutes with the maps asserted in the
  statement of the proposition, which we leave to the reader.
\end{proof}

\subsection{Pro-representable functors}
\label{sec:pro-repr-funct}

The class of pro-representable functors $\smallC_k \to s\Sets$ can be
defined in several equivalent ways.  Let us pick the following as the
official definition.
\begin{Definition}\label{defn:pro-rep}
  A functor $\mathcal{F}: \smallC_k \to s\Sets$ is pro-representable
  if there exists a pro-object $R = (j \mapsto R_j)$ indexed by a
  filtered category $J$ and with all $R_j \in \smallC_k$ cofibrant, such that
  $\mathcal{F}$ is naturally weakly equivalent to the functor
  \begin{equation*}
    A \mapsto \colim_{j \in J^{\mathrm{op}}} \Hom(R_j,A).
  \end{equation*}

  $\mathcal{F}$ is \emph{sequentially pro-representable} if it is
  possible to choose $J$ countable.  (It is well known  that this
  implies that in fact one can choose $J = (\N,<)$: argue as in Lemma
\ref{lemma:replace-by-nice} below to replace $J$ by a partially ordered set,
  and then the argument of \cite[Tag 0597]{StacksProject} to convert it to $(\N,<)$.)  
\end{Definition}

Let us point out that any pro-representable functor is automatically
homotopy invariant, since, as we already observed,  $\Hom(R_j,-)$ is homotopy invariant for cofibrant $R_j$,  and
since filtered colimit of commutes with homotopy groups.  Following
Lurie and Schlessinger, we shall discuss general criteria for
pro-representability.  First let us discuss some more easily deduced
reformulations.
\begin{Lemma} \label{lemma: pro representing objects}
  Let $\mathcal{F}: \smallC_k \to s\Sets$ be a simplicially enriched
  functor with values in Kan simplicial sets.  Then the following
  conditions are equivalent.
  \begin{enumerate}[(i)]
  \item\label{item:3} $\mathcal{F}$ is pro-representable.
  \item\label{item:4} There exists a natural weak equivalence
    \begin{equation*}
      \hocolim_{j \in J^\mathrm{op}} \Hom(R_j,-) \to \mathcal{F}
    \end{equation*}
    with $R = (j \mapsto R_j)$ as above.
  \item\label{item:5} There exists a filtered category $J$ and a functor $F:
    J^\mathrm{op} \times\smallC_k \to s\Sets$ with $F(j,-)$
    representable (in the sense of
    Definition~\ref{defn:representability}) for all $j \in J$ and such
    that $\mathcal{F}$ is naturally weakly equivalent to $A \mapsto
    \colim F(j,A)$
  \end{enumerate}
\end{Lemma}
\begin{proof}
  That~(\ref{item:4}) implies~(\ref{item:3}) follows from the fact that
  $\hocolim \to \colim$ is a weak equivalence for filtered indexing
  categories.  That~(\ref{item:3}) implies~(\ref{item:5}) is clear, so it
  remains to see that~(\ref{item:5}) implies~(\ref{item:4}).

  Applying the functorial procedure from
  Proposition~\ref{prop:approximation} to each $F(j,-)$ we obtain
  cofibrant $R_j \in \smallC_k$ and zig-zags of weak equivalences
  \begin{equation*}
    \Hom(R_j,-) \leftarrow M_j \rightarrow F(j,-),
  \end{equation*}
  all natural with respect to $j \in J$.  To split the map
  $M_j \to \Hom(R_j,-)$ up to homotopy it suffices to pick a lift of
  the identity under $M_j(R_j) \to \Hom(R_j,R_j)$, which is possible,
  but a bit of care is needed to assemble these to a natural transformation
  $\hocolim_j \Hom(R_j,-) \to M_j$.  This follows from
  Lemma~\ref{Lemma:hocolim-is-cofibrant} below, which implies that any
  (simplicially enriched, Kan valued, homotopy invariant) functor
  $\mathcal{F}$ connected to $\hocolim \Hom(R_j,-)$ by a zig-zag of
  natural weak equivalences in fact admit a
  single natural weak equivalence $\hocolim \Hom(R_j,-)$.
\end{proof}
\begin{Lemma}\label{Lemma:hocolim-is-cofibrant}
  Let $J$ be a
  filtered category and $R: J^\mathrm{op} \to \smallC_k$ a pro-object
  with $R_j \in \smallC_k$ cofibrant for all $j \in J$.  The for any
  diagram of (simplicially enriched, Kan valued, homotopy invariant)
  functors $\smallC_k \to s\Sets$ and natural transformations
  \begin{equation*}
    \xymatrix{
      & \mathcal{F} \ar[d]^{\simeq} \\
      \hocolim_{j \in J} \Hom(R_j,-) \ar[r] & \mathcal{G}
    }
  \end{equation*}
  there exists a natural transformation $\hocolim_j \Hom(R_j,-) \to
  \mathcal{F}$ and a natural simplicial homotopy making the diagram
  homotopy commute.
\end{Lemma}
\begin{proof}
  The universal property of $\hocolim$ implies that the set of natural
  transformations  $\hocolim_j \Hom(R_j,-) \to \mathcal{G}$ is in
  bijection with vertices in $\holim_j \mathcal{G}(R_j)$ and that
  natural simplicial homotopies are in bijection with 1-simplices in
  this homotopy limit.  Hence the claim follows from the weak equivalence
  \begin{equation*}
    \holim_j \mathcal{F}(R_j) \to \holim_j \mathcal{G}(R_j).\qedhere
  \end{equation*}
\end{proof}

If $(i \mapsto R_i)$ is an object of pro-$\smallC_k$ with all $R_i$
cofibrant, then the natural map
$\hocolim_i \Hom(R_i,-) \to \colim_i \Hom(R_i,-)$ is an objectwise
weak equivalence so from a homotopical point of view the two functors
should be considered interchangeable.  Nevertheless the homotopy colimit and
colimit have different technical properties, in particular the
homotopy colimit is ``easier to map out of'', cf.\
Lemma~\ref{Lemma:hocolim-is-cofibrant}, and the strict colimit has the
convenient property that it sends fibrations in $\smallC_k$ to Kan fibrations of
simplicial sets.  In Lemma~\ref{lem:factor-through-strict-colim}
below, we shall establish a criterion for factoring a natural
transformation out of $\hocolim_i \Hom(R_i,-)$ through a natural
transformation out of the strict colimit.

\begin{Lemma}\label{lemma:replace-by-nice}  
  Any pro-object of $\smallC_k$ is isomorphic to an 
  $R: I^\mathrm{op} \to \smallC_k$ whose indexing category is a
  directed set, equipped with a strictly increasing map $I \to \N$,
  and satisfying that $\{j \in I \mid j \leq i\}$ is finite for all
  $i\in I$.  For any such diagram $R$ there exists another diagram
  $R': I^\mathrm{op} \to \smallC_k$ (with the same indexing category,
  if desired) and a natural transformation $R \to R'$ such that each
  $R_i \to R'_i$ is a weak equivalence and also a cofibration (in
  particular $R'_i$ is cofibrant if $R_i$ is), and such that the
  natural map
  \begin{equation*}
    R'_i \to \lim_{j < i} R'_j
  \end{equation*}
  is a fibration.
\end{Lemma}

\begin{proof}
  The indexing category map be replaced in the usual way
  (cf. \cite[Expose 1, Prop 8.1.6]{SGA4}): if $I$ does not already
  satisfy this, let $J$ be the poset of finite subdiagrams of $I$
  having a unique terminal object, ordered by inclusion (where a
  ``diagram'' is a set $O$ of objects and a set $F$ of morphisms such
  that the source and target of any element of $F$ are in $O$, and
  ``terminal object'' means an object $x \in O$ such that any object
  $y \in O$ admits precisely one morphism $y \to x$ in $F$).  Then $J$
  is a directed set and $J \to \N$ sends a diagram to its cardinality.
  The functor $J \to I$, which sends a subdiagram to its terminal
  object, is cofinal and hence the original pro-object
  $R: I^\mathrm{op} \to \smallC_k$ is isomorphic to the composition
  with $J \to I$.

  Now suppose $I$ satisfies the assumption and
  $R: I^\mathrm{op} \to \smallC_k$.  Then we construct $R'$ by the
  same argument as when constructing the ``Reedy model structure'' on
  $I$-shaped diagrams.  Namely, filter $I$ by the full subcategories
  $I_{\leq n}$ on those objects with image in $\{0, \dots, n\} \subset
  \N$, and assume
  the restriction $R': I_{\leq n-1}^\mathrm{op} \to \smallC_k$ and the
  maps $R_i \to R'_i$ have been defined for $i \in I_{\leq n-1}$.
  Then for $i \in \mathrm{Ob}(I_{\leq n}) \setminus
  \mathrm{Ob}(I_{\leq n-1})$ define $R_i \to R'_i$ by factoring
  the composition
  \begin{equation*}
    R_i \to \lim_{j < i} R_j \to \lim_{j < i} R'_j
  \end{equation*}
  into an acyclic cofibration $R_i \to R'_i$ followed by a fibration
  $R'_i \to \lim_{j < i} R'_j$.  This is possible in
  $\mathrm{SCR}_{/k}$ and the result will be in $\smallC_k$ because
  $R_i$ is.  The resulting object comes with a map to
  $\lim_{j < i} R'_j$ which, as $i$ varies, precisely contains the
  information to make $i \mapsto R'_i$ into a \emph{functor}
  $I_{\leq n}^\mathrm{op} \to \smallC_k$ extending the previously
  constructed functor on $I_{\leq n-1}$.
\end{proof}

The conclusion of the Lemma above motivates the following definition,
which is a generalization of ``tower of fibrations between cofibrant
objects''.

\begin{Definition} \label{nice ring definition}
  Let us say that a filtered diagram $R: I^\mathrm{op} \to \smallC_k$
  is \emph{nice} if $I$ is a poset such that $\{j \in I \mid j \leq
  i\}$ is finite for all $i$, and there exists a strictly increasing
  poset map $I \to \N$, that each $R_i \in \smallC_k$ is cofibrant,
  and that each of the maps $R_i \to \lim_{j < i} R_j$ is a fibration.
\end{Definition}

By the Lemma above we may replace any pro-object
$R: I^\mathrm{op} \to \smallC_k$ by one given by a nice diagram
$R': J \to \smallC_k$.  In general this replacement is a zig-zag
$R \leftarrow R'' \to R'$ of morphisms in pro-$\smallC_k$, the arrow
$R'' \to R$ changing $I$ to $J$ and cofibrantly replacing all $R_i$,
the arrow $R'' \to R'$ as in the lemma.

\begin{Lemma}\label{lem:factor-through-strict-colim}
  Let $R: I \to \smallC_k$ be a nice pro-object, and let
  $\mathcal{F}: \mathrm{SCR}_{/k} \to s\Sets$ be a simplicially enriched
  functor which (strictly) preserves finite limits and takes
  fibrations in $\smallC_k$ to Kan fibrations.  Then any natural
  transformation
  \begin{equation*}
    T: \hocolim_{i \in I} \Hom(R_i,-) \to \mathcal{F}
  \end{equation*}
  is naturally homotopic to one which factors through a natural
  transformation out of the (strict) colimit.
\end{Lemma}
Before giving the proof, let us mention that for us the main example
of a functor satisfying the assumption in the Lemma is
$\mathcal{F} = \colim_{j \in J} \Hom(R'_j,-)$ for
$R' \in \text{pro-}\smallC_k$ with all $R'_j$ cofibrant.  In this case
we conclude that (when the individual $R'_j$ are cofibrant and the
pro-object $R$ is nice in the sense defined above)
any natural transformation
\begin{equation*}
  \hocolim_{i \in I} \Hom(R_i,-) \to \colim_{j \in J} \Hom(R'_j,-)
\end{equation*}
is naturally homotopic to a transformation which factors through the
colimit, and hence is induced by a morphism $R' \to R$ in
pro-$\smallC_k$.  
\begin{proof}
  By the enriched Yoneda lemma, $T$ corresponds (bijectively) to a
  vertex in $\holim_{i \in I} \mathcal{F}(R_i)$, and we want a
  1-simplex connecting $T$ to a vertex in $\lim_{i \in I}
  \mathcal{F}(R_i) \subset \holim_{i \in I} \mathcal{F}(R_i)$.  Now
  the assumptions imply that all the maps
  \begin{equation*}
    \mathcal{F}(R_i) \to \lim_{j < i}\mathcal{F}(R_j) = 
    \mathcal{F}(\lim_{j < i} R_j)
  \end{equation*}
  are fibrations.  This condition in turn implies that
  $\lim_{i \in I} \mathcal{F}(R_i) \to \holim_{i \in I}
  \mathcal{F}(R_i)$
  is a weak equivalence (indeed,
  $\lim_{i \in I_{\leq n}} \mathcal{F}(R_i) \to \lim_{i \in I_{\leq
      n-1}} \mathcal{F}(R_i)$
  is a Kan fibration and
  $\lim_{i \in I_{\leq n}} \mathcal{F}(R_i) \to \holim_{i \in I_{\leq
      n}} \mathcal{F}(R_i)$
  a weak equivalence, by induction on $n$).
\end{proof}

Next we study natural transformations between pro-representable
functors.  As in the representable case, natural transformations
correspond to morphisms between the pro-objects representing them, but the precise statement is a
bit more complicated in this case.
\begin{Lemma}\label{Lemma:represent-natural-trans}
  Let $\mathcal{F}$ and $\mathcal{G}$ be pro-representable with Kan
  values and suppose given natural weak equivalences $\hocolim_i
  \Hom(R_i,-) \to \mathcal{F}$ and $\hocolim_j \Hom(R_j',-) \to
  \mathcal{G}$ with all $R_i$ and $R'_j$ cofibrant and $(i \mapsto
  R_i) \in \text{pro-}\smallC_k$ nice in the sense above.

  For any natural transformation $T: \mathcal{F} \to \mathcal{G}$
  there exists a morphism $T: R' \to R$ between the representing
  pro-objects, and a
diagram of functors and natural
  transformations
  \begin{equation*}
    \xymatrix{
      \mathcal{F} \ar[r]^-T 
& \mathcal{G}\\
      \hocolim_j \Hom(R_j,-)
      \ar[r]\ar[u]_\simeq\ar[d]^\simeq 
&
      \hocolim_j \Hom(R'_j,-)
      \ar[u]_\simeq\ar[d]^\simeq\\
      \colim_j \Hom(R_j,-) \ar[r]^-{T'} &
      \colim_j \Hom(R'_j,-),
    }
  \end{equation*}
  where both squares commute up to natural simplicial homotopy.
\end{Lemma}
\begin{proof}[Proof sketch]
  This is an easy consequence of
  Propositions~\ref{Lemma:hocolim-is-cofibrant} and~\ref{lem:factor-through-strict-colim}. \end{proof}

Although we shall not explicitly use it, we remark that there is a
model structure on the category of simplicially enriched functors $\smallC_k \to s\Sets$ in
which the fibrations and weak equivalences are defined
objectwise, cf.\ e.g.\ \cite{Toen}.
In particular there is a well defined homotopy category of such
functors.  The above Lemma then implies that $T$ and
$\colim_j \Hom(R_j,-) \to \colim_j \Hom(R'_j,-)$ are isomorphic in the
category of arrows in the homotopy category of functors.

\textbf{Remark}.  By general theory, the bottom natural transformation
corresponds to a morphism from the object of pro-$\smallC_k$ given by
the $R_j'$ to the object given by the $R_j$.  Similarly, the natural
transformation $T'$ corresponds to a zero-simplex in $\holim_j
\hocolim_i \Hom(R'_i,R_j)$.

\section{More on representable functors}

We continue our study of representable functors on the category $\Art_k$. 

\subsection{Postnikov truncations}
\label{sec:postnikov}

A basic construction in homotopy theory associates functorially
to any topological space $X$ a space $\tau_{\leq n} X$ with vanishing
homotopy groups in dimensions above $n$ and an $(n+1)$-connected map
$X \to \tau_{\leq n} X$.  In fact, the map $X \to \tau_{\leq n} X$ is
unique up to weak equivalence in an appropriate sense.  In topological
spaces the usual proof constructs $\tau_{\leq n}$ from $X$ by
attaching $(n+2)$-cells along \emph{all possible} maps
$\partial D^{n+2} \to X$, then $(n+3)$-cells to the result along all
possible maps from $\partial D^{n+3}$, etc.

The same idea as in topological spaces may be used to construct a
functor $\tau_{\leq n}: \SCR \to \SCR$ and a natural map
$R \to \tau_{\leq n} R$ in $\SCR$, inducing an isomorphism in $\pi_i$
for $i \leq n$ and such that $\pi_i \tau_{\leq n} R = 0$ for $i > n$.
For technical reasons it is convenient to have $R \to \tau_{\leq n}R$
always be a cofibration, such that e.g.\ $\tau_{\leq n} R$ is
cofibrant when $R$ is.  One construction of such a functor proceeds as
for topological spaces, using cell attachments in $\SCR$: first attach
cells along all possible maps $\partial \Delta^{n+2} \to R$, then
attach cells along all possible maps from $\partial \Delta^{n+3}$ to
the result, etc.  Then $\tau_{\leq n} R$, defined as the union of
these cell attachments, has the required homotopical properties and
$R \to \tau_{\leq n} R$ is a cofibration by construction.

Alternatively we can use the
\emph{coskeleton} functors: for any simplicial set $Y$ there is a
natural transformation $Y \to \mathrm{cosk}^n(Y)$ with the universal
property that maps $X \to \mathrm{cosk}^n(Y)$ are in natural bijection
with maps from the $n$-skeleton of $X$ to $Y$.  When $Y$ is Kan it is
easy to see that $Y \to \mathrm{cosk}^n(Y)$ is a model for
$Y \to \tau_{\leq n}Y$.  In particular, since the underlying
simplicial set of any $R \in \SCR$ is automatically Kan, we could
alternatively define $\tau_{\leq n} R$ by a (functorial) factorization
$R \to \tau_{\leq n} R \to \mathrm{cosk}^n(R)$ into a cofibration
followed by an acyclic fibration.  This construction has the mild
advantage that as $n$ varies the $\tau_{\leq n} R$ fit into a tower
$\dots \to \tau_{\leq n+1} R \to \tau_{\leq n} R \to \dots \to
\tau_{\leq 0} R$.   We will use this as the ``official'' definition of truncation, although it is slightly less intuitive than the previous one. 

If $R \in \SCR_{/k}$, then $\tau_{\leq n} R$ also comes with a
natural morphism to $k$ so (at least for $n \geq 0$) we may also
regard $\tau_{\leq n}$ as a functor $\SCR_{/k} \to \SCR_{/k}$.  If
$R \in \smallC_k$ then also $\tau_{\leq n} R\in \smallC_k$.

\subsection{Tensor product of simplicial rings and pro-simplicial
  rings}
\label{sec:tensor-product-pro}

Recall that if $R' \leftarrow R \rightarrow R''$ is any pushout diagram in $\SCR$ we defined $R' \otimes_R R''$ as the tensor product
applied in each degree.  This tensor product is the (strict) pushout
in simplicial commutative rings, and hence for any simplicial ring
$A$, the diagram of simplicial sets
\begin{equation*}
  \xymatrix{
    \Hom(R' \otimes_R R'',A) \ar[r] \ar[d] & \Hom(R',A)\ar[d]\\
    \Hom(R'',A) \ar[r] & \Hom(R,A)
  }
\end{equation*}
is pullback (not necessarily homotopy pullback).  Indeed, since $A$
may be replaced by $A^{\Delta[p]}$ it suffices to prove this on
zero-simplices of mapping spaces, where it follows by applying the
analogous result for discrete rings in each degree.  If either $R \to
R'$ or $R \to R''$ is a cofibration in simplicial commutative rings,
the square above is also homotopy cartesian.

We prove the following homotopical analogue.
\begin{Proposition}\label{prop:tensor-of-representables}
  Let $\mathcal{F}_0$, $\mathcal{F}_1$, $\mathcal{F}_{01}$ be
  representable functors $\smallC_k \to s\Sets$ and let $T_i:
  \mathcal{F}_i \to \mathcal{F}_{01}$ be natural transformations, $i =
  0,1$.  Then the objectwise homotopy pullback   \begin{equation*}
    A \mapsto \mathcal{F}_0(A) \times^h_{\mathcal{F}_{01}(A)}
    \mathcal{F}_1(A)
  \end{equation*}
  is sequentially pro-representable.
\end{Proposition}
\begin{proof}[Proof sketch]
  By homotopy invariance of the homotopy pullback, we may replace
  $\mathcal{F}_i$ and $\mathcal{F}_{01}$ by naturally weakly
  equivalent functors and also the natural transformations with
  naturally homotopic ones.

  Suppose $R_i, R_{01} \in \smallC_k$ are cofibrant and $\iota_i \in
  \mathcal{F}(R_i)$ and $\iota_{01} \in \mathcal{F}(R_{01})$ induce
  natural weak equivalences, and suppose $R_{01} \to R_i$ represent  
  the natural transformations $\mathcal{F}_i \to \mathcal{F}_{01}$, as
  in Lemma~\ref{Lemma:represent-natural-trans}.  We may suppose that
  $R_{01} \to R_i$ are cofibrations, in which case $\Hom(R_{i},A) \to
  \Hom(R_{01},A)$ are fibrations, so the objectwise strict pullback
  \begin{equation*}
    A \mapsto \Hom(R_0,A) \times_{\Hom(R_{01},A)} \Hom(R_1,A)
  \end{equation*}
  is naturally weakly equivalent to the homotopy pullback
  $\mathcal{F}_0 \times^h_{\mathcal{F}_{01}}\mathcal{F}_1$.  This
  object-wise fiber product is naturally isomorphic to the functor
  \begin{equation*}
    A \mapsto \Hom(R_0 \otimes_{R_{01}} R_1,A),
  \end{equation*}
  and $R_0 \otimes_{R_{01}} R_1$ is cofibrant.  Unfortunately $R_0
  \otimes_{R_{01}} R_1$ need not be an object of $\smallC_k$, but it
  is close.  Clearly the discrete ring
  \begin{equation*}
    \pi_0(R_0 \otimes_{R_{01}} R_1) = \pi_0(R_0)
    \otimes_{\pi_0(R_{01})} \pi_0(R_1)
  \end{equation*}
  is Artin local in the usual sense, and the spectral
  sequence \cite[Theorem 6, \S 6, II]{QuillenHomotopicalAlgebra} 
  \begin{equation*}
    E^2 = \Tor_{\pi_*(R_{01})}(\pi_*(R_0),\pi_*(R_1)) \Rightarrow \pi_*(R_0
    \otimes_{R_{01}} R_1),
  \end{equation*}
  shows that $\pi_k(R_0 \otimes_{R_{01}} R_1)$ has finite length
  as a module over $\pi_0$ for each $k > 0$, but it could be non-zero
  for infinitely many $k$.  However, any map 
  \begin{equation*}
    R_0 \otimes_{R_{01}} R_1 \to A
  \end{equation*}
  with $A \in \smallC_k$ admits a factorization over some finite  
  Postnikov truncation $\tau_{\leq n} (R_0 \otimes_{R_{01}} R_1)$, so
  the inverse system given by (cofibrant approximations to)
     $\tau_{\leq n}(R_0 \otimes_{R_{01}}
  R_1)$ will pro-represent the functor, and each truncation is Artin.
\end{proof}

The same statement holds for pro-representable functors:  
\begin{Proposition}\label{prop:functor-pullback}
  Let $\mathcal{F}_0$, $\mathcal{F}_1$, $\mathcal{F}_{01}$ be
  pro-representable functors $\smallC_k \to s\Sets$ and let $T_i:
  \mathcal{F}_i \to \mathcal{F}_{01}$ be natural transformations, $i =
  0,1$.  Then the homotopy pullback 
  \begin{equation*}
    A \mapsto \mathcal{F}_0(A) \times_{\mathcal{F}_{01}(A)}^h
    \mathcal{F}_1(A)
  \end{equation*}
  is also pro-representable.

  If all three functors are sequentially pro-representable, then so is
  the pullback.  
\end{Proposition}
\begin{proof}

  As before we may use homotopy invariance to replace $\mathcal{F}_i$
  by a functor $F_i = \colim_j \Hom(R_i^j,-)$ for $(j \mapsto R_i^j) \in
  \text{pro-}\smallC_k$ and similarly for $\mathcal{F}_{01}$.  We may also assume the representing pro-objects are \emph{nice}, and then 
 replace (by Lemma \ref{Lemma:represent-natural-trans})
 the natural transformations
  by those induced by morphisms
  \begin{equation*}
    \phi_i \in \lim_k \colim_j \Hom(R_{01}^j,R_i^k)
  \end{equation*}
  in the pro-category. 
  
  Now arrange that
  $\colim_j\Hom(R_{01}^j,A) \to \colim_j\Hom(R_1^j,A)$ is a fibration
  for all $A$, for example in the following way.  After replacing with
  isomorphic objects, we may assume that they have a common indexing
  category $J$, and that the morphism in the pro-category is given by
  a natural transformation of functors $R_{1} \Rightarrow R_{01}:
  J^\mathrm{op} \to \smallC_k$, and then functorially replacing each
  constituent morphism $R_1^j \to R_{01}^j$ by a cofibration.  (The
  pro-objects may not stay nice, but that is no longer important.)
  Since filtered colimits preserve Kan fibrations, the resulting map
  of simplicial sets
  $\colim_j\Hom(R_{01}^j,A) \to \colim_j\Hom(R_1^j,A)$ is a Kan
  fibration for all $A$.  After this replacement the pullback is weakly
  equivalent to the homotopy pullback, and we see that it is
  ``represented'' by the object $(j \mapsto R_1^j \otimes_{R_{01}^j}
  R_0^j) \in \text{pro-}\SCR_{/k}$.  As in
  Proposition~\ref{prop:tensor-of-representables}, the levels of this
  need not be objects of $\smallC_k$, but we can then instead consider
  the pro-object
  \begin{equation*}
    c \bigg(\tau_{\leq n} \big(R_1^j
    \otimes_{R_{01}^j} R_0^j \big)\bigg)_{(j,n) \in J \times \N},
  \end{equation*}
  where $c(-)$ denotes a choice of functorial cofibrant replacement.
\end{proof}

Inspired by this, we will use the following {\em ad hoc} version of a
derived tensor product in pro-$\smallC_k$.  The definition depends on
a choice of functorial factorization of morphisms in $\SCR_{/k}$ (or
just $\smallC_k$) into a cofibration followed by an acyclic fibration,
as well as a sequence of Postnikov truncation functors $\tau_{\leq
  n}$.
 \begin{Definition} \label{defn-tensor-product} Given morphisms
   $\phi: R \to R'$ and $R \to R''$ in pro-$\smallC_k$, represent them
   by natural transformations $(R_j \to R'_j)_{j \in J}$ and
   $(R_j \to R''_j)_{j \in J}$ of pro-objects indexed by the same
   category, apply functorial factorization to replace these maps by
   $R_j \to \widehat{ R'_j} \stackrel{\sim}{\rightarrow} R'_j$ and
   $R_j \to \widehat{R''_j} \stackrel{\sim}{\rightarrow} R''_j$ (where
   the first map of each is a cofibration, and the second a weak
   equivalence) and define
   \begin{equation*}
     R' \dotimes_R R'' = \mathrm{nice}\, c\left( \tau_{\leq n} (\widehat{R'_j} \otimes_{R_j}
       \widehat{R''_j}) \right)_{(j,n) \in J\times \N},
   \end{equation*} where as usual $c$ denotes the chosen cofibrant
   replacement functor, and $\mathrm{nice}$ refers to the procedure from
   Lemma \ref{lemma:replace-by-nice}, replacing a pro-object by an
   equivalent nice one.
\end{Definition}

By our discussion above, $R' \dotimes_{R} R''$ pro-represents the functor 
$$ \colim \Maps(R'_j, -) \times^h_{\colim \Maps(R_j, -)} \colim
\Maps(R''_j, -).$$
Let  us point out that the definition
above is a bit of a kludge, and the notation $R' \dotimes_R R''$ is
shorthand for a construction whose isomorphism class depends on many
choices.  In particular, it does not define a functor from pushout
diagrams in pro-$\smallC_k$, because the choice of how to represent
$R'' \leftarrow R \to R'$ by natural transformations of functors out
of the same indexing category is not functorial.  Different choices
lead to non-isomorphic pro-objects $R' \dotimes_R R''$, but all of
them are nice and represent the same functor so they are at least
related by morphisms in pro-$\smallC_k$ inducing objectwise weak
equivalences of represented functors.

\subsection{Some pro-representable functors}
\label{sec:some-pro-repr}

We discuss a few examples of functors for which we may describe
explicit pro-representing objects.

\begin{Proposition}
  The terminal functor, defined as $\mathcal{F}(A) = \{\ast\}$, is sequentially
  pro-representable.
\end{Proposition}
\begin{proof}
  Let $p = \mathrm{char}(k)$, let $F(n,A)$ be the subspace of
  $\Map(\Delta[1],A)$ consisting of 1-simplices starting at 0 and
  ending at $p^n 1$, and define $F(n,A) \to F(n+1,A)$ as $x \mapsto px$.  Since $\pi_0(A)$ is Artinian, $F(n,A)$ is non-empty
  for large enough $n$, in which case it has the homotopy type of the
  loop space of $(A,0)$.  More explicitly, if we pick $s_n \in F(n,A)$ and define $s_{n+k} = p^k s_n \in F(n+k,A)$, we may define an isomorphism of simplicial sets $\Omega (A,0) \to F(n+k,A)$ by $x \mapsto x + s_{n+k}$.  With respect to these isomorphisms, the map $F(n,A) \to F(n+1,A)$ is 
  identified with simplicial loops of
  multiplication by $p: A \to A$ and hence 
  induces
  multiplication by $p$ on homotopy groups.  Since the homotopy groups are finite and hence $p$-torsion, $\colim_n \pi_k(F(n,A)) = 0$ and hence $\colim_n F(n,A)$ is
  contractible for all $A$.

  Finally we have a natural weak equivalence $F(n,A) \simeq
  \Hom(R_n,A)$, where $R_n$ is a cofibrant approximation to
  $W(k)/p^n$  -- e.g., if $k =\mathbb{F}_p$, we may take
  $R_n$ to be the ring obtained  by freely adjoining to the discrete simplicial ring $\mathbb{Z}$
  a generator $y_1$ satisfying $d_0 y_1 =0, d_1 y_1 = p^n$. \end{proof}

We shall sometimes simply write $W(k)$ for a pro-representing object
of the terminal functor, but this is somewhat sloppy.  The constant
object of $\SCR_{/k}$ given by $W(k)$ is not cofibrant and neither is
the constant object $W(k)/p^n$.

\textbf{Remark.}  It is easy to see that the terminal functor is only
pro-representable, not representable.  Indeed, any putative
representing object $R$ would have $p^n = 0 \in \pi_0(R)$ for some $n$
and hence have $\Hom(R,A) = \emptyset \not\simeq \ast$ when $\pi_0(A)
= W(k)/p^{n+1}$.

\begin{Lemma}
  The functor $A \mapsto \mathfrak{m}(A) = \Ker(A \to k)$ is
  sequentially pro-representable.
\end{Lemma}
\begin{proof}
  Following the same strategy as in the previous lemma, we let
  $F(n,A)$ be the subspace of $\mathfrak{m} \times A^{\Delta[1]}
  \times A^{\Delta[1]}$ consisting of triples $(a,\lambda,\sigma)$
  where $\lambda$ is a path from $p^n1$ to 0 as before and $\sigma$ is
  a path from $a^n$ to 0.  There is a natural transformation $F(n,A)
  \to \mathfrak{m}$ sending $(a,\lambda,\sigma)$ to $a$ and a
  compatible one $F(n,A) \to F(n+1,A)$ sending $(a,\lambda,\sigma)
  \mapsto (a,p \lambda, a \sigma)$, inducing a weak equivalence
  $\colim_n F(n,A) \to \mathfrak{m}(A)$.  Each $F(n,-)$ is
  representable by a cofibrant approximation to $W(k)[x]/(x^n,p^n)$
  and hence $\mathfrak{m}$ is sequentially pro-representable.
\end{proof}

\begin{Lemma}
  For $n \geq 0$ the functor $A \mapsto \Omega^n \mathfrak{m}(A)$ is
  sequentially pro-representable.
\end{Lemma}
\begin{proof}
  By Proposition~\ref{prop:functor-pullback} this follows inductively
  by writing $\Omega^n\mathfrak{m}(A)$ as the homotopy pullback of
  $\ast \to \Omega^{n-1}\mathfrak{m}(A) \leftarrow \ast$.
\end{proof}

\textbf{Remark.}  A representing pro-object in the above Lemma can be
constructed more explicitly using the ``cell attachment'' construction
from section~\ref{sec:simplicial-rings}.  In fact, if $n > 0$ and
$\Z[y]$ is the simplicial commutative ring obtained by adjoining a
(trivially attached) $n$-cell to the initial object $\Z$, i.e.\ the
level-wise free commutative algebra on the pointed simplicial set
$S^n$, then this object ``represents'' in the sense that it is
cofibrant in $\SCR$ and $\SCR(\Z[y],A)$ is naturally isomorphic to
$\Omega^n A = \Omega^n \mathfrak{m}(A)$.  Of course $\Z[y] \to k$ is
not an object of $\smallC_k$ because the homotopy groups of $\Z[y]$
are too large, and in fact the functor is not representable.  But if
we write $W = W(k)$ we may instead use the pro-objects formed by
cofibrant approximations to the truncations
$\tau_{\leq n} (W/p^n)[y]$, $n \in \N$.  If as usual we write $c$ for
a chosen cofibrant replacement functor, we shall occasionally denote
the resulting pro-object $(n \mapsto c(\tau_{\leq n} (W/p^n)[y]))$
informally by $W[[y]]$, thinking of it as power series over $W$ in one
variable $y$ of degree $n$.

More generally we have the following definition
\begin{Definition}\label{defn:attach-cell-to-functor}
  Suppose $\mathcal{F}: \smallC_k \to s\Sets$ is a homotopy invariant
  functor and $e: \mathcal{F}(A) \to \Omega^n \mathfrak{m}(A)$ is a
  natural transformation.  We may then define a functor
  $\mathcal{F}': \smallC_k \to s\Sets$ by letting $\mathcal{F}'(A)$ to
  be the homotopy fiber of
  $e: \mathcal{F}(A) \to \Omega^n\mathfrak{m}(A)$ over the basepoint
  $0$, i.e. $\mathcal{F}'$ is defined by the following homotopy pullback square:
   \begin{equation}\label{eq:19}
   \begin{aligned}
     \xymatrix{ \mathcal{F}'(A) \ar[d]\ar[r] &
    0 \ar[d]\\
       \mathcal{F}(A) \ar[r]_-e &  \Omega^n \mathfrak{m}(A).  }
   \end{aligned}
 \end{equation}
  
    In this case we shall say that $\mathcal{F}'$ is obtained from
  $\mathcal{F}$ by attaching an $(n+1)$-cell to $\mathcal{F}$ along
  $e$.  For $n = -1$ we shall simply define $\mathcal{F}'(A) =
  \mathcal{F} \times \mathfrak{m}(A)$ and say that $\mathcal{F}'$ is
  obtained from $\mathcal{F}$ by attaching a 0-cell.
\end{Definition}

If $e: \mathcal{F} \to \Omega^n\mathfrak{m}$ is a natural
transformation and $\mathcal{F}'$ is defined as above, then if
$\mathcal{F}$ is (pro-)representable by $R \in \text{pro-}\smallC_k$,
the functor $\mathcal{F}'$ is (pro-)representable by the derived
tensor product $W \dotimes_{W[[x]]} R$, where $W[[x]] \to R$ is a
morphism in pro-$\smallC_k$ classifying
$[e] \in \pi_n\mathfrak{m}(R)$.

\subsection{Formally cohesive functors}  
\label{sec:tangent-complex}

In 
section~\ref{sec:luri-deriv-schl} below, we shall review a
verifiable criterion for a homotopy invariant functor
$\mathcal{F}: \smallC_k \to s\Sets$ to be pro-representable.  The
criterion, due to Lurie, is a simplicial version of Schlessinger's
criterion.  The criterion makes two assumptions on $\mathcal{F}$, one
of which we discuss in this section.

Any functor $\mathcal{F}: \smallC_k \to s\Sets$ will send a
commutative square
\begin{equation}
 \label{eq:10}
 \begin{aligned}
 \xymatrix{
    A \ar[r]\ar[d] & B\ar[d]\\C \ar[r] & D
  }
 \end{aligned}
\end{equation}
in $\smallC_k$ to a commutative square
\begin{equation}\label{eq:11}
  \begin{aligned}
    \xymatrix{ \mathcal{F}(A) \ar[r]\ar[d] &
      \mathcal{F}(B)\ar[d]\\\mathcal{F}(C) \ar[r] & \mathcal{F}(D) }
  \end{aligned}
\end{equation}
of simplicial sets, by definition of ``functor''.  Recall that a
strictly commutative square is said to be \emph{homotopy cartesian} if
the induced map from the initial vertex to the homotopy pullback of
the rest of the diagram is a weak equivalence.  We shall be
particularly interested in homotopy invariant functors satisfying the
following properties.
\begin{Definition}  Let $\mathcal{F}: \smallC_k \to s\Sets$ be a
  homotopy invariant functor.
  \begin{enumerate}[(i)]
  \item $\mathcal{F}$ \emph{preserves homotopy pullback} if the
    square~\eqref{eq:11} is homotopy cartesian whenever the
    diagram~\eqref{eq:10} is (strictly) cartesian and $C \to D$ and $B
    \to D$ are surjective in each simplicial degree.
  \item\ [Lurie] $\mathcal{F}$ is \emph{formally cohesive} (or
    ``good'') if it preserves homotopy pullback and if
    $\mathcal{F}(k)$ is contractible.
  \end{enumerate}
\end{Definition}
Let us point out that a map of simplicial rings is surjective in all
simplicial degrees if and only if it is a fibration and induces a
surjection in $\pi_0$.
\begin{Lemma}
  If a homotopy invariant functor $\mathcal{F}: \smallC_k \to s\Sets$
  preserves homotopy pullback in the sense above, it also preserves
  homotopy pullback in the stronger sense that the
  square~\eqref{eq:11} is homotopy cartesian whenever the square~\eqref{eq:10} is
  (strictly) cartesian and \emph{either} $C \to D$ or $B \to D$ are
  surjections in all simplicial degrees.
\end{Lemma}
\begin{proof}
  See Lemma 6.2.7 in Lurie's thesis \cite{LurieThesis}. 
\end{proof}
\begin{Example}\label{ex:functors-represented-by-general-rings}
  If $R \in \smallC_k$ is any object, the functor $\mathcal{F}_R =
  \Hom(R,-)$ may not be homotopy invariant, but if $R$ is cofibrant it
  will be (Lemma   \ref{simpilcially enriched representable Lemma}). In that case the functor $\mathcal{F}_R$ will
  preserve homotopy pullback (because it preserves actual pullbacks and also Kan
    fibrations) but will not necessarily have $\mathcal{F}_R(k)$
  contractible.

  If $\overline \rho: R \to k$ is any zero-simplex of
  $\mathcal{F}_R(k)$, we may obtain a formally cohesive functor
  $\mathcal{F}_{R,\overline \rho}$ which takes $\epsilon: A \to k$ to
  the homotopy fiber of $\Hom(R,A) \to \Hom(R,k)$ over
  $\overline{\rho}$.

  More generally, if $\mathcal{F}: \smallC_k \to s\Sets$ is homotopy
  invariant and preserves homotopy pullbacks and $\overline{\rho} \in
  \mathcal{F}(k)$ is a zero-simplex, then the functor
  $\mathcal{F}_{\overline{\rho}}$ which takes $(A \to k) \in
  \smallC_k$ to the homotopy fiber of $\mathcal{F}(A) \to
  \mathcal{F}(k)$ over $\overline\rho$ is formally cohesive.
\end{Example}

For formally cohesive functors we may check whether a natural
transformation $T: \mathcal{F} \to \mathcal{G}$ is a natural weak
equivalence without checking on all $A$.
\begin{Lemma} \label{lem:tangent-complex-detects-equivalence}
  Let $T: \mathcal{F} \to \mathcal{G}$ be a natural transformation
  between formally cohesive functors.  Then $T$ is a natural weak
  equivalence if and only if $T: \mathcal{F}(k \oplus k[n]) \to
  \mathcal{G}(k \oplus k[n])$ is a weak equivalence for large $n$.
\end{Lemma}
\begin{proof}
  By the pullback
  \begin{equation*}
    \xymatrix{ k \oplus k[n] \ar[r]\ar[d] & k \oplus \widetilde{k[n+1]}
 \ar[d]\\
      k \ar[r] & k \oplus k[n+1]
    }
  \end{equation*}
  we see that $T$ induces an equivalence on $A = k \oplus k[n]$ for
  all $n \geq 0$, provided it does so for large $n$.  The case of
  general $A$ then follows from Lemma~\ref{lem:build-Artin-rings-Postnikov}.
\end{proof}

\subsection{Homotopy categories}
\label{sec:hocat}
To conclude this section, let us discuss the \emph{homotopy category} of $\smallC_k$ and
its relation to functors $\smallC_k \to s\Sets$.  
This discussion is not strictly necessary, but it may be helpful to orient the reader. 

The homotopy
category is the (non-simplicial) category $\mathrm{Ho}(\smallC_k)$
whose objects are the objects of $\smallC_k$, but whose morphism sets
are given by
\begin{equation*}
  \mathrm{Ho}(\smallC_k)(A,B) = \pi_0 \smallC_k(c(A),B) = \pi_0
  \smallC_k(c(A),c(B))
\end{equation*}
where $c: \SCR_{/k} \to \SCR_{/k}$ is some choice of cofibrant
approximation.  Up to canonical isomorphism of categories,
$\mathrm{Ho}(\smallC_k)$ does not depend on the choice of $c$.  The
canonical functor $\smallC_k \to \mathrm{Ho}(\smallC_k)$ sends weak
equivalences to isomorphisms and is universal with that property.  It
also has the property that two objects $A,B\in \smallC_k$ are weakly
equivalent (i.e.\ there is a zig-zag of weak equivalences connecting
them) if and only if their image in $\mathrm{Ho}(\smallC_k)$ are
isomorphic.

For any homotopy invariant functor
$\mathcal{F}: \smallC_k \to s\Sets$, there is an associated functor
\begin{align*}
  \mathrm{Ho}(\smallC_k) & \to \Sets\\
  A  & \mapsto \pi_0(\mathcal{F}(A)),
\end{align*}
which we shall denote $\pi_0\mathcal{F}$.  Of course the passage from
$\mathcal{F}$ to $\pi_0\mathcal{F}$ loses much information in general,
but for formally cohesive functors we have the following result.
\begin{Lemma}\label{lem:pi-zero}
  Let $\mathcal{F} \to
  \mathcal{G}$ be a natural transformation of formally cohesive
  functors $\smallC_k \to s\Sets$.  Assume that $\pi_0\mathcal{F}(A)
  \to \pi_0 \mathcal{G}(A)$ is a bijection for all $A \in \smallC_k$.
  Then $\mathcal{F}(A) \to \mathcal{G}(A)$ is a weak equivalence for
  all $A \in \smallC_k$.
\end{Lemma}
\begin{proof}
  From the homotopy pullback square in the proof of the previous lemma
  we obtain a natural weak equivalence $\mathcal{F}(k \oplus k[n])
  \simeq \Omega \mathcal{F}(k \oplus k[n+1])$.  Hence $\pi_i
  \mathcal{F}(k \oplus k[n]) = \pi_0 \mathcal{F}(k \oplus k[n+i])$,
  and similarly for $\mathcal{G}$.  It follows that $T: \mathcal{F}(A)
  \to \mathcal{G}(A)$ induces an isomorphism in all homotopy groups
  for $A = k \oplus k[n]$ for all $n$, and hence by the previous Lemma
  for all $A \in \smallC_k$.
\end{proof}
\begin{Corollary}\label{cor:homotopy-representability}
  A formally cohesive functor $\mathcal{F}: \smallC_k \to s\Sets$ is
  representable if and only if
  $\pi_0\mathcal{F}: \mathrm{Ho}(\smallC_k) \to \Sets$ is
  representable.
\end{Corollary}
\begin{proof}
  Let $\mathcal{F}$ be a formally cohesive functor and suppose   $\pi_0\mathcal{F}$ is representable.  Without loss of generality we
  may assume $\mathcal{F}$ is simplicially enriched.  Since
  $\pi_0\mathcal{F}$ is representable we may pick an object
  $R \in \smallC_k$ and
  $\iota_0 \in \pi_0 \mathcal{F}(R) = \pi_0\mathcal{F}(c(R))$ such that
  the corresponding natural transformation
  $\mathrm{Ho}(\smallC_k)(R,-) \to \pi_0\mathcal{F}$ is a natural
  isomorphism.  Now any choice of zero-simplex
  $\iota \in \mathcal{F}(c(R))$ in the path component $\iota_0$ gives
  rise to a natural transformation $\smallC_k(c(R),-) \to \mathcal{F}$
  by the simplicial enrichment, and Lemma~\ref{lem:pi-zero} shows that
  it is a natural weak equivalence.

  The other direction is clear.
\end{proof}

Since any representable functor $\mathcal{F}: \smallC_k \to s\Sets$ is
automatically formally cohesive, we see that the question of
representability splits into two: whether the functor is formally
cohesive, and whether
$\pi_0\mathcal{F}: \mathrm{Ho}(\smallC_k) \to \Sets$ is representable
in the usual sense.   However, it is often \emph{easier} to work with
$\mathcal{F}: \smallC_k \to s\Sets$ than
$\pi_0\mathcal{F}: \mathrm{Ho}(\smallC_k) \to \Sets$, even when it's
known a priori that $\mathcal{F}$ is formally cohesive.

There is a parallel discussion for pro-objects and
pro-representability.  The functor
$\smallC_k \to \mathrm{Ho}(\smallC_k)$ induces a functor
pro-$\smallC_k$ to pro-$\mathrm{Ho}(\smallC_k)$.  Of course some
information is again lost in this process, but the next lemma shows
that for $R \in \text{pro-}\smallC_k$ we may still recover the functor
$\pi_0\mathcal{F}_R: \smallC_k$ from the image of $R$ in
pro-$\mathrm{Ho}(\smallC_k)$.
\begin{Lemma}\label{lem:passage-to-pro-Ho}
  Let $I$ be a filtered category, let
  $R = (i \mapsto R_i) \in \text{pro-}\smallC_k$, and let
  $\mathcal{F}_R: \smallC_k \to s\Sets$ be the functor pro-represented
  by $R$.  Then the natural maps
  \begin{equation*}
    \smallC_k(R_i,A) \to \pi_0\smallC_k(R_i),A) \to \colim_i \pi_0
    \smallC_k(R_i, A) = (\text{pro-}\mathrm{Ho}(\smallC_k))(R,A)
  \end{equation*}
  induce a natural transformation
  \begin{equation*}
    \pi_0\mathcal{F}_R \to (\text{pro-}\mathrm{Ho}(\smallC_k))(R,-)
  \end{equation*}
  between functors $\mathrm{Ho}(\smallC_k) \to \Sets$.  It is a
  natural isomorphism if each $R_i$ is cofibrant.
\end{Lemma}
\begin{proof}
  This is just the fact that $\pi_0$ takes filtered colimits (not just
  homotopy colimits) of simplicial sets to colimits of sets.
\end{proof}

Again
 we advise the reader that it is better to study
$\mathcal{F}: \smallC_k \to s\Sets$ directly than to ``reduce'' to
$\pi_0 \mathcal{F}$.  For example, it is likely easier to prove that
$\mathcal{F}$ is (pro-) representable in the homotopical sense than to
prove directly that $\pi_0\mathcal{F}$ is (pro-) representable in the
ordinary categorical sense.

A useful corollary to our prior discussions is that one can meaningfully talk about the homotopy groups of a 
representing ring for a functor:
 Suppose that $R: I \rightarrow \Art_k$ and $R': J \rightarrow \Art_k$ are levelwise cofibrant pro-objects of $\Art_k$,
and the functors  $\mathcal{F}_{R}, \mathcal{F}_{R'}$ that they represent are   naturally weakly equivalent. 
This natural weak equivalence induces, by Lemma \ref{lem:passage-to-pro-Ho},  an equivalence $\pi_0 \mathcal{F}_R \simeq \pi_0 \mathcal{F}_{R'}$ of  functors 
$\Ho(\Art_k) \rightarrow \Sets$; thus we have an induced isomorphism
between the images of $R$ and $R'$ in the pro-category of $\Ho(\Art_k)$. 
 In particular,  we obtain an isomorphism
\begin{equation} \label{pistar iso} (i \mapsto \pi_* R_i) \cong  (j \mapsto  \pi_* R'_j)\end{equation}
 of pro-graded rings.

 In the later parts of this paper, we will often use the
 ``naive'' definition
\begin{equation} \label{naive pi} \pi_* R = \lim_i \pi_* R_i\end{equation}
 for the homotopy groups of an object of pro$\Art_k$; one does not expect
 this definition to be well-behaved in general, but in the context we will work,
 all the $\pi_* R_i$ are finite, and this definition has reasonable formal properties. (In general, it seems  more reasonable to either work with $\pi_* R$ as a pro-object in graded rings, or to remember the topology on the limit.) The above discussion shows
 that, at least, $\pi_* R$ is determined, up to a unique isomorphism, by the natural weak equivalence class of $\mathcal{F}_R$.

\section{Tangent complexes of rings and functors}
\label{sec:tang-compl-funct}

Lurie's \emph{derived Schlessinger criterion} will play an important
role in this paper.  It is an (in practice often verifiable) criterion
on a functor $\mathcal{F}: \smallC_k \to s\Sets$, guaranteeing that it
is pro-representable.  We shall state it in a form suitable for our
applications and outline a proof, following Lurie's.  First we must
recall the \emph{tangent complex} $\bigtangent R$ of an object $R \in
\smallC_k$ and more generally the tangent complex $\bigtangent
\mathcal{F}$ of a functor $\mathcal{F}: \smallC_k \to s\Sets$.

\subsection{Cohomology of (pro-) Artin rings}

There are two useful ways to associate a graded $k$-vector space to an
object $R \in \smallC_k$, one behaving like ``cohomology'' of $R$ and one behaving like ``dualized homotopy groups''.  In
fact, we shall define analogues of \emph{relative} ``cohomology/cohomotopy'' for a
morphism $R \to R'$ in $\smallC_k$.

We 
begin with the \emph{relative tangent complex} which we shall study in
much more detail in the following sections.  Let $R \rightarrow k$ be
a cofibrant object of $\smallC_k$ and consider the simplicial set
$\smallC_k(R,k \oplus k[n])$ of homomorphisms $R \to k \oplus k[n]$
lifting the given homomorphism $R \to k$.  We shall be interested in
the set of homotopy classes of such homomorphisms, i.e.\ the set
$\pi_0 \smallC_k(R,k \oplus k[n])$.  This set is canonically a
$k$-vector space, and in fact $\smallC_k(R,k \oplus k[n])$ is
canonically a simplicial $k$-vector spaces; indeed, we have a natural
isomorphism of simplicial sets
\begin{equation} \label{AQ1}
  \smallC_k(R,k \oplus k[n]) \cong R\text{-Mod}(\Omega_{R/\Z},k[n]),
\end{equation}
where $\Omega_{R/\Z}$ is the simplicial $R$-module whose $p$-simplices
are $\Omega^1_{R_p/\Z}$, the K\"ahler differentials of $\Z \to R_p$,
$R$-Mod denotes the category of simplicial $R$-modules, simplicially
enriched in the usual way, and the simplicial $k$-module $k[n]$ is
made into a simplicial $R$-module using the given augmentation
$R \to k$.  (As usual we use the shorthand
$R = (R \to k) \in \smallC_k$ when typographically convenient, but of
course the definition depends on the homomorphism $R \to k$.)
\begin{Definition}
  For a cofibrant object $R \to k$ of $\smallC_k$, let us write
  $\pi_{-n} \bigtangent R$ for the $k$-vector space
  $\pi_0 \smallC_k(R,k \oplus k[n])$.  For a cofibration $R \to R'$
  between cofibrant objects of $\smallC_k$ define for $n \geq 0$ a
  $k$-modules $\pi_{-n}\bigtangent(R',R)$ as $\pi_0$ of the fiber of
  the fibration
  \begin{equation*}
    \smallC_k(R',k \oplus k[n]) \to \smallC_k(R,k \oplus k[n])
  \end{equation*}
  over the point given by the composition $R \to k \to k \oplus
  k[n]$.  (By a similar reasoning as before, this fiber is a
  simplicial $k$-module.)

  For general objects $R \in \smallC_k$ we define
  $\pi_{-n}\bigtangent R$ by first taking cofibrant approximation.
  For a pro-object
  $R = (i \mapsto R_i)_{i \in I} \in \text{pro-}\smallC_k$ we define
  $\pi_{-n}\bigtangent R$ as the colimit of $\pi_{-n}\bigtangent R_i$.
  In the relative case we similarly define
  $\pi_{-n} \bigtangent(R',R)$ for an arbitrary morphism $R \to R'$ in
  pro-$\smallC_k$.
\end{Definition}
By the discussion above the definition, the $k$-vector space
$\pi_{-n}\bigtangent R$ is identified with the Andr\'e--Quillen
cohomology of $\Z \to R$ with coefficients in the $R$-module $k$.  We
shall later explain in what sense $\pi_{-n} \bigtangent R$ is
$\pi_{-n}$ of an object $\bigtangent R$.

The functor $\pi_{-*} \bigtangent$ can be regarded as ``cohomology'' of the object
or pro-object $R$  (or the  ``relative cohomology'' of $R \rightarrow R'$) and manifestly depends only on the map of functors represented by $R \to
R'$.  These groups fit in a long exact sequence
\begin{equation} \label{t-long-exact-sequence} 
\dots \to \pi_{-n} \mathfrak{t}(R', R)  \rightarrow \pi_{-n} \mathfrak{t} R'
\rightarrow \pi_{-n} \mathfrak{t} R \rightarrow \pi_{-n-1}
\mathfrak{t}(R', R) \rightarrow \cdots.
\end{equation}

\begin{Definition}
  Let $R \to R'$ be a morphism in $\smallC_k$, assume the underlying
  map of simplicial abelian groups is levelwise injective and let   $R'/R$ be cokernel, calculated levelwise in simplicial abelian
  groups.  Then $R' \to R'/R$ is a Kan fibration with fiber $R$, and
  $\pi_*(R'/R)$ is a graded module over $\pi_*(R)$ and in particular
  over $\pi_0(R)$.  In this case, define
  \begin{equation*}
    \overline\pi^n(R',R) = \Hom_{\pi_0 R}(\pi_n(R'/R),k)
  \end{equation*}
  For a general morphism $R \to R'$ in $\smallC_k$ we first replace
  $R \to R'$ by a cofibration of $R$-modules (or even just simplicial
  abelian groups).

  For a morphism $R \to R'$ in pro-$\smallC_k$ we define
  $\overline\pi^n(R',R)$ as the colimit of the level-wise
  $\overline\pi^n$.
\end{Definition}

The functor $\overline \pi^n$ manifestly depends only on the underlying simplicial abelian groups.
While perhaps a little bit easier to define than $\pi_{-n} \bigtangent
R$, it seems to be less conceptually important, appearing only in a
few technical proofs. 

The canonical isomorphism $\pi_n(k \oplus k[n]) \to \pi_n(k \oplus k[n],k) \cong k$ gives an
element $\iota \in \overline\pi^n(k \oplus k[n],k)$ and hence a
canonical natural transformation
\begin{equation}\label{eq:18}
  \pi_{-n}\bigtangent(R',R) \to \overline\pi^{n}(R',R)
\end{equation}
which could perhaps be thought of as an analogue of the homomorphism
$H^n(X,Y;k) \to \Hom_{\pi_1(Y,y)}(\pi_n(X,Y,y),k)$ dual to the usual
Hurewicz homomorphism.

The category $\smallC_k$ also has an analogue of the
Hurewicz theorem.
\begin{Proposition}\label{Prop:hurewicz}
  Let $R \to R'$ be a morphism in pro-$\smallC_k$.
  \begin{enumerate}[(i)]
  \item For $n=0$ the homomorphism~\eqref{eq:18} is always injective,
    with image $(m/m^2)^\vee \subset m^\vee = \overline\pi^0(R',R)$,
    where $m \subset \pi_0(R' \otimes_R k)$ is the maximal ideal when
    $R \to R'$ is a morphism in $\smallC_k$, and $m^\vee$ and
    $(m/m^2)^\vee$ are defined as the colimit when $R \to R'$ is a
    morphism in pro-$\smallC_k$.
  \item If $n \geq 1$ and $\overline\pi^l(R',R) = 0$ for $l < n$, then
    the homomorphism~\eqref{eq:18} is an isomorphism.
  \item For $n \geq 0$, we have $\overline\pi^l(R',R) = 0$ for all $l \leq n$
    if and only if $\pi_{-l} \bigtangent(R',R) = 0$ for all $l \leq
    n$.
  \end{enumerate}
\end{Proposition}

 {\bf Example.}   Suppose that $R, R' \in \Art_k$ (just for simplicity in discussing homotopy groups). 
If the map $\pi_{-k} \mathfrak{t}(R') \rightarrow \pi_{-k} \mathfrak{t}(R)$
is an isomorphism for $k=0$ and an injection for $k=1$,   then
\begin{equation} \label{pi0iso}  \pi_0 R \stackrel{\sim}{\longrightarrow} \pi_0 R'\end{equation}
 must be an isomorphism. This is a well-known statement in deformation theory, cf. e.g. \cite[Theorem 2.4]{GoldmanMillson}. 
 
To deduce this,  observe that
\eqref{t-long-exact-sequence} implies that 
$\pi_{-k} \mathfrak{t}(R', R)$ is vanishing for  $k \leq 1$; the same conclusion
then holds for $\overline \pi^k$. That means that in fact
$\pi_k(R'/R)$ is vanishing for $k=0, 1$, by Nakayama's lemma.
Then \eqref{pi0iso} follows from the long exact sequence in homotopy.

\begin{proof}
  Statement (i) reduces to a well known property of the map $\pi_0(R)
  \to \pi_0(R')$ of discrete Artin rings, and the case of pro-objects
  follows by taking colimit. 

  To prove statement (ii) we first reduce to the case $R=k$ by proving
  that 
  \begin{equation} \label{reduced-pi-iso} \overline\pi^n(R',R) \to \overline\pi^n(R' \otimes_R k,k)\end{equation} is
  an isomorphism when $R \to R'$ is a morphism with
  $\overline\pi^l(R',R) = 0$ for $l < n$.  (The corresponding result for $\mathfrak{t}(R, R')$ is straightforward.)  
    
   We first consider  \eqref{reduced-pi-iso} in the case
  where $R$ and $R'$ are in $\smallC_k$.  
  To this end we first reduce to the case where $k \subset \pi_0(R) = \pi_0(R')$, since both
  groups are unchanged by tensoring both $R$ and $R'$ over $\Z$ with a
  cofibrant approximation to $\Z/p$.  We have the usual spectral
  sequence \cite[II, \S 6]{QuillenHomotopicalAlgebra}
  \begin{equation}\label{eq:20}
    \Tor_{\pi_*(R)}(\pi_*(R'/R),k) \Rightarrow \pi_*((R'/R) \otimes_R
    k) = \pi_*((R' \otimes_R k)/k),
  \end{equation}
  from which it is not hard to deduce that the element of lowest
  degree in the target arises from $\Tor^0_{\pi_0(R)}(-,k)$ applied to
  the lowest degree non-zero group in $\pi_*(R'/R)$.  Dualizing this
  and using the tensor--Hom adjunction gives the claimed result 
  \eqref{reduced-pi-iso}. 

  Next we consider the case of a morphism $R \to R'$ in
  pro-$\smallC_k$: again we claim that if $\overline\pi^l(R',R) = 0$
  for $l < n$, then \eqref{reduced-pi-iso} holds.  The spectral
  sequence~\eqref{eq:20} has no direct analogue in the case of
  pro-objects, but it may be dualized to a more well behaved object as
  follows.
  In the case $R, R' \in \smallC_k$ and $\pi_0(R) = \pi_0(R') = k$,
  the spectral sequence~\eqref{eq:20} is a spectral sequence of finite
  dimensional $k$-vector spaces, and hence may be formally dualized to
  a spectral sequence
  \begin{equation*}
    \Cotor_{\pi_*(R)^\vee}(\pi_*(R'/R)^\vee,k) \Rightarrow
    \pi^*(R'\otimes_R k,k).
  \end{equation*}
  (The bidegrees and differentials in this spectral sequence are as in
  the cohomology Serre spectral sequence.)  For a morphism $R \to R'$
  in pro-$\smallC_k$ we have such a spectral sequence levelwise, and
  may form the direct limit and obtain a spectral sequence, since
  filtered colimit is an exact functor.  Here
  $\colim_j \pi_*(R_j)^\vee$ is a coalgebra in $k$-vector spaces and
  $\colim_j \pi_*(R'_j/R_j)^\vee$ is a comodule, when $\pi_0(R_j) = k$
  for all $j$.  The functor $\Cotor$ is the derived functor of
  cotensor product, and is calculated by a cochain complex formally
  dual to the ``bar complex'' calculating $\Tor$.  Again the lowest
  degree element arises from the actual cotensor product (not the
  higher derived ones) of $\pi_n(R'/R)^\vee$ for the smallest possible
  $n$ for which that group is non-zero.  In that bidegree, the
  $\Cotor$ is the direct limit of the $k$-vector spaces
  \begin{equation*}
    \Cotor_{\pi_0(R_j)^\vee}(\pi_n(R'_j/R_j)^\vee,k) =
    \Hom_{\pi_0(R_j)}(\pi_n(R'_j/R_j),k),
  \end{equation*}
  and this colimit is precisely $\overline\pi^n(R',R)$.

  We have now reduced (ii) to the case $R = k$, i.e.\ $R'$ is a $k$-algebra.   For $R' \in 
  \smallC_k$ it   now  follows    that $\tau_{\leq n} R' \simeq k \oplus
  V^\vee[n]$ for $V = \overline\pi_n(R',k) =
  \pi_{-n}\bigtangent(R',k)$, as in the proof of Lemma~\ref{lem:build-Artin-rings-Postnikov}.
    
  However, the pro-case is not so easy.  We  shall use the truncations $\tau_{\geq n} R'$ defined for
  a simplicial $k$-algebra $R'$ as the (homotopy) pullback of
  $k \to \tau_{\leq n-1} R' \leftarrow R'$.

  The main ingredient in the proof of (ii) for $R = k$ for a
  pro-object $R'$, is that if $R' = (j \mapsto R'_j)_{j \in J}$, then
  for all $i \in J$ there exists $j > i$ such that $R'_j \to R'_i$ is
  homotopic to a ring map which factors through
    $\tau_{ \geq n} R'_i \to R'_i$.  This main
  ingredient is proved by inductively factoring it as $R_{j_l}' \to
  \tau_{\geq l} R'_i \to R'_i$, with $i = j_0 < j_1 < \dots < j_n =
  j$.  These truncations from below fit in homotopy pullback squares
  \begin{equation*}
    \xymatrix{
      \tau_{\geq l}R'_i \ar[r] \ar[d] & k\ar[d]\\
      \tau_{\geq l-1}R'_i \ar[r] & k \oplus V[l-1]
    }
  \end{equation*}
  where $V = \pi_{l-1}R'_i$, so to lift from $R'_{j_{l-1}} \to
  \tau_{\geq l-1} R'_i$ to $R'_{j_{l}} \to \tau_{\geq l} R'_i$
  amounts %
  to finding $j_l$ large
  enough such that $R_{j_l} \to k \oplus V[l-1]$ is a null homotopic
  map of simplicial commutative rings.  But this is possible, since
  otherwise it would represent a non-trivial element in $\pi_{-(l-1)}
  \bigtangent(R',k)$.  By induction this group is zero for $l \leq n$
  and we obtain the factorization $R_j \to \tau_{\geq n} R'_i$.

  Having proved this main ingredient, we prove the induction step in
  the statement of (ii) as follows.  For surjectivity, we wish to show
  that any homomorphism $\phi : \pi_n(R_i') \to k$ arises from some
  homomorphism $R'_i \to k \oplus k[n]$, after possibly increasing
  $i$.  It is clear that $\phi$ is induced by a homomorphism
  $\tau_{\geq n} R'_i \to k \oplus k[n]$, since
  $\tau_{\geq n} R'_i \to \tau_{\leq n} \tau_{\geq n} R'_i \simeq k
  \oplus V[n]$
  with $V = \pi_n(R'_i)$, where this last weak equivalence follows
  since the objects $k \oplus V[n]$ are  ``characterized by their homotopy groups'' as in
  the proof of Lemma~\ref{lem:build-Artin-rings-Postnikov}.

  But then any lift $R'_j \to \tau_{\geq n} R'_i \to k \oplus k[n]$
  represents $\phi$.  Injectivity is similar: if
  $\phi,\psi: R'_i \to k \oplus k[n]$ induce the same map in $\pi_n$,
  argue that their restrictions to
  $\tau_{\geq n} R'_i \to k \oplus k[n]$ are weakly homotopic, i.e.\
  in the same path component of the derived mapping space of
  simplicial $k$-algebras (we point out that this is not quite the
  same as the derived space of morphisms in $\smallC_k$).  Then use the
  factorization. 

  Statement (iii) is proved by induction on $n$.  For $n=0$ the ``if''
  part is immediate from the injectivity in (i), and the ``only if''
  part follows from Nakayama's lemma.  For $n>0$ the induction step
  follows from (ii).
\end{proof}

\textbf{Remark.}   Just as 
in the discussion following Lemma~\ref{lem:build-Artin-rings-Postnikov},  
there is a topological analogue involving   $p$-finite spaces. 
For any inclusion $X
\to Y$ of such spaces we have a Hurewicz homomorphism
\begin{equation*}
  H^n(Y,X;k) \to \Hom_{k[\pi_1(X)]}(\pi_n(Y,X),k).
\end{equation*}
It is true that this homomorphism is an isomorphism in the first
degree in which the target is non-zero (even when the spaces are not
simply connected).  It follows that $H^*(Y,X;k) = 0$ if and only if
$(\pi_*(Y,X))^\vee = 0$.

Finally, the following Corollary is an immediate consequence of  \eqref{t-long-exact-sequence} and the Proposition:
\begin{Corollary} \label{Corollary:compare}
Let $R$ be an object of pro-$\Art_k$, and let $\pi_0 R$ be the object of pro-$\Art_k$
obtained by applying $\pi_0$ level-wise to $R$. 
Then one has $\pi_0 \bigtangent(\pi_0 R) = \pi_0 \bigtangent(R)$ and
an exact sequence
\begin{equation} \label{pi0RRcompare}\pi_{-1}  \mathfrak{t} (\pi_0 \mathcal{R})  \stackrel{\iota}{\hookrightarrow } \pi_{-1}( \mathfrak{t}  R) \rightarrow  \Hom_{\pi_0 R}(\pi_1 R, k) \rightarrow \pi_{-2} \mathfrak{t} (\pi_0 R)\rightarrow
\pi_{-2} (\mathfrak{t} R) \rightarrow \cdots  \end{equation}
 \end{Corollary}

\subsection{Cell structures on pro-rings}
\label{sec:cell-structures-pro}

For simply connected topological space, there always exists a CW
approximation with the smallest number of cells consistent with its
integral singular homology.  The easy direction, which does not use
simple connectivity, is that the number of cells is \emph{at least}
that much: this follows from a calculation of $H_*(D^n,\partial D^n)$.
The other direction is more difficult and requires simple connectivity
and the Hurewicz theorem: if $X \to Y$ is $(n-1)$-connected, the
Hurewicz theorem allows us to lift generators of $H_n(Y,X)$ to maps
from $(D^n,\partial D^n)$ along which we may attach more cells to $X$
in order to make the map $n$-connected.

For morphisms $R \to R'$ in pro-$\smallC_k$ the relative tangent  complex plays a similar role with respect to the cell attachments  described in Definition \ref{defn:attach-cell-to-functor}.  Indeed, if $R'$ is
obtained from $R$ by attaching a single $n$-cell, then it follows from
the pullback square that $\pi_{-*}\bigtangent(R',R)$ is
one-dimensional in degree $*=n$ and vanishing otherwise.  It follows
that $\Hom(R,-)$ cannot be obtained from $\Hom(R',-)$ by fewer than
$\dim_k(\pi_{-*}\bigtangent(R',R))$ cell attachments.  In this section
we shall explain why this minimal number of cells can always be
realized.  In analogy with what happens for spaces, the main
ingredient is the Hurewicz type result in Proposition~\ref{Prop:hurewicz}.

Indeed, if $R \to R'$ is a morphism in pro-$\smallC_k$, and
$\pi_{-l}\bigtangent(R,R')$ vanishes for $l < n$ and is finite
dimensional for $l=n$, then the Hurewicz theorem proved above implies
that the groups $\overline\pi^l\bigtangent(R',R)$ also vanish for $l<
n$ lets us lift elements of a dual basis for
$\pi_{-n}\bigtangent(R',R)$ to maps $(\Delta^n, \partial \Delta^n) \to
(R',R)$.  If $R''$ is the pro-ring obtained by attaching $n$-cells to
$R$ along these finitely many maps, we obtain a factorization $R \to
R'' \to R'$, where $\pi_{-*} \bigtangent(R'',R)$ is concentrated in
degree $n$ and $\pi_{-n}\bigtangent(R',R) \to
\pi_{-n}\bigtangent(R'',R)$ is an isomorphism.  Hence by the long
exact sequence, $\pi_{-*}\bigtangent(R',R'')$ vanishes for $* \leq n$
and $\pi_{-*}\bigtangent(R',R'') \to \pi_{-*}\bigtangent(R',R)$ is an
isomorphism for $* > n$.  By induction on $n$, this proves the following result.
\begin{Corollary}\label{cor:Corollary hurewicz analogue} 
  Let $R \to R'$ be a morphism in pro-$\smallC_k$.  If
  $\pi_*\bigtangent(R',R)$ is finite dimensional, then $R'$ is
  obtained from $R$ by finitely many cell attachments, with one
  $n$-cell for each element in a dual basis for
  $\pi_{-n}\bigtangent(R',R)$.  In other words, the functor
  $\Hom(R,-)$ is naturally weakly equivalent to a functor obtained
  from the functor $\Hom(R',-)$ by attaching cells (Definition \ref{defn:attach-cell-to-functor}) precisely $\dim_k(\pi_{-*}\bigtangent(R',R))$
  times.

  In particular, if $R \in \text{pro-}\smallC_k$ is an object with $N
  = \dim_k(\pi_{-*}\bigtangent(R,W)) < \infty$, then $\Hom(R,-)$ is
  naturally weakly equivalent to a functor obtained from the terminal
  functor by precisely $N$  cell attachments (Definition \ref{defn:attach-cell-to-functor}). 
  
  If $\mathcal{F}: \smallC_k \to s\Sets$ is pro-representable, then it
  is sequentially pro-representable if and only if
  $\pi_*\bigtangent(R',R)$ is countable. 
\end{Corollary}
\begin{proof}[Proof sketch]
  We already explained how to achieve a cell structure, so it remains
  to discuss countability.  Adjoining a single cell, or a countable
  number of cells, does not change whether or not the indexing set of
  the pro-object may be chosen countable, by our constructions.
\end{proof}

\subsection{The tangent complex of a formally cohesive functor} \label{tangent complex setup}

In this section, we will define the {\em tangent complex} of a
formally cohesive functor
$\mathcal{F}: \smallC_k \to s\Sets$. It will be  a chain complex  
$\bigtangent \mathcal{F}$, possibly unbounded in both directions.  It
generalizes the previously defined $\pi_{-n}\bigtangent(R',R)$ for a
morphism $R \to R'$ in $\smallC_k$ in the sense that
$\pi_{-n} \bigtangent(R',R)$ is the homotopy groups of the mapping
cone of $\bigtangent \Hom(R',-) \to \bigtangent \Hom(R,-)$.  Whenever
we speak of a ``chain complex'' we shall always mean one with
degree-decreasing boundary map, and vice versa for cochain complexes.

It seems difficult to directly define a
chain complex $\bigtangent \mathcal{F}$ from a formally cohesive
functor $\mathcal{F}$.  Instead, we construct an essentially
equivalent incarnation of it, as a \emph{spectrum} with the structure
of a module spectrum over the Eilenberg--MacLane spectrum $Hk$.  We first review the relationship between spectra and chain complexes.

\subsubsection{The Dold--Kan correspondence and spectra}
\label{sec:dold-Kan}

\begin{Definition}
  A \emph{spectrum} is a
  sequence $E$ of based simplicial sets $E_n$ together with maps
  $\epsilon_n: E_n \to \Omega E_{n+1}$.  It is an $\Omega$-spectrum if
  all the compositions
  $E_n \to \Omega E_{n+1} \to \Omega \Exinf E_{n+1}$ are weak
  equivalences.  The \emph{homotopy groups} of a spectrum $E$ are
  defined for $k \in\Z$ as
  \begin{equation*}
    \pi_k(E) = \colim_{n \to \infty} \pi_{n+k} E_n.
  \end{equation*}
\end{Definition}

Following Lurie, we can now define the tangent complex of a formally
cohesive functor $\mathcal{F}$ as an $\Omega$-spectrum.
\begin{Example}\label{example:defn-of-tangent-complex-as-spectrum}
  Let $\mathcal{F}$ be a formally cohesive functor, and suppose for
  convenience of notation that is it Kan valued.  Then the tangent
  complex of $\mathcal{F}$ is the $\Omega$-spectrum
  $\bigtangent\mathcal{F}$ whose $n$th space is
  $\mathcal{F}(k \oplus k[n])$, and whose structure maps are given by
  the composition
  \begin{align*}
    \mathcal{F}(k \oplus k[n]) & \xrightarrow{\simeq} \mathcal{F}(k)
    \times^h_{\mathcal{F}(k \oplus k[n+1])} \mathcal{F}(k \oplus
    \widetilde{k[n]})\\ & \xrightarrow\simeq \{\ast\}
    \times^h_{\mathcal{F}(k \oplus k[n+1])} \{\ast\} = \Omega
    \mathcal{F}(k \oplus k[n+1]),
  \end{align*}
  where the first equivalence comes from the fact that $\mathcal{F}$
  preserves homotopy pullback and the second comes from the fact that
  $\mathcal{F}(k) \simeq \ast$.
\end{Example}

Next we recall, in Example~\ref{Example:dold-Kan-spectrum} below, how
an unbounded chain complex gives rise to an $\Omega$-spectrum via the
Dold--Kan correspondence.  Recall first that to a simplicial
$k$-module $A$ there is an associated non-negatively graded $k$-linear
chain complex $NA$: for $p \geq 0$ the $p$th term is
$N_pA = \cap_{i =1}^p \Ker(d_i: A_p \to A_{p-1})$ and the boundary map
$\partial: N_pA \to N_{p-1}A$ is the restriction of $d_0$.  If we
write $\mathrm{Ch}_+(k)$ for the category of non-negatively graded
$k$-linear chain complexes whose morphisms are the chain maps, then it
is a fundamental insight of Dold (\cite[\S 1]{MR0097057}) and Kan that
the functor
\begin{equation*}
  N: s\mathrm{Mod}_k \to \mathrm{Ch}_+(k)
\end{equation*}
is an \emph{equivalence} of categories.  The inverse functor
$\mathrm{Ch}_+(k) \to s\mathrm{Mod}_k$ is often called the Dold--Kan
functor.  The composition
\begin{equation*}
  \mathrm{Ch}_+(k) \xrightarrow{\text{Dold--Kan}}  s\mathrm{Mod}_k
  \xrightarrow{\text{forget}} s\Sets
\end{equation*}
sends a chain complex $(C_*,\partial)$ to a Kan simplicial set (or
topological space, after realization) with basepoint given by
0-element, whose homotopy groups are canonically isomorphic to the
homology groups of the chain complex.

If $C = (C_*,\partial) \in \mathrm{Ch}_+(k)$, then the space
$|\text{Dold--Kan}(C)|$ is an example of an \emph{infinite loop
  space}: if $\Sigma C$ denotes the shifted chain complex with
$(\Sigma C)_0 = 0$ and $(\Sigma C)_{n+1} = C_n$, then there is a
canonical weak equivalence
\begin{equation}\label{eq:8}
  \DK(C) \xrightarrow\simeq \Omega \DK(\Sigma C)
\end{equation}
reflecting the fact that the homology groups of $\Sigma C$ are the
shifted homology groups of $C$.  By iteration, we get weak
equivalences $\DK(C) \to \Omega^n \DK(\Sigma^n C)$.

If $C = (C_*,\partial)$ is an object in the category $\mathrm{Ch}(k)$
of \emph{unbounded} chain complexes we can do something similar.
Firstly, let $\tau_{\geq 0} C \in \Ch_+(k)$ denote the
\emph{truncation}, i.e.\ the chain complex with
$(\tau_{\geq 0} C)_0 = \Ker(\partial: C_0 \to C_{-1})$ and
$(\tau_{\geq 0} C)_n = C_n$ for $n \geq 1$.  We may then apply the
Dold--Kan functor to the non-negatively graded chain complexes
$\tau_{\geq 0}(\Sigma^n C)$ and as in~\eqref{eq:8} above, the
underlying based simplicial sets come with weak equivalences
\begin{equation*}
  \DK(\tau_{\geq 0} \Sigma^n C) \xrightarrow\simeq \Omega
  \DK(\tau_{\geq 0} \Sigma^{n+1} C).
\end{equation*}
\begin{Example}\label{Example:dold-Kan-spectrum}
  An unbounded chain complex $C = (C_*,\partial)$ gives rise to an
  $\Omega$-spectrum with $n$th space $\DK(\tau_{\geq 0} \Sigma^n C)$.
  The homotopy groups of this spectrum are canonically isomorphic to
  the homology groups of $C$.

  This defines a functor from $\Ch(k)$ to $\Omega$-spectra which we
  shall also sometimes call the Dold--Kan functor.
\end{Example}

We take the point of view that the Dold--Kan functor from $\Ch(k)$ to
spectra defined by the above example is a ``forgetful'' functor.  It
remembers enough about a chain complex to recover its homology groups
(viz.\ as the homotopy groups of the spectrum) and in particular it
detects quasi-isomorphisms, but it does not remember enough
information to recover the $k$-module structure on these homology
groups.  In the next few subsections we shall explain how recognize on
a given spectrum $E$ the ``extra structure'' of a weak equivalence
$E \simeq \DK(C)$ for an unbounded chain complex $C$; such an extra
structure implies among other things a $k$-module structure on
$\pi_*(E)$.  The goal of reviewing this theory is to prove that the
tangent complex of a formally cohesive functor $\mathcal{F}$ naturally
has this structure: thus, up to weak equivalence, the tangent complex
spectrum arises under the Dold--Kan functor from a more fundamental
object of $\Ch(k)$ which we consider to be ``the'' tangent complex of
$\mathcal{F}$.

\subsubsection{$\Gamma$-sets and $\Gamma$-spaces}
\label{sec:gamma-spaces}

\begin{Definition}
  Let $\mathrm{FinSet}_*$ be the category of finite based sets, and
  let $\Gamma^\mathrm{op} \subset \mathrm{FinSet}_*$ be the full
    subcategory on the objects
  $\underline{n} = \{\ast, 1, \dots, n\}$ for $n \geq 0$.  A
  $\Gamma$-set is a functor $X:\Gamma^\mathrm{op} \to \Sets$ with
  $X(\underline 0)$ terminal (i.e., having a single element). \end{Definition}

Any $\Gamma$-set arises as the restriction of a functor
$\mathrm{FinSet}_* \to \Sets$, and this extension is unique up to
unique natural isomorphism.  For example, if $S$ is a finite based set
we may define
\begin{equation*}
  X(S) = (X(\underline{n}) \times
  \mathrm{Iso}_{\mathrm{FinSet}_*}(\underline{n},S))/S_n
\end{equation*}
for $|\underline{n}| = |S|$.  Henceforth we shall often (tacitly)
assume that $\Gamma$-spaces $X: \Gamma^\mathrm{op} \to s\Sets$ have
been extended to all finite pointed sets in this way.

\begin{Definition}[Segal] \label{GammaSpaceDefinition} A
  $\Gamma$-space is a simplicial $\Gamma$-set\footnote{A.\ Connes has
    argued that it is better to consider $\Gamma$-spaces as simplicial
    objects in $\Gamma$-sets rather than $\Gamma$-objects in
    simplicial sets.}, or in other words a functor
  $X:\Gamma^\mathrm{op} \to s\Sets$ with $X(\underline 0)$ terminal.

  Let $p_i: \underline{n} \to \underline{1}$ be the morphism with
  $p_i^{-1}(1) = \{i\}$ and let
  \begin{equation}
    \label{eq:13}
    X(\underline{n}) \to \prod_{i=1}^n X(1)
  \end{equation}
  be the map whose $i$th coordinate is $X(p_i)$.  Then the
  $\Gamma$-space $X$ is \emph{special} when these maps are weak
  equivalences for all $n \geq 0$.

  Let $\nabla: \underline{2} \to \underline{1}$ be the unique map with
  $\nabla^{-1}(1) = \{1,2\}$.  If $X$ is special, then the diagram
  \begin{equation*}
    X(\underline{1}) \times X(\underline{1})
    \xleftarrow{(X(p_1),X(p_2))} X(\underline{2})
    \xrightarrow{X(\nabla)}
    X(\underline{1}).
  \end{equation*}
  induces the structure of an abelian monoid on the set
  $\pi_0(X(\underline{1}))$, with unit arising from the unique map
  $\underline{0} \to \underline{1}$.  The $\Gamma$-space $X$ is \emph{very
  special} if this monoid is a group.
\end{Definition} 
 
\textbf{Remark.}  In Segal's original definition, the requirement that
$X(\underline{0})$ be a terminal simplicial set was not included, but
many later authors have added this requirement.  His less restrictive
definition still implies that special $\Gamma$-spaces have
$X(\underline{0})$ contractible, by the weak equivalence~\eqref{eq:13}
for $n=0$.  If $X$ is a $\Gamma$-space in this less restrictive sense,
the unique morphisms $\underline{0} \to \underline{n} \to
\underline{0}$ in $\Gamma^\mathrm{op}$ gives a canonical factorization
of $X$ through a functor into simplicial sets over and under $A =
X(\underline{0})$, in which $X(\underline{0})$ is terminal in that
category.  Now the $\Gamma$-space defined by $X'(\underline{n}) =
X(\underline{n})/A$ has $X'(A)$ terminal.  If $A$ is contractible the
natural map $X \to X'$ is an objectwise weak equivalence, so $X'$ is
(very) special if and only if $X$ is.  Therefore, since we're
mainly interested in special $\Gamma$-spaces it does not matter much
whether we use the less restrictive notion or not.

It may be helpful to think of the simplicial set $X(\underline{1})$ as
the ``underlying space'' of $X$, and the fibers of $X(\underline{n})
\to X(\underline{1})$ as the space of ways to decompose an element as
the ``sum'' of $n$ other elements.  A ``$\Gamma$-space structure'' on
a pointed space $Y$ is a $\Gamma$-space $X$ and an isomorphism (or
perhaps a weak equivalence) $Y \approx X(\underline{1})$.

\begin{Example}   The $\Gamma$-space $\mathbb{S}$ given by the functor which sends
  $\underline{n}$ to the constant simplicial set $\underline{n}$ is
  \emph{not} special.  Indeed, $\mathbb{S}(\underline{n})$ has $n+1$
  elements whereas $\prod_1^n \mathbb{S}(\underline{1})$ has $2^n$
  elements.

  This particular $\Gamma$-space is nevertheless quite important, and
  is known as the \emph{sphere spectrum}. 
\end{Example}

Segal then proves that very special $\Gamma$-spaces model
$\Omega$-spectra.  Let us summarize the construction, following the
later exposition by Bousfield and Friedlander.
\begin{enumerate}[(i)]
\item If $X$ is any $\Gamma$-space, let $\Omega X$ be the
  $\Gamma$-space obtained by taking objectwise (simplicial) loop
  space, i.e.\ $\underline{n} \mapsto \Omega X(\underline{n})$, where
  the basepoint comes from the unique map $\underline{0} \to
  \underline{n}$.
\item If $X$ is any $\Gamma$-space, let $BX$ be the $\Gamma$-space
  defined as
  \begin{equation*}
    (BX)(S) = |[p] \mapsto X(S^1_p \wedge S)|,
  \end{equation*}
  where $S^1_\bullet$ is the pointed simplicial circle, i.e.\ $S^1_p =
  \Delta^1([p])/ (\partial \Delta^1)([p])$, regarded as a functor
   $\Delta^\mathrm{op} \to \Gamma^\mathrm{op}$.
\item The adjoint of the map $$S^1 \times X(\underline{n}) \cong S^1
  \times X(S^1_1 \wedge \underline{n}) \to (BX)(\underline{n})$$
  induces a natural transformation $X \to \Omega BX$. 
  
 More explicitly,  at the level of $p$-simplices, this map sends
    $(u,v) \in S^1_p \times X( \underline{n })_p$ to the image of $v$
    under the map
    $\underline{n} \rightarrow S^1_p \wedge \underline{n}$ induced by
    $t \mapsto u \wedge t$. 
  
  \item Translating into this language, Segal's theorem is firstly
    that if $X$ is very special (and takes values in Kan complexes),
    then $X \to \Omega BX$ induces a weak equivalence
    $X(\underline{n}) \to \Omega (\Exinf(BX(\underline{n})))$, secondly
    that if $X$ is special then so is $BX$, and thirdly that
    $(BX)(\underline{1})$ is $n$-connected when $X(\underline{1})$ is
    $(n-1)$-connected.
\end{enumerate}
In particular when $X$ is very special the space $X(\underline{1})$
comes with canonical deloopings $(BX)(\underline{1})$,
$(B^2X)(\underline{1})$, $\dots$, which depend only on the
$\Gamma$-space structure, and hence $X(\underline{1})$ is canonically
the 0-space in an $\Omega$-spectrum.  If $X$ is special but not very
special, the ``group completion'' of $X(\underline{1})$ (i.e.\
$\Omega B(X(\underline{1}))$) is the zero-space of an
$\Omega$-spectrum.

\begin{Definition}[Bousfield--Friedlander]\label{defn:stable-eq}
  Let $X$ be any $\Gamma$-space, and define
  \begin{equation*}
    \pi_k(X) = \colim_n \pi_{n+k} B^nX.
  \end{equation*}
  A map $X \to Y$ of $\Gamma$-spaces is a stable equivalence if it
  induces an isomorphism in all homotopy groups.
\end{Definition}

\subsubsection{Smash product, $\Gamma$-rings and module spectra}
\label{sec:smash-product-gamma}

Segal's category of $\Gamma$-spaces was later studied further by
Bousfield and Friedlander, who constructed a closed model category
structure on $\Gamma$-spaces whose weak equivalences are the stable
equivalences from Definition~\ref{defn:stable-eq}  and
proved that its homotopy category is equivalent to the homotopy
category of connective spectra in other known models of that category. (See also
\cite[\S 4, Chapter II]{Quillen}). 
Later, Lydakis \cite{Lydakis} gave a model for the
\emph{smash product} $E \wedge F$ of two spectra $E$ and $F$ arising
from $\Gamma$-spaces: for $\Gamma$-spaces defined on all finite based
sets, the smash product has the universal property that a map of
$\Gamma$-spaces $E \wedge F \to G$ is the same (up to isomorphism) as
maps
\begin{equation*}
  E(S) \times F(T) \to G(S \wedge T)
\end{equation*}
forming a natural transformation of functors of $(S,T) \in
\Gamma^\mathrm{op} \times \Gamma^\mathrm{op}$, where $\wedge$ is the
smash product of based sets.

\begin{Example}
  For any $\Gamma$-space $X$, there are canonical maps
  \begin{equation*}
    \underline{n} \times X(T) \to X(\underline{n} \wedge T),
  \end{equation*}
  whose restriction to $\{i\} \times X(T)$ is given by the map $X(T)
  \to X(\underline{n} \wedge T)$ induced by the injection $T \cong
  \{\ast,i\} \wedge T \to \underline{n} \wedge T$.  This produces a
  canonical map of $\Gamma$-spaces
  \begin{equation*}
    \mathbb{S} \wedge X \to X,
  \end{equation*}
  which is in fact an isomorphism.
\end{Example}

\begin{Definition}
  \begin{enumerate}[(i)]
  \item A $\Gamma$-ring is a triple $(R,\mu,\nu)$ consisting of a
    $\Gamma$-space $R$, a (strictly!)\ associative and commutative map
    $\mu: R \wedge R \to R$, and a unit map $\nu: \mathbb{S} \to R$
    satisfying that $\mu \circ (\nu \wedge \mathrm{id}): \mathbb{S}
    \wedge R \to R$ agrees with the canonical isomorphism.
  \item A module over a $\Gamma$-ring $R$ consists of a $\Gamma$-space
    $M$ together with a map $R \wedge M \to M$ satisfying the obvious
    axioms.
  \end{enumerate}
\end{Definition}
For example, any $\Gamma$-space is canonically an $\mathbb{S}$-module.
A $\Gamma$-ring give rise a connective \emph{ring spectrum}, although
the two notions are not quite the same \cite{Lawson}.  Important
examples come from the  \emph{Eilenberg--MacLane
  spectrum} construction, which associates a spectrum
$HV$ to a simplicial abelian group $V$. \index{$HV$}
\begin{Definition}
  \begin{enumerate}[(i)]
  \item Let $V$ be an abelian group.  The Eilenberg--MacLane space
    $HV$ is the $\Gamma$-set defined by
    \begin{equation*}
      HV(S) = H_0(S,\ast;V) \cong \prod^{S \setminus \{\ast\}} V.
    \end{equation*}
    The description as relative singular homology $H_0(S,\ast;V)$
    makes the functoriality clear, and it is clear from the product
    description that it is special (the map~\eqref{eq:13} is a
    bijection of sets).  Then $\pi_0(HV(\underline{1})) = V$ and the
    monoid structure agrees with vector space addition so $HV$ is very
    special.
  \item If $V$ is a simplicial abelian group, define $HV:
    \Gamma^\mathrm{op} \to s\Sets$ by applying the previous
    construction degreewise.
  \end{enumerate}
\end{Definition}
\begin{Example}\label{example:EM-spectra}
  For abelian groups $V,W$ the map
  \begin{align*}
    HV(S) \times HW(T) = H_0(S,\ast;V) \times H_0(T,\ast;W) &\to
    H_0(S,\ast;V) \otimes H_0(T,\ast;W) \\& \stackrel{\cong}\to H_0(S\wedge T, \ast;V
    \otimes W),
  \end{align*}
  obtained by composing the ``K\"unneth isomorphism'' and the
  canonical bilinear map defining the tensor product, is a natural
  transformation of functors of the based finite sets $S$, $T$, and
  hence defines a map of spectra
  \begin{equation}\label{eq:16}
    HV \wedge HW \to H(V \otimes W).
  \end{equation}
  This construction works degreewise for simplicial abelian groups,
  and the map is natural in the simplicial abelian groups $V$ and $W$.
  
  If $k$ is a simplicial commutative ring, then the multiplication $k
  \otimes k \to k$ gives rise to a map $H(k \otimes k) \to Hk$, which
  composed with~\eqref{eq:16} makes $Hk$ into a $\Gamma$-ring.  If $V$
  is a simplicial $k$-module, then $HV$ inherits an $Hk$-module
  structure in a similar way.
\end{Example}

A map $X \to Y$ of $Hk$-module spectra is a map of $\Gamma$ spaces
which commutes (strictly) with the module structure maps.  Such a map
is a weak equivalence if the underlying map of $\Gamma$-spaces is
(i.e.\ is a stable equivalence in the sense of
Definition~\ref{defn:stable-eq}; we emphasize that this is \emph{not}
the same as objectwise weak equivalence of functors
$\Gamma^\mathrm{op} \to s\Sets$).  It is a result of Stefan Schwede
that this notion can be extended to a full-fledged model category
structure.  The following result plays an important technical role in
this paper.
\begin{theorem}[Robinson, Schwede] \label{thm:rsss}
  Let $k$ be a (possibly simplicial) commutative ring.  The
  Eilenberg--Maclane functor  induces an equivalence of categories from
  the homotopy category of simplicial $k$-modules to the homotopy
  category of $Hk$ modules. In fact, this functor is part of a Quillen equivalence between
  model categories. 

  In particular (for $k = \Z$) it induces an equivalence from
  simplicial abelian groups to $H\Z$-modules.
\end{theorem}

By this result, the homotopy theory of simplicial $k$-modules (or
equivalently non-negatively graded chain complexes of $k$-modules) is
``equivalent'' to the homotopy theory of $Hk$-modules.  In particular
there is for each $Hk$-module $E$ a ``corresponding'' simplicial
$k$-module $V$ and a zig-zag of weak equivalences of $Hk$-modules
between $HV$ and $E$.  The usefulness of this result is that an
object may arise quite naturally and explicitly as an $Hk$-module $E$
but not explicitly as a simplicial $k$-module.  For example, in this
paper this is the case for the tangent complex of a formally coherent
functor $\mathcal{F}: \smallC_k \to s\Sets$.

Although we shall not strictly need it, let us briefly discuss the
extent to which this relation between $Hk$-modules and simplicial
$k$-modules may be made functorial.  In fact, Schwede promotes the relation to a Quillen equivalence
of model categories, where $V \mapsto HV$ is the right adjoint.
He defines a functor in the other direction, sending an \index{$L$ \mbox{ and} $Lc$}
$Hk$-module $E$ to a simplicial $k$-module $LE$.  The functor $L$ is
left adjoint to $H$ in the strict sense, and has an explicit and
rather simple construction which we shall omit here.  The
\emph{homotopy} inverse to $H$ then sends an $Hk$-module spectrum $E$
to the simplicial $k$-module $Lc(E)$, where $c$ is a cofibrant
replacement functor on $Hk$-module spectra (in a certain model
category structure, see Schwede's paper for details).  Such a functor
$c$ is proven to exist using the small object argument, but
unfortunately there does not seem to be a known explicit formula.
Consequently we do not know an explicit functorial formula for the
simplicial $k$-module ``corresponding'' to an $Hk$-module spectrum
$E$, but nevertheless it may be useful (or at least consoling) to
remember the direction of the arrows in the comparison zig-zags: to an
$Hk$-module $E$ the ``corresponding'' simplicial $k$-module is
$V = Lc(E)$, and we have natural weak equivalences of $Hk$-modules
\begin{equation} \label{QEV1}
  \xymatrix{
    E & \ar[l]_\simeq cE \ar[r]^-\simeq & HV.
  }
\end{equation}

\subsubsection{Non-connective spectra}
\label{sec:non-conn-spectra}

A spectrum $E = (E_n,\epsilon_n)_{n \in \N}$ is \emph{connective} if
$\pi_k(E) = 0$ for $k < 0$.  For $\Omega$-spectra this is equivalent
to each space $E_n$ being $(n-1)$-connected.  As already alluded to,
connective $\Omega$-spectra are, up to weak equivalence, the same as
very special $\Gamma$-spaces.  More precisely there are functors in
both directions given as follows.
\begin{Example}
  Any $\Gamma$-space $X$ gives rise to a connective spectrum with $E_n
  = (B^nX)(\underline{1})$.  It is an $\Omega$-spectrum when $X$ is
  very special (after possibly applying fibrant replacement to the $E_n$s).

  Conversely, to a spectrum $E$ there is an associated very
  special 
  $\Gamma$-space defined by   \begin{equation*}
    S   \mapsto \colim_{n \to \infty} \Omega^n \Exinf(S \wedge E_n).
  \end{equation*}
  Here $S \wedge E_n = \vee^{S \setminus \{\ast\}} E_n$ is the smash
  product of pointed simplicial sets, where $S$ is regarded as a
  constant simplicial set.

  These constructions are inverse up to
  homotopy, and give an equivalence of categories between the
  homotopy category of very special $\Gamma$-spaces and connective
  $\Omega$-spectra.
\end{Example}

To get a model for \emph{all} spectra, including those with homotopy
groups in negative degrees, we may consider spectrum objects in
$\Gamma$-spaces.
\begin{Definition}
  A $\Gamma$-spectrum is a sequence $E$ consisting of $\Gamma$-spaces
  $E_n$ and maps $\epsilon_n: E_n \to \Omega E_{n+1}$ of
  $\Gamma$-spaces.
\end{Definition}
This notion of $\Gamma$-spectrum is somewhat uncommon, probably for
the following reason.  The forgetful functor from $\Gamma$-spectra to
spectra induced by sending $E$ to the spectrum with $n$th space
$E_n(\underline{1})$ induces an equivalence of homotopy categories, so
in this sense the ``$\Gamma$-structure'' is redundant.  Nevertheless,
it seems convenient to keep track of this extra redundant data, for
instance in dealing with $Hk$-module structures.
\begin{Definition}  
  A (non-connective) $Hk$-module spectrum is a $\Gamma$-spectrum $E$
  together with maps of $\Gamma$-spaces $\mu_n: Hk \wedge E_n \to E_n$
  making $E_n$ into an $Hk$ module, such that the maps $\epsilon_n:
  E_n \to \Omega E_{n+1}$ are maps of $Hk$ modules.
\end{Definition}

As already explained, the Dold--Kan functor may be promoted to a
functor from unbounded chain complexes to spectra.  In fact this
functor naturally takes values in $Hk$-module spectra.
\begin{Example}\label{example:chain-complex-to-spectrum}
  Recall from Example~\ref{example:EM-spectra} that the
  Eilenberg--MacLane functor takes a simplicial $k$-module $V$ to an
  $Hk$-module $HV$.  If $C \in \Ch_+(k)$ is a non-negatively graded
  $k$-linear chain complex, then we obtain an $Hk$-module $H(\DK(C))$.
  If $C \in \Ch(k)$ is an unbounded chain complex, then we obtain an
  $Hk$-module spectrum with $n$th space
  \begin{equation*}
    H(\DK(\tau_{\geq 0} \Sigma^n C)).
  \end{equation*}

  It  remains to define the
  structure maps in this $Hk$-module spectrum.  First recall that for
  simplicial $k$-modules $A$ and $B$ the Alexander--Whitney formula
  gives rise to a map of $k$-linear chain complexes
  $N(A \otimes_k B) \to NA \otimes_k NB$.  If we write $kS^1$ for the
  free $k$-module on the simplicial circle, modulo the span of the
  basepoint and its degeneracies, and apply the Alexander--Whitney map
  in the special case $A = kS^1$ and $B = \DK(D)$ for some
  non-negatively graded $k$-linear chain complex $D$, we get a natural
  transformation
  \begin{equation*}
    N(kS^1 \otimes \DK(D)) \to N(kS^1) \otimes D \cong \Sigma D.
  \end{equation*}
  Applying the Dold--Kan functor to this map, and using that it is
  an inverse functor to $N$ up to isomorphism, we then get a natural
  transformation of simplicial $k$-modules
  \begin{equation*}
    kS^1 \otimes \DK(D) \to \DK(\Sigma D).
  \end{equation*}
  The loop-space functor $\Omega$, when regarded as a endo-functor of
  simplicial $k$-modules, is right adjoint to the functor
  $kS^1 \otimes -$, so the above map is adjoint to a map
  \begin{equation}
    \label{eq:23}
    \DK(D) \to \Omega (\DK(\Sigma D))
  \end{equation}
  of simplicial $k$-modules, natural in the non-negatively graded
  $k$-linear chain complex $D$.  Finally, substitute $D = \tau_{\geq
    0} \Sigma^n C$ into~(\ref{eq:23}) and use that the inclusion
  \begin{equation*}
    \DK(\Sigma \tau_{\geq 0} \Sigma^n C) = \DK(\tau_{\geq 1} \Sigma^{n+1} C)
    \subset \DK(\tau_{\geq 0} \Sigma^{n+1}C)
  \end{equation*}
  becomes an isomorphism after applying $\Omega$.  This defines a
  natural transformation between simplicial $k$-modules, and the
  Eilenberg--MacLane functor turns this into a map of $Hk$-modules.
\end{Example}

\begin{theorem}
  The functor from unbounded
  $k$-linear chain complexes to $Hk$-module spectra defined in the
  above example induces an equivalence of homotopy categories.
\end{theorem}
\begin{proof}[Proof sketch]
  This is just assembling other equivalences: the Dold--Kan functor
  gives an equivalence between $\Ch_+(k)$ and $s\mathrm{Mod}_k$,
  Robinson and Schwede's results give an equivalence between
  $s\mathrm{Mod}_k$ and $Hk$-modules.  There is an induced equivalence
  between ``$\Omega$-spectrum objects'' in $\Ch_+(k)$ and
  $Hk$ module spectra.  Spelling this out, we have an equivalence from
  $\Ch(k)$ to $Hk$-module spectra.
\end{proof}

The above functor from $\Ch(k)$ to $Hk$-module spectra, defined
explicitly in terms of the Dold--Kan functor and the
Eilenberg--MacLane functor, again has a homotopy inverse
\emph{functor} which turns an $Hk$-module spectrum $E$ into an
unbounded chain complex.  This inverse functor is defined by turning
each $Hk$-module $E_n$ into a $k$-linear positively graded chain
complex $Lc(E_n)$; the maps $E_n \to \Omega E_{n+1}$ then induces maps
of chain complexes $\Sigma Lc(E_n) \to Lc(E_{n+1})$ and we may form a
chain complex as the colimit (or homotopy colimit) of
$\Sigma^{-n} Lc(E_n)$.  Up to
weak equivalence we get the $Hk$-module spectrum back by applying the 
functor in Example~\ref{example:chain-complex-to-spectrum};
conversely, turning an unbounded $k$-linear chain complex into an
$Hk$-module spectrum and then back to a chain complex results in a
chain complex quasi-isomorphic to the one we started with.
Unfortunately the functor from $Hk$-module spectra to chain complexes
is not nearly as explicit as the one in the other direction, since it
involves the inexplicit cofibrant approximation $c(E_n) \to E_n$ of
$Hk$-modules.

\subsection{The tangent complex of a formally cohesive functor as an
  $Hk$ module spectrum.}
\label{sec:hk-module-spectra}

Let us now finally return to formally cohesive functors.  If
$\mathcal{F}: \smallC_k \to s\Sets$ is such a functor, we have already
explained how to define $\bigtangent \mathcal{F}$ as an
$\Omega$-spectrum, whose $n$th space is
\begin{equation*}
 ( \bigtangent \mathcal{F})_n= \mathcal{F}(k \oplus k[n]).
\end{equation*}
This is an
$\Omega$-spectrum by the cohesiveness assumption, and is usually not
connective.  The goal of this section (and the goal of recalling all the
$\Gamma$-space machinery!)\ is to explain why the based simplicial sets
$(\bigtangent \mathcal{F})_n$ have canonical $\Gamma$-space structures,
and why the resulting $\Gamma$-spectrum $\bigtangent\mathcal{F}$ is
naturally an $Hk$-module spectrum.

\begin{Lemma}
  Let $\mathcal{F}: \smallC_k$ be formally cohesive and let $V$ be a
  simplicial $k$-module with $\pi_*(V)$ finite dimensional.  Then the
  $\Gamma$-space $S \mapsto \mathcal{F}(k \oplus HV(S))$ is special.
\end{Lemma}
\begin{proof}
  Strictly speaking $\mathcal{F}(k) = \mathcal{F}(k \oplus
  HV(\underline{0}))$ need not be terminal, so we should first
  quotient it out, as described in the remark following
  Definition \ref{GammaSpaceDefinition}. 

  If $W_1$, \dots, $W_n$ are simplicial $k$-modules, the projection
  maps $p_j: \oplus_i W_i \to W_j$ induce a map
  \begin{equation*}
    \mathcal{F}(k \oplus \bigoplus_{i=1}^n W_i) \to \prod_{i=1}^n
    \mathcal{F}(k \oplus W_i),
  \end{equation*}
  and by induction the formal cohesiveness of $\mathcal{F}$ implies
  that these are weak equivalences.  Applying this to $W_1 = \dots =
  W_n = V$ we see that the map
  \begin{equation*}
    \mathcal{F}(k \oplus HV(\underline{n})) \to \prod_{i=1}^n
    \mathcal{F}(k \oplus HV(\underline{1}))
  \end{equation*}
  is a weak equivalence and hence $S \mapsto \mathcal{F}(k \oplus
  HV(S))$ is special.
\end{proof}
The following Lemma and its proof shows the advantage of working with
$Hk$-modules as opposed to simplicial $k$-modules.
\begin{Lemma}
  Let $S$ and $T$ be based finite sets and let $\mu: Hk(S) \times
  HV(T) \to HV(S \wedge T)$ be the maps defining the multiplication
  $Hk \wedge HV \to HV$ in the $Hk$-module structure on $HV$.  From
  the definition of $\mu$ we see that for each $v \in Hk(S) = \prod^{S
    \setminus \{\ast\}} k$ the map $\mu(v,-): HV(T) \to HV(S \wedge T)$
  is a map of simplicial $k$-modules, and hence induces maps
  \begin{equation*}
    \mathcal{F}(k \oplus -)(\mu(v,-)): \mathcal{F}(k \oplus HV(T)) \to
    \mathcal{F}(k \oplus HV(S \wedge T)),
  \end{equation*}
  and as $v \in Hk(S)$ varies, these assemble to a map
  \begin{equation*}
    Hk(S) \times \mathcal{F}(k \oplus HV(T)) \to \mathcal{F}(k \oplus
    HV(S \wedge T))
  \end{equation*}
  making the special $\Gamma$-space $T \mapsto \mathcal{F}(k \oplus
  HV(T))$ into an $Hk$-module spectrum.
\end{Lemma}
\begin{proof}
  This follows from the functoriality of the constructions involved.
\end{proof}
We have not yet verified that the $\Gamma$-spaces constructed above
are very special.  We shall do that by providing deloopings, which are
in fact (usually) non-connective.  We shall only need this in the case
$V = k[n]$, but the following Lemma may be generalized to
other $V$.
\begin{Lemma}
  Let $(\bigtangent \mathcal{F})_n$ be the special $\Gamma$-space $S
  \mapsto \mathcal{F}(k \oplus H(k[n])(S))$, for $n \geq 0$.  The
  natural maps%
  \begin{equation*}
    \mathcal{F}(k \oplus H(k[n])(S)) \to \Omega \mathcal{F}(k \oplus
    H(k[n+1])(S)),
  \end{equation*}
  arising from the weak equivalence from $k \oplus H(k[n])(S)$ to the
  homotopy pullback of $k \rightarrow k \oplus H(k[n+1])(S) \leftarrow
  k$, define maps of special $\Gamma$-spaces $(\bigtangent 
  \mathcal{F})_n \to \Omega (\bigtangent \mathcal{F})_{n+1}$ which are
  $Hk$-module maps and also weak equivalences.
\end{Lemma}

\begin{proof}
  The maps are weak equivalences for each $S$ by the formal
  cohesiveness of $\mathcal{F}$.  It is easy to check that the
  deloopings commute (strictly, as usual) with the $Hk$-module
  structure maps.
\end{proof}

\begin{Definition}
  The tangent complex $\bigtangent\mathcal{F}$ is the chain complex
  associated (as in
  Example~\ref{example:chain-complex-to-spectrum}) to the $Hk$-module
  spectrum defined by the $\Gamma$-space
  $S \mapsto \mathcal{F}(k \oplus Hk(S))$ and the deloopings
  $\mathcal{F}(k \oplus H(k[n])(S))$ provided above.
\end{Definition}

In many cases of interest the space $\mathcal{F}(k \oplus k)$ is
discrete, and then the chain complex $\bigtangent \mathcal{F}$ has
homology groups only in non-positive degrees.  For example, this is
the case for the functors $\mathcal{F}_{R,\overline{\rho}}$ from
Example \ref{ex:functors-represented-by-general-rings}.

The following important example explains
the connection to the usual \emph{cotangent complex} of Quillen and
Illusie:  To a map $A \rightarrow B$ of  commutative rings, or indeed of simplicial commutative rings, 
and a $B$-module $M$, 
we can define Andr{\'e}--Quillen cohomology groups $\Der^i_A(B, M)$;
if $A \rightarrow B$ is cofibrant, this is the cohomology of the cosimplicial abelian group obtained by computing (levelwise) $A$-linear derivations of $B$ with targets in $M$.  
 Similarly one dually defines a simplicial $B$-module $L_{B/A}$, the ``cotangent complex.''     Thus if $A, B$ are usual rings and $M$
 a usual $B$-module, $\Der^0_A(B, M)$ is the set of derivations $B \rightarrow M$ and $\Der^1_A(B, M)$  classifies  commutative $A$-algebra extensions $M \rightarrow ?  \rightarrow B$, 
 whereas $\pi_0 L_{B/A} \simeq \Omega_{B/A}$. 
 For details see \cite{Quillen}.

\begin{Example}\label{ex:tangent-complex-Hom-R}
  If $R$ is a cofibrant simplicial commutative ring and $\overline
  \rho: \pi_0 R \to k$ is a homomorphism, we may consider the functor
  $\mathcal{F}_{R,\overline\rho}: \smallC_k \to s\Sets$ from  
  Example~\ref{ex:functors-represented-by-general-rings}.  
By \eqref{AQ1} and the subsequent discussion, 
  $$ \pi_{-i} \mathfrak{t} \mathcal{F}_{R, \rhobar} \simeq \Der^i_{\Z}(R, k).$$
Indeed, the tangent complex of $\mathcal{F}_{R, \overline \rho}$ is
   quasi-isomorphic to     \begin{equation}  \label{DKLdual}
    \Hom_{\Ch(k)}(\DK(L_{R/\Z} \otimes_{R} k), \ k),
  \end{equation}
  where the  $\Hom$ is the internal hom of chain complexes. 
\end{Example}

\subsection{Constructions on cohesive functors and their effect on
  tangent complexes}
\label{sec:verify-deriv-schl}

In this section we discuss various constructions which produce new
cohesive functors out of old ones, and discuss the effect of these
functors on tangent complexes.  Both of the following two Lemmas are
special cases of the more general statement that the class of formally
cohesive functors are closed under taking objectwise homotopy limits.
\begin{Lemma}
  Let $\mathcal{F}:\smallC_k \to s\Sets$ be formally cohesive and let
  $X$ be a simplicial set, or even a pro-simplicial set.  Then the
  functor $A \mapsto \Map(X,\mathcal{F}(A))$ is formally cohesive,
  where $\Map$ denotes the space of unbased maps.  If $\mathcal{F}$
  takes values in pointed spaces and $X$ is pointed, then the same is
  true for the space of based maps.
\end{Lemma}  As usual, if $X = (\alpha \mapsto X_{\alpha})$ is a pro-simplicial set, 
we define $\Map(X, -)$  to be the colimit $\colim_{\alpha} \Map(X_{\alpha}, -)$.  
\begin{proof}
  $\Map(X,-)$ preserves homotopy pullbacks.   
\end{proof}

\begin{Lemma}\label{lemma:2.50}
  Let $\mathcal{F}_0, \mathcal{F}_1, \mathcal{F}_{01}: \smallC_k \to
  s\Sets$ be homotopy invariant functors, equipped with natural
  transformations $\mathcal{F}_0 \to \mathcal{F}_{01} \leftarrow
  \mathcal{F}_1$, and define $\mathcal{F}$ by the objectwise homotopy
  pullback
  \begin{equation*}
    \mathcal{F}(A) = \mathcal{F}_0(A) \times^h_{\mathcal{F}_{01}(A)}
    \mathcal{F}_1(A).
  \end{equation*}
  If $\mathcal{F}_0$, $\mathcal{F}_1$, and $\mathcal{F}_{01}$ are
  formally cohesive, then so is $\mathcal{F}$.  More generally, if
  $\mathcal{F}_0$, $\mathcal{F}_1$, and $\mathcal{F}_{01}$ preserve
  homotopy pullback, then so does $\mathcal{F}$.
\end{Lemma}
\begin{proof}
  This is a formal consequence of homotopy limits commuting with each
  other up to weak equivalence; in particular any homotopy limit
  preserves homotopy pullback.
\end{proof}

\begin{Lemma}\label{lem:2.51}
  Let $\mathcal{F}: \smallC_k \to s\Sets$ be formally cohesive.
  \begin{enumerate}[(i)]
  \item\label{item:2} The tangent complex of the ``forgetful'' functor
    $A \mapsto \mathfrak{m} = \Ker(A \to k)$ is one-dimensional and
    concentrated in degree 0 (i.e.\ the corresponding $Hk$-module
    spectrum is equivalent to $Hk$ as an $Hk$-module spectrum).
  \item If $\mathcal{F}$ takes values in pointed simplicial sets,
    then 
    $\bigtangent \Omega \mathcal{F} \simeq \Sigma^{-1}\bigtangent
    \mathcal{F}$.
  \item\label{item:1} For a space $X$, we have
    $\bigtangent\Map(X,\mathcal{F}(-)) \simeq
    C^*(X;\bigtangent\mathcal{F})$.  Similarly when $X$ is a
    pro-space, where both sides are then interpreted as the direct
    limits of maps from finite levels of the pro-object.
  \item The tangent complex takes homotopy pullback squares of
    functors are taken to pullback squares of spectra.  Hence for
    $\mathcal{F} = \mathcal{F}_0 \times^h_{\mathcal{F}_{01}}
    \mathcal{F}_1$
    as in Lemma~\ref{lemma:2.50} above, there is a (co)fiber sequence
    of spectra      \begin{equation*}
      \bigtangent\mathcal{F} \to \bigtangent\mathcal{F}_0 \times
      \bigtangent\mathcal{F}_1 \to \bigtangent\mathcal{F}_{01}
    \end{equation*}
    and a corresponding ``Mayer--Vietoris'' sequence in homotopy
    groups
    \begin{equation*}
      \dots \to \pi_{-n}\bigtangent \mathcal{F} \xrightarrow{((f_0)*,(f_1)_*)}
      \pi_{-n}\bigtangent \mathcal{F}_0 \oplus \pi_{-n}\bigtangent
      \mathcal{F}_1 \xrightarrow{(p_0)_* - (p_1)_*} \pi_{-n} \bigtangent
      \mathcal{F}_{01} \to
      \pi_{-(n+1)}\bigtangent \mathcal{F} \to \dots,
    \end{equation*}
    where $p_i: \mathcal{F}_i \to \mathcal{F}_{01}$ are the specified
    natural transformations and $f_i: \mathcal{F} \to \mathcal{F}_i$
    is the induced projections from the homotopy fiber product to its
    factors.
  \end{enumerate}
\end{Lemma}

 We will often use (iv) in the following form:  a homotopy commutative square of functors
$$  \xymatrix{
\mathcal{G} \ar[r]\ar[d]  &   \mathcal{F}_0 \ar[d] \\ 
 \mathcal{F}_1  \ar[r] &      \mathcal{F}_{01} \\ 
 }$$ 
 is an object-wise homotopy pullback square exactly when the induced map $$\bigtangent \mathcal{G} \rightarrow \hofib(\tangent \mathcal{F}_0 \oplus \tangent \mathcal{F}_1 \longrightarrow
 \tangent \mathcal{F}_{01})$$ is an isomorphism on homotopy groups. Indeed by our discussion above, the latter condition is satisfied exactly when 
 $\mathcal{G} \rightarrow \mathcal{F}_0 \times_{\mathcal{F}_{01}}^h \mathcal{F}_1$ induces an isomorphism on  homotopy groups of tangent complexes, thus is an equivalence by Lemma \ref{lem:tangent-complex-detects-equivalence}.

\begin{proof}
  For~(\ref{item:2}), clearly the functor takes $k \oplus k[n]$ to
  $k[n]$ and hence $\pi_i \bigtangent$ has the same homotopy groups as
  $Hk$ so they must be weakly equivalent.

  The remaining statements are special cases of the general statement
  that if $C$ is any category and $\mathcal{F} = \holim_{c \in C}
  \mathcal{F}_c$, then $\bigtangent \mathcal{F} = \holim_{c \in C}
  \bigtangent \mathcal{F}_c$.  This follows because (homotopy) limits
  of $\Omega$-spectra may be computed levelwise.
\end{proof}

\textbf{Remark}.  By combining these results, we see that the tangent
complex of the functor $A \mapsto s\Sets_*(X,\mathfrak{m})$ is
$C^*(X,\ast;k)$.

These result shows that the tangent complex behaves well with respect
to homotopy pullback of functors, or equivalently derived tensor
product of (pro-)simplicial rings.  By contrast, the behaviour under
pullback of rings is much more complicated.  Similarly, the tangent
complexes of $k \oplus k[n]$ are quite complicated.

Later we shall define for any algebraic group $G$ a functor $A \mapsto
BG(A)$.  The following results will let us prove that that functor
preserves homotopy pullbacks.  %
\begin{Lemma}
  For a commutative square
  \begin{equation*}
    \xymatrix{
      A \ar[r]^f\ar[d]^g & B\ar[d]^h\\
      C\ar[r]^j & D
    }
  \end{equation*}
  with all four spaces are path connected, and such that either $\pi_1(B) \to
  \pi_1(D)$ or $\pi_1(C) \rightarrow \pi_1(D)$ is surjective, the square is homotopy cartesian if and
  only if the induced square of loop spaces is homotopy cartesian.
\end{Lemma}
\begin{proof}
  To see that the induced map $A \to B \times_D^h C$ is a weak
  equivalence, we check that $\pi_i(A) \to \pi_i(B \times^h_D C)$ is a
  bijection for all $i$ and for all basepoints.  In fact, we claim
  that for $i=0$ both homotopy sets have exactly one element.  Given
  that, the bijection on higher homotopy groups follow since $\pi_i A
  = \pi_{i-1}\Omega A$ and $\pi_i(B \times_D C) = \pi_{i-1}(\Omega B
  \times_{\Omega D}^h \Omega A)$.

  $\pi_0(A)$ has one element by assumption.  The set $\pi_0(B \times_D
  C)$ fits in a short exact sequence of pointed sets $\pi_0(\hofib(j))
  \to \pi_0(B \times_D C) \to \pi_0 B$, where $\pi_0 B$ has one
  element by assumption, and $\hofib(j)$ has one element since $\pi_1
  C \to \pi_1(D)$ is surjective and $D$ is path connected. 
  \end{proof}

\begin{Corollary}\label{cor:2.54}
  Let $\mathcal{F}: \smallC_k \to s\Sets_*$ be a homotopy invariant
  functor into based simplicial sets, such that $\mathcal{F}(A)$ is
  path connected for all $A$, that $A \mapsto \Omega \mathcal{F}(A)$
  is preserves homotopy pullback, and such that $\pi_0
  \Omega\mathcal{F}(A) \to \pi_0\Omega\mathcal{F}(B)$ is surjective
  whenever $\pi_0 A \to \pi_0 B$ is surjective.  Then $\mathcal{F}$
  preserves homotopy pullback.
\end{Corollary}

\subsection{Lurie's derived Schlessinger criterion}\label{sec:luri-deriv-schl}
The following result is from \cite[6.2.14]{LurieThesis}.  In our
applications to number theory, we will only use the
``countable-dimensional'' case of the result.  In the following we
shall use the notation  $\pi_i \bigtangent\mathcal{F}$ for the
$i$th homology group of the chain complex $\bigtangent\mathcal{F}$
(which is the $i$th homotopy group of the corresponding $Hk$-module
spectrum).

\begin{theorem}\label{thm:Lurie-Schlessinger}
  Let $\mathcal{F}: \smallC_k \to s\Sets$ be formally cohesive.  Then
  $\mathcal{F}$ is pro-representable if and only if the $k$-vector
  space $\pi_i\bigtangent\mathcal{F}$ vanishes for $i > 0$.

  If the $k$-vector spaces $\pi_i \bigtangent\mathcal{F}$ have
  countable dimension, then the pro-representing object may be chosen
  to have countable indexing category.
\end{theorem}

Recall that a spectrum or chain complex $E$ is called
\emph{connective} if $\pi_i E = 0$ for $i < 0$ and is called
\emph{co-connective} if $\pi_i E = 0$ for $i > 0$.  Thus the theorem
asserts that a functor is pro-representable if and only if the tangent
complex is co-connective; this happens precisely when
$\mathcal{F}(k \oplus k[0])$ is homotopy discrete.  This is a variant
of Theorem 6.2.13 in Lurie's thesis and is proved by essentially the
same argument.  Since his setup and assumption are mildly different
from ours (e.g.\ we do not discuss any ``Noetherian-ness'' of the
representing pro-system, and correspondingly we do not assume the
homotopy groups of the tangent complex are finite dimensional), we
shall outline the proof.
\begin{proof}[Sketch of proof of Theorem~\ref{thm:Lurie-Schlessinger}]
  In this proof we shall write $\bigtangent\mathcal{F}$ for the
  $Hk$-module spectrum version of the tangent complex.  It follows
  from Example~\ref{ex:tangent-complex-Hom-R} that the tangent complex
  of any representable functor is co-connective.  The tangent complex
  takes filtered homotopy colimits of functors to filtered homotopy
  colimits of spectra, so the tangent complex of any pro-representable
    functor is
  co-connective.

  Conversely, suppose $\pi_i\bigtangent\mathcal{F} = 0$ for $i > 0$,
  i.e.\ that $\pi_i\mathcal{F}(k \oplus k[n]) = 0$ for $i > n$.  We
  shall produce a natural weak equivalence from a filtered homotopy
  colimit of representable functors to $\mathcal{F}$.  Without loss of
  generality, we may assume that $\mathcal{F}$ is a simplicially
  enriched functor and takes Kan values.
  
  The construction is by a (generally transfinite) recursive recipe,
  providing an ``improvement'' to any pair $(R,\iota)$ consisting of a
  cofibrant $R \in \smallC_k$ and a zero-simplex $\iota \in
  \mathcal{F}(R)$.  Suppose given an integer $n \geq 0$, a finite
  dimensional $k$-vector space $V$ and a point in the homotopy fiber
  of the map $\Hom(R,k \oplus V^\vee[n+1]) \to \mathcal{F}(k \oplus
  V^\vee[n+1])$, given by a diagram
  \begin{equation}\label{eq:14}
    \begin{aligned}
      \xymatrix{
        \partial \Delta[1] \ar[r]^-f \ar[d] & {\Hom(R,k \oplus
          V^\vee[n+1])} \ar[d]^{\iota}\\
        \Delta[1] \ar[r]^-h & {\mathcal{F}(k \oplus V^\vee[n+1])}, }
    \end{aligned}
  \end{equation}
  where $f$ and $h$ are pointed maps.  We then construct an object $R'
  \in \smallC_k$ and a morphism $R' \to R$ by the  cartesian
  square
  \begin{equation*}
    \xymatrix{
      R' \ar[r]\ar[d] & k \oplus \widetilde{V^\vee[n+1]} \ar[d]\\
      R \ar[r]^-f & k \oplus V^\vee[n+1],
    }
  \end{equation*}
  where as usual we write
  $k \oplus \widetilde{V^\vee[n+1]} = k \oplus (V^\vee \otimes_k
  \widetilde{k[n+1]}) \simeq k$.  Since $f$ induces a surjection on
  $\pi_0$ and $\mathcal{F}$ is formally cohesive we obtain a fiber
  sequence
  $\mathcal{F}(R') \to \mathcal{F}(R) \to \mathcal{F}(k \oplus
  V^\vee[n+1])$, where the second map sends $\iota$ to $f_*(\iota)$.
  The path $h: D^1 \to \mathcal{F}(k \oplus V^\vee[n+1])$ then
  provides a lift of $\iota$ to the the homotopy fiber of
  $\mathcal{F}(R) \to \mathcal{F}(k \oplus V^\vee[n+1])$, and hence a
  homotopy lift of $\iota$ to $\iota' \in \mathcal{F}(R')$.  As we
  shall explain next, this pair $(R',\iota')$ can be regarded as an
  ``improvement'' of $(R,\iota)$.

  In terms of represented functors, we have constructed a
  natural diagram of functors $\smallC_k \to s\Sets$
  \begin{equation*}
    \xymatrix{
      {\Hom(R,-)} \ar[r]^-\iota \ar[d] & {\mathcal{F}} \ar@{=}[d] \\
      {\Hom(R',-)} \ar[r]^-{\iota'}\ar@{=>}[ur]  & {\mathcal{F}} \\
    }      
  \end{equation*}
  where the double arrow denotes a natural homotopy.  Evaluating at $k
  \oplus k[n]$ these functors define the spectrum underlying the
  cotangent complex, and we have an induced homotopy commutative
  diagram
  \begin{equation*}
    \xymatrix{
      {\bigtangent R} \ar[r]^-\iota\ar[d] &
      {\bigtangent\mathcal{F}} \ar[r]\ar@{=}[d] & C\ar[d]\\
      {\bigtangent R'} \ar[r]^-{\iota'} &
      {\bigtangent\mathcal{F}(-)} \ar[r] & C'
    }
  \end{equation*}
  where $C$ and $C'$ are the cofibers in the category of $Hk$-module
  spectra (they should be thought of as the \emph{relative} cotangent
  complexes of $\mathcal{F}$ relative to $R$ or
  $R'$).  The homotopy type of the maps $(f,h)$ from~\eqref{eq:14}
  correspond bijectively to $k$-linear maps $V \to \pi_{-n} C$, and by
  construction of $(R',\iota')$ the composition $V \to \pi_{-n} C \to
  \pi_{-n} C'$ is zero.

  By iterating this procedure, possibly transfinitely many times, we
  obtain an inverse system of simplicial rings $R_j$, and a point in
  the homotopy inverse limit of $\mathcal{F}(R_j)$, such that the
  corresponding relative cotangent complexes $C_j$ have contractible
  direct limit (any element at some stage is always killed at some
  later stage).  Hence we have produced a natural weak equivalence
  $\hocolim_{j\in J} \Hom(R_j,-) \to \mathcal{F}$, as desired.
\end{proof}

\subsection{Non-reduced functors and local systems}
\label{sec:non-reduced-functors}

In this section we discuss how to generalize the definitions and
results above to functors $\mathcal{F}: \smallC_k \to s\Sets$ which
preserve homotopy pullback, but which do not necessarily have
$\mathcal{F}(k)$ contractible and hence is not formally cohesive.

Such functors are important for our later constructions: for example,
in \S \ref{sec:gener-semis-algebr} we shall define, for $G$ an
algebraic group over $W(k)$, a functor $A \mapsto BG(A)$ which is
homotopy invariant and preserves homotopy pullbacks, but it is not
reduced.  Our representation functor is a modification of
$A \mapsto s\Sets(X,BG(A))$, and it is convenient first study
$BG(A)$ and then afterwards study the effect of taking mapping space
from $X$.

Rather than forcing a functor $\mathcal{F}: \smallC_k\to s\Sets$
preserving homotopy pullbacks to have contractible value at $k$, we
change the target category in order for the object $\mathcal{F}(k)$ to
be \emph{homotopy terminal} (i.e.\ have contractible derived mapping
space from any other object).  Since any object
$A = (A \to k) \in \smallC_k$ comes with a unique morphism to
$k = (\mathrm{id}: k \to k)$ there is an induced map
$\mathcal{F}(A) \to \mathcal{F}(k)$, and we may regard $\mathcal{F}$
as taking values in the over category $s\Sets_{/Z}$ for
$Z = \mathcal{F}(k)$, or when technically convenient
$Z = \Exinf(\mathcal{F}(k))$.  In this setting it still makes sense to
define the tangent complex of $\mathcal{F}$, but instead of a
$k$-linear chain complex (or equivalently an $Hk$-module spectrum), it
will be a \emph{local system} of $k$-linear chain complexes on $Z$, in
the following sense.
\begin{Definition}
  The category $\mathrm{Simp}(Z)$ of
  simplices of a simplicial set $Z$ has objects pairs of an object
  $[p] \in \Delta$ and a morphism $\sigma: \Delta[p] \to Z$, and
  morphisms those morphisms in $\Delta$ that commute with the maps to
  $Z$.  In particular any $p$-simplex $\sigma$ has face inclusion
  morphisms $d_i \sigma \to \sigma$ for each $i = 0, \dots p$.

  A local system (of $k$-linear chain complexes) on a simplicial set
  $Z$ is a functor
  \begin{align*}
    \mathcal{L}: \mathrm{Simp}(Z) &\to \mathrm{Ch}(k)\\
    \sigma &\mapsto \mathcal{L}_\sigma
  \end{align*}
  to the category of chain complexes of $k$-modules, possibly
  unbounded in both directions, sending all morphisms in
  $\mathrm{Simp}(Z)$ to quasi-isomorphisms.
\end{Definition}
If $\mathcal{L}$ is such a system, we may define cochains of $Z$ with
coefficients in $\mathcal{L}$.  In order to avoid confusion, let us
decide that all differentials \emph{decrease} degrees, and that
``cochains'' therefore tends to live in negative degrees.  Let
\begin{equation*}
  C^*(Z;\mathcal{L}) = \prod_{\sigma \in \mathrm{Simp}(Z)}
  \mathcal{L}_\sigma,
\end{equation*}
which a priori is a bigraded $k$-vector space (one grading from the
dimension of $\sigma$, one from the grading in $\mathrm{Ch}(k)$) with
two commuting coboundary maps (one from the internal boundary maps of
the $\mathcal{L}_\sigma$, one from the alternating sum of the maps
$\delta^i_i: \mathcal{L}_{d_i \sigma} \to \mathcal{L}_\sigma$).  By
$C^*(Z;\mathcal{L})$ we shall usually mean the associated total
complex.
\begin{Remark}
  Alternatively, a local coefficient system could be defined as a
  functor from $\mathrm{Simp}(Z)$ into the category of $Hk$-module
  spectra, sending all morphisms into weak equivalences.  The
  Dold--Kan functor and its homotopy inverse functor $Lc$ discussed in
  Section~\ref{tangent complex setup} translate back and forth
  between these  definitions. 

  Taking cochains of a local system then corresponds to taking
  \emph{homotopy limit}: if $\mathcal{L}$ is a local system of
  $Hk$-module spectra, then applying $Lc$ to the \emph{homotopy limit}
  of $\mathcal{L}$, which is naturally again an $Hk$-module spectrum,
  results in a chain complex which is canonically quasi-isomorphic to
  the cochains of the local system of $k$-linear chain complexes
  $\sigma \mapsto Lc(\mathcal{L}_\sigma)$.
\end{Remark}
Let us comment on the relationship between this notion of local
system and the more classical one as a module with an action of the
fundamental group.  In fact the classical one is sufficient for local
systems of chain complexes whose homology is concentrated in a single
degree, as follows.
\begin{Remark}\label{remark:local-system-classical-defn}
  In the case where the base $Z$ is path connected and $z \in Z$ is a
  vertex, the group $\pi_1(Z,z)$ acts on the homology of the chain
  complex $\mathcal{L}_z$.  If $H_*(\mathcal{L}_z)$ is concentrated in
  a single degree, then this action in fact classifies the local
  system: any $\pi_1(Z,z)$-module may be realized as
  $H_n(\mathcal{L}_z)$ for a local system $\mathcal{L}$ such that
  $\mathcal{L}_z$ has homology concentrated in degree $n$, and if
  $\mathcal{L}$ and $\mathcal{L}'$ are two local systems such that
  both $\mathcal{L}_z$ and $\mathcal{L}'_z$ have homology concentrated
  in degree $n$ then any isomorphism
  $H_n(\mathcal{L}_z) \cong H_n(\mathcal{L}'_z)$ may be realized by a
  zig-zag of weak equivalences of functors
  $\mathrm{Simp}(Z) \to \mathrm{Ch}(k)$.
\end{Remark}

Let us return to the tangent complex.  Assume that
$\mathcal{F}: \smallC_k \to s\Sets_{/Z}$ preserves homotopy pullbacks
and that $\mathcal{F}(k) \in s\Sets_{/Z}$ is a terminal object.  For
each $\sigma: \Delta[p] \to Z \cong \mathcal{F}(k)$ we may define a
functor $\mathcal{F}_\sigma: \smallC_k$ by the homotopy pullback
square
\begin{equation*}
  \xymatrix{
    {\mathcal{F}_\sigma(A)} \ar[r]\ar[d] & {\mathcal{F}(A)} \ar[d]\\
    \Delta[p] \ar[r]^-\sigma & {\mathcal{F}(k)}.
  }
\end{equation*}
Then each $\mathcal{F}_\sigma$ has a tangent complex $\bigtangent
\mathcal{F}_\sigma$, and we have obtained a functor
\begin{align*}
  \mathrm{Simp}(Z) &\to \mathrm{Ch}(k)\\
  \sigma &\mapsto \bigtangent \mathcal{F}_\sigma
\end{align*}
sending all morphisms to weak equivalences.
\begin{Definition}\label{defn:tangent-complex-non-reduced}
  Let $Z$ be a simplicial set and
  $\mathcal{F}: \smallC_k \to s\Sets_{/Z}$ a homotopy invariant
  functor preserving pullbacks, such that $\mathcal{F}(k)$ is a
  terminal object.  Then the tangent complex of $\mathcal{F}$ is the
  local system of $k$-linear unbounded chain complexes on $Z$ given by
  as $\sigma \mapsto \bigtangent\mathcal{F}_\sigma$, as defined above.
\end{Definition}

Our results and constructions from Section~\ref{sec:verify-deriv-schl}
all have a generalizations to non-reduced functors, where the tangent
complex is a local system on $\mathcal{F}(k)$ as in
Definition~\ref{defn:tangent-complex-non-reduced}.  Most importantly
for the next section, we have the following construction.
\begin{Example}\label{example:mapping-into-functor}
  Let $\mathcal{F}:\smallC_k \to s\Sets$ be homotopy invariant and
  preserve homotopy pullbacks.  Without loss of generality, suppose
  also that all values of $\mathcal{F}$ are Kan.  If we let
  $Z = \mathcal{F}(k)$ and regard $\mathcal{F}$ as a functor
  $\smallC_k \to s\Sets_{/Z}$ the value at $k$ is then terminal.  Let
  $\bigtangent\mathcal{F}: \mathrm{Simp}(Z) \to \mathrm{Ch}(k)$ be its tangent
  complex, regarded as a local system of $k$-linear chain complexes on
  $Z$, as defined above.

  Let $X$ be any simplicial set and
  $\overline{\rho}: X \to Z = \mathcal{F}(k)$ a map.  We may then
  define a new functor
  $\mathcal{F}_{X,\overline\rho}: \smallC_k \to s\Sets$ by
  \begin{equation*}
    \mathcal{F}_{X,\overline\rho}(A) = \hofib_{\overline\rho}
    \big(s\Sets(X,\mathcal{F}(A)) \to s\Sets(X,\mathcal{F}(k)\big).
  \end{equation*}
  Then $\mathcal{F}_{X,\overline\rho}$ is formally cohesive, and its
  tangent complex is given by the formula
  \begin{equation}\label{eq:24}
    \bigtangent \mathcal{F}_{X,\overline\rho} \simeq C^*(X,\overline \rho^*
    \bigtangent \mathcal{F}),
  \end{equation}
  where $\overline\rho^* \bigtangent \mathcal{F}$ denotes the
  pulled-back local system defined by precomposing
  $\bigtangent \mathcal{F}$ by
  $\overline\rho_*: \Simp(X) \to \Simp(Z)$.

  The  formula~\eqref{eq:24} is perhaps best explained before
  turning $Hk$ modules into chain complexes, where the definition of
  $\bigtangent\mathcal{F}$ is more explicit.  In that setting, a local
  system on $X$ is simply a functor from the category of simplices of
  $X$ to the category of $Hk$-module spectra, and taking cochains
  translates to a particular model for the homotopy limit.  For the
  particular local system relevant for~\eqref{eq:24}, the $n$th space
  of the functor sends a simplex $\sigma: \Delta^p \to X$ to the space
  of homotopy commutative diagrams
  \begin{equation*}
    \xymatrix{
      \Delta^p \ar[r] \ar[d]_\sigma & \mathcal{F}(k \oplus k[n])  \ar[d] \\
      X \ar[r]_{\overline\rho} & \mathcal{F}(k),
    }
  \end{equation*}
  where ``space of homotopy commutative diagrams'' means the
  simplicial set of ways to supply the top horizontal arrow as well as
  a simplicial homotopy $\Delta^1 \times \Delta^p \to \mathcal{F}(k)$
  in the diagram.  Taking homotopy limit over all $\sigma$ then
  gives the space of homotopy commutative diagrams
  \begin{equation*}
    \xymatrix{
       & \mathcal{F}(k \oplus k[n]) \ar[d] \\
      X \ar[ur] \ar[r]_{\overline\rho} & \mathcal{F}(k),
    }    
  \end{equation*}
  which is the $n$th space of $Hk$-module spectrum
  $\bigtangent\mathcal{F}_{X,\overline\rho}$.
\end{Example}

For later use, we remark that the construction in the above example
works just as well when $X = (i \mapsto X_i)$ is a pro-simplicial
set.  In that case the example applies to each $X_i$ and we may pass
to the (filtered) colimit.

Finally, let us describe how to define $B\GL_n(A)$, for $A$ a
simplicial ring. 
The definition we give is {\em ad hoc} and depends on the specific realization of $\GL_n(A)$: 

\begin{Example}\label{GLn ad hoc}
  For a simplicial ring $A$, let $\GL_n(A)\subset M_n(A) = A^{n^2}$
  denote union of the path components in $\GL_n(\pi_0 A) \subset
  M_n(\pi_0 A)$.  Then the functor $A \mapsto B\GL_n(A)$ is homotopy
  invariant and preserves homotopy pullbacks, but $B\GL_n(k)$ is not
  contractible, since it is a $K(\pi,1)$ for $\GL_n(k)$.
\end{Example}
 
To see that $B\GL_n$ preserves homotopy pullback, we use
Corollary~\ref{cor:2.54}.  Indeed, the looped functor
$A \mapsto \Omega B\GL_n(A)$ is just
$\GL_n(A) \subset M_n(A) = A^{n^2}$.  This looped functor
$A \mapsto M_n(A)$ clearly preserves homotopy pullback, so it remains
to 
see that $\pi_0(\GL_n)$ preserves surjections of Artin
rings.  This is because a square matrix with coefficients in an Artin
local ring is invertible if and only if the reduction to a matrix with
coefficients in the residue field is invertible.

\begin{Lemma}
  The tangent complex of $A \mapsto B\GL_n(A)$ is the local system on  
  $B\GL_n(k)$ given by the conjugation action of $\GL_n(k)$ on
  $M_n(k)$.  (Here $M_n(k)$ is discrete, i.e.\ the homotopy groups of
  the tangent complex vanish except in degree 0.)  
\end{Lemma}

We omit the proof, because in the next section we shall give a more general result: 
 we define a functor $A \mapsto BG(A)$ for an
arbitrary algebraic group $G$, and will prove the analogue of the Lemma in that context.
(The definition of the next section will not specialize to Example \ref{GLn ad hoc} in the case $G = \GL_n$, but it is naturally weakly equivalent to it.)

\section{Representation functors}
\label{sec:gener-semis-algebr}

In this section we define and study an ``infinitesimal representation 
variety'' functor, parametrizing representations into an algebraic
group $G$ defined over $W(k)$.  
Assume for the moment that $G$ has trivial center.  Let $\Gamma$ be a group. 
We wish to define a functor
\begin{equation} \label{FGAdef}
  A \mapsto  \mathcal{F}_G(A) = \Hom(\Gamma,G(A))/G(A),
\end{equation}
where $G(A)$ acts by conjugation on the space of homomorphisms.
Better yet, we want to study the subfunctor where the composition $\Gamma \to G(A)
\to G(k)$ is fixed to be some $\overline{\rho}$ on which $G(k)$ acts
without isotropy.  
This makes perfect sense as stated for a discrete
(pro)-group $\Gamma$ and a discrete commutative $W(k)$-algebra $A$,
and for a homomorphism $A \to k$.  
In this section we will define such
a functor also for simplicial $A$, and possibly (pro-)simplicial
$\Gamma$ as well.  

 The final definition of our representation functors are given in
Definition \ref{defn-repfunctor} -- the role of $\Gamma$ is replaced by $B\Gamma$, 
which is also allowed to be in fact an arbitrary pro-simplicial set,  as we need in our applications; 
and their tangent complexes are computed in Lemma
\ref{Lemma:tangent space computation}.

\subsection{$BG(A)$ for a simplicial ring $A$} \label{BGA simplicial}  
Let us first discuss how one might try to define  $G(A)$, if $A$ is a simplicial ring. 
The
simplicial group $[p] \mapsto G(A_p)$ is not a useful object
to consider, even for $G = \G_m$: it is not homotopy invariant. 

If $\mathcal{O}_G$ is the (discrete) $W(k)$-algebra with $G =
\Spec(\mathcal{O}_G)$, we may instead define $A \mapsto G(A)$ for $A \in
\smallC_k$ by
\begin{equation*}
  G(A) = \Hom_\SCR(c(\mathcal{O}_G),A),
\end{equation*}
where $c: \SCR \to \SCR$ denotes cofibrant approximation.  This is a
reasonable definition if we only care about $G$ as a scheme, and not a
group scheme.  In fact, $G: \smallC_k \to s\Sets$ is pro-representable
in the sense of Definition~\ref{defn:pro-rep}, e.g.\ by inverse system
given by the objects $c(\mathcal{O}_G /p^m)$.  
Unfortunately, $A \mapsto
G(A)$ defined this way does not take values in simplicial groups or
even simplicial monoids.  

For $G = \GL_n$ there is an \emph{ad hoc}
definition of $\GL_n(A)$ as the union of the components of $M_n(A) =
A^{n^2}$ with invertible image in $M_n(\pi_0 A)$. This is monoid
valued and weakly equivalent to the abstract definition above.  

For a
general algebraic group $G$, it seems better to directly define the
functor $A \mapsto BG(A)$ instead of first defining $G(A)$.  (If
desired, a simplicial group valued functor $A \mapsto G(A)$ can be
obtained by applying ``Kan loop group''.)

For an ordinary commutative ring $A$, the space $BG(A)$ is the realization of the nerve of $G(A)$;
the functor  $A \mapsto N_p(G(A))$  (i.e.\ $p$-simplexes of the nerve of $G(A)$)
 is represented by $\mathcal{O}_{N_pG} = \mathcal{O}_G^{\otimes p}$,
the $p$-fold tensor product over $W(k)$.  
Our recipe in general is to replace each $A \mapsto N_p G(A)$ by a functor which
is representable by a simplicial ring which is cofibrant and homotopy
discrete, in a way that also preserves the simplicial structure $[p]
\mapsto N_p G(A)$.

In the simplicial case we
define for $A \in \smallC_k$ for each $p$ the simplicial set \begin{equation*}
\Hom_\SCR(c(\mathcal{O}_{N_p G}),A). \end{equation*}
This simplicial set is a functor of $[p]$ as well as $A$,
and we have obtained a
functor from simplicial rings to bisimplicial sets.
\begin{Definition} \label{BG simplicial ring definition}
  Let $BG(A)$ denote  $\Exinf$ of the geometric realization (i.e.\ the diagonal) of   the bisimplicial set $\Hom_\SCR(c(\mathcal{O}_{N_\bullet G}),A)$
  defined above.
\end{Definition}

Note that, if $A$ is homotopy discrete, then $BG(A)$ is weakly
equivalent to the usual classifying space of the discrete group
$G(\pi_0 A)$.  Let us also note the following ``sanity
check'' of our definition.
\begin{Lemma}\label{lemma:sanity-check}
  Let $G(A) = \Hom_\SCR(c(\mathcal{O}_G),A)$ as above.  Then the
  natural inclusion
  \begin{equation*}
    G(A) \to \Omega BG(A)  \end{equation*}
  is a weak equivalence.
\end{Lemma}
\begin{proof}
For $R \in \SCR$, we say that  $\Z \rightarrow R$  is ``cellular'' if it is isomorphic to a   possibly transfinite)  composition  of cell attachments as in Definition \ref{defn:attach-cell}. 

  Let us also note that $(c(\mathcal{O}_G))^{\otimes p}$ is cellular  
  if $c(\mathcal{O}_G)$ is and hence for purposes where the strict
  functoriality in $[p] \in \Delta$ is unnecessary we may use
  $(c(\mathcal{O}_G))^{\otimes p}$ as a model for
  $c(\mathcal{O}_{N_p G})$.  In particular we see that the morphisms
  $[1] \cong (i-1 < i) \subset (0 < \dots < p)$ in $\Delta$ for
  $i = 1, \dots, p$ induce weak equivalences
  \begin{equation*}
    \Hom_\SCR(c(\mathcal{O}_{N_p G}),A) \to \prod_{i = 1}^p
    \Hom_\SCR(c(\mathcal{O}_{N_1 G}),A).
  \end{equation*}
  This condition on a simplicial space, often called the ``Segal
  condition'' after \cite{SegalCategoriesAndCohomologyTheories},
  implies that the natural map from the 1-simplices to the loop space  
  of the geometric realization is a weak equivalence.
\end{proof}
\begin{Corollary}\label{cor:connectivity-of-BG}
  If $A \to B$ is a morphism in $\smallC_k$ whose underlying map of
  simplicial sets is $n$-connected, then $BG(A) \to BG(B)$ is
  $(n+1)$-connected.
\end{Corollary}

The proof of this Corollary is not hard. Before giving it, let us
point out that we defined $G(A) = \SCR(c(\mathcal{O}_G),A)$ and not
$\SCR_{/k}(c(\mathcal{O}_G),A)$.  Hence we have an interesting map
$G(A) \to G(k)$.  Each element $\rho \in G(k) = \pi_0 G(k)$ gives a
homomorphism $\rho: \mathcal{O}_G \to k$ and hence an object
$(\mathcal{O}_G,\rho) \in \SCR_{/k}$, and we have an isomorphism of
simplicial sets 
\begin{equation}\label{eq:9}
  G(A) \cong \coprod_{\rho \in G(k)} \SCR_{/k}((c(\mathcal{O}_G),\rho),A)
\end{equation}
which is in fact a natural isomorphism of functors
$\smallC_k \to s\Sets$.
\begin{proof}
  By
  Lemma~\ref{lemma:sanity-check} it suffices to see that
  $G(A) \to G(B)$ is $n$-connected, and by~\eqref{eq:9} it suffices to
  see that
  \begin{equation*}
    \SCR_{/k}((c(\mathcal{O}_G),\rho),A) \to
    \SCR_{/k}((c(\mathcal{O}_G),\rho),B)
  \end{equation*}
  is $n$-connected for all $\rho \in G(k)$, which we may check one $\rho$
  at a time.

  This holds because the coordinate ring $\mathcal{O}_G$ is smooth at
  the $k$-point $\rho \in G(k)$.  Hence if we choose formal parameters
  $x_1, \dots, x_d \in \mathcal{O}_G$ of $G$ at $\rho$, the object of
  pro-$\smallC_k$ representing the functor
  $\SCR_{/k}((c(\mathcal{O}_G),\rho),-)$ may be taken to be
  $W(k)[[x_1,\dots,x_d]]$ for $d = \dim(G)$.  This choice of
  parameters induce a natural weak equivalence
  $\SCR_{/k}((c(\mathcal{O}_G),\rho),A) \simeq (\mathfrak{m}(A))^d$. 
 \end{proof}

Our definition of $A \mapsto BG(A)$ is also functorial in the
algebraic group $G$.    
\subsection{Framed and unframed deformation functors}

We can now define analogues of the framed and unframed representation
functors for general $G$.  Let us postpone a discussion of
irreducibility (hence our functors need not be representable in general -- they are ``derived stacks'').

Before we proceed to the general definitions, we give first the simplest case. 
Let $\mathcal{X}$ be a scheme -- in our applications, the spectrum of a ring of integers in a number field -- and let $y_0$ a geometric basepoint.
  If $A \in \Art_k$ is discrete,  and so $G(A)$ is simply a discrete group, then the set of conjugacy classes of representations
$\pi_1(\mathcal{X},x_0) \rightarrow G(A)$ are in correspondence with the {\'e}tale torsors over $\mathcal{X}$
with structure group $G(A)$.   In turn, these correspond to
homotopy classes of maps from the {\'e}tale homotopy type $X$ of $\mathcal{X}$ to $BG(A)$. 
Here the  {\'e}tale homotopy type is understood
to be the pro-simplicial set $X$ that  is associated to $\mathcal{X}$ by Friedlander \cite{Friedlander}\footnote{In the cases of interest for number theory, we don't need to appeal to the {\'e}tale homotopy type --
see the footnote on page \pageref{explicitetalehomotopytype} for a more straightforward construction.}.
This motivates the basic version of our representation functor, which appears in
\eqref{basic_case} below. 

   In terms of the discussion above, 
 part (i) of the following definition  corresponds to keeping track of representations together with a basis for the underlying module  (this is usually called
 ``framed deformation rings'' in number theory) and part (ii) 
corresponds to studying only lifts of a fixed representation $\pi_1(\mathcal{X}_0, x) \rightarrow G(k)$. 
The most important case for us will be the functor $\mathcal{F}_{X, \rhobar}$
  from (iii), in the case when $X$ is the etale homotopy type of the ring of integers in a number field.  
  
  In the following definition, if $X = (\alpha \mapsto X_{\alpha})$
  is a pro-object in simplicial sets, we write    $\Hom(X, -)$ for the colimit 
\begin{equation} \label{prohomdef} \colim_{\alpha} s\Sets(X_{\alpha} , -).\end{equation}
  
\begin{Definition} \label{defn-repfunctor}
  \begin{enumerate}[(i)]
  \item Let $X$ be a pro-object in based simplicial sets, and let
    $\mathcal{F}_{X,G}^\fram$ be the functor whose value on a
    simplicial ring $A$ is the space of based maps $X \to   BG(A)$.  Equivalently, it is the homotopy fiber of the
    \begin{equation*}
      \Map(X, BG(A)) \to  BG(A), 
    \end{equation*}
    given by evaluation at the basepoint.

    \item Let $X$ be a pro-object in simplicial sets, and let $\mathcal{F}_{X,G}$  
    be the functor whose value on a simplicial ring $A$ is the space of maps
 (now unbased):
        \begin{equation} \label{basic_case}  \mathcal{F}_{X,G}(A) =  \Map(X,  BG(A)).\end{equation}

  \item If $\overline\rho: \pi_1 X \to G(k)$ is a homomorphism (from
    the fundamental pro-groupoid, in the unpointed situation), we 
    define functors $\mathcal{F}_{X,\overline\rho}$ and
    $\mathcal{F}^\square_{X,\overline\rho}$ by taking homotopy fibers
    over the zero-simplices of $\mathcal{F}_{X,G}(k)$ and
    $\mathcal{F}_{X,G}^\square(k)$ given by $\overline\rho$.    \end{enumerate}
\end{Definition}

It follows from the discussion of   Example~\ref{example:mapping-into-functor},
that these functors are homotopy invariant and preserve homotopy pullback,
and that the functors of (iii) are formally cohesive.

\subsection{Calculation of tangent complexes}
\label{sec:calc-tang-compl}

The functor $\mathcal{F}_{X,G}$ is
not reduced except in trivial cases.  Hence the tangent complex is a
local system (\S \ref{sec:non-reduced-functors}) on $\mathcal{F}_{X,G}(k)$, and in fact in
Section~\ref{sec:non-reduced-functors} we have developed all the tools
to calculate it at an arbitrary vertex
$\overline\rho \in \mathcal{F}_{X,G}(k)$.  

To this end, it is 
convenient to first treat the functor $A \mapsto  BG(A)$, which
we shall write as simply $BG$.  It has the same properties as the
functors defined above, namely it is homotopy invariant, Kan valued,
preserves pullbacks, but is not reduced.

\begin{Lemma}\label{lemma:tangent-complex-of-BG}
  The tangent complex of $A \mapsto BG(A)$ is the local system on
  $BG(k)$ which at the basepoint gives a chain complex with homology
  concentrated in degree 1, where it is the $k$-module
  $\mathfrak{g} = \Ker(G(k[\epsilon]/\epsilon^2) \to G(k))$ with the
  conjugation action of $G(k) = \pi_1(BG(k))$.
\end{Lemma}
In other words it is the Lie algebra of $G$ with the adjoint action,
concentrated in degree 1 (where as usual we insist on grading so that
boundary maps decrease degrees; in cohomological notation it would be
in degree $-1$).  By Remark~\ref{remark:local-system-classical-defn},
the Lemma uniquely describes $\bigtangent(BG)$ as a local system on
$BG(k)$ up to weak equivalence of local systems.
\begin{proof}
  Let us work with the corresponding spectra instead of chain
  complexes.  Then the tangent complex $\bigtangent(BG)$ at the
  basepoint vertex of $BG(k)$ is the spectrum whose $n$th space is
  \begin{equation}\label{eq:6}
    \hofib(BG(k \oplus k[n]) \to BG(k)).
  \end{equation}
  Now $BG(k \oplus k[n]) \to BG(k)$ is $(n+1)$-connected by
  Corollary~\ref{cor:connectivity-of-BG} and the $(n+1)$st loop space is
  the homotopy fiber of $G(k \oplus k[0]) \to G(k)$ or in other words
  the kernel of that surjection of discrete groups.  Hence the
  space~\eqref{eq:6} is an Eilenberg--MacLane space
  $K(\mathfrak{g},n+1)$ for
  $\mathfrak{g} = \Ker(G(k[\epsilon]/\epsilon^2) \to G(k))$, and this
  description also captures the action of $G(k)$.
\end{proof}

From the Lemma and Example~\ref{example:mapping-into-functor} it is
then immediate to read off the tangent complex of associated
functors, in particular:

\begin{Example} 
With notation as in Definition \ref{defn-repfunctor},  the tangent complex of
$\mathcal{F}_{X,G,\overline{\rho}}$ is (quasi-isomorphic to)
$C^{*+1}(X;\overline\rho^*\mathfrak{g})$, where 
we regard the Lie algebra $\mathfrak{g}$ as a local system on $X$ by means of $\rhobar$.  
\end{Example}
The indexing may be somewhat confusing: we mean that $\pi_{-i}$ of the tangent complex is isomorphic to the $H^{i+1}(X; \overline\rho^*\mathfrak{g})$; in particular the homotopy groups of the tangent complex 
are  concentrated in degrees $(-\infty,1]$.

Observe, in particular, that if the center of $G$ is positive-dimensional, then the functor $\mathcal{F}_{X,G,\overline{\rho}}$
is never pro-representable, because its tangent complex has nonvanishing $\pi_1$.  We now explain how to modify the foregoing
discussion in this case.  

\subsection{Modifications for center}
If $G$ has a nontrivial center (e.g. $G=\GL_n$) the functor $\mathcal{F}_{X,G}$ will never be pro-representable, because of automorphisms. 
 We now explain how to modify the previous discussion,  along the lines of \eqref{FGAdef}. 
 (This section will not be used in the remainder of the paper.)
 
 Suppose, then, that $Z \to G$ is a central algebraic subgroup. Then we have an algebraic
group $PG = G/Z$ and hence $BPG(A)$ for simplicial $A$.  It shall be
convenient to give another model for $BPG(A)$, in which the short
exact sequence $Z(A) \to G(A) \to PG(A)$ corresponds to a fibration
sequence
\begin{equation*}
  BG(A) \to BPG(A) \to B^2Z(A).
\end{equation*}
We must first explain the definition of the functor $A \mapsto B^2Z(A)$.  For discrete
rings $A$, $Z(A)$ is a commutative group and hence $N_\bullet Z(A)$ is
a simplicial group so we may form the bisimplicial set $N_\bullet
N_\bullet Z(A)$.  Extend this to a functor on simplicial rings which
in each bi-degree is representable by a cofibrant simplicial ring, as
before.  Then $B^2Z(A)$ is defined as  $\Exinf$ of the diagonal of the
trisimplicial set.

\begin{Lemma}
  In the above setting, there is a natural fibration sequence
  \begin{equation*}
    BG(A) \to BPG(A) \to B^2Z(A).
  \end{equation*}
\end{Lemma}
\begin{proof}[Proof sketch/Definition]
  What we mean is that the functors are related by zig-zags of natural
  weak equivalences to functors which actually form a fibration.

  The main thing to explain
     is how to define $BPG(A) \to B^2Z(A)$, to which end we
  replace $BPG(A)$ by the Borel construction $BG(A)\modd BZ(A)$.  More
  explicitly, for discrete $A$ the multiplication $Z(A) \times G(A)
  \to G(A)$ is a group homomorphism and hence gives an action $(N_p
  Z(A)) \times (N_pG(A)) \to N_p(G(A))$ of the simplicial group
  $N_\bullet Z(A)$ on the simplicial set $N_\bullet G(A)$.  Taking bar
  construction of this actions gives a bisimplicial set
  \begin{equation*}
    (p,q) \mapsto N_q(\ast, N_pZ(A), N_p G(A)),
  \end{equation*}
  which is represented by a bi-cosimplicial ring which in bidegree
  $(p,q)$ is $\mathcal{O}_Z^{\otimes pq} \otimes
  \mathcal{O}_G^{\otimes p}$.  Then apply a functorial cosimplicial  
  replacement as before, and redefine $BPG(A)$ as the geometric
  realization of the tri-simplicial set represented (this is naturally
  weakly equivalent to the previous definition).  There is a
   natural weak
  equivalence to the previously defined $BPG(A)$.

  In this model the obvious map $N_q(\ast, N_pZ(A), N_p G(A)) \to
  N_q(\ast, N_p Z(A), \ast)$ gives the desired map $BPG(A) \to B^2
  Z(A)$, with fiber $BG(A)$. (Note that, as before, we should apply $\Exinf$ to all these functors so that they have Kan values.)
\end{proof}

\begin{Definition} \label{defn-repfunctor-Z}
   Let $X$ be a pro-object in spaces, and let $\mathcal{F}_{X,G,Z}$
    be defined by the homotopy pullback diagram
    \begin{equation}\label{eq:7}
      \begin{aligned}
        \xymatrix{ {\mathcal{F}_{X,G,Z}(A)}\ar[r]\ar[d] & {\Map(\pi_0
            X, B^2
            Z(A))}\ar[d]^{\mathrm{const}}\\
          {\Map(X, PG(A))} \ar[r] & {\Map(X, B^2
            Z(A))} }
      \end{aligned}
    \end{equation}
    where the right-hand vertical map is the inclusion of the
    (componentwise) constant maps; and $\Hom(X, -)$ is as in \eqref{prohomdef}, but taken in unpointed spaces. 

\end{Definition}

Observe that when $Z$ is trivial, $\mathcal{F}_{X,G,Z}$ reduces to the functor $\mathcal{F}_{X,G}$ previously studied.  
More generally $\mathcal{F}_{X,G,Z}$ fits into a natural fibration sequence 
\begin{equation}\label{eq:21}
  \Map(\pi_0(X),BZ(A)) \to \Map(X, BG(A)) \to \mathcal{F}_{X,G,Z}(A).
\end{equation}
In the case $X = B\Gamma$ for a pro-simplicial group $\Gamma$, this
could be though of as modeling a bundle of stacks
\begin{equation*}
  \ast\modd Z \to (\Hom(\Gamma,G))\modd G \to (\Hom(\Gamma,G))\modd PG.
\end{equation*}

\begin{Lemma}
$\mathcal{F}_{X,G,Z}$ is homotopy invariant and preserves homotopy pullbacks. 
\end{Lemma}
\begin{proof}
  The three functors in the homotopy pullback diagram defining
  $\mathcal{F}_{X,G,Z}$ are all special cases of the construction in
  Example~\ref{example:mapping-into-functor}, hence they are all
  homotopy invariant and preserve homotopy pullback.  These properties
  are then inherited by $\mathcal{F}_{X,G,Z}$.
\end{proof}

 If $\overline\rho: \pi_1 X \to G(k)$ is a homomorphism, we may define functors $\mathcal{F}_{X,Z, \overline{\rho}}$  and $\mathcal{F}^{\square}_{X,Z,\overline{\rho}}$ similarly to the previous discussion. 
We may compute the   tangent complexes of $\mathcal{F}_{X,G,Z}$ and $\mathcal{F}_{X, Z, \overline{\rho}}$:
the fiber sequence \eqref{eq:21} 
 induces a cofiber sequence of tangent complexes and the tangent
complexes of $ \Map(\pi_0(X), B Z(A))$ and $\Map(X,BG(A))$ are
determined by the above discussion.  We deduce a cofiber
sequence 
\begin{equation*}
  C^{*+1}(\pi_0 X;\mathrm{Ad}(Z)) \to C^{*+1}(X;\mathrm{Ad}(G)) \to
  \bigtangent \mathcal{F}_{X,Z,\overline \rho},
\end{equation*}
where we have written $\mathrm{Ad}(G)$ for the Lie algebra of $G$ with
action of $\pi_1(X)$ pulled back along $\overline\rho$ and similarly
for $\mathrm{Ad}(Z)$.

\begin{Lemma} \label{Lemma:tangent space computation}    The tangent complex of $\mathcal{F}_{X,Z,\overline{\rho}}$  is quasi-isomorphic to the  mapping cone of      \begin{equation*}
    C^{*+1}(\pi_0(X);\mathrm{Ad}(Z)) \to C^{*+1}(X; \mathrm{Ad}(G)),  \end{equation*}    where $\mathrm{Ad}(G)$ denotes the Lie algebra of $G$ over $k$,
  acted on by $\pi_1(X)$ via $\overline \rho$, and $C^*$ denotes cochains.  The map is induced by
  the inclusion $Z \subset G$ and by the canonical map $X \to
  \pi_0(X)$ of (pro-)simplicial sets, and the domain is non-vanishing
  only for $\ast = -1$.

  In particular, the tangent complex is co-connective precisely when  
  the map of $k$-vector spaces $H^0(\pi_0(X);\mathrm{Ad}(Z)) \to
  H^0(X; \mathrm{Ad}(G))$ is an
  isomorphism, i.e.\ only central vectors in the Lie algebra of $G(k)$
  are fixed by the action of the pro-groupoid $\pi_1(X)$.  In that
  case we have an isomorphism
  \begin{equation*}
    \pi_{-n}\bigtangent \mathcal{F}_{X,Z,\overline\rho} =
    H^{n+1}(X; \mathrm{Ad}(G))
  \end{equation*}
  for all $n \geq 0$.
\end{Lemma}
\begin{proof}
  This follows from the discussion in
  Section~\ref{sec:verify-deriv-schl}, together with the fiber
  sequence~\eqref{eq:21}.
\end{proof}

\begin{Definition}
  A homomorphism $\overline\rho: \pi_1 X \to G(k)$ is
  \emph{irreducible} if the map of $k$-vector spaces $H^0(\pi_0 X;
  \mathrm{Ad}(Z)) \to  H^0(X, \Ad(G))$ is an isomorphism.
\end{Definition}

For $G = \GL_n$ this notion of irreducibility is sometimes called
\emph{absolute} irreducibility: it amounts to the representation of
$\pi_1 X$ being irreducible in the sense of representation theory,
even after extending scalars to an algebraic closure of $k$.

\def\acts{\mathrel{\reflectbox{$\righttoleftarrow$}}}
\newcommand{\usbS}{\overline{\mathrm{S}}}
\newcommand{\tame}{\mathrm{tame}}
\newcommand{\usbrS}{\overline{\mathrm{S}}^{\circ}}
\newcommand{\usbR}{\overline{\mathrm{R}}}
\newcommand{\rS}{ \mathrm{S}^{\circ}}
\newcommand{\usS}{\mathrm{S}}
\newcommand{\rmM}{\mathrm{M}}
\newcommand{\Dinf}{\mathsf{D}_{\infty}} 
\newcommand{\rhoglob}{\rho_{\mathcal{O}}}
\newcommand{\scc}{\mathrm{sc}}
\newcommand{\usrS}{\mathrm{S}^{\circ}}
\newcommand{\sigmabar}{\overline{\sigma}}
\newcommand{\ab}{\mathrm{ab}}
\newcommand{\usArt}{\mathrm{discArt}}
\newcommand{\Iinf}{\mathsf{I}}
\newcommand{\usrbS}{\overline{\mathrm{S}^{\circ}}}
\newcommand{\usR}{\mathrm{R}}
\newcommand{\weakequ}{\stackrel{\sim}{\dashrightarrow}}
\newcommand{\GSp}{\mathrm{GSp}}
\newcommand{\univ}{\mathrm{univ}}
\newcommand{\rmA}{\mathrm{A}}
\newcommand{\Lifts}{\mathrm{Lifts}}
\newcommand{\temp}{\mathrm{temp}}
\newtheorem{Assumption}{Assumption}
   \setcounter{tocdepth}{1}

          \section{ Number-theoretic notation: Galois representations} \label{numtheorynotn}
          The remainder of this paper is devoted to  studying the derived deformation ring of a Galois representation.
Both this section and the next one (\S \ref{numbertheorynotn2}) set up  notation that we will use.
The current section collects notation related to Galois representations and their cohomology, while
\S \ref{numbertheorynotn2} collects notation related to deformation rings and problems.

         \subsection{General setup}

In our number theory sections, we shall fix a prime $p$ and a finite field $k$ of characteristic $p$ with Witt vectors $W(k)$. 
(In the prior sections we have often used $p$ to denote ``simplicial degree,'' but we will never do so  hereafter.)
We let $W_n = W(k)/p^n$ be the corresponding ring of length $n$ Witt vectors.  \index{$W_n$} \index{$W(k)$}

   Let $S$ be a finite set of primes of $\mathbb{Q}$ containing $p$.  Let $T = S-\{p\}$. 
   
 Let $G$ be a split  semisimple algebraic
group over $W(k)$, e.g. $G = \mathrm{PGL}_n$.   This will be the target for our Galois representations.
 To avoid various  (not particularly important) subtleties with square roots, it is convenient for us to restrict to the case when $G$ is actually adjoint.\footnote{It might be worth remarking here
 that in, for example, the theory of modular forms for $\GL_2$ one usually studies Galois representations to $\GL_2$ with fixed determinant.
 However, this deformation functor maps to the corresponding $\PGL_2$ deformation functor, and, away from characteristic $2$, the map
 is an isomorphism by inspection of tangent spaces. Thus in such situations, so long as the characteristic
 does not divide the order of $\pi_1(G)$,  nothing is  lost by working with the adjoint form.}

 Let $T \subset G$ be a
maximal $k$-split torus, and let the Lie algebras of $T, G$ be denoted by $\Lie(T), \Lie(G)$ respectively; 
these are free $W(k)$-modules, and we denote   $\Lie(T)_k = \Lie(T) \otimes_{W(k)} k$ 
and $\Lie(G)_k = \Lie(G) \otimes_{W(k)} k$. Where typographically convenient we will use
the shorthand $\mathfrak{g}$ for $\Lie(G)_k$; we don't use this shorthand for $T$ because it clashes with notation for the tangent complex.

Since we will be passing often between simplicial rings and usual rings, in these sections we will
adhere to the following convention: script letters $\mathcal{A}, \mathcal{B}$ etc.\ denote
simplicial commutative rings, and roman letters $\mathrm{A}, \mathrm{B}$ etc.\ always denote usual rings.
(A normally italicized letter $R$ could denote either.)

\subsection{Galois groups} \label{sec:galoisgroups}

With $S$ as above, a finite set of primes containing $p$, we put
$$\Gamma_S := \pi_1^{\mathrm{et}}(\Z[\frac{1}{S}], x_0) \cong \Gal(\Q^{(S)}/\Q)$$
where $x_0$  is a fixed basepoint, and $\Q^{(S)}$ is the largest extension of $\Q$ unramified outside $S$. 
 
 It is traditional in number theory to think in terms of Galois groups, but given that our definition
 of representation rings is actually specified in terms of maps out  of the {\'e}tale homotopy type,
 it seems more consistent for us to use $\pi_1$ instead. We will usually be lazy about specifying basepoints,
 just as one is often lazy about algebraic closures when discussing Galois groups. (To be precise,
 we should fix, once and for all, basepoints for  the spectrum of $\Z[\frac{1}{S}], \Z_q$ and the other rings we use;
 no compatibility between these basepoints is required.)
 
   For every finite prime  $q$ there is a distinguished conjugacy class of mappings  
\begin{equation} \label{iotavdef} \iota_q:   \pi_1(\Q_q, x_1)  \longrightarrow \Gamma_S\end{equation} 
where again $x_1$ is a fixed basepoint; 
we have $\pi_1(\Q_q, x_1)  \cong \Gal(\overline{\Q}_q/\Q_q) . $
Moreover for  $q \notin S$ the map $\iota_q$ factors through the unramified quotient of  $\pi_1(\Q_q, x_1)$. 
This unramified quotient is pro-cyclic, generated by a Frobenius element at $q$: 
$$ \pi_1(\mathbb{F}_q, x_2) \cong  \Gal(\overline{\mathbb{F}_q}/\mathbb{F}_q) =  \langle \Frob_q \rangle.$$  

The abelianization  $\pi_1 \Q_q^{\ab}$ of $\pi_1 \Q_q$ is isomorphic, by class field theory, to the profinite completion of $\Q_q^*$.
 The tame part  $\pi_1 \Q_q^{\ab, \tame}$ of this abelianization is, by definition, the quotient by the maximal pro-$q$ subgroup $1+q\Z_q$ of $\Z_q^*$. 
  Thus it fits into a short exact sequence  
\begin{equation} \label{iqdef}  \underbrace{ I_q }_{\cong (\Z/q)^*} \hookrightarrow \pi_1 \Q_q^{\ab, \tame} \twoheadrightarrow  \underbrace{\pi_1 \mathbb{F}_q}_{\cong \hat{\Z}},\end{equation}
where we have defined $I_q$ to be the kernel of the natural surjection on the right. 
 The notation $I_q$ here is intended to suggest ``inertia,'' but it is  only a small quotient of the full inertia group: it is just the tame
 part of the abelian Galois group.

For any $p$-torsion  abelian group $M$ equipped with a $\Gamma_S$-action,  the natural map
  $$H^*(\Gamma_S,M) \stackrel{\sim}{\longrightarrow} H^*_{\mathrm{et}}(\Z[\frac{1}{S}], M)$$
  is an isomorphism, 
and we will denote this group    by $H^*(\Z[\frac{1}{S}], -)$ where convenient,
without explicitly writing the subscript ``etale.''
Similarly for $H^*(\Q_q, -)$  and  $H^*(\mathbb{F}_q, -)$.

   \subsection{Cohomology with local conditions} \label{local-cohomology}
   See \S \ref{GlobalDuality} for details:  
   
 Let $M$ be a $p$-torsion module under $\Gamma_S$.  \index{$H^*_{?}$} 
 Let $S' \subset S$.  
For each $v \in S'$ suppose given a subgroup $\mathfrak{l}_v \subset H^1(\Q_v, M)$.     One defines $H^1_{\mathfrak{l}}(\Z[\frac{1}{S}], M)$ (``cohomology with local conditions'')  to be the subgroup of  $x \in H^1(\Z[\frac{1}{S}], M)$ such that
the restriction  $x_v$ to $H^1(\Q_v, M)$ lies in  $\mathfrak{l}_v$ for every $v \in S'$. 
There is an exact sequence
{\small \begin{eqnarray} \label{long local cohomology}  0 \rightarrow H^1_{\mathfrak{l}}(\Z[\frac{1}{S}], M) \rightarrow H^1 (\Z[\frac{1}{S}], M) \rightarrow \prod_{v \in S'} H^1(\Q_v, M)/\mathfrak{l}_v  \\  
\rightarrow H^2_{\mathfrak{l}}(\Z[\frac{1}{S}], M) \rightarrow H^2(\Z[\frac{1}{S}], M) \rightarrow \prod_{v \in S'} H^2(\Q_v, M)  \rightarrow .
\end{eqnarray}
}where one {\em defines} 
 $H^2_{\mathfrak{l}}(\Z[\frac{1}{S}],M)$ by a cone construction  (see \S \ref{GlobalDuality} or \cite{CHT} -- there is a specific lift of $\mathfrak{l}$ to the chain level being used here,   prescribed in \S \ref{ExampleAB}). 
 
 If $S' = S$ it follows from duality in Galois cohomology (see \S \ref{GlobalDuality}) that
there is a perfect pairing
$$H^1_{\mathfrak{l}}(M) \times H^2_{\mathfrak{l}^{\perp}}(M^*) \longrightarrow \mathbb{Q}_p/\Z_p,$$
where $M^*$ is the dual group of homomorphisms $M \rightarrow \mu_{p^{\infty}}$ and \index{$M^*$}
 $\mathfrak{l}^{\perp}$ is the dual local condition, i.e.\  $\mathfrak{l}_v^{\perp}$ is the orthogonal complement to $H^1_{\mathfrak{l}_v}$ inside $H^1(\Q_v, M^*)$,
 with respect to the pairing of local Tate duality. \index{$\mathfrak{l}^{\perp}$} \index{$H^*_{\mathfrak{l}}$}
 
 \subsection{Fontaine--Laffaille and the $f$ cohomology} \label{FLreview}  
 
    We briefly recall the theory of Fontaine and Laffaille.   Fix once
    and for all an interval $[a,b] \subset \mathbb{Z}$ of
    ``Hodge weights'', 
   where $b-a \leq p-2$; for us, the interval $[-\frac{p-3}{2}, \frac{p-3}{2}]$ will do nicely. 
We say that a representation of $\Gal(\overline{\Q}_p/\Q_p)$ with coefficients in a finitely generated $\Z_p$-module $M$ is {\em crystalline}
   if it is  isomorphic to the quotient $L_1/L_2$, where $L_1, L_2$ are lattices inside a crystalline representation  $L_1 \otimes \Q_p=L_2 \otimes \Q_p$ with weights in the interval $[a,b]$.

   Fontaine and Laffaille describe an explicit category $MF$ of semilinear algebraic data
   (the ``Fontaine--Laffaille modules'').  See \cite{FL} or for a summary \cite[p.\ 363]{BK}.   
There is an equivalence of categories
   $$ MF \stackrel{FL}{ \longrightarrow} \mbox{ crystalline Galois modules} $$

 If $M$ is  a crystalline  module, we define the ``$f$-cohomology'' $H^1_f(  \Q_p, M)$ to be that 
 subset of the Galois cohomology group $H^1(\Q_p,M)$ consisting of classes classifying extensions $M \rightarrow \tilde{M} \rightarrow \Z_p$
 with $\tilde{M}$ crystalline. This is known to be a subgroup of $H^1(\Q_p, M)$.

{\em Globally} for $M$ a module under $\pi_1 \Z[\frac{1}{S}]$ we define the
 global $f$-cohomology
\begin{equation} \label{globalfcohomology}H^*_f(\Z[\frac{1}{S}], M)\end{equation}  by   imposing conditions $\mathfrak{l}$ only at $v=\{p\}$,  taking $\mathfrak{l}_p= H^1_f(\Q_p, M) \subset H^1(\Q_p, M)$. 

{\bf Remark.} Note global $f$-cohomology is often used to denote imposing, in addition to crystalline conditions at $p$, {\em unramified} conditions at all places except $p$. 
If we want to use this convention we will write instead $H^*_f(\Q, M)$.  When we write $H^*_f(\Z[\frac{1}{S}], M)$
or similar we always mean that we impose {\em only} the crystalline condition at $p$, and impose no restrictions at $v \neq p$. 

        \subsection{The arithmetic manifold $Y(K)$ and the existence of Galois representations} \label{YKdef}  \index{$Y(K)$} 
     
     We recall below the definition of the arithmetic manifold $Y(K)$,  associated to an algebraic $\Q$-group  $\mathbf{G}$ 
 whose dual is $G$, and an open compact subgroup  $K$ of its adelic points. 

  Galois representations into the group $G$ are related, by the Langlands program, to automorphic forms on a group {\em dual} to $G$.  
  Therefore, let $\mathbf{G}$ be the split reductive $\Q$-group whose root datum is dual to that of $G$,\footnote{We apologize for not introducing the customary ``dualizing'' signs,
  such as $\vee$, into the notation here:    it seems  too  much of a conflict with standard notation to introduce  use  such notation for the group on which the automorphic forms live;
  and it would introduce too much typographical difficulty to call our $G$ instead $G^{\vee}$. Since $\mathbf{G}$ is always used in boldface and comes up rather rarely,
  we hope this will not be too confusing for the reader.}  and let $\mathbf{B} \supset \mathbf{T}$ be a Borel and maximal split torus in $\mathbf{G}$.\index{$\mathbf{G}$} \index{$\mathbf{T}$}
  (Note that, to be more canonical, we could always replace $\mathbf{T}$ in our considerations with the torus quotient of $\mathbf{B}$).

 Fix a maximal compact subgroup $K_{\infty}^{\circ}$ of the connected component of $\mathbf{G}(\R)$. 
 A ``level structure'' $K$ for us is an 
   open compact subgroup $K \subset \mathbf{G}(\adele_f)$ of the points of $\mathbf{G}$
 over the finite adeles $\adele_f$; we will only consider examples
where $K$ is a product $K = \prod_{q} K_q$, with $K_q \subset \mathbf{G}(\Q_q)$,
 and where the ``level is prime to $p$'' --    we require that  $K_p$ is a hyperspecial  subgroup  of $\mathbf{G}(\Z_p)$
 (this means: the $\Z_p$-points of a smooth reductive model over $\Z_p$). 
 
 For such a level structure $K$,  define the ``arithmetic manifold of level $K$:'' 
 $$Y(K) = \mathbf{G}(\Q) \backslash \mathbf{G}(\adele) / K_{\infty}^{\circ} K.$$

 In fact, we can identify $Y(K)$ with a finite union of quotients of the  (contractible) symmetric space $\mathcal{Y}$ for $\mathbf{G}(\R)$,  by various arithmetic subgroups $\Gamma_i \leqslant \mathbf{G}(\Q)$:
 
 $$ Y(K) = \coprod_{i} \Gamma_i \backslash  \mathcal{Y}. $$
 In particular, the orbifold  cohomology of $Y(K)$ is the direct sum $\bigoplus_i H^*(\Gamma_i)$. 
 It is one candidate for the ``space of modular forms for $\mathbf{G}$.''   In particular, the Langlands program predicts that, to a Hecke eigenclass on $H^* (Y(K),-)$,
 there should be attached a Galois representation into $G$. We now formulate this  prediction precisely, in the form of the conjecture below.

\subsection{}
\label{sec:galois-rep-conjecture}  We now formulate more precisely the conjecture on Galois representations to be used.  We follow the version used by Khare--Thorne \cite{KT};
 the version of Calegari--Geraghty does not use derived categories. 
 
 Let $E$ denote a pro-$p$ coefficient ring, with a map $E \rightarrow k$. 
  
 We consider  pairs  $(K \vartriangleleft K')$ of  level structures, with $\Delta := K'/K$ abelian.
 Now $\Delta$ acts in the natural way on $Y(K)$,  and 
we  may consider $C_*^{\Delta}(Y(K), E)$,   the chain complex of $Y(K)$ with $E$ coefficients, 
 as an object in the derived category of $E \Delta$-modules.    Each Hecke operator gives an endomorphism of this object. 
Let $\widetilde{\mathrm{T}}_K$ be the ring of endomorphisms thus generated by all (prime-to-the-level) Hecke operators.  It is commutative.  
Also  by \cite[Lemma 2.5]{KT}
the natural map from $\widetilde{\mathrm{T}}_K$ to the usual (homological) Hecke algebra has nilpotent kernel.

\begin{Conjecture} \label{GaloisRepConjecture}  
Fix a surjection  $\widetilde{\mathrm{T}}_K \twoheadrightarrow k$ with  kernel the maximal ideal 
$\mathfrak{m}$ of $\widetilde{\mathrm{T}}_K$.  Let $\widetilde{\mathrm{T}}_{K, \mathfrak{m}}$
be the completion of $\widetilde{\mathrm{T}}_{K}$ at the maximal ideal $\mathfrak{m}$. 

Then there exists a Galois representation
$\overline{\sigma}: \mathrm{Gal}(\overline{\Q}/\Q) \longrightarrow   G(k)$ with the following properties:

\begin{itemize} 
\item[(a)] For all primes $q \neq p$ at which $K_q$ is hyperspecial, the representation $\sigmabar$ is unramified; moreover,  if we fix a representation $\tau$ of $G$,
 the trace $\mathrm{trace}(\tau \circ \sigma)(\mathrm{Frob}_q) \in   k$
coincides with the image of  the  associated Hecke operator 
$T_{q, \tau} \in \widetilde{\mathrm{T}}_K$ under $\widetilde{\mathrm{T}}_K \rightarrow k$.
(Explicitly, $T_{q, \tau}$  comes from the Satake isomorphism of the Hecke algebra at $q$ with the representation ring of $G$). 

\item[(b)]
$\sigmabar$ is odd at $\infty$, i.e., the image of complex conjugation   in $G(k)$ may be lifted to an involution in $G(W(k))$
whose trace, acting on $\mathrm{Lie}(G)$ in the adjoint action,  is minimal amongst all involutions.

 \item[(c)] At $p$, let $\bar{\sigma}_p: \Gal(\overline{\Q}_p/\Q_p) \rightarrow G(k)$ 
be the restriction of $\sigmabar$.  Then $\Ad( \bar{\sigma}_p): \Gal(\overline{\Q_p}/\Q_p) \rightarrow
\GL(\mathfrak{g})$ is torsion crystalline (see \S \ref{FLreview}).

\end{itemize}

Moreover, if  the image of $\sigmabar$ is large enough, e.g.\ if
\begin{equation} \label{bigimage} \mathrm{image}(\sigmabar)\supset \mathrm{image} (G^{\mathrm{sc}}(k) \rightarrow G(k)), \end{equation}
where $G^{\mathrm{sc}}$ is the simply connected cover of $G$, 
  then $\sigma$ can be lifted to a representation 
  $\sigma: \mathrm{Gal}(\overline{\Q}/\Q) \longrightarrow   G(\widetilde{\mathrm{T}}_{K, \mathfrak{m}})$
which continues to  have the obvious analogue of property (a), and moreover: 
 
 \begin{itemize} 
 \item[(d)] At places $v$ for which $K_v \neq K_v'$, 
 one has the ``local-global compatibility'' formulated later: see  Assumption \ref{LGasssumption} in \S~\ref{mprimeconstruc}.  This assumption is in fact only formulated for
 specific pairs $(K_v, K'_v)$, and these are the only type for which we will apply it; we regard (d) as being vacuous in other cases.

  \item[(e)]  
  Let $\mathrm{Def}_{\bar{\sigma}_p}$ be the usual (Mazur) deformation functor, from usual Artin rings augmented over $k$, to sets.
  There exists   a unobstructed  subfunctor  $\mathrm{Def}^{\crys}_{\bar{\sigma}_p}  \subset \mathrm{Def}_{\bar{\sigma}_p}$ 
  (``unobstructed'' means that \  square zero extensions $\tilde{\rmA} \rightarrow \rmA$ induce surjections  $\mathrm{Def}^{\crys}_{\bar{\sigma}_p}(\tilde{\rmA}) \twoheadrightarrow \mathrm{Def}^{\crys}_{\bar{\sigma}_p}(\rmA)$)
 with tangent space $H^1_f(\Q_p, \Ad \rho_p) \subset H^1(\Q_p, \Ad \rho_p)$,  such that
   $$ \sigma \in \mathrm{Def}^{\crys}_{\bar{\sigma}_p}(\widetilde{\mathrm{T}}_{K, \mathfrak{m}}).$$

 \end{itemize} 
 \end{Conjecture}

Note that asking about $\widetilde{\mathrm{T}}$ is a slightly stronger statement than asking about the usual Hecke algebra $T$. 
The idea of considering $\widetilde{\mathrm{T}}$ is due to Khare and Thorne \cite{KT}, and evidence for this  stronger statement  has been given by Newton and Thorne \cite{NT}.

The necessity of condition (e) is that the notion of crystalline deformation has not been explicated for a general group, although it  has been verified in some important cases
(e.g.\ $\GL_n$ in \cite{CHT}, and the case of $\mathrm{GSp}$ is apparently analyzed in Patrikis' undergraduate thesis). We just assume this as an axiom. 
Note also that our assumption (c) that the adjoint representation is crystalline forces $p$ to be ``large'' relative
to $\mathbf{G}$:  one expects the Hodge weights of $\Ad(\bar{\rho}_p)$ to lie in the interval $[1-h, h-1]$
where $h$ is the  Coxeter number of $G$. 
 
\subsection{Taylor--Wiles primes} \label{TWprimesdef}
We summarize briefly the idea of Taylor--Wiles primes, which will play a central role in our analysis. 
Let 
$$\rho: \pi_1 \Z[\frac{1}{S}] \rightarrow G(k)$$
be a representation.

   A {\em Taylor-Wiles prime}
 $q$, for  the Galois representation $\rho$, is a prime $q \notin S$ with the property $q = 1$ in $k$ and
 $\rho(\mathrm{Frob}_q)$ is conjugate to a {\em strongly regular} element of $T(k)$ -- as usual, a strongly regular element  $t \in T(k)$ is  an element $t$ whose centralizer in $G$ coincides with $T$. 
   \index{Taylor-Wiles prime} 
  Therefore, if $G = \PGL_n$,  a Taylor--Wiles prime is simply one for which 
  the Frobenius has a representative in $\GL_n(k)$ with distinct eigenvalues, all belonging to $k$. 
   
 In particular, if $q$ is a Taylor-Wiles prime, one may choose a representation \begin{equation} \label{rqTd} \rho_{\Q_q}^{T}: \pi_1 \Q_q \rightarrow T(k)\end{equation}   factoring through $\pi_1 \Z_q$, 
   such that the composition of $\rho_{\Q_q}^T$ with $T \hookrightarrow G$ is isomorphic to $\rho_{\Q_q}$.  We denote by $\rho_{\Z_q}^T$
   the corresponding representation of $\pi_1 \Z_q$. Later we will study the $T$-valued deformation theory of $\rho_{\Q_q}^T$. 
   This depends on the choice of $\rho_{\Q_q}^T$; it is often convenient to think of a Taylor--Wiles prime as  being equipped with  such a choice.    Thus our basic objects are actually pairs $(q, t)$,
   where $t \in T(k)$ is regular and conjugate to $\rho(\mathrm{Frob}_q)$. That then fixes $\rho_{\Q_q}^T$,
by requiring that $\rho_{\Q_q}^T$ carry a Frobenius element to $t$.    
   
   What is really important to us are certain sets of Taylor--Wiles primes:
   \begin{Definition}  \label{Def:TWdefn}Let $\rho: \pi_1 \Z[\frac{1}{S}] \rightarrow G(k)$. 
An {\em    allowable   Taylor--Wiles datum of level $n$}, for  the Galois representation $\rho$,  
is a set of Taylor--Wiles  primes $Q=\{\ell_1, \dots, \ell_k\}$, disjoint from $S$, and each equipped with a regular element $t_{\ell_i} \in T(k)$ conjugate to $\rho(\mathrm{Frob}_{\ell_i})$,  
and  satisfying the  further conditions:   \begin{itemize}
 \item[(a)]  $p^n$ divides each $\ell_i-1$ 
\item[(b)] 
We have $H^2_{\mathfrak{l}}(\Z[\frac{1}{S  Q}], \Ad \rho)=0$, where   we impose (see \S \ref{local-cohomology},
\S \ref{FLreview})  local conditions $\mathfrak{l}_v $  at $v \in S\cup Q$,  namely  $\mathfrak{l}_v = \begin{cases}   H^1_f, \ v =p \\ 
H^1, v \in Q   \\
0, v \in S-p.  \end{cases} $.   \end{itemize} 
\end{Definition}

Note that although the choice of elements of $T(k)$ is part of the datum, we will often just informally  say ``let $Q$ be a Taylor--Wiles datum,'' with the understanding
that one actually carries along this extra information.  
More explicitly, the vanishing condition in (b) means that  
 in the sequence   (we write $SQ$ instead of $S \coprod Q$): 
\begin{equation}  
  \label{vanishing_sequence}  
  \begin{aligned}
    & H^1(\Z[\frac{1}{SQ}], \Ad \rho) \stackrel{A}{\longrightarrow}
    \frac{H^1(\Q_p, \Ad \rho)}{H^1_f(\Q_p, \Ad \rho)} \oplus
    \bigoplus_{v \in S-\{p\}} H^1(\Q_v, \Ad \rho) \\ & \rightarrow
    \underbrace{ H^2_{\mathfrak{l}} }_0 \rightarrow
    H^2(\Z[\frac{1}{SQ}], \Ad \rho) \stackrel{B}{\longrightarrow}
    \bigoplus_{SQ } H^2(\Q_v, \Ad \rho),
  \end{aligned}
\end{equation}   
that $A$ is surjective and $B$ is injective.  
Actually, the cokernel of $B$
is dual to $H^0(\Ad^* \rho(1))$; this
will vanish so long as $\rho$ has sufficiently large image, so if $B$ is injective it is an isomorphism.

{\bf Remark. }  One can find allowable Taylor--Wiles data of any level $n$ as long as $\rho$ satisfies a ``big image'' criterion; for example,
it suffices to assume   that $\rho$
restricted to $\Q(\zeta_{p^{\infty}})$ has image 
that contains the image of $G^{\mathrm{sc}}(k) \rightarrow G(k)$, where $G^{\mathrm{sc}}$ is the simply connected cover.
One proves this as in \cite[Proposition 5.9]{Gee}.
In particular, write $E$ for the fixed field of the kernel of $\rho$ and $E' = E(\zeta_{p^n})$. 
The main point of the proof is to check that the restriction map  from the cohomology of $\Z[\frac{1}{S}]$, with $\Ad^* \rho(1)$ coefficients, 
to the corresponding cohomology group for $E'$ is injective. To see this, in turn, it is enough to verify that the cohomology of
$ G  := \Gal(E'/F)$ acting on $\Ad^* \rho(1)$ is trivial.   Let $N$ be the  $(\Z/p)^*$ subgroup naturally  embedded in $\Gal(E'/E)$; it is a normal subgroup of
$G$, and  the claim follows easily from  the fact that $N$ has no invariants on $\Ad^* \rho(1)$. 
   
\section{Deformation-theory notation}  \label{numbertheorynotn2}

          We continue our summary of notation for the rest of the paper. Here we briefly repeat some facts
          of homotopical algebra from the first part and set up notation for deformation rings and deformation functors.

\subsection{Representable functors} \label{intro:functors}
  As discussed at length in \S \ref{sec:functors-artin-rings} we will deal extensively with objects of the category 
 \begin{equation} \label{Mdef} \mathcal{M} = \mbox{simplicially enriched functors $\Art_k \longrightarrow s\Sets$},\end{equation}
 namely, the deformation functors associated to various Galois representations. (We
 will almost always want to work with functors that are valued in Kan simplicial sets.)
 
 A morphism is a natural transformation of simplicially enriched functors. 
 There is a natural notion of weak equivalence for morphisms, namely, an object-wise weak equivalence, i.e.\ $F \rightarrow G$ is a weak equivalence
 when $F(A) \rightarrow G(A)$ is a weak equivalence for all $A \in \Art_k$.  Indeed $\mathcal{M}$ has a compatible model category structure (mentioned in  \cite[\S 2.3.1]{Toen}) with object-wise fibrations but we won't use it.     
 
 We will however use the following shorthand: Given objects $F,G \in \mathcal{M}$, we will write
 $$F \dashrightarrow G$$
 to abbreviate an explicit zig-zag of maps
\begin{equation} \label{roofy} F \stackrel{\sim}{ \leftarrow }F_1 \rightarrow F_2 \stackrel{\sim}{ \leftarrow} F_3 \rightarrow F_4 \cdots \rightarrow G, \end{equation} 
  where all left arrows are weak equivalences.  {\em Wherever we use this notation we have in mind an explicit zig-zag, not merely an morphism in the homotopy category of $\mathcal{M}$. }
In practice it would be overly cumbersome to write out this zig-zag in every case, thus our shorthand.

 Recall that, for any such zig-zag, we can find a ``roof'': a functor $F^*$ equipped with a natural weak equivalence $F^* \stackrel{\sim}{\rightarrow} F$,
 and a map $F^* \rightarrow G$; we can replace the long-zig zag above by \begin{equation} \label{shortzigzag} F  \stackrel{\sim}{\leftarrow} F^* \rightarrow G.\end{equation}
 Explicitly, we can take $F^*(A)$ to be the homotopy limit of the diagram $F(A) \leftarrow F_1(A) \rightarrow F_2(A) \dots$, that is to say,
 a collection of points  $x \in F(A), z_1 \in F_1(A), z_2 \in F_2(A)$ etc., together with a collection of paths: a path from the image of $z_1$ to $x$, etc.

We will use homotopy fiber products of functors:  for $F_1 \rightarrow F_2 \leftarrow F_3$
we denote by $F_1 \times_{F_2}^h F_3$ the  functor that assigns to $A \in \Art_k$
the homotopy pullback $F_1(A) \times_{F_2(A)}^h F_3(A)$.  (The homotopy pullback is defined in Definition \ref{homotopy pullback square Example}:
informally speaking, it consists of  points $v_i \in F_i(A)$ together with paths in $F_2$ between $v_2$ and the images of $v_1, v_3$.)
We say that a square in $\mathcal{M}$  \begin{equation} \label{hpssquarereview}
\xymatrix{
G
  \ar[r] \ar[d]  & F_1  \ar[d] 
  \\ F_3
 \ar[r]  &  F_2
}
\end{equation} 
is a homotopy pullback square if the induced map $G_1 \rightarrow F_1 \times^h_{F_2} F_3$
is an object-wise weak equivalence.

Recall (Definition \ref{defn:pro-rep})
 also that a functor is {\em pro-representable} if it is naturally weakly equivalent to 
a functor of the form
$$ \varinjlim_{\alpha} \Map_{\Art_k} (\mathcal{R}_{\alpha}, -)$$
where the $\mathcal{R}_{\alpha}$ are a projective system of  cofibrant objects of $\Art_k$, indexed by  $ \alpha \in \mathrm{Ob} \mathcal{C}$ for some cofiltered category $\mathcal{C}$. 
 (Each  $\Map$ space above is the simplicial set of maps in the category $\SCR_{/k}$, i.e.\ those maps of simplicial commutative rings that commute with the map to $k$.)

 In this case, the pro-ring $\alpha \mapsto \mathcal{R}_{\alpha}$ is said to represent the  given functor (for more precision, see
 Lemma \ref{lemma: pro representing objects}); we note here that 
 such a representing pro-ring is not determined up to unique isomorphism, but it is unique in a homotopical sense (see Lemma
\ref{Lemma:represent-natural-trans} or   \eqref{pistar iso}). We will nonetheless slightly abuse notation
and refer to ``the'' representing ring for a functor. Note   that although we often write ``deformation ring'' or ``representing ring,''
these are always permitted to be pro-objects of $\SCR$.

We say ({\em loc.\ cit.}) that a functor is {\em sequentially pro-representable} if it is possible to choose the indexing category $\mathcal{C}$ to be the natural numbers.
By  Corollary \ref{cor:Corollary hurewicz analogue}  this is automatically the case if the tangent complex of the functor is of countable dimension.
{\em This will be the case in all the applications to number theory,  and thus we will freely take our functors to be represented by pro-systems indexed by the natural numbers.} 
 
 The sequentially indexed projective system $\mathcal{R}_n \ (n \in \mathbb{N})$ is said to be {\em nice} when
 each $\mathcal{R}_n$ is cofibrant and all the transition maps $\mathcal{R}_n \rightarrow \mathcal{R}_m$ are fibrations; this is a 
 special case of Definition \ref{nice ring definition}.   Any sequentially pro-representable functor can 
 be represented by a nice pro-ring  (Lemma 
\ref{lemma:replace-by-nice}). 

It will be therefore be convenient to use the following notation: for a pro-object of $\Art_k$, say  $\mathcal{R} =  ( \mathcal{R}_{\alpha})$,    we understand the functor $\Map$ to mean      \index{cofibrant replacement functor $R\mapsto R^c$}
\begin{equation} \label{map_definition} \Map(\mathcal{R}, -) := \colim_{\alpha} \Map_{\Art_k}(\mathcal{R}_{\alpha} , - ).\end{equation}
 However, if $\mathcal{R}$ is not level-wise cofibrant this definition is not homotopy invariant,  and we  should  only use the functor  
 after first applying a level-wise cofibrant replacement.  We fix throughout a cofibrant replacement functor
 $c$ on the category $\SCR$; we will write it as $\mathcal{S} \mapsto c(\mathcal{S})$ or (more often)
 $\mathcal{S} \mapsto \mathcal{S}^c$ on the category $\SCR$; by  $\mathcal{R}^c$ we shall mean the pro-simplicial ring
 $(\mathcal{R}_{\alpha}^c)$, obtained by applying $c$ level-wise.

  We will quite frequently want to relate    the functor represented by $\mathcal{R}$ to the functor represented by $\pi_0 \mathcal{R}$.  We shall need this comparison both when $\mathcal{R} \in \smallC_k$ and when $\mathcal{R} \in \text{pro-}\smallC_k$; in the latter case $\pi_0 \mathcal{R}$ is a pro-finite set.   Even if $\mathcal{R} \in \smallC_k$ is cofibrant, the ordinary ring $\pi_0 \mathcal{R}$, regarded as a (constant) simplicial ring is not, but using the chosen cofibrant approximation functor we have a comparison zig-zag $(\pi_0 \mathcal{R})^c \leftarrow \mathcal{R}^c \rightarrow \mathcal{R}$, inducing a zig-zag of functors
  $$  \Map_{\Art_k}((\pi_0 \mathcal{R})^c, -) \rightarrow \Map_{\Art_k}(\mathcal{R}^c, -) \stackrel{\sim}{\longleftarrow} \Map_{\Art_k}(\mathcal{R}, -).$$
  Thus, for example, if $\mathcal{F}$ is representable with representing object $\mathcal{R}$,
  we get 
  \begin{equation} \label{sbpq} \Map_{\Art_k}((\pi_0 \mathcal{R})^c, -) \dashrightarrow \mathcal{F}.\end{equation}
    If $\mathcal{R}$ is a pro-object of $\Art_k$ the same construction applies levelwise.

  Finally,
   we will often encounter situations of the following type. Consider a sequence of maps in pro-$\Art_k$ \begin{equation} \label{invertingringmaps} \mathcal{R}_1   \mathop{\leftarrow}_{g}^{\sim}   \mathcal{R}_2   \stackrel{h}{\rightarrow} \mathcal{R}_3. \end{equation}
  where $\sim$ in the  first map means that it induces an objectwise weak equivalence $\Map(\mathcal{R}_1, -) \rightarrow \Map(\mathcal{R}_2, -)$
  of functors $\Art_k \rightarrow s\Sets$. 
If $\mathcal{R}_3$ is nice, we can ``invert'' the first arrow, in a homotopical sense,  to get a morphism $\mathcal{R}_1 \rightarrow \mathcal{R}_3$ in pro$\Art_k$,
in the sense that we produce
$$ f: \mathcal{R}_1 \rightarrow \mathcal{R}_3$$ 
such that $f \circ g$, although not literally equal to $h$, induces the same map on $\pi_0 \Map(-, \mathcal{A})$.
Thus, for example, $f \circ g$ and $h$ induce the same map on tangent complexes, 
and the same map in the pro-homotopy category (see discussion of \S \ref{sec:hocat}).

To produce $\mathcal{R}_1 \rightarrow \mathcal{R}_3$, consider the map of functors $$\Map(\mathcal{R}_3, -)  \rightarrow 
  \Map(\mathcal{R}_2, -)  \stackrel{\sim}{\leftarrow} \Map(\mathcal{R}_1, -)$$
  and as in \eqref{shortzigzag} replace $\Map(\mathcal{R}_3, -)$ by a naturally weakly equivalent functor $\mathcal{F}_3 \stackrel{\sim}{\rightarrow} \Map(R_3, -)$
  together with a map $\mathcal{F}_3 \rightarrow \Map(\mathcal{R}_1, -)$. Now $\mathcal{R}_3$ is still a representing ring for $\mathcal{F}_3$: there is 
a natural weak equivalence $ \hocolim \Map(\mathcal{R}_{3,\alpha}, -) \rightarrow \mathcal{F}_3$ by Lemma \ref{Lemma:hocolim-is-cofibrant}. 
Then apply Lemma \ref{Lemma:represent-natural-trans}
to produce a map $\mathcal{R}_1 \rightarrow \mathcal{R}_3$; the resulting diagrams of functors all commute objectwise in the homotopy category.

 \subsection{Tangent complex}  \label{tcrecall}

For a (possibly pro-) simplicial ring $\mathcal{R}$ augmented over $k$,  we define the tangent complex $\tangent \mathcal{R}$
to be the tangent complex of the associated functor $\Map(\mathcal{R}, -)$, where
the maps are taken in $\SCR/k$, and we use $\tangent^i \mathcal{R}$ to abbreviate $\pi_{-i} \tangent \mathcal{R}$. 
(As before we shall apply this definition only when $\mathcal{R}$ is cofibrant.)
Recall that $\tangent \mathcal{R}$ is a chain complex with degree-decreasing differentials, with homology
supported entirely in degrees $\leq 0$; therefore $\tangent^i \mathcal{R}$ is supported in $i \geq 0$.

  To be explicit,  if $\mathcal{R} = (\mathcal{R}_{\alpha})$,  where $\mathcal{R}_{\alpha}$ are all cofibrant, we have 
\begin{equation} \label{tidef}\tangent^i \mathcal{R} = \varinjlim_{\alpha} \pi_{j-i} \Map_{\Art_k}(\mathcal{R}_{\alpha},  k \oplus k[j]), \mbox{ any fixed } j  \geq  i. \end{equation}

Let $\mathrm{D}^i_A(B, M)$ denote the Andr{\'e}-Quillen cohomology functors  of the $A$-algebra $B$ with coefficients in the $B$-module $M$ (see discussion before
Example \ref{ex:tangent-complex-Hom-R}, or \cite{Quillen} -- in particular $\mathrm{D}^0$ is ``usual'' derivations). 
In the case of a simplicial ring $\mathcal{R}$ augmented over $k$, we have  by Example \ref{ex:tangent-complex-Hom-R} an identification  
   $$ \tangent^i \mathcal{R}  =  \Der^{i}_{\Z}(\mathcal{R}, k) $$ %
and in the ``pro'' case there is a corresponding assertion with $\varinjlim$.  

If $\mathcal{R} \rightarrow \mathcal{S}$ is a morphism of simplicial commutative rings (we won't need the ``pro'' case) we get a long exact sequence  
\begin{equation} \label{AQ} \Der^{i}_{\mathcal{R}}(\mathcal{S}, k) \rightarrow \tangent^i\mathcal{S} \rightarrow \tangent^i \mathcal{R} \stackrel{[1]}{\rightarrow} \end{equation} 
(this follows from the long exact sequence  for Andr{\'e}-Quillen cohomology, \cite[Theorem 5.1]{Quillen}). 

Let $\mathcal{R} \in \Art_k$. 
Let $\pi_0 \mathcal{R}/p^n$ be the quotient of $\pi_0 \mathcal{R} $ by the ideal generated by $p^n$
in the usual (non-derived) sense of quotient of a ring by an ideal; we  shall consider the result  as a discrete 
simplicial ring.  
  \label{pagerefB}  Then we  have natural isomorphisms
\begin{equation} \label{pagrefB}\tangent^0 \mathcal{R} \cong \tangent^0 (\pi_0 R) \cong \tangent^0 (\pi_0 R/p^n) \ \ (n \geq 1), \end{equation} 
because  the map $\mathcal{R} \rightarrow \pi_0 \mathcal{R} \rightarrow \pi_0 \mathcal{R}/p^n$ 
 all induce isomorphisms when we consider homomorphisms into $k[x]/x^2$, considered as a discrete simplicial ring.
 Indeed since any homomorphism $\pi_0\mathcal{R} \rightarrow k[x]/x^2$ is trivial on $(p, \mathfrak{m}_{\pi_0 R}^2)$, 
 with $\mathfrak{m}_{\pi_0 \mathcal{R}}$ the  kernel of $\pi_0 \mathcal{R} \rightarrow k$, we could have equally well 
  replaced $\pi_0 \mathcal{R}/p^n$ by $\pi_0 \mathcal{R}/\mathfrak{b}$ so long as $\mathfrak{b} \subset (p, \mathfrak{m}_{\pi_0 \mathcal{R}}^2)$. 

Similarly, the map $\mathcal{R} \to \pi_0\mathcal{R}$ always induces an injection
\begin{equation} \label{tangent-inclu} \tangent^1 (\pi_0 \mathcal{R}) \hookrightarrow \tangent^1 \mathcal{R}.\end{equation}

 Informally  \eqref{tangent-inclu} arises from the fact that $\pi_0 \mathcal{R}$ can be presented as a $\mathcal{R}$-algebra
 by freely adding relations in degree $2$ and higher to kill all higher homotopy. More formally, 
this  follows from Proposition \ref{Prop:hurewicz}
 part (ii):  
 writing $\mathcal{R}' = \pi_0 \mathcal{R}$, considered
as a discrete simplicial ring, 
we have an exact sequence
$$ \cdots \rightarrow \pi_1(\mathcal{R}') \rightarrow \pi_1(\mathcal{R}', \mathcal{R})   \rightarrow \pi_0 \mathcal{R} \rightarrow \pi_0 \mathcal{R}' \rightarrow \pi_0(\mathcal{R}', \mathcal{R}) \rightarrow 0.$$
and thus  $\pi_j(\mathcal{R}', \mathcal{R}) = 0$ for $j = 0,1$;
now the quoted Proposition shows that 
  $\pi_{-1} \tangent (\mathcal{R}', \mathcal{R})  = 0,$
  and then    
\eqref{t-long-exact-sequence}
 gives
the desired result.
 
 The assertion of \eqref{pagrefB} and \eqref{tangent-inclu} continue to hold for a pro-simplicial commutative ring  $(\mathcal{R}_{\alpha})$, where we define:
\begin{equation} \label{pi0def} \pi_0 \mathcal{R} =  \mbox{ the pro-ring} (\alpha \mapsto \pi_0 \mathcal{R}_{\alpha}).\end{equation}

  \subsection{The Galois representation } \label{setup}

  We will be interested in analyzing the derived deformation ring of a representation \index{$\rho$} 
      $$\rho: \Gamma_S \rightarrow G(k),$$
Eventually $\rho$ will be a Galois representation associated to a cohomology class,
      i.e.\ it will be the $\sigmabar$ from  Conjecture \ref{GaloisRepConjecture},
      and it will moreover be required to satisfy some ``big image'' and ``nice at $p$'' conditions
      that are formulated in \S \ref{minimalevel}. However, for the moment, 
    we do not impose any restrictions on $\rho$.

  We write $\Ad \  \rho$ for the   adjoint representation of $\Gamma_S$
        on the Lie algebra $\mathfrak{g}$ of $G$ over $k$, 
        i.e.\ the composition of   $\rho$ with the adjoint representation of $\mathbf{G}$: 
  $$\Ad \ \rho: \Gamma_S \rightarrow \GL(\mathfrak{g}).$$
  We will also be concerned with the $k$-linear dual of this representation, which acts on the $k$-linear dual
  $\mathfrak{g}^*$ to $\mathfrak{g}$; we denote it $\Ad^*$:
  $$ \Ad^* \ \rho:  \Gamma_S \rightarrow \GL(\mathfrak{g}^*).$$

  Our basic object of interest will be the deformation functor 
  corresponding to  deforming $\rho: \Gamma_S \rightarrow G(k)$.  
In the notation of Definition \ref{defn-repfunctor} (iii) this is 
the functor  $\mathcal{F}_{X, \rho}: \Art_k \rightarrow s\Sets$ where\footnote{\label{explicitetalehomotopytype}in fact, for our purposes, we could also replace $X$ by the pro-simplicial set $X'$
  which is the inverse system $(BG_{\alpha})_{\alpha}$, where we write $\pi_1 \Z[\frac{1}{S}] = \varprojlim_{\alpha} G_{\alpha}$
  as an inverse limit of finite groups. There is a natural map $X \rightarrow X'$ which need not be an isomorphism; nonetheless,
  they give the same deformation functor, because the {\'e}tale cohomology groups coincide with $p$-torsion coincides by the equality of {'e}tale and Galois cohomology with such coefficients \cite[Proposition 2.9]{MilneADT}.} 
$$ X = \mbox{{\'e}tale homotopy type of $\Spec \Z[\frac{1}{S}]$} \in \mathrm{pro}\left(s\Sets\right).$$

 We will denote this deformation functor by     $\mathcal{F}_{\Z[\frac{1}{S}], \rho}$. Explicitly it's given as 
\begin{equation} \label{explicit Frho} \mathcal{F}_{\Z[\frac{1}{S}], \rho}: A \mapsto   \mbox{ fiber of $\varinjlim_{\alpha} \Maps_{s\Sets}( X_{\alpha},   BG(A))$ above $\rho$},\end{equation}
  where we wrote $X \left(= \mbox{{\'e}tale homotopy type of $\Z[\frac{1}{S}]$}\right)$ as the pro-simplicial set $(X_{\alpha})$, and 
  where $BG(A)$ for a simplicial ring $A$ has been described in 
  Definition \ref{BG simplicial ring definition}.
  Informally $\mathcal{F}_{\Z[\frac{1}{S}], \rho}$  sends an Artinian simplicial ring $A$  
  to  a simplicial set of ``conjugacy classes of deformations of $\rho: \Gamma_S \rightarrow G(k)$ to $A$.'' 
  
  There is also a framed
  version where one does not quotient by conjugacy:
     $$\mathcal{F}_{\Z[\frac{1}{S}], \rho}^{\fram}: \Art_k \rightarrow s\Sets$$
which has been described in Definition \ref{defn-repfunctor} (i),
but informally sends an Artinian simplicial ring $A$  
  to  a simplicial set of ``lifts of $\rho: \Gamma_S \rightarrow G(k)$ to $A$,''
  i.e.\ we don't quotient by conjugacy.

Just as for rings (\S \ref{tcrecall})  we use the notation $\tangent^j \mathcal{F}_{\Z[\frac{1}{S}], \rho}$ for the $(-j)$th homotopy group $\pi_{-j} \tangent \mathcal{F}_{\Z[\frac{1}{S}], \rho}$ \index{$\tangent^j$}
  of the tangent complex of $\mathcal{F}_{\Z[\frac{1}{S}],\rho}$.  Recall  that this is $\pi_0\mathcal{F}_{\Z[\frac{1}{S}], \rho}(k \oplus k[j])$ by definition and that we have an identification    (Lemma \ref{Lemma:tangent space computation}, specialized to the case of trivial center)   $$ \tangent^j \mathcal{F}_{\Z[\frac{1}{S}], \rho} =   \pi_{-j} \tangent \mathcal{F}_{\Z[\frac{1}{S}], \rho} \cong  H^{j+1} (  \Z[\frac{1}{S}],  \Ad \rho),   j \geq -1 $$  between the  the homotopy groups of the tangent complex of $\mathcal{F}_{\Z[\frac{1}{S}],\rho}$ and the 
  adjoint cohomology.   Recall our convention above that $H^{*}(\Z[\frac{1}{S}], -)$ denotes the {\'e}tale cohomology.

  Recall  (\S \ref{tangent complex setup})   that the groups $\tangent^j$ are {\em a priori} the homotopy groups of a certain spectrum, but,  by the discussion
  of \S \ref{tangent complex setup}, we can canonically (but not very explicitly) consider them as the homology of a 
  chain complex of $k$-vector spaces, which we shall refer to as the tangent complex.

  We say that $\rho$ is {\em Schur} if its centralizer coincides with the center of $G$.  (For example, for $G =\GL_n$, this would be implied
  by absolute irreducibility of $\rho$; in our case, we are supposing $G$ is adjoint, and so we are requiring that the centralizer of $\rho$ is trivial.) 
In that case $\tangent^j$ is vanishing for negative $j$ and so Lurie's derived Schlessinger criterion (Theorem
\ref{thm:Lurie-Schlessinger}) 
 implies that $\mathcal{F}_{\Z[\frac{1}{S}],\rho}$ is represented by a pro-object in $\Art_k$;
we have reviewed what that means in 
  \S  \ref{intro:functors}. 
  
   Finally, 
an important remark is that $\mathcal{F}_{\Z[\frac{1}{S}],\rho}$ recovers the usual deformation functor of Mazur
upon passage to $\pi_0$: 
\begin{Lemma} \label{MazurPi0} Let $\usArt_k$ be the category of usual (i.e.\ discrete) Artin local rings $\rmA$ equipped with an identification of their residue field to $k$. 
Suppose also that $\rho$ is Schur, i.e.\ has trivial centralizer. 
Then the functor
\begin{equation} \label{moonb1}  \pi_0 \mathcal{F}_{\Z[\frac{1}{S}],\rho} :  \usArt_k \longrightarrow \mathrm{Sets} \end{equation}
obtained by sending a usual Artin ring $\rmA$ to $\pi_0 \mathcal{F}_{\Z[\frac{1}{S}], \rho}(\rmA)$
 is naturally isomorphic to the usual, underived deformation functor,
 that is to say
\begin{equation} \label{moonb2}  \rmA \rightarrow \mbox{ lifts of $\rho$ to $G(A)$}/\mbox{conjugation by $\ker(G(\rmA) \rightarrow G(k))$,} \end{equation}
\end{Lemma}
This result implies, in particular,  that 
$\pi_0$ of a representing ring for $\mathcal{F}_{\Z[\frac{1}{S}], \rho}$
 recovers Mazur's deformation ring.  There is a very similar result for framed deformation rings, where one does not need to assume that $\rho$ has trivial centralizer.

 \proof 
Let  $X $ be  the etale homotopy type of $\Z[\frac{1}{S}]$. Then the functor
 $\mathcal{F}_{X,G}$  of Definition
 \ref{defn-repfunctor} 
 sends $\rmA$ to the mapping space $\Map_{s\Sets}(X, BG(\rmA))$ (see remark after Definition \ref{BG simplicial ring definition}).
 
 The components of the mapping space $\Map_{s\Sets}(X, BG(\rmA))$ are identified 
 with the set of $G(\rmA)$-torsors over $X$; these are, in turn, in bijection with conjugacy classes of maps
\begin{equation} \label{mb}  \tilde{\rho}: \pi_1 \Z[\frac{1}{S}] \rightarrow G(\rmA).\end{equation}
 
 Thus the functor $\mathcal{F}_{X, \rho}$ of Definition \ref{defn-repfunctor}, sends $\rmA \in \usArt_k$ to the
subset of \eqref{mb}  comprising those $\tilde{\rho}$ whose reduction to $G(k)$ is conjugate to $\rho$;
that is identified with the right-hand side of \eqref{moonb2}.
 \qed

\subsection{Notation for deformation rings and deformation functors} \label{Defringdeffunctors} 
    
    We will use the phrase ``deformation ring'' to mean a ring representing a deformation functor, if the
    functor is representable. For example, the deformation ring of $\rho: \pi_1 \Z[\frac{1}{S}] \rightarrow G(k)$, 
    is, by definition, the  pro-object of $\Art_k$ representing the functor $\mathcal{F}_{\Z[\frac{1}{S}], \rho}$, if that functor is indeed representable.

We need to study not just the deformation ring of $\rho: \pi_1 \Z[\frac{1}{S}] \rightarrow G(k)$, 
but also the deformation  ring of $\rho$ pulled back
   to various other groups,  e.g.\ the fundamental group of $\Q_q$.  This gives rise to a small zoo of deformation functors and rings.
   We briefly summarize the notation for them, although we will also define them again
   as we use them. In particular, we shall follow the

   {\em Notational convention:}  We  will always be deforming the same representation $\rho$, but pulled back to various other groups besides $\pi_1 \Z[\frac{1}{S}]$.
Therefore, we will omit completely $\rho$ from the notation for deformation rings, and rather keep track of
what $\rho$ has been pulled back to.  Thus, we will abridge $\mathcal{F}_{\Z[\frac{1}{S}], \rho}$
simply to $\mathcal{F}_{\Z[\frac{1}{S}]}$, or even to $\mathcal{F}_S$, and similarly for the representing ring; thus 
$$\mathcal{F}_{\Z[\frac{1}{S}]} = \mathcal{F}_S \mbox{ for short }, \mathcal{R}_{\Z[\frac{1}{S}]} = \mathcal{R}_S \mbox{ for short}$$
will denote the deformation functor and deformation ring for $\rho$ considered as a representation of $\pi_1 \Z[\frac{1}{S}]$.   As a warning (see discussion below) we note that the meaning of these notations is changed after \S \ref{minimalevel}: they will always
refer to the  functor and ring with {\em crystalline} conditions imposed.

Now we can pull back $\rho$ under $\pi_1 \Q_q \rightarrow \pi _1 \Z[\frac{1}{S}] $  or, if $q \notin S$, under $\pi_1 \Z_q \rightarrow  \pi_1 \Z[\frac{1}{S}]$; 
  we  denote the resulting representations as $\rho_{\Q_q}$ and $\rho_{\Z_q}$, and the resulting deformation functor (for $\pi_1 \Q_q$ or $\pi_1 \Z_q$, respectively)  \index{$\rho_{\Q_q}$}
  \index{$\rho_{\Z_q}$}
by $\mathcal{F}_{\Q_q}$ or $\mathcal{F}_{\Z_q}$.  
These are defined just as in
\eqref{explicit Frho}, but replacing the role of $\Z[\frac{1}{S}]$ by $\Q_q$ or $\Z_q$. 
These functors are rarely representable. The corresponding framed functors are representable,  leading to
a deformation functor and representing ring  
$$ \mathcal{F}_{\Q_q}^{\fram}, \mathcal{R}_{\Q_q}^{\fram}$$
and similarly for $\Z_q$. 

Later on we will consider certain sets $Q_n$ of auxiliary primes, disjoint from $S$, and  will want to consider $\rho$ as a representation
of $\pi_1 \Z[\frac{1}{SQ_n}]$. Again we can define deformation functors as in \eqref{explicit Frho}, replacing the role of $\Z[\frac{1}{S}]$
by $\Z[\frac{1}{SQ_n}]$.  These deformation functors and rings will be denoted by
$$ \mathcal{F}_{S \coprod Q_n} = \mathcal{F}_n \mbox{ for short}, \mathcal{R}_{S \coprod Q_n}=\mathcal{R}_n  \mbox{ for short},$$
where we will only use the abridged forms when $Q_n$ is understood. 

Finally, from  \S \ref{minimalevel} onwards, we will {\em only} consider the deformation functors with crystalline conditions imposed at $p$
(the precise meaning of this is spelled out in  \S \ref{sec:localconditions}). 
To avoid notational overload we will not explicitly include this in the notation. 
Therefore, from \S \ref{minimalevel} onwards, the notations $\mathcal{F}_{S}, \mathcal{R}_S, \mathcal{F}_{S \coprod Q_n}$ etc.\ always denote
the versions of these functors and rings with  crystalline conditions imposed at $p$.

\subsubsection{Representations with target $T$: the rings $\mathcal{S}$}  Finally, later on, we will study various deformation rings with targets not in the algebraic group $G$ but just in its torus $T$.

In particular, our representation functors
 with target $T$ will {\em never} be representable, because they always have automorphisms; but the framed versions will be. 
  We will use the notation $\mathcal{S}$ for deformation rings representing framed deformation functors with target $T$.
 This notation will be specified precisely when we use it, but typical examples will be the following:

For certain ``Taylor--Wiles'' primes $q \notin S$, we will choose a representation $\rho_{\Z_q}^T: \pi_1 \Z_q \rightarrow T(k)$
which is conjugate inside $G(k)$, to $\rho_{\Q_q}$ (see \eqref{rqTd}); let $\rho_{\Q_q}^T: \pi_1 \Q_q \rightarrow T(k)$
be its pullback to $\Q_q$. We write 
$$ \mathcal{S}_q = \mbox{ framed deformation ring  of $\rho_{\Z_q}^T$ pulled back to $\pi_1 \Q_q$},$$
 $$ \mathcal{S}_q^{\ur}= \mbox{ framed deformation ring of $\rho_{\Z_q}^T$,} $$
  the superscript $\ur$ is for ``unramified.'' 
  
For example, just to  be completely explicit  $\mathcal{S}_q$ represents 
the functor $\mathcal{F}_{X, \rho_{\Z_q}^T}^{\fram}: \Art_k \rightarrow s\Sets$
from  Definition \ref{defn-repfunctor}, where $X$ is now the 
  etale homotopy type of $\Q_q$,  and we replace $BG$ by  $BT$ in that Definition.

 \subsubsection{Mnemonics} 
In summary we have the following general mnemonic for our notations:
$$ \mathcal{F}_{?} = \mbox{ deformation functor for $\rho$, with target $G$,  at level $?$}, \\$$
$$ \mbox{ (if $? = n$, this is a shorthand for level $S \coprod Q_n$ for a set of auxiliary primes $Q_n$);}$$ 
$$ \mathcal{R}_{?} = \mbox{ representing ring for $\mathcal{F}_{?}$}$$
$$ \mathcal{R}_{?}^{\fram} = \mbox{representing ring for framed version of $\mathcal{F}_{?}$},$$
  $$ \mathcal{S}= \mbox{a representing ring for a framed deformation problem with target $T$}$$
$$\mbox{ superscript }  \ur = \mbox{``unramified'', e.g.\ deformations of $\pi_1 \Z_q$ as opposed to $\pi_1 \Q_q$} $$

  We will also use roman letters for the corresponding ``usual'' deformation rings. For example, 
  just as $\mathcal{R}_S$ denotes the deformation functor for $\rho$ with target $G$ at level $\Z[\frac{1}{S}]$,
  $\mathrm{R}_S = \pi_0 \mathcal{R}_S$ denotes the corresponding (Mazur) deformation ring.

 \subsection{Pro-rings and complete local rings}

 In the usual (non-simplicial)
 setting of deformation theory, it does not matter much whether one allows  functors from Artin rings  to be represented by a pro-object in
 Artin rings, or whether one allows the ``representing'' object to be
 the complete ring slightly outside the category of Artin rings. 
 The latter  is the traditional choice.  
  
   In our simplicial setting, 
 we have chosen to consistently work with pro-objects.  It is certainly possible to work with complete objects in this setting too, 
 as does Lurie \cite{LurieThesis}; however, we will always stick to the pro-object point of view
 when we deal with simplicial rings and derived deformation problems.

It  will be convenient for us to review certain aspects of the passage between
the ``complete'' and ``pro'' point of view in the case of usual (non-simplicial) deformation theory:

 \begin{Lemma}  \label{lem:complete local ring}
 Let $\mathrm{R} = (\mathrm{R}_{\alpha})$ be a sequentially indexed pro-object of the category of (usual) Artin local rings augmented over $k$.    Suppose that $\varinjlim \Hom(\mathrm{R}_{\alpha}, k[\epsilon]/\epsilon^2)$ 
 is finite-dimensional (i.e., $\mathfrak{t}^0 \mathrm{R}$ is finite-dimensional).  Then $\mathrm{R} := \varprojlim \mathrm{R}_{\alpha}$ is a complete local Noetherian ring. 
 \end{Lemma}
 
 We will refer to $\mathrm{R}$ as the ``associated complete local ring'' to the pro-object $\mathrm{R}_{\alpha}$.

 \proof   We can suppose all the transition maps on $\mathrm{R}_{\alpha}$ are actually surjective, otherwise we just replace $\mathrm{R}_{\alpha}$ by the intersection of all the images of the maps from $\beta > \alpha$. 
 
Write $\mathfrak{m}_{\mathrm{R}} = \ker(\mathrm{R} \rightarrow k) = \varprojlim \mathfrak{m}_{\alpha}$, where $\mathfrak{m}_{\alpha}$
is the kernel of $\mathrm{R}_{\alpha} \rightarrow k$.  Then $\mathrm{R}/\mathfrak{m}_{\mathrm{R}} \cong k$,
and moreover $\mathfrak{m}_{\mathrm{R}}$ is clearly the unique maximal ideal, for any element not in it is invertible .

Write $J_{\alpha}$ for the kernel of the natural map from $\mathrm{R}$
to $\mathrm{R}_{\alpha}$.  Thus $\mathrm{R} = \varprojlim \mathrm{R}/J_{\alpha}$. 
 Because $\mathrm{R}_{\alpha}$ is Artin local, $\mathfrak{m}_{\alpha}^{k_{\alpha}} = 0$ for some $k_{\alpha}$. Therefore, 
 $\mathfrak{m}_{\mathrm{R}}^{k_{\alpha}} \subset J_{\alpha}$.  Therefore, $\mathrm{R}$
 is complete with respect to the topology defined by powers of the maximal ideal.
(In fact, this topology is readily verified to coincide with the profinite topology.)

By assumption,  $\varinjlim ( \mathfrak{m}_{\alpha}/\mathfrak{m}_{\alpha}^2)^{\vee}$ is finite dimensional, say, of dimension $s$.  Since  the transition maps are injective, the  $k$-dimension of $\mathfrak{m}_{\alpha}/\mathfrak{m}_{\alpha}^2 \leq s$ for all $\alpha$.  

 We claim that in fact
\begin{equation} \label{moreannoyingthanIthought} \mathfrak{m}_{\mathrm{R}}^t = \varprojlim \mathfrak{m}_{\alpha}^t. \end{equation}
Assuming this we get  (everything is finite!) 
  $\mathfrak{m}_{\mathrm{R}}/\mathfrak{m}_{\mathrm{R}}^2 = \varprojlim \mathfrak{m}_{\alpha}/\mathfrak{m}_{\alpha}^2$, 
 which implies that $\mathrm{R}$ is a quotient of the ring $W(k)[[X_1, \dots, X_s]]$ and so Noetherian.

To verify \eqref{moreannoyingthanIthought}, first note that multiplication induces a surjection of $k$-vector spaces  $\left(\mathfrak{m}_{\alpha}/\mathfrak{m}_{\alpha}^2 \right)^{\otimes t} 
 \rightarrow \mathfrak{m}_{\alpha}^t/\mathfrak{m}_{\alpha}^{t+1}$ 
and so the map $$ \mathfrak{m}_{\mathrm{R}}^t \rightarrow \mathfrak{m}_{\alpha}^t/\mathfrak{m}_{\alpha}^{t+1}$$
is surjective.  Suppose $Y_{\alpha} \in \mathfrak{m}_{\alpha}^t$  form a compatible system; there is $X_0 \in \mathfrak{m}_{\mathrm{R}}^t$
 such that $X_0 \equiv Y_{\alpha} $ modulo $\mathfrak{m}_{\alpha}^{t+1}$.
 We can then find $X_1 \in \mathfrak{m}_{\mathrm{R}}^{t+1}$ such that $X_1 \equiv Y_{\alpha} -X_0$ modulo $\mathfrak{m}_{\alpha}^{t+2}$. Proceeding in this way, 
the series  $X_0 + X_1 + \dots $  converges to an element of $\mathrm{R}$, and its limit is equal to $Y_{\alpha}$ modulo all powers of $\mathfrak{m}_{\alpha}$, 
as desired.
\qed

   In that setting the following lemma will be useful, to compare
 the tangent complex of a pro-object with a tangent complex of the
 corresponding complete ring: 
  
 \begin{Lemma} \label{complete} Suppose that $\mathrm{R}$ is a (usual)
   complete local Noetherian ring with maximal ideal $\mathfrak{m}$, equipped with an isomorphism $\mathrm{R}/\mathfrak{m} \to k$.
   Then $\tangent^* \mathrm{R} \cong \varinjlim \tangent^*
   (\mathrm{R}/\mathfrak{m}^N)$.
 \end{Lemma}

In particular, with the notations of the Lemma, 
the functor $\Art_k \rightarrow s\Sets$ given by  $$A \mapsto \colim_{N \to \infty}\smallC_k(\mathrm{R}/\mathfrak{m}^N,A)$$
has $i$th tangent cohomology 
 given by $\tangent^i \mathrm{R} = \mathrm{D}^i_{\mathbb{Z}}(\mathrm{R},k) = \mathrm{D}^i_{W(k)}(\mathrm{R}, k)$. 
 (For the last equality we look at the long exact sequence in Andr{\'e}--Quillen cohomology for $\mathbb{Z} \rightarrow W(k) \rightarrow k$:
 that shows that $\mathrm{D}^i_{\mathbb{Z}}(W(k),k)= 0$.     This allows us to compare $\mathrm{D}^i_{\mathbb{Z}}(\mathrm{R},k)$
 and $\mathrm{D}^i_{W(k)}(\mathrm{R}, k)$, as desired.)
  
We will apply this later in the case where  $\mathrm{R} = W(k)[[X_1, \dots, X_s]]/(Y_1, \dots, Y_t)$  and the $Y_i$ are a regular sequence
 inside the maximal ideal  $(p, X_1, \dots, X_s)$ of $W(k)[[X_1, \dots, X_s]]$. 
 Then, if the reductions $\overline{Y}_i$ to $k[[X_1, \dots, X_s]]$ have only quadratic and higher terms,   standard computations of Andr{\'e}--Quillen cohomology imply that 
\begin{equation} \label{numerical} \dim_k \tangent^i \mathrm{R} = \begin{cases} s, i=0 \\ t, i=1, \\  0,  \mbox{else}  \end{cases}\end{equation}
 
 \proof (of Lemma \ref{complete}):
 It is enough to show that the induced map of Andr{\'e}-Quillen cohomology (denoted $\mathrm{D}^i$), with coefficients in $k$, 
 $$ \varinjlim \mathrm{D}^i_{\mathbb{Z}}(\mathrm{R}_N,k) \rightarrow \mathrm{D}^i_{\mathbb{Z}}(\mathrm{R},k)$$
 is an isomorphism.  
  From the long exact sequence \cite[Theorem 5.1]{Quillen}
 associated to $\mathbb{Z} \rightarrow \mathrm{R} \rightarrow \mathrm{R}_N $ 
 it is enough to see that, for all $i$, 
 $$\varinjlim \mathrm{D}^i_{\mathrm{R}}(\mathrm{R}_N,k)  =0$$
 which is a consequence of \cite[Theorem 6.15]{Quillen} (unfortunately there is no proof given). 
  \qed

We will want a reasonable way to say that a  pro-simplicial ring is discrete. We do not need to discuss any homotopy theory
of pro-simplicial rings: for us their  their importance is solely in the functors they represent, and the following definition
is adequate: 
    
   \begin{Definition} \label{homotopydiscretedefinition}
  We say that a pro-object $\mathcal{R} = \{ \mathcal{R}_{\alpha} \}_{\alpha}$ of $\Art_k$
  is {\em homotopy discrete} if 
  the map  $p: \mathcal{R} \rightarrow \pi_0 \mathcal{R}$  induces an equivalence on represented functors (after applying   level-wise cofibrant replacement), i.e.
  if the map
  $$ \colim_{\alpha} \Map_{\Art_k}( \pi_0 \mathcal{R}_{\alpha}^c, A) \to \colim_{\alpha} \Map_{\Art_k}(\mathcal{R}^c_{\alpha}, A)$$
  is a weak equivalence for all $A \in \Art_k$. 
  \end{Definition}

The following Lemma will show that, in the arithmetic contexts where we apply it,    the coincidence of ``derived'' and ``usual'' deformation rings  is equivalent to the  usual deformation ring being a complete intersection of the expected size:

 \begin{Lemma} \label{uvd}
 Suppose $\mathcal{R} = \{ \mathcal{R}_{\alpha} \}_{\alpha}$ is an object of pro-$\Art_{/k}$
 such that $b_i := \dim (\tangent^i \mathcal{R}) $ is finite for all $i$, and zero if  $i \notin \{0,1\}$. 
  Denote by $\pi_0 \mathcal{R} = (\alpha \mapsto \pi_0 \mathcal{R}_{\alpha})$ the associated pro-ring.
  In this case, the following are equivalent: 
\begin{itemize}
\item[(i)]  $\mathcal{R}$ is homotopy discrete, in the sense of Definition  \ref{homotopydiscretedefinition}.
  \item[(ii)]  The  complete local  ring $\varprojlim_{\alpha} \pi_0 \mathcal{R}_{\alpha}$ associated to $\pi_0 \mathcal{R}$ is isomorphic to  $$W(k)[[X_1, \dots, X_{b_0}]]/(Y_1, \dots, Y_{b_1})$$ 
for a regular sequence $Y_i$ of elements all belonging to 
$(p,\mathfrak{m}^2)$ (where $\mathfrak{m}$ is the maximal ideal of $W(k)[[X_*]]$). \end{itemize}\end{Lemma}

\proof

For (ii) implies (i): the map $\mathcal{R} \rightarrow \pi_0 \mathcal{R}$   always induces an isomorphism on $\tangent^0$, 
and an injection on $\tangent^1$ (see \eqref{tangent-inclu}); it is an isomorphism on $\tangent^1$  then by dimension-counting
(see Lemma \ref{complete} for why we can compute the tangent complex of the pro-ring $\pi_0 \mathcal{R}$
in terms of the tangent complex of the associated complete local ring.) 

For (i) implies (ii): Let $\mathrm{R} =\varprojlim_{\alpha} \pi_0 \mathcal{R}_{\alpha}$. 
Under the quoted finiteness assumption,   $\mathrm{R}$ is a complete  Noetherian local ring (Lemma \ref{lem:complete local ring}).
By assumption (i), together with Lemma \ref{complete}, we have  $\dim \Der^i_{W(k)}(\mathrm{R}, k) = b_i$. 
 Thus there  is a surjection  of complete local Noetherian rings 
 $$ W(k)[[X_1, \dots, X_{b_0} ]] \twoheadrightarrow \mathrm{R}$$
 compatible with the augmentations to $k$. 
 Any element $Y$ in the kernel  has the property that $\bar{Y} \in k[[X_i]]$ doesn't involve any linear terms,
otherwise $\dim \mathfrak{t}^0 \mathrm{R}$ would be too small.   Thus $Y \in (p, \mathfrak{m}^2)$.   What we must show is that the kernel
is generated by a regular sequence;
once this is so, the length of the regular sequence must be $  \dim \Der^1_{W(k)}(\mathrm{R}, k)=b_1$.  For  this,  see \cite[Theorem 8.5]{Iyengar} and its proof in particular.  
   \qed

   We conclude with a minor result about projective limits of Tor-groups (which is mainly cosmetic). 
 
\begin{Lemma} \label{Tor_limit}
Let $S$ be a complete local ring, and $I_n$ a sequence of ideals which
form a basis for the topology of $S$.
Let $M,N$ be finitely generated complete $S$-modules.  
Set $M_k =M/I_k, N_k = N/I_k, S_k=S/I_k$ and suppose these are all finite (as sets). Then the natural map
  $$\Tor_{S}(M,N) \rightarrow \varprojlim_k \Tor_{S_k}(M_k, N_k)$$ 
 is an isomorphism. 

 \end{Lemma}

\proof  
 
 First of all, the map
\begin{equation} \label{Firstcheck} \Tor_S(M,N) \rightarrow \varprojlim_k  \Tor_S(M, N_k)\end{equation} 
 is an isomorphism. To see this, choose a resolution $F_{\bullet}$ of $M$ by finite free $S$-modules; tensoring with $N$,  
 our result follows from the fact that $N=\varprojlim N_k$ and homology commutes with inverse limits for complexes of finite abelian groups.

 We are now reduced to verifying that  the natural map  
 $$\Tor_S(M,N) \rightarrow  \varprojlim_{n \geq k}  \Tor_{S_n}(M_n, N)$$ 
is an isomorphism when $N$ is an $S_k$-module.

By change of rings  \cite[Tag 068F]{StacksProject}
(in the derived sense, we have $M \dotimes_S N = M \dotimes_{S} S_n \dotimes_{S_n} N$) we get 
there is  a spectral sequence computing  $\Tor_S^r(M,N)$: 
\begin{equation}   \bigoplus_{p+q=r} \Tor^p_{S_n}(\Tor_S^q(M, S_n), N)\end{equation}  
  All the terms in this spectral sequence are finite abelian groups. 
 Also, this spectral sequence is ``compatible with change of $n$,''  i.e.\ there is a map  of spectral sequences from the spectral sequence at level $N$ to the spectral sequence at level $n < N$. 
Thus we can take the inverse limit and obtain a spectral sequence with first page 
\begin{equation} \label{invlimit}  \varprojlim_n \bigoplus_{p+q=r} \Tor^p_{S_n}(\Tor_S^q(M, S_n), N)\end{equation} 
which still computes $\Tor_S(M,N)$.

Because $\varprojlim_n \Tor_q^S(M, S_n) =0$ in positive degree $q$, by \eqref{Firstcheck} with $N=S$, 
the natural maps
$$\Tor_q^S(M, S_N) \rightarrow \Tor_q^S(M, S_n)$$
must be zero for fixed $n$ and big enough $N$. 
 
 That means that, in the \eqref{invlimit}  all the terms with $q  > 0$ go away.  Thus the natural edge map 
 $$  \Tor_S(M,N) \stackrel{\sim}{\rightarrow} \varprojlim_n \Tor^p_{S_n} (M \otimes_S S_n, N)$$
 which is what we wanted.\qed

%
%
%
%
%
%
%
%
%
%
%
%
%
%
%
%
%
%
%
%
%
%
%
%
%
%
%
%
%
%
%
%
%
%
%
%
%
%
%
%
%
%
%
%
%

  \section{ Allowing ramification at extra primes } \label{sec:TWprimes} 

  In this section we examine the behavior of the derived deformation ring when
  adding a prime to the set of ramification.  The statement
  is the obvious one -- it is a translation, at the level of
  deformation rings, of the statement ``a representation of
  $\pi_1(\Z[\frac{1}{Sq}])$ unramified at $q$ is actually a
  representation of $\pi_1(\Z[\frac{1}{S}])$.''
 
   We need this only for certain primes $q$, Taylor--Wiles primes -- see \S \ref{TWprimesdef} -- which present a particularly simple situation
 as far as local deformation theory goes.

  \subsection{Notation and the main result} \label{TWprimes notation section} 
  We briefly recall the relevant notation from the previous section, especially \S \ref{setup} and  \S \ref{Defringdeffunctors}: 
 we're deforming a fixed representation $\rho: \Gamma_S
 \rightarrow G(k)$ from $\Gamma_S = \pi_1^{\mathrm{et}}(\Z[1/S])$ into
 the $k$-points of the reductive group $G$. 
    We denote by $\mathcal{F}_S = \mathcal{F}_{\Z[\frac{1}{S}]}$
 the corresponding deformation functor (as discussed in \S \ref{Defringdeffunctors}, because $\rho$ is always fixed, we omit it from the notation). 
  Let now $q$ be a Taylor--Wiles prime, as in \S \ref{TWprimesdef}. Denote by $\rho_{\Z_q}$ the pullback of $\rho$ under (the map of {\'e}tale fundamental groups associated with)
$\Spec \Z_q \rightarrow \Spec \Z[\frac{1}{S}]$.  By assumption, $\rho_{\Z_q}$  is conjugate, in $G(k)$, to a 
representation of the {\'e}tale fundamental group of $\Z_q$ into $T(k)$. 
As part of the data associated to a Taylor--Wiles prime, we have fixed such a representation  $\rho_{\Z_q}^T$, {\em together with} an isomorphism of 
\begin{equation} \label{above} \mathrm{inclusion}_{T}^G \circ \rho_{\Z_q}^T      \cong 
 \rho_{\Z_q},\end{equation}
 and  our constructions that follow will depend on this specific isomorphism.

 We define similarly $\rho_{\Q_q}, \rho_{\Q_q}^T$
 by pullback via $\pi_1 \Q_q \rightarrow \pi_1 \Z_q$.  
 
 Then make the following definitions:
 
  \begin{itemize}
  \item[-] $\mathcal{F}_{\Z_q}, \mathcal{F}_{\Q_q}$ are the deformation functors for 
 $\rho_{\Z_q}$ and $\rho_{\Q_q}$, considered as representations into $G(k)$; 
\smallskip 
  
    \item[-] $\mathcal{F}_{\Z_q}^T$ or $\mathcal{F}_{\Q_q}^T$  are deformation functors for $\rho_{\Z_q}^T$ or $\rho_{\Q_q}^T$ {\em valued in $T$} (and
    recall the convention about $T$-valued deformation  functors from \S \ref{Defringdeffunctors}).

     The choice of isomorphism of \eqref{above}
    induces a map $\mathcal{F}_{\Z_q}^T \rightarrow \mathcal{F}_{\Z_q}$, similarly for $\Q_q$. 
    \smallskip
    
  \item[-] $\mathcal{F}_{\Z_q}^{T, \fram}$ or $\mathcal{F}_{\Q_q}^{T, \fram}$ are framed deformation functors for $\rho_{\Z_q}^T$ or $\rho_{\Q_q}^T$  {\em valued in $T$}.
  \smallskip
  
  \item[-] (Ignore this one until you need to read it:) Let $\usS_q^{\circ}$ be the usual  (underived) framed deformation ring  for the trivial representation $I_q \rightarrow T$.  
See  \eqref{iqdef} for the definition of $I_q \cong (\Z/q)^*$.    
    \end{itemize}

  Our main results can be summarized in the following diagram, where
  all squares are  object-wise homotopy pullback squares (see \eqref{hpssquarereview} for definition), 
 the $s$-maps are sections of  the natural maps that they are adjacent to, and the symbol $\sim$ means objectwise weak equivalence.    The maps are the ``natural ones'', except for $s, j$ which are discussed  in \S \ref{TWdiagsections}
and \S \ref{jdef} respectively; the notation $(\usS_q^{\circ})^c$ means a cofibrant replacement for $\usS_q^{\circ}$, where the latter
is considered as a discrete simplicial commutative ring.  The square (c) is a homotopy pullback  considered both ways, i.e.\ both with the ``left to right'' horizontal
arrows and with the ``right to left'' horizontal actions.

        \begin{equation}  \label{TWdiag} 
 \xymatrix{
\ar @{} [dr] |{(a)}      \mathcal{F}_{\Z[\frac{1}{S}]}   \ar[r]\ar[d]  &   \ar @{} [dr] |{(b)}        \mathcal{F}_{\Z_q} \ar[d]    & \ar @{} [dr] |{(c)}    \ar[d] \ar[l]^{\sim}  \mathcal{F}_{\Z_q}^{T} \ar@/^1pc/[r]^{s_{\Z_q}}   &  \ar @{} [dr] |{(d)}  \ar[l] \ar[d]  \mathcal{F}_{\Z_q}^{T, \fram}\ar[r]    &   \mbox{$*$} \ar[d]  \\ 
 \mathcal{F}_{\Z[\frac{1}{Sq}]}   \ar[r]  &   \mathcal{F}_{\Q_q}  &  \ar[l]^{\sim}  \mathcal{F}_{\Q_q}^{T} \ar@/_1pc/[r]_{s_{\Q_q}}  &  \ar[l] \mathcal{F}_{\Q_q}^{T, \fram} \ar[r]^{j\qquad }   &
 \Maps_{\Art_k}((\usS_q^{\circ})^c, -)  
 }
 \end{equation}

  \begin{Remark} \label{lawyerly}
  Strictly speaking, this diagram is not correct:
  it exists only after replacing some of the functors by naturally weakly equivalent ones; the diagram
  as above exists literally only in the homotopy category, i.e.\ at the cost of replacing functors by their $\pi_0$.
We have allowed ourselves to be imprecise about this
   since they do not affect the key deduction  
   \eqref{rugelach} that follows. 

We briefly describe the sources of this: Firstly,   to produce the splittings of square (c), we must replace 
the functors occurring there by 
  naturally weakly equivalent functors,  as mentioned in  the discussion around \eqref{BigUglyDiagram}.   Secondly, in square (d), 
we must insert some homotopy commutative squares with horizontal weak equivalences; these arise from passing from functors to representing rings, as in Lemma \ref{Lemma:represent-natural-trans}.
Also, in square (d),  the map $j$ exists  only as a zig-zag, for reasons similar to those mentioned near \eqref{sbpq}. 
  \end{Remark}

\begin{Remark} The first square  (a) is a homotopy pullback square for any prime $q$.  It is only in the analysis of the other squares that use the fact that $q$ is a ``Taylor--Wiles'' prime. 
 In fact, square (a) already expresses 
 that ``a representation of
  $\pi_1(\Z[\frac{1}{Sq}])$ unramified at $q$ is actually a
  representation of $\pi_1(\Z[\frac{1}{S}])$.''     The reason to go to all the trouble of going through squares (b), (c), 
  is that the functors $\mathcal{F}_{\Z_q}$ and $\mathcal{F}_{\Q_q}$ are not representable -- they have too many automorphisms. 
  But the framed $T$-functors,, and of course the  trivial functor $*$ and the functor $ \Maps_{\Art_k}((\usS_q^{\circ})^c, -)$,
  are pro-representable. 
   \end{Remark}

By concatenating all the squares, we arrive at a homotopy pullback square  
\begin{equation} \label{rugelach}
  \begin{aligned}
    \xymatrix{ \mathcal{F}_{\Z[\frac{1}{S}]}' \ar[r] \ar[d] &
      \mbox{$*$} \ar[d] \\ \mathcal{F}_{\Z[\frac{1}{Sq}]}' \ar[r] &
      \Maps_{\Art_k}((\usS_q^{\circ})^c, -) }
  \end{aligned}
\end{equation} 
 where the functor $\mathcal{F}_{\Z[\frac{1}{Sq}]}'$ comes with a natural weak equivalent to $\mathcal{F}_{\Z[\frac{1}{Sq}]}$, etc.;
 we have to make this naturally equivalent replacements to ``invert'' the horizontal weak equivalences of \eqref{TWdiag} (cf. discussion after \eqref{roofy}).

 If we were dealing with the usual deformation rings, the corresponding statement is (cf. \eqref{Descent}) 
\begin{equation} \label{rugelach2} \left( \mbox{ deformation ring at level $S$ } \right)=  \left( \mbox{ deformation ring at level $S q$} \right)\otimes_{\usS_q^{\circ}} W(k).\end{equation}
The diagram \eqref{rugelach}  will imply a similar result for the derived deformation rings, with tensor products now taken in the derived sense. 

 Now let us introduce notation for representing rings: 
\begin{itemize}

  \item[-] We denote representing rings for $\mathcal{F}_{\Q_q}^{T,\fram}$ and $\mathcal{F}_{\Z_q}^{T, \fram}$ by $\mathcal{S}_q$ and $\mathcal{S}_q^{\ur}$ respectively
  (the superscript ``ur'' is for ``unramified,'' and we omit any square $\fram$ on the representing rings for typographical simplicity -- $\mathcal{S}$-rings will always denote framed deformations into $T$.)  Thus these are representing rings for the framed deformation functors of $\rho_{\Q_q}^T$ and $\rho_{\Z_q}^T$ with targets in $T$. 
  
  \item[-] We set  also 
\index{$\usS_q$}   $\usS_q = \pi_0 \mathcal{S}_q$
and $\usS_q^{\ur} = \pi_0 \mathcal{S}_q^{\ur}$. \index{$\usS_q^{\ur}$}
These are, therefore, the (underived) framed deformation rings  for $\rho_{\Q_q}^T$ and $\rho_{\Z_q}^T$,  again with targets in $T$.

  \item[-]   Recall the definition of $I_q \cong (\Z/q)^*$ from \eqref{iqdef}. The usual (underived) framed deformation ring of the trivial representation:
$ I_q \longrightarrow T(k)$
will be denoted by $\rS_q$.   \index{$\rS_q$}
  \end{itemize}

   The remainder of the section will prove that each square (a), (b), (c), (d)  in \eqref{TWdiag} is homotopy pullback, construct the sections $s$ and the map $j$. 
   We will analyze the squares (a), (b), (c), (d) in order. 

\subsection{The square (a) from \eqref{TWdiag} is a homotopy pullback square}  \label{SquareA}

 Examine just square (a):
 \begin{equation}  \label{pi1fiberA}
   \begin{aligned}
     \xymatrix{
       \ar@{} [dr] |{(a)}      \mathcal{F}_{\Z[\frac{1}{S}]}   \ar[r]\ar[d]  &          \mathcal{F}_{\Z_q} \ar[d]    \\
       \mathcal{F}_{\Z[\frac{1}{Sq}]} \ar[r] & \mathcal{F}_{\Q_q} }
   \end{aligned}
  \end{equation} 
 
This  arises from a corresponding  diagram of pro-simplicial sets, the {\'e}tale homotopy types of the following diagram
   \begin{equation}  \label{pi1fiberB}
     \begin{aligned}
       \xymatrix{
         \Spec \  \Q_q \ar[r]\ar[d]  &  \Spec \  \Z[\frac{1}{Sq}]   \ar[d] \\
         \Spec \  \Z_q   \ar[r] &     \Spec \  \Z[\frac{1}{S}]  \\
      } 
     \end{aligned}
 \end{equation}
 
 To verify that square (a) is homotopy pullback, it is enough
 to check that the corresponding map on   tangent complexes is  also
a pullback square (see discussion after Lemma \ref{lem:2.51})
that is to say, that the induced map  
\begin{equation} \label{hofibdef} \mathfrak{t} \mathcal{F}_{\Z[1/S]} \rightarrow \hofib \left(\tangent \mathcal{F}_{\Z_q}  \oplus  \tangent \mathcal{F}_{\Z[1/Sq]} \rightarrow  \tangent \mathcal{F}_{\Q_q} \right)\end{equation}
is an isomorphism.   Here the ``homotopy fiber'' $\hofib(A_{\bullet} \rightarrow B_{\bullet})$ of a map of chain complexes   \index{``hofib'' (map of complexes)}
is the 
chain complex $C_{\bullet} := A_{\bullet} \oplus B_{\bullet}[-1]$, with the usual mapping cone differential;
it fits into a triangle $C_{\bullet} \rightarrow A_{\bullet} \rightarrow B_{\bullet}  \stackrel{[1]}{\rightarrow}$. 
In other words  (Lemma  \ref{Lemma:tangent space computation})
we should get a homotopy pullback square when we apply  {\'e}tale cochains, valued in $\Ad \rho$, to the {\'e}tale homotopy type of the diagram  \eqref{pi1fiberB}.

To verify that, we can replace $\Spec \  \Z_q$ and $\Spec \  \Q_q$ 
 by $\Spec \  \Z_q^{hs}$ and $\Spec \  \Q_q^{hs}$, where $\Z_q^{hs}$ is the henselization of $\Z$ at the closed point corresponding to $q$
 (one can present $\Z_q^{hs}$ as the ring of integers
 inside the algebraic closure of $\Q$ within $\Q_q$)
 and $\Q_q^{hs}$ is its quotient field, obtained by inverting $q$.  
In fact the maps  $\Z_q^{\hs} \rightarrow \Z_q $ and $\Q_q^{\hs} \rightarrow \Q_q $ induce
  isomorphisms in etale cohomology.  (Actually this step is wholly cosmetic; we could
  replace everywhere $\Z_q$ by $\Z_q^{hs}$.) 
  
  In turn the henselization can be presented
 as a direct limit of the ring of regular functions on etale coverings $V \rightarrow \Spec \Z[\frac{1}{S}]$ equipped with a lift of $\Spec  \ \Z/q \rightarrow \Spec \Z[\frac{1}{S}]$;
  and the induced map
 $$\varinjlim C^*_{et}(V; \Ad \rho) \rightarrow C^*_{et} (\Z_q^{\hs}; \Ad \rho) $$ is
 a quasi-isomorphism; there is a similar assertion for $\Q_q^{hs}$ replacing $V$ by $V \times_{\Z[\frac{1}{S}]} \Z[\frac{1}{Sq}]$ (this can be seen directly or see
\cite[VIII, Corollaire 5.8]{SGA4}). 

 We are reduced  in this way to the ``Mayer--Vietoris'' sequence: if $V \rightarrow X$ is an {\'e}tale map, and $U \subset X$ is Zariski-open,  the corresponding sequence of {\'e}tale cochains
 $$C^*(X) \rightarrow C^*(U) \oplus C^*(V) \rightarrow C^*(U \times_{X} V)$$ induces an isomorphism from $C^*(X)$ to the homotopy fiber of the latter morphism.    After  checking the compatibility of connecting homomorphisms, this follows from the cohomological version proved in \cite[Tag 0A50]{StacksProject};
 there both $U, V$ are Zariski-open and $U \times_X V  = U \cap V$, but the proof only requires {\em one} of them to be Zariski-open.  
 We apply this with $X = \Spec \Z[\frac{1}{S}], 
 U = \Spec \Z[\frac{1}{Sq}]$, and $V$ as above.     
\qed

 \subsection{Analysis of square (b) from diagram \eqref{TWdiag}}  \label{subsec:TWprimes}

The following Lemma shows that the  horizontal maps of square (b) are weak equivalences:

\begin{Lemma}  \label{TWcohomology}
Let $q$ be a Taylor-Wiles prime for $\rho$, as described above.  Assume by conjugating $\rho$
that $\rho(\mathrm{Frob}_q)$ lies in $T(k)$.  Let $\pi_1 \Z_q$ and $\pi_1 \Q_q$
act on $\Lie(T)_k$ and $\Lie(G)_k$ by means of the adjoint action, composed with $\rho$.  
   Then all  horizontal maps in the following diagram are isomorphisms:  

    \begin{equation}  \label{TWdiagEASY} 
      \begin{aligned}
        \xymatrix{
          H^*(\mathbb{Z}_q, \Lie(T)_k)   \ar[r]\ar[d]  &   H^*(\mathbb{Z}_q, \Lie(G)_k) \ar[d] \\
          H^*(\Q_q, \Lie(T)_k)   \ar[r] &   H^*(\Q_q, \Lie(G)_k)   \\
        }
    \end{aligned}
          \end{equation}
  \end{Lemma} 

 For the notation, see the last paragraph of \S \ref{sec:galoisgroups}: in particular $H^*(\mathbb{Z}_q)$ is the {\'e}tale cohomology of $\Spec \ \mathbb{Z}_q$, etc.    
 
   \proof

Indeed, for $H^0$ on either the top row or the bottom row, this is just the assertion that the fixed space of $\rho(\mathrm{Frob}_q)$
on $\Lie(G)_k$ is exactly $\Lie(T)_k \subset \Lie(G)_k$, which follows from  the  assumed (\S \ref{TWprimesdef})
strong regularity
of $\rho(\mathrm{Fr ob}_q)$. 

Note in what follows that the inclusion $\Lie(T)_k \subset \Lie(G)_k$ is $T(k)$-equivariantly split (by means of root spaces). 

For $H^2$, the top row is identically zero: the {\'e}tale cohomology of $\Z_q$ is in dimensions $0,1$ only.
  Bijectivity of the bottom 
map  for $H^2$ amounts (under Poitou--Tate duality) to the fact that
  $\mathrm{Frob}_q$ fixes only $  \Hom(\Lie(T)_k, \mu_{p^{\infty}})$ inside $\Hom(\Lie(G)_k, \mu_{p^{\infty}})$: 
  this follows from the assumption that $q=1$ in $k$, so that the
  cyclotomic character is trivial, and the assumption of
  regularity. Note in particular that
\begin{equation} \label{H2comp} \dim H^2(\Q_q, \Lie(G)_k) = \mathrm{rank} \ G.\end{equation} 
It remains to check for $H^1$. 

On  the bottom row: By the Euler characteristic formula
 the map  goes between groups of the same sizes; it is injective because the inclusion  of $\pi_1 \Q_q$-modules $\Lie(T)_k \rightarrow \Lie(G)_k$
is split, so it is an isomorphism. 

On the top row:  we check root space by root space:
our assertion comes down to the triviality
of  the first group cohomology $H^1$,  for the profinite group $\hat{\Z}$
acting nontrivially on $k$.  
\qed

\begin{Remark} \label{factoring remark}
The Lemma implies that any lift of $\rho_{\Q_q}: \pi_1 \Q_q \rightarrow G(k)$ 
 actually factors through $\pi_1 \Q_q^{\tame,\ab}$.  Indeed, the Lemma means that
 the map from the $T$-valued deformation functor of $\rho_{\Q_q}^T$ to the $G$-valued deformation functor 
 of $\rho_{\Q_q}^T$
 is an equivalence. This implies that any lift of $\rho_{\Q_q}^T$
 into $G$ admits a unique $T$-valued conjugate;   in particular, such a lift factors through $\pi_1 \Q_q^{\tame, \ab}$, as claimed. 
\end{Remark}

 Later on we will use the following terminology: 
 \begin{Definition} \label{inertialleveldef}
 Let $q$ be a Taylor--Wiles prime and $\rmA$ a (usual) Artin local ring  over $k$. 
  A (usual, underived) lift  $\tilde{\rho}: \pi_1 \Q_q \rightarrow G(\rmA)$
of $\rho_{\Q_q}$ is of {\em inertial level $\leq r$} \index{inertial level} 
if the restriction of $\tilde{\rho}$ to $I_q$:
\begin{equation} \label{ILL def} I_q \rightarrow G(\rmA)\end{equation}
 factors through $I_q/p^r$.
Here $I_q \leqslant \pi_1 \Q_q^{\tame,\ab}$ is as in \eqref{iqdef}. 
\end{Definition}

 Thus, for example, if   the $p$-valuation of $q-1$ is exactly equal to $n$, then every deformation  of $\rho_{\Q_q}$  is certainly of inertial level $\leq n$,
 because in that case the $p$-part of $I_q$ has size precisely $p^n$.

\subsection{Analysis of square (c) from diagram \eqref{TWdiag}. Construction of the sections $s_{\Z_q}, s_{\Q_q}$ in the diagram \eqref{TWdiag}} \label{TWdiagsections}
 We now examine the square (c), which we draw at first without the splittings: 
      \begin{equation*} 
 \xymatrix{
  \ar @{} [dr] |{(c)}    \ar[d]    \mathcal{F}_{\Z_q}^{T}   &  \ar[l] \ar[d]  \mathcal{F}_{\Z_q}^{T, \fram}      \\ 
   \mathcal{F}_{\Q_q}^{T}   &  \ar[l] \mathcal{F}_{\Q_q}^{T, \fram} \  }
 \end{equation*}
This is a homotopy pullback square: both above and below, one can obtained the framed
version from the unframed version by taking a homotopy fiber of a map to  $BT(A)$.
 
 We will construct splittings for   $$\mathcal{F}_{\Q_q}^T \leftarrow \mathcal{F}_{\Q_q}^{T,\fram} $$
 and the $\Z_q$ analog, compatibly.  
 Actually, strictly speaking, the splitting we construct is in the sense of zig-zags; explicitly we construct   functors $(\dots)^*$ naturally weakly equivalent to the corresponding unstarred functors $(\dots)$  and  a diagram   
      \begin{equation}   \label{BigUglyDiagram}   \xymatrix{
  \mathcal{F}_{\Z_q}^{T,*}  \ar[d]  \ar[r]  &   \mathcal{F}_{\Z_q}^{T,\fram*}   \ar @{} [dr] |{(c)^*}   \ar[d]   \ar[r] &  \mathcal{F}_{\Z_q}^{T,*}  \ar[d] \\ 
    \mathcal{F}_{\Q_q}^{T,*}  \ar[r]  & \mathcal{F}_{\Q_q}^{T,\fram*}   \ar[r]  & \mathcal{F}_{\Q_q}^{T,*}    \\ 
 }
 \end{equation}
 where the horizontal compositions are the identity, and the square  $(c)^*$ is connected by natural weak equivalences of squares,  to square (c). 
It  then follows that the left-hand square is also a homotopy pullback square. 
   
To make this splitting, we note that
 the $A$-points of framed and unframed representation functors into $T$ correspond to based and unbased  maps of (the etale homotopy type of) $\Spec \  \Q_q$ into $BT(A)$.  Since $T$ is commutative,  it is possible to make a weakly equivalent model $BT(A)^*$ of $BT(A)$ which 
is a simplicial group (for example, take the simplicial loop group of $B^2T(A)$; the construction of $B^2T(A)$ is discussed after Definition \ref{BG simplicial ring definition}).

  Now, for $\mathcal{E}$ a simplicial group and $X$ an arbitrary based simplicial set the natural map
$$ \mbox{ based }\Maps_{s\Sets}(X, \mathcal{E})\rightarrow  \Maps_{s\Sets}(X, \mathcal{E})$$ 
is naturally split, because one can use the group structure to take the image of the basepoint in $X$ to the identity.  
This discussion gives rise to a splitting of the map
$(\mathcal{F}_{X, T}^{\fram})^* \rightarrow (\mathcal{F}_{X,T})^*$
 (notation as in Definition \ref{defn-repfunctor}) whenever $G$ is commutative;
 we have superscripted with $*$ because we replaced $BT$ by $(BT)^*$ in the definitions. 

  Finally, the actual deformation functors are obtained from these (up to replacement by a naturally weakly equivalent functor) by taking a homotopy fiber over
  a certain vertex $v_{\rho} \in  (\mathcal{F}_{X,T}^{\fram})^*(k)$  and its image $\bar{v}_{\rho} \in \mathcal{F}_{X,T}^*(k)$;
  we may choose any vertex $v_{\rho}$   whose class in $\pi_0 \left((\mathcal{F}_{X,T}^{\fram})^*(k) \right)\simeq  \pi_0  (\mathcal{F}_{X,T}(k))$
is the class corresponding to  $\pi_1(X, x_0) \rightarrow T(k)$.

\subsection{Analysis of the square (d); Construction of the map $j$.}   \label{jdef}  
A very useful fact is that the pro-rings $\mathcal{S}_q^{\ur}$ and $\mathcal{S}_q$ 
representing the functors $\mathcal{F}_{\Z_q}^{T, \fram}$ and $\mathcal{F}_{\Q_q}^{T, \fram}$
are already homotopy discrete:

\begin{Lemma}  \label{HDiscreteness} Let $q$ be a Taylor--Wiles prime (\S \ref{TWprimesdef}). 
The pro-rings $\mathcal{S}_q^{\ur}$ and $\mathcal{S}_q$ are homotopy discrete in the sense of Definition \ref{homotopydiscretedefinition},
i.e.\ the maps $\mathcal{S}_q^{\ur } \rightarrow \usS_q^{\ur}$  and 
    $\mathcal{S}_q \rightarrow \usS_q$
induce a weak equivalence of represented functors (as usual, after applying level-wise cofibrant replacement).
  \end{Lemma}

A word of preparation before the proof.
Any (usual, underived) $T$-valued deformation of $\rho^T_{\Q_q}$  factors through the tame abelian quotient of $\pi_1 \Q_q$,
which fits into the exact sequence \eqref{iqdef}; we have, non-canonically,  
\begin{equation} \label{presentation} \pi_1 \Q_q^{\tame, \ab} \cong \langle \Frob_q \rangle \times  I_q \end{equation} 
where $I_q$ is as in \eqref{iqdef}, a cyclic group of order $q-1$.

\begin{proof}[Proof (of the Lemma)]
  We apply Lemma \ref{uvd}. It is enough to
  check that the complete local rings associated to $\usS_q^{\ur}$ and
  $\usS_q$ (as in Lemma \ref{uvd}, the inverse limit of the
  corresponding projective system) are complete intersection ``of the
  expected size,'' i.e.\ their tangent spaces and dimensions are as one
  would predict from the tangent complex.

  By a mild abuse of terminology, in the remainder of this proof, we
  will use $\usS_q^{\ur}$ and $\usS_q$ to actually denote these
  complete local rings.  We emphasize that in the rest of the proof we
  are working with underived deformation rings.

  It will be convenient to recall the following simple fact: For a
  finitely generated discrete group $\Gamma$, with profinite
  completion $\widehat{\Gamma}$, the usual (underived) framed
  deformation ring of a representation
  $\rho_0: \widehat{\Gamma} \rightarrow G(k)$ is (after passing to the
  associated complete ring) the completed local ring
  $\widehat{\mathcal{O}}_{X,\rho_0}$ of the $W(k)$-scheme $X$
  parameterizing maps $\Gamma \rightarrow G$, at the point
  $\rho_0 \in X(k)$.

  First of all, $\usS_q^{\ur}$ is a power series ring over $W(k)$
  because it arises from the deformation ring of the pro-cyclic group
  $\widehat{\Z}$.  It obviously has the ``expected size.''
 
  Now for $\usS_q$.  Let $p^N$ be the highest power of $p$ dividing
  $q-1$.  It follows from \eqref{presentation} that $\usS_q$ is the
  completed local ring of functions on
  $$\left( (F, t_2) \in T \times T : t_2^{p^N} = 1 \right)$$ at the
  maximal ideal corresponding to $F =\rho(\Frob_q)$ and
  $t_2=\mathrm{identity}$.  In other words, if $\rmA$ is the
  completion of ring of functions on $(t_1, t_2) \in T \times T$ at
  the identity point of $T(k) \times T(k)$, and $J$ the ideal defined
  by $t_2^{p^N}=e$, we have
$$\usS_q \cong \rmA/J.$$

In suitable coordinates, with $r=\dim(T)$, we have
\begin{equation} \label{coordinate_ref} \rmA \cong W[[X_1, \dots, X_r,
  Y_1, \dots, Y_r]], \ \ J = \langle (1+Y_i)^{p^N}-1
  \rangle \end{equation} This is visibly a complete intersection, with
$r$ relations; and the dimension of
$ \mathfrak{t}^1 \mathcal{F}_{\mathcal{S}_q} $ equals
$ \dim H^2(\Q_q, \Ad \rho) = r$ by \eqref{H2comp}.
\end{proof}

  We are now ready to construct the final square (d):

Clearly, any $T$-valued deformation $\tilde{\rho}$ of $\rho|_{\Q_q}^T$
actually factors through $\pi_1(\Q_q)^{\tame,\ab}$.  
 Thus, from the sequence \eqref{iqdef} $I_q  \hookrightarrow \pi_1 \Q_q^{\tame,\ab} \rightarrow
\pi_1 \mathbb{F}_q$  we get  a sequence of 
maps of the associated $T$-valued framed usual deformation rings
\begin{equation} \label{have-you-not-heard-of-this?} \underbrace{ \rS_q}_{W[[ Y_1, \dots, Y_r]]/J} \rightarrow \underbrace{  \usS_q }_{W[[X_1, \dots, X_r, Y_1, \dots, Y_r]]/J} \rightarrow \underbrace{ \usS_q^{\ur}}_{W[[X_1, \dots, X_r]]}\end{equation}
Here, recall that  $\rS_q$ was  the framed deformation ring  
 for the trivial representation of $I_q$ with targets in $T$;
 intrinsically, we can identify   $\rS_q$ with  the group algebra of $\mathbf{T}(\mathbb{F}_q)_p$. 
 The maps of \eqref{have-you-not-heard-of-this?} are really maps of  pro-finite rings, but we wrote below how the sequence of associated complete local rings  
  looks in the coordinates of \eqref{coordinate_ref}; the second map sends $Y_i$ to $0$.  
  
  \begin{Remark} \label{canonicalS}  
  It  is convenient to identify these rings more canonically.  For any  finitely generated abelian group $\Gamma$, 
  the usual, un-derived framed deformation ring of a representation $\sigma: \Gamma \rightarrow T(k)$ is identified
  with the completed group algebra    \begin{equation} \label{cga} W(k) [[ X_*(\mathbf{T}) \otimes \Gamma_{(p)}]]\end{equation}
   where $\Gamma_{(p)}$ is the quotient of $\Gamma$ by prime-to-$p$ torsion. 
   
   In fact, the unique splitting of $T(W(k)) \rightarrow T(k)$ gives us a lift $\tilde{\sigma}: \Gamma \rightarrow T(W)$,
   which permits us to reduce to the case (by twisting) when $\sigma$ is trivial. 
   Then, for any Artin ring $A \twoheadrightarrow k$, 
  the homomorphisms from $\Gamma$ to $T(A)$, reducing to the trivial homomorphism, are identified with 
  group homomorphisms 
  $$  \Hom(\Gamma_{(p)}, X_*(T) \otimes (1+\mathfrak{m}_A)^{\times})
  = \Hom(X^*(T) \otimes \Gamma_{(p)}, 1+\mathfrak{m}_A),$$
  i.e.\ homomorphisms from \eqref{cga} to $A$ preserving augmentations to $k$. 
  
  In this way,   writing $\mathbf{T}$ for the dual torus to $T$, 
  $$ \rS_q =  \mbox{ completed group algebra of $\mathbf{T}(\F_q)$}$$
  $$ \usS_q = \mbox{ completed group algebra of $\mathbf{T}(\Q_q)^{\tame}$},$$
  $$ \usS_q^{\ur} = \mbox{ completed group algebra of $\mathbf{T}(\Q_q)^{\ur}$}$$
    Here the 	``tame'' quotient of, e.g.\ $\mathbb{G}_m(\Q_q) = \Q_q^*$ is the profinite completion of the quotient by $1+q\Z_q$, and the ``unramified''
  quotient of the same group is the profinite completion of the quotient by $\Z_q^*$.
 \end{Remark}

In  fact the following diagram
\begin{equation}   \label{inertia_pullback} 
  \begin{aligned}
    \xymatrix{
      \rS_q  \ar[d]   \ar[r] & \usS_q \ar[d] \\
      W(k) \ar[r] & \usS_q^{\ur} }
  \end{aligned}
 \end{equation}
 induces a homotopy pullback square of represented functors
(equivalently the map $\usS_q \dotimes_{\rS_q} W(k) \rightarrow \usS_q^{\ur}$, where $\dotimes$ is the derived tensor product, is an equivalence, in
the sense of inducing a weak equivalence on represented functors).  
  This can be checked
 easily from the expressions above:
 $\usS_q^{\ur}$ is obtained from $W(k)$ by freely adding the generators $X_i$, 
 and $\usS_q$ is obtained from $\rS_q$ by the same procedure.  
 
 Passing from \eqref{inertia_pullback} to the corresponding square of represented functors, and using 
 Lemma \ref{HDiscreteness}, we get square (d) from \eqref{TWdiag}. 
As we have mentioned in Remark \ref{lawyerly}, this can strictly speaking only be done
after replacing the functors with weakly equivalent ones, for the reasons already mentioned there.

{\bf Remark:}  In our later presentation of the Taylor--Wiles method, it 
will be useful to observe that 
exactly the same holds true if we quotient all the rings by $p^n$, 
that is 
 to say, writing  $   \usbrS_q = \rS_q/p^n.$  
 \begin{equation}   \label{inertia_pullback2} 
   \begin{aligned}
     \xymatrix{
       \usbrS_q  \ar[d]   \ar[r] & \usS_q/p^n \ar[d] \\
       W_n \ar[r] & \usS_q^{\ur}/p^n }
   \end{aligned}
 \end{equation}
again  induces a homotopy pullback square of represented functors.

\section{The deformation ring with local conditions imposed} \label{sec:localconditions}
In the theory of Galois deformations much of the subtlety comes from the local theory at $p$. 
Thus we discuss how to impose local conditions on a derived deformation ring.  
After an abstract discussion in \S \ref{derivedlocalconditions}, we specify which conditions we actually use in  \S \ref{minimalevel}.

\subsection{Imposing local conditions} \label{derivedlocalconditions} 
  Let $\rho: \pi_1 \Z[\frac{1}{S}]  \rightarrow G(k)$. 
As discussed, for  $v$ a finite place\footnote{When we will apply this discussion later on, we will first have enlarged $S$ to a larger set $S' = S \coprod Q$, and
we will actually take $v \in Q$.  }    in $S$ we let $\mathcal{F}_{\Q_v}$ be the deformation functor for $\rho$ pulled back to $\Q_v$;
thus, there is a natural  transformation 
$$ \mathcal{F}_{\Z[\frac{1}{S}]} \rightarrow \mathcal{F}_{\Q_v}$$ 
of functors from $\Art_k$ to $s\Sets$. 
Note that  $\mathcal{F}_{\Q_v}$ is often not representable
because of infinitesimal automorphisms,  even when $\mathcal{F}_{\Z[\frac{1}{S}]}$ is, but it doesn't matter for our discussion. 

\begin{Definition} \label{Local conditions} 
A {\em local deformation condition} will be, by definition, a  simplicially enriched functor
$\mathcal{D}_v$ from $\Art_k$ to $s\Sets$, equipped with a map$$ \mathcal{D}_v  \longrightarrow \mathcal{F}_{\Q_v}$$ 
  We define the {\em global deformation  functor with local conditions} as the homotopy fiber product (see \S \ref{sec:holim}) 
  \begin{equation} \label{DFdef} \mathcal{F}_{\Z[\frac{1}{S}]}^{\mathcal{D}} = \mathcal{F}_{\Z[\frac{1}{S}]} \times^h_{\mathcal{F}_{\Q_v}} \mathcal{D}_v.\end{equation}
This comes with natural maps to $\mathcal{D}_v$ and to
$\mathcal{F}_{\Z[\frac{1}{S}]}$. 
\end{Definition}

 {\bf Remark.} It will happen (indeed, in the main case we use it) that a deformation condition
 is not presented as a map $\mathcal{D}_v \longrightarrow \mathcal{F}_{\Q_v}$,
 but rather as a zig-zag: 
 \begin{equation} \label{Zigzag} \mathcal{D}_v  \stackrel{\sim}{\leftarrow} \mathcal{D}_v^* \rightarrow \mathcal{F}_{\Q_v}^* \stackrel{\sim}{\leftarrow} \mathcal{F}_{\Q_v},\end{equation} 
 where $\sim$ denotes an object-wise weak equivalence. In this case, we convert this to a local deformation condition in the sense above
 by setting $\mathcal{D}_v' = \mathcal{D}_v^* \times_{\mathcal{F}_{\Q_v}^*}^h \mathcal{F}_{\Q_v} $;
 then $\mathcal{D}_v'$ is naturally weakly equivalent to $\mathcal{D}_v$, and is equipped with a map $\mathcal{D}_v' 
 \rightarrow \mathcal{F}_{\Q_v}$. 
  One can proceed in a similar way with longer zig-zags; in all cases one obtains a functor naturally weakly equivalent to $\mathcal{D}_v$
  with an actual map to $\mathcal{F}_{\Q_v}$.   Cf. discussion around \eqref{roofy}.

\medskip

 {\bf Example (unramified local condition):}  
Suppose that $\rho$ is actually unramified at $v$, i.e., it extends to a representation of $\Z[\frac{1}{S'}]$, where
$S' = S-\{v\}$. 
Thus it can be pulled back to $\Spec  \ \Z_v$. 
The natural maps  $\Spec \ \Q_v \rightarrow \Spec \ \Z_v$ 
induces a map of deformation functors
$ \mathcal{F}_{\Z_v} \rightarrow \mathcal{F}_{\Q_v}$ (as always, we are deforming $\rho$, but we omit this from our notation).
If we take
$$ \mathcal{D}_v = \mathcal{F}_{\Z_v}, \mbox{ with its natural map to $\mathcal{F}_{\Q_v}$},$$
the corresponding deformation functor with local conditions $\mathcal{F}_{\Z[\frac{1}{S}]}^{\mathcal{D}}$
is then naturally weakly equivalent  to $\mathcal{F}_{\Z[\frac{1}{S'}]}$, i.e.\ the deformations of $\rho$ but considered
now as a representation of $\pi_1 \Z[\frac{1}{S'}]$.  This  statement is precisely what we proved
in the last section: it follows from the homotopy pullback
square constructed in  \S \ref{SquareA}  (the assumption of $q$ being a Taylor-Wiles prime was not used for this part).

\medskip

 Return now to the general case. 
We get  from Lemma \ref{lem:2.51}
\begin{equation} \label{hofib-calc}  \mathfrak{t} \mathcal{F}^D_{\Z[\frac{1}{S}]} = \mbox{hofib}(  \tangent \mathcal{F}_{\Z[\frac{1}{S}]} \oplus \tangent \mathcal{D}_v \rightarrow \tangent \mathcal{F}_{\Q_v}) \end{equation}
(for $\hofib$ see \eqref{hofibdef}) 
with associated exact sequence \begin{equation} \label{tangent_def}   \tangent^i \mathcal{F}^D_{\Z[\frac{1}{S}]} \rightarrow \tangent^i \mathcal{F}_{\Z[\frac{1}{S}]} 
\oplus \tangent^i \mathcal{D}_v \rightarrow \tangent^i \mathcal{F}_{\Q_v} \stackrel{[1]}{\longrightarrow} \end{equation}

One can generalize this definition and discussion,  replacing the role of $v$ by a finite set of places.

\subsection{Lifting an underived local condition to a derived local condition} \label{lifting}

In \cite[\S 2]{CHT} there is an extensive study of various types of (underived) local conditions $\mathrm{D}$. Each such
gives rise to a quotient of the usual (underived) framed deformation ring $\mathrm{R}_v^{\fram} \twoheadrightarrow \mathrm{R}_v^{\mathrm{D}, \fram}$;
if the representation functor is representable, we get in fact 
a quotient of the underived deformation ring: 
$$ \mathrm{R}_v \twoheadrightarrow \mathrm{R}_v^{\mathrm{D}}.$$
(These objects are actually pro-rings; we use the notation $\twoheadrightarrow$ to denote that the induced map of functors is injective, 
or equivalently that the map on associated complete local rings is surjective.)

Assume, for simplicity, that $\mathcal{F}_{\Q_v}$ is representable, with representing ring $\mathcal{R}_v$;
then $\pi_0 \mathcal{R}_v \simeq \mathrm{R}_v$.  We can lift an underived local deformation condition to a derived deformation condition, in the sense of \S \ref{derivedlocalconditions}, 
by using the sequence  (for the framed case, see below):
$$ \mathcal{R}_v \rightarrow  \pi_0 \mathcal{R}_v  = \mathrm{R}_v \twoheadrightarrow  \mathrm{R}_v^{\mathrm{D}}.$$

More precisely, 
considering  now $\mathrm{R}_v^{\mathrm{D}}$ as a 
pro-simplicial ring, we may form
$\mathcal{D}_v= \Map((\mathrm{R}_v^{\mathrm{D}})^c, -)$,  as a $s\Sets$-valued functor on $\Art_k$  
(the superscript $c$ is for level-wise cofibrant replacement, as discussed in the final paragraph of
\S \ref{intro:functors}). 
Then we get an natural zig-zag:  \footnote{Explication:  Apply the cofibrant replacement 
to the series of rings to get   $$ \mathcal{R}_v \stackrel{\sim}{\leftarrow} \mathcal{R}_v^c \rightarrow  \mathrm{R}_v^c \rightarrow \mathrm{R}_v^{D, c} $$ 
where the left $\sim$ means level-wise weak equivalence. Now apply $\Maps$.  } 
\begin{equation} \label{Zigzag2} \mathcal{D}_v \dashrightarrow \Map(\mathcal{R}_v, -)  \end{equation}

As in the Remark of 
\S \ref{derivedlocalconditions}  we obtain from  \eqref{Zigzag2} a map 
\begin{equation} \mathcal{D}_v^* \longrightarrow \mathcal{F}_{\Q_v}, \end{equation}
where $\mathcal{D}_v^*$ is naturally weakly equivalent to $\Map((\mathrm{R}_v^{\mathrm{D}})^c,  -)$. 
We call this  the {\em derived deformation condition associated with the (usual) deformation condition $\mathrm{D}$.}  

Before we formulate our main theorem, note that  $\mathrm{R}_v \twoheadrightarrow \mathrm{R}_v^{\mathrm{D}}$
defines a subspace  $H^1_{\mathrm{D}} \subset H^1(\Q_v, \Ad \rho)$, where $H^1_{\mathrm{D}}$ is the (usual)
tangent space to the (usual) functor represented by $\mathrm{R}_v^{\mathrm{D}}$. 

\begin{theorem} \label{tangent complex with def conditions}
Suppose that $\mathrm{R}_v^{\mathrm{D}}$ is actually formally smooth (i.e.\ its tangent complex is nonvanishing only in degree $0$, or equivalently  the complete local ring associated to $\mathrm{R}_v^{\mathrm{D}}$
is isomorphic to $W(k)[[Y_1, \dots, Y_{m}]]$).   Form from $\mathrm{R}_v^{\mathrm{D}}$ a local deformation condition, as described above, and then a global deformation functor
$\mathcal{F}_{\Z[\frac{1}{S}]}^{\mathrm{D}}$ with this local condition imposed, as in \eqref{DFdef}. 

The cohomology of the tangent complex of $\mathcal{F}_{\Z[\frac{1}{S}]}^{\mathrm{D}}$ 
is naturally identified  with the cohomology with local conditions (\S  \ref{local-cohomology}): 
\begin{equation} \label{earned} \mathfrak{t}^i \mathcal{F}_{\Z[\frac{1}{S}]}^{\mathrm{D}} \cong H^{i+1}_{\mathrm{D}}(\Z[1/S], \Ad \rho).\end{equation} 
 where the  local conditions at $v$ are prescribed by the subspace
   $H^1_{\mathrm{D}} \subset H^1(\Q_v, \Ad \rho)$. 
\end{theorem}

\proof 
  Consider the map of tangent  complexes $\mathfrak{t} \mathcal{D}_v \rightarrow \mathfrak{t} \mathcal{F}_{\Q_v}$.
 The canonical map $\tau_{\geq 0} \left( \mathfrak{t} \mathcal{D}_v \right) \rightarrow \mathfrak{t} \mathcal{D}_v $
 is a quasi-isomorphism; and so the tangent complex of $\mathcal{F}_{\Z[\frac{1}{S}]}^{\mathcal{D}}$ is naturally 
 equivalent to  the homotopy fiber of
 $$ \mathfrak{t} \mathcal{F}_{\Z[\frac{1}{S}]} \oplus \tau_{\geq 0} \mathfrak{t} \mathcal{D}_v \rightarrow  \mathfrak{t} \mathcal{F}_{\Q_v}.$$
 
 However, the map   $\tau_{\geq 0} ( \mathfrak{t} \mathcal{D}_v) \rightarrow  \mathfrak{t} \mathcal{F}_{\Q_v}$
 can be factored
 $$ \tau_{\geq 0} (\mathfrak{t} \mathcal{D}_v)  \stackrel{\varpi}{\rightarrow} \tau_{\geq 0} (\mathfrak{t} \mathcal{F}_{\Q_v}) \rightarrow
 \mathfrak{t} \mathcal{F}_{\Q_v}$$
 The source and target of the first map $\varpi$
 both have homotopy only in degree $0$; thus,  $\varpi$ induces
 a quasi-isomorphism
 onto the subcomplex of 
 $\tau_{\geq 0} \mathfrak{t} \mathcal{F}_{\Q_v}$  corresponding to the subgroup  $H^1_{\mathrm{D}} \subset H^1(\Q_v, \Ad \rho)
 \cong \pi_{0}  \left( \tau_{\geq 0} \mathfrak{t} \mathcal{F}_{\Q_v} \right)$. 
 
 This discussion gives a zig-zag of weak equivalences (i.e.\ quasi-isomorphisms) between the tangent complex of 
 $\mathcal{F}_{\Z[\frac{1}{S}]}^{\mathrm{D}}$, and the chain complex (see \eqref{local cohomology definition}) that computes $H^i_{\mathrm{D}}$. 
  \qed

{\bf Remark.}  When we will use this construction,
$\mathcal{R}_v \simeq \pi_0 \mathcal{R}_v$ will be homotopy discrete, and so
both $\mathcal{R}_v$ and $\mathrm{R}_v^{\mathrm{D}}$ will actually be formally smooth. The definition makes sense without the assumption that $\mathcal{R}_v$ is homotopy discrete 
but it seems unlikely to be of much use if this fails.

 {\bf Remark.}  If the local representation functors are not representable, we can  proceed 
 in a similar fashion using the framed functors instead. Namely,  the usual framed
deformation ring $\mathrm{R}_{v}^{\fram}$ is equipped with a $G$-action, i.e.\ 
$$\mathrm{R}_{v}^{\fram} \rightarrow \mathrm{R}_{v}^{\fram} \widehat{\otimes} \mathcal{O}_G,$$
where the $\widehat{\otimes}$ denotes completed tensor product, i.e.\ the level-wise tensor product for pro-objects; 
also, given a quotient $\mathrm{R}_v^{\fram} \twoheadrightarrow \mathrm{R}_{v}^{\fram, \mathrm{D}}$
together with a compatible $G$-action, we may then form a derived deformation condition
by using -- just as in \S \ref{BGA simplicial} -- a bar construction to ``quotient by $G$.''

\section{Restrictions on $\rho$} \label{minimalevel}
Recall that we have been discussing the deformation theory of a Galois representation $\rho: \pi_1 \Z[\frac{1}{S}] \rightarrow G(k)$.  Thus far, $\rho$ has been quite general. 
Now, to simplify our life as far as possible,  we will impose the  following conditions: %
   (Recall the shorthand that $T = S-\{p\}$, where $p$ is the characteristic of $k$.)
 \begin{Assumption}
 Assumptions on $\rho$: 
\begin{itemize}
  \item[(a)]   $H^0(\Q_p, \Ad \rho_{\Q_p}) = H^2(\Q_p, \Ad \rho_{\Q_p}) = 0$. 
 \item[(b)]  For $v \in T$, the local cohomology $H^j(\Q_v, \Ad \rho_{\Q_v})=0$ for $j=0,1,2$.  
 \item[(c)]   $\rho$ has big image: the image of $\rho$ restricted to $\Q(\zeta_{p^{\infty}})$ contains 
 the image of $G^{\mathrm{sc}}(k)$ in $G(k)$ (here $G^{\mathrm{sc}}$ is the simply connected cover).
 
 Moreover, we suppose that $G^{\mathrm{sc}}(k)$ has no invariants in the adjoint action.\footnote{This is automatic, by results of Steinberg \cite{Steinberg}, if $p$ is a very good prime for $\mathbf{G}$;
 in particular if $p >5$ and does not divide $r+1$ for any $\mathfrak{sl}_r$ factor of the Lie algebra, cf. \cite[6.4(b)]{Jantzen}). }
  
  \item[(d)] $\rho_p $ (the restriction of $\rho$ to $G_{\Q_p}$) satisfies (c) and (e) from Conjecture~\ref{GaloisRepConjecture} -- in particular,  
 it is equipped with a sub-functor  $\Def^{\crys}_{\rho_p}$ of the usual deformation functor, the ``crystalline deformations,''
 whose tangent space is the local $f$-cohomology. 
    \end{itemize} 
 \end{Assumption}

In particular, (a) means that the local deformation ring at $p$ is formally smooth, isomorphic to $W(k)[[Y_1, \dots, Y_s]]$ for $s=\dim H^1(\Q_p, \Ad \rho_{\Q_p})$; and the local deformation ring at $v$, for $v \in T$, is even isomorphic to $W(k)$. 
  We will sometimes refer to these conditions as ``minimal level'', although they are a bit stronger than the usual usage of that term.

Now,   in this context, we may  use  $\Def^{\crys}_{\rho_p}$ and the process of    \S \ref{lifting} to {\em define}  a derived
deformation condition at $p$, and then we get as in
Definition \ref{Local conditions}
a global deformation functor $\mathcal{F}_{\Z[\frac{1}{S}]}^{\crys} $ with crystalline conditions imposed. 
 So by \eqref{earned}  and the definition \eqref{globalfcohomology} of the global $f$-cohomology we obtain  
\begin{equation} \label{polanski} \tangent^i \mathcal{F}_{\Z[\frac{1}{S}]}^{\crys} \cong H^{i+1}_f(\Z[\frac{1}{S}], \Ad \rho). \end{equation} The functor and its representing ring will be our main object of study in the rest of this paper:

{\bf {\em In the rest of the paper, whenever we refer to a representation $\rho: \Gamma_S \rightarrow G(k)$, we assume that it satisfies conditions (a)--(d) above; whenever we refer to a deformation ring or deformation functor for $\rho$, considered as a global Galois representation,  we always understand the ``crystalline deformation ring''  or ``crystalline deformation functor'' -- that is to say, we have imposed crystalline local conditions at $p$, 
in the sense just described. We henceforth drop the subscript ``crys'' from the notation. }}

   {\bf Remark.}  Why should the ``lifting'' process of \S \ref{lifting} give a reasonable definition of the derived version of the crystalline deformation functor?
In the  Fontaine--Laffaille range, one can reasonably guess that the tangent space of a putative derived version of the local crystalline  functor 
    should be given by the $f$-cohomology (\S \ref{FLreview}) and this is enough to force this  functor to be formally smooth, thus with representing ring that is homotopy discrete.
Therefore, it must arise by means of the construction of \S \ref{lifting}.

\index{$\mathcal{R}_S$} \index{$\mathcal{R}^{\crys}$} \index{crystalline deformation ring}

  \section{Descending from Taylor--Wiles level } \label{main}

  We   continue with a representation $\rho: \Gamma_S \rightarrow
  G(k)$ satisfying the assumptions fixed in \S \ref{minimalevel}.    It is ramified at a set $S = T \cup \{p\}$.    We denote
  by $\mathcal{R}_S$ the crystalline deformation ring of $\rho$; see  \S \ref{Defringdeffunctors} 
for a summary of our various notations concerning deformation functors and rings. 

In the Taylor-Wiles method, one studies $\mathcal{R}_S$ by relating
it to a larger deformation ring, wherein one allows extra ramification.
Namely, we consider the deformation ring which 
also allows ramification at an auxiliary set $Q_n$ of primes, which is allowable
in the sense of Definition  \ref{Def:TWdefn} -- i.e., $Q_n$ satisfies some carefully chosen cohomological criteria. 
Although one eventually uses a sequence of such sets of auxiliary primes (thus the notation $Q_n$), 
{\em the set $Q_n$ of auxiliary primes can be regarded as fixed in this section.} It is only in later sections that
we will study the situation as one varies the set of auxiliary primes. 

  We will study the deformation rings further under these cohomological conditions.  More precisely,
  we have seen in \S  \ref{sec:TWprimes}   -- see especially \eqref{rugelach2}   -- how to  recover the derived deformation ring at base level from the derived deformation ring at level $S Q_n$. 
 Here we will see that the derived deformation ring  at level $S$ can actually be ``well--approximated,'' at least as far as $\tangent^0$ and $\tangent^1$ go, just using the {\em usual} deformation ring at level $SQ_n$,
 or even a sufficiently deep Artinian quotient of it; 
 the final result is Theorem \ref{TWcloseness}.  
   
\subsection{Review of tensor products}  \label{tpreview}

In Definition \ref{defn-tensor-product} we have  given a definition of the derived tensor products  $\mathcal{A} \dotimes_{\mathcal{B}} \mathcal{C}$ of pro-simplicial rings.  
This definition is somewhat {\em ad hoc}, and not even entirely functorial; however, it {\em is} functorial
for diagrams with a common level representation (i.e., where all of the pro-objects are indexed by the same category),
which will be enough for our needs; we will often briefly abusively still say ``by functoriality.''  The derived tensor product also has the following property (see Definition
\ref{defn-tensor-product} and the discussion following it): 
 
Suppose given functors $\mathcal{F}_{\mathcal{A}} \rightarrow \mathcal{F}_{\mathcal{B}} \leftarrow \mathcal{F}_{\mathcal{C}}$
 and representing pro-rings $\mathcal{A}, \mathcal{B}, \mathcal{C}$
 which are {\em nice} in the sense of Definition \ref{nice ring definition}.
According to Lemma \ref{Lemma:represent-natural-trans}
we may promote the functor maps to ring maps  (i.e., maps in pro-$\Art_k$) $\mathcal{A} \leftarrow \mathcal{B} \rightarrow \mathcal{C}$.  
In this case\begin{equation} \label{tprep} \mbox{ $\mathcal{A} \dotimes_{\mathcal{B}} \mathcal{C}$
pro-represents the functor   $\mathcal{F}_{\mathcal{A}} \times_{\mathcal{F}_{\mathcal{B}}}^h \mathcal{F}_{\mathcal{C}}$.}  \end{equation}
Although this looks obvious, this requires a word of explanation to navigate the various homotopies, so let us talk through the technical details:

{\small
Denote by (e.g.)\ $\Maps'(\mathcal{A}, -)$ the mapping space   defined as  in \eqref{map_definition} but replacing colimit by homotopy colimit.
Lemma \ref{Lemma:represent-natural-trans} gives us a diagram \begin{equation}  
 \xymatrix{
\mathcal{F}_{\mathcal{A}}  \ar[r]  &  \mathcal{F}_{\mathcal{B}}      &  \ar[l] \mathcal{F}_{\mathcal{C}} \\ 
\Maps'(\mathcal{A}, -) \ar[u]\ar[d] \ar[r]&   \Maps'(\mathcal{B},-)  \ar[u]\ar[d]  &  \ar[l]    \Maps'(\mathcal{C}, -) \ar[u]\ar[d]  \\
\Maps(\mathcal{A}, -)  \ar[r] &   \Maps(\mathcal{B},-)   &   \ar[l]  \Maps(\mathcal{C}, -) }
\end{equation} 
The squares commute only  up to natural simplicial homotopy, as in  Lemma \ref{Lemma:represent-natural-trans}. 
 By inserting some extra weak equivalences, we may replace this diagram
with a strictly commutative diagram: given a  ``homotopy coherent'' collection of maps $X_i \rightarrow Y_i$ indexed by a diagram $I$ (in our case, the arrows along a row)
we replace it by the diagram $X_i \stackrel{\sim}{\leftarrow} X_i' \rightarrow Y_i$, where we may take $X_i' = \hocolim_{j \to i} X_j$ (see e.g.\ \cite[\S 10]{Shulman}). 
Thus we have enlarged the above diagram, by inserting various weak equivalences, to obtain a strictly commutative diagram. 
Take homotopy limits of each row to obtain
a zig-zag of weak equivalences $$ \Map(\mathcal{A}, -) \times^h_{\Map(\mathcal{B},-)} \Map(\mathcal{C},-) \stackrel{\sim}{\dashrightarrow} \mathcal{F}_{\mathcal{A}} \times_{\mathcal{F}_{\mathcal{B}}}^h \mathcal{F}_{\mathcal{C}},$$
and this gives a zig-zag of equivalences exhibiting \eqref{tprep}. }
  \subsection{Sets of Taylor-Wiles primes} \label{TWintro}

Let $Q_n$ be an allowable Taylor--Wiles datum  of level $n$, as in \S \ref{TWprimesdef}.    We'll extend the functor notation of 
\S \ref{TWprimes notation section} to many primes thus:
 \index{$ \mathcal{F}_{S \coprod Q_n}$} \index{$ \mathcal{R}_n$}  
 \index{$\mathcal{F}_n^{\loc}$} \index{$\mathcal{F}_n^{\loc,\ur}$}
  $$\mathcal{F}_{S \coprod Q_n} \mbox{ or } \mathcal{F}_n \mbox{ when clear } = \mbox{ crystalline deformation  functor at level $S \coprod Q_n$} $$ 
 $$    \mathcal{R}_{S \coprod Q_n}  \mbox{ or } \mathcal{R}_n = \mbox{ representing ring for $\mathcal{F}_{S \coprod Q_n}$} $$ 
 $$\mathcal{F}_n^{\loc} = \prod_{q \in Q_n} \mathcal{F}_{\Q_q}^{T,\fram} = \mbox{ ``framed def. functor at primes in $Q$ into the torus''} $$ 
 $$\mathcal{S}_n =  \mbox{ representing ring for $\mathcal{F}_n^{\loc}$} $$
 $$\mathcal{F}_n^{\loc,\ur} =\prod_{q \in Q}  \mathcal{F}_{\Z_q}^{T,\fram}= \mbox{ ``unramified framed  def.functor at primes in $Q$ into the torus''} $$ 
  $$  \mathcal{S}_n^{\ur} =  \mbox{ representing ring for $\mathcal{F}_n^{\loc,\ur}$} $$

   We apologize for one feature of the notation: 
   the $\mathcal{S}$-rings  and the $\mathcal{F}^{\loc}$-functors are framed, but
  we have not put explicit $\fram$ into the notation, to keep the typography simple. 
 
Using  squares (a), (b), (c) of  \eqref{TWdiag}, or rather an analogous diagram  imposing crystalline conditions at each stage and using $Q_n$ instead of just $\{q\}$, gives
\begin{equation} \label{moonrock}   \mathcal{F}_{S \coprod Q_n}' \times^h_{{ \mathcal{F}_n^{\loc}}'}{ \mathcal{F}_n^{\loc,\ur}}'  \weakequ \mathcal{F}_{S} \end{equation}
where a prime denotes a weakly equivalent functor.   

The functors on the left are pro-representable, with representing pro-rings $\mathcal{R}_n, \mathcal{S}_n, \mathcal{S}_n^{\ur}$ (by definition, if a functor is pro-representable, so is a weakly equivalent functor, and they can be taken to be represented by the same ring.)
We may suppose that these representing rings are nice, in the sense of Definition \ref{nice ring definition}.
Just as in the discussion of \S \ref{tpreview},
we obtain a corresponding diagram of rings $\mathcal{R}_n \leftarrow \mathcal{S}_n \rightarrow \mathcal{S}_n^{\ur}$.

The maps $   \mathcal{F}_{S \coprod Q_n}' \rightarrow {\mathcal{F}_n^{\loc}}' $ gives a map $\mathcal{S}_n \rightarrow \mathcal{R}_n$, by Lemma \ref{Lemma:represent-natural-trans} and the Remark following it; 
it is compatible with the original map of functors in the sense specified in that Lemma. The corresponding map 
$\pi_0 \mathcal{S}_n \rightarrow \pi_0 \mathcal{R}_n$ is then the natural map of usual (underived) deformation rings. Similarly
we have $\mathcal{S}_n \rightarrow \mathcal{S}_n^{\ur}$.

The equivalence \eqref{moonrock}  implies that \begin{equation} \label{ringmap} \mathcal{R}_S \simeq \mathcal{R}_n \dotimes_{\mathcal{S}_n} \mathcal{S}_n^{\ur}\end{equation}
where $\simeq$ here means that the functors that they represent are naturally weakly equivalent. A brief word of warning about this notation is in order. We have not defined
``directly'' the notion of weak equivalence for pro-simplicial rings, and nor do we need it: we think only in terms of functors represented.  
However,   Lemma \ref{Lemma:represent-natural-trans}  shows at least that we may find a map of pro-objects  $\mathcal{R}_S  \rightarrow  \mathcal{R}_n \dotimes_{\mathcal{S}_n} \mathcal{S}_n^{\ur}$
 which induces an isomorphism on tangent complexes; the assumptions of that Lemma are satisfied because we built ``niceness'' into the definition of derived tensor product. 

\subsection{Setup} \label{patchingdiscussion}

    In what follows we suppose given  pro-Artinian quotients   \index{$\usbR_n$} \index{$\usbS_n$} \index{$\usbS_n^{\ur}$}
\begin{equation}\label{Artin-quotients}\pi_0 \mathcal{R}_n \twoheadrightarrow \usbR_n, \pi_0 \mathcal{S}_n \twoheadrightarrow  \usbS_n, \  \pi_0 \mathcal{S}_n^{\ur} \twoheadrightarrow  \usbS_n^{\ur} \end{equation}
    which are compatible, in that there is a diagram $\usbR_n \leftarrow \usbS_n \rightarrow \usbS_n^{\ur}$ which is compatible with the same diagram for the $\pi_0$ rings.

We get   \begin{equation} \label{ringmap22} \mathcal{R}_S  \stackrel{\eqref{ringmap}}{\longrightarrow} \mathcal{R}_n \dotimes_{\mathcal{S}_n} \mathcal{S}_n^{\ur} 
\rightarrow \usbR_n \dotimes_{\usbS_n} \usbS_n^{\ur}. \end{equation}   
where $\dotimes$ is derived tensor product.  Recall that derived tensor product was not genuinely functorial: it is only so when we have fixed
a common level representation for all the pro-objects appearing. However, we will allow ourselves to ignore this issue. For one thing,
it is easy to fix such a level representation in the case at hand. For another, one could avoid the language
of derived tensor product of rings entirely and just work with the functors; we talk about rings just in the hope that they are easier to relate to. 
At the level of represented functors, \eqref{ringmap22} corresponds to the diagram (see \eqref{sbpq} for an explication of this)
       {\small  \begin{equation} \label{fibersquare0}  \Map(\usbR_n^c, -) \times^h_{ \Map(\usbS_n^c, -)} \Map((\usbS_n^{\ur})^c, -)  \dashrightarrow   (\mathcal{F}_{S \coprod Q_n}')  \times_{( \mathcal{F}_n^{\loc} )' } (\mathcal{F}_n^{\loc,\ur})'  \weakequ \mathcal{F}_{S} \end{equation}  } 
     where the second map came from \eqref{moonrock}.  

We show that, for allowable Taylor-Wiles data $Q_n$ (cf. \S \ref{TWprimesdef}),  the  
map of  \eqref{ringmap22} (equivalently the composite of \eqref{fibersquare0}) 
is isomorphic on $\tangent^0$ and surjective on $\tangent^1$.    The point of this result  is that the Taylor-Wiles limit process will give very tight control on   $\usbR_n$ and the other rings
   appearing on the right-hand side of \eqref{ringmap22}.

\begin{theorem}  \label{TWcloseness}  
Let $Q_n = \{\ell_1, \dots, \ell_q\} $ be an allowable Taylor--Wiles datum  of level $n$, in the sense defined in \S \ref{TWprimesdef}.
Let other notation be as above.  
If the maps  of \eqref{Artin-quotients} all induce isomorphisms on $\tangent^0$, then 
the   map   $\mathcal{R}_S \rightarrow  \usbR_n \dotimes_{\usbS_n} \usbS_n^{\ur}$ (or equivalently
the map
 \eqref{fibersquare0} of functors)  induces an isomorphism on $\tangent^0$ and a
surjection on $\tangent^1$.
\end{theorem}

\proof Recall our notation for places: $p$ is the characteristic of
$k$, $S$ is the ramification set of $\rho$, $T = S-\{p\}$, $Q$ is the
set of Taylor-Wiles primes. We also put $S' = S \cup Q$.  Our
``allowability'' condition on Taylor--Wiles primes means (see
\eqref{vanishing_sequence} and the conditions on $\rho$ from \S \ref{minimalevel})
that we get
\begin{equation} \label{starredeqn}
  \begin{aligned}
    H^1(\Z[\frac{1}{S'}], \Ad \rho ) &\stackrel{A}{\twoheadrightarrow}
    \frac{ H^1(\Q_p , \Ad \rho)}{H^1_f(\Q_p,\Ad \rho)},\\    
    H^2(\Z[\frac{1}{S'}], \Ad \rho) &\stackrel{\sim}{\rightarrow}
    \bigoplus_{S' } H^2(\Q_v, \Ad \rho).
  \end{aligned}
\end{equation}
  
The tangent complexes of the  derived rings can be computed via \eqref{polanski}; here we get:  $$ \tangent^0 \mathcal{R}_n = \ker(H^1(\Z[\frac{1}{S'}], \Ad \rho) \rightarrow H^1(\Q_p, \Ad \rho)/H^1_f ),$$
 \begin{equation} \label{tangentRn} \tangent^1 \mathcal{R}_n \stackrel{\sim}{\rightarrow}  H^2(\Z[\frac{1}{S'}], \Ad \rho). \end{equation}
 Note that \eqref{polanski} actually says that $\tangent^1$ is the kernel of $H^2(\Z[\frac{1}{S'}], \Ad \rho) \rightarrow H^2(\Q_p, \Ad \rho)$, 
 augmented by the cokernel of $A$, 
 but $H^2(\Q_p, \Ad \rho)$ vanishes by assumption (\S \ref{minimalevel}) and the cokernel of $A$ is zero by \eqref{starredeqn}. 
 
  From \eqref{tangentRn}, \eqref{starredeqn} and the assumed vanishing (\S \ref{minimalevel}) of $H^2(\Q_v, \Ad \rho)$ for $v \in T$ and for $v=p$, we see that  
 the map  \begin{equation} \label{TWkey} \tangent^1 \mathcal{R}_n \rightarrow  \tangent^1 \mathcal{S}_n  \cong  \prod_{v  \in  Q} H^2(\Q_v, \Ad \rho)\end{equation} 
 is an isomorphism.    (As concerns the latter equality: $\tangent^1 \mathcal{S}_n$ is a priori a product of Galois cohomology groups
 with coefficients in the Lie algebra of the torus, but then we can apply  Lemma \ref{TWcohomology}). 

Also the maps 
\begin{equation} \label{inj}  \mathfrak{t}^i \mathcal{R}_n \rightarrow \mathfrak{t}^i \usbR_n,  \mathfrak{t}^i \mathcal{S}_n \rightarrow \mathfrak{t}^i \usbS_n, \mathfrak{t}^i \mathcal{S}_n^{\ur} \rightarrow \mathfrak{t}^i \usbS_n^{\ur}\end{equation}
 are isomorphisms for $i=0$ by assumption; and also
\begin{equation} \label{sophie} \mathfrak{t}^1 \mathcal{S}_n^{\ur} = 0,\end{equation}
 since $\mathcal{S}_n^{\ur}$ is formally smooth (see Lemma \ref{HDiscreteness}, and \eqref{have-you-not-heard-of-this?}).

 We can now verify the claim:

For $\tangent^0$,   we must check that $j_1$ is an isomorphism in the following diagram, where the rows
are the long exact sequences arising from Lemma \ref{lem:2.51} (iv): 
 \begin{equation} 
      \begin{aligned}
        {\small
        \xymatrix{
          0 \ar[r] \ar[d]  & \tangent^0(\usbR_n  \otimes_{\usbS_n } \usbS_n^{ur} )  \ar[r]  \ar[d]^{j_1} &\tangent^0(\usbR_n )  \oplus \tangent^0(\usbS_n^{ur}  ) \ar[r]  \ar[d]^f & \tangent^0(\usbS_n )  \ar[d]^g \\
          0 \ar[r] & \tangent^0 (\mathcal{R}_n \otimes_{\mathcal{S}_n}
          \mathcal{S}_n^{ur}) \ar[r]^{ \alpha \qquad} &
          \tangent^0(\mathcal{R}_n) \oplus
          \tangent^0(\mathcal{S}_n^{ur}) \ar[r] &
          \tangent^0(\mathcal{S}_n) }
        }
      \end{aligned}
 \end{equation}
Indeed, by \eqref{inj},  $f,g$ are both isomorphisms. Since $f$ is injective, $j_1$ is injective. Since $\alpha$ is injective,  
 $f$ is surjective and $g$ is injective, we get that  $j_1$ is surjective.  Therefore $j_1$ is an isomorphism, as desired.

For $\tangent^1$ we must check that $j_2$ is surjective in the following diagram, which is  just the  continuation of the previous one:
 \begin{center}
 {\small 
\begin{equation} \label{foot2} 
 \xymatrix{
\tangent^0(\usbR_n )  \oplus \tangent^0(\usbS_n ) \  \ar[r]  \ar[d]^f  & \tangent^0(\usbS_n ) \ar[r] \ar[d]^g & \tangent^1(\usbR_n  \otimes_{\usbS_n } \usbS_n)  \ar[r] \ar[d]^{j_2} &\tangent^1( \usbR_n )  \oplus \tangent^1( \usbS_n )\ar[d] \\ 
\tangent^0(\mathcal{R}_n)  \oplus \tangent^0(\mathcal{S}_n^{ur}) \  \ar[r]  &\tangent^0(\mathcal{S}_n) \ar[r]^{\beta \qquad} & \tangent^1 (  \mathcal{R}_n \otimes_{\mathcal{S}_n} \mathcal{S}_n^{ur})    \ar[r]^{\gamma} & \tangent^1( \mathcal{R}_n )  \oplus \tangent^1(\mathcal{S}_n^{ur}) \ar[r]^{\qquad h} & \tangent^1( \mathcal{S}_n ) \\ 
 }
 \end{equation}}
 \end{center}
   Because of \eqref{TWkey}  and \eqref{sophie} the kernel of $h$ is zero. So
 $\gamma$ is zero.  So $\beta$ is surjective.  We saw above that $g$ is an isomorphism. So $j_2$ is surjective.  \qed

\section{A patching theorem}  \label{sec:patching} 

The current section proves a ``patching'' theorem  for the derived deformation ring. 
In terms of the discussion of the introduction, we describe the compactness argument required to extract a limit in the context
of \eqref{patching_baby}.

\begin{theorem} \label{Patching_Theorem} 
Suppose given $\mathcal{R}_0$ a pro-object of $\Art_k$ such that 
  $\tangent^i \mathcal{R}_0$ is supported in degrees $0$ and $1$.\index{$\usrS_{\infty}$}\index{$\usR_{\infty}$} 
 Suppose also given a continuous map of (usual)  complete local rings  $\iota:   \usrS_{\infty} \rightarrow \usR_{\infty}$
where $\usrS_{\infty} = W(k)[[X_1, \dots, X_s]], \usR_{\infty} = W(k)[[X_1, \dots , X_{s-\delta}]]$,  
the map $\iota$ makes $\usR_{\infty}$ a finite $\usrS_{\infty}$-module, and 
\begin{equation} \label{assump:dim} \dim \  \mathfrak{t}^0 \mathcal{R}_0 - \dim \  \mathfrak{t}^1 \mathcal{R}_0 = \dim(\usR_{\infty}) - \dim (\usrS_{\infty}) (=\delta).\end{equation}
  Let $\mathfrak{a}_n$ be the descending sequence  of ideals of $\usrS_{\infty}$ defined as
$\mathfrak{a}_n = (p^n, (1+X_i)^{p^n}-1)$.

Regard the Artinian rings $\usR_{\infty}/\mathfrak{a}_n, \usrS_{\infty}/\mathfrak{a}_n, W_n$ 
as constant objects of pro-$\Art_k$, indexed by a category $J$ that is independent of $n$.
 Set 
$$\mathcal{C}_n :=  \usR_{\infty}/\mathfrak{a}_n \dotimes_{\usrS_{\infty}/\mathfrak{a}_n} W_n$$
where $\dotimes$ is derived tensor product, as in Definition~\ref{defn-tensor-product}, and where the map $\usrS_{\infty} \rightarrow W_n$ is the natural augmentation. 
By functoriality we obtain $e_{n,m}: \mathcal{C}_n \rightarrow \mathcal{C}_m$ for $n  > m$.

  Finally, suppose given     a collection of maps in pro-$\Art_k$  
\begin{equation} \label{input-to-theorem} f_n :   \mathcal{R}_0 \rightarrow  \mathcal{C}_n  ,\end{equation} 
   such that, for every $n > m$,
if we write $f_{n,m} = e_{n,m} \circ f$ for the composite 
$$f_{n,m}: \mathcal{R}_0 \rightarrow \mathcal{C}_n   \stackrel{e_{n,m}}{\rightarrow} \mathcal{C}_m   $$
we have 
\begin{equation}\label{input-to-theorem-2} \mbox{$\mathfrak{t}^0 f_{n,m}$ is an isomorphism and  $\mathfrak{t}^1 f_{n,m}$ is a surjection.}\end{equation} 
 
Then there is  an isomorphism of graded rings $$ \pi_* \mathcal{R}_0 \cong \Tor^*_{\usrS_{\infty}}(\usR_{\infty}, W(k)),$$
where we understand $\pi_i \mathcal{R}_0$  of the pro-object $(\mathcal{R}_{0,n})$ to be defined as
$\varprojlim \pi_i \mathcal{R}_{0,n}$ (for discussion of this  definition, see
around \eqref{naive pi}).   \end{theorem}

Note we don't assume any type of compatibility between the $f_n$. What we will do instead is to  extract a compatible sequence by compactness.   
Note  also that if the map $\usrS_{\infty} \rightarrow \usR_{\infty}$ is surjective,  the $\Tor$-algebra above is actually an exterior algebra on $s-r$ generators.

\subsection{The proof of Theorem \ref{Patching_Theorem}}  \label{patching-proof} 

In this proof we shall use a simplicial enrichment of the category
pro-$\smallC_k$, defined as
pro-$\smallC_k(A,B) = \lim_i \colim_j \smallC_k(A_j,B_i)$, if
 $A = (j \mapsto A_j)$ and $B = (i \mapsto B_i)$.  
 
 Note that  this has a nice meaning in terms of  functors represented:
the $q$-simplices of pro-$\smallC_k(A,B)$ are natural transformations from $\Delta^q \times  \colim_i \smallC_k(A_i, -)$ to $\colim_j \smallC_k(B_j, -)$.

 If each $A_i$ is
cofibrant and $B$ is \emph{nice}, then the natural map
\begin{equation}\label{eq:22}
  \holim_i \colim_j
  \smallC_k(A_j,B_i) \to \text{pro-}\smallC_k(A,B)
\end{equation}
is a weak equivalence,    by the argument of
Lemma \ref{lem:factor-through-strict-colim}; 
 we shall only ever consider the simplicial
set pro-$\smallC_k(A,B)$ when~\eqref{eq:22} holds.

Let us also note
that if $\pi_*\bigtangent A$ is finite dimensional, then the space
pro-$\smallC_k(A,B_i)$ has finite homotopy groups for all $i$:   this 
reduces to the case $B_i = k\oplus k[n]$  by applying  Lemma  \ref{lem:build-Artin-rings-Postnikov}.
Hence in that case we have (assuming~\eqref{eq:22} as always, and   using
Lemma \ref{pi0 commutes holim 1}): 
\begin{equation*}
  \pi_0 \big(\text{pro-}\smallC_k(A,B)\big) = \lim_i \pi_0 \big(\text{pro-}\smallC_k(A,B_i)\big).
\end{equation*}
We shall write $[A,B]$ for this profinite set.  The above equation
then implies $[A,B] = \lim_i [A,B_i]$.

We wish to apply the above discussion for $A = \mathcal{R}_0$ and
$B = \mathcal{C}_n$, and see that $[\mathcal{R}_0,\mathcal{C}_n]$ is
in fact a finite set.  By definition, the levels of the pro-object
$B = \mathcal{C}_n$ have the homotopy types
$B_i \simeq \tau_{\leq i}\big(\usR_{\infty}/\mathfrak{a}_n
\otimes_{\usrS_{\infty}/\mathfrak{a}_n} W_n^c\big)$, where $W_n^c$ is
a cofibrant replacement of $W_n$ as an algebra over  
$\usrS_\infty/\mathfrak{a}_n$.    Up to homotopy $B_{i+1}$ may be
obtained from $B_i$ by taking homotopy fibers of a map to
$k \oplus k[i+2]$ finitely many times (see 
Lemma \ref{lem:build-Artin-rings-Postnikov} again), so the assumption 
that $\pi_* \bigtangent \mathcal{R}_0$ vanishes in degrees besides $0,-1$    implies that the map
$[A,B_{i+1}] \to [A,B_i]$ is bijective for $i \geq 1$.  
Just as above, $[A,B_i]$ is
finite for all $i$.  Hence $[A,B]$ is finite in this case.

For any
$[f] \in [\mathcal{R}_0,\mathcal{C}_n]$ there are well defined induced
maps
\begin{equation*}
  \bigtangent^i(e_{n,m} \circ f): \bigtangent^i \mathcal{C}_m \to
  \bigtangent^i \mathcal{C}_n \to \bigtangent^i\mathcal{R}_0
\end{equation*}
for all $m \leq n$ and all $i$.  As $n$ varies the finite
sets $[\mathcal{R}_0,\mathcal{C}_n]$ form an inverse system, and we shall consider the
subsystem
\begin{equation*}
  X_n = \{[f] \in [\mathcal{R}_0,\mathcal{C}_n] \mid
  \text{$\forall m \leq n:$ $\bigtangent^i (e_{n,m} \circ f)$ is iso for
    $i=0$ \& epi for $i=1$}\}.
\end{equation*}
By assumption $X_n \neq \emptyset$ for all $n \geq 0$.  Hence this is
an inverse system of non-empty finite sets, so by compactness the
inverse limit is non-empty.

   We may therefore pick morphisms  $g_n: \mathcal{R}_0 \to \mathcal{C}_n$ representing compatible elements of the subset $X_n \subset [\mathcal{R}_0,\mathcal{C}_n]$ and hence simplicial homotopies
 between $g_n$ and the composition  $e_{n+1,n} \circ g_{n+1}$; 
 that is to say, there
 exists a $1$-simplex   in the mapping space $\text{pro-}\smallC_k(\mathcal{R}_0, \mathcal{C}_n)$
 with these vertices. This $1$-simplex induces a natural simplicial homotopy between the    natural transformations of functors 
 $$ \Hom(\mathcal{C}_n, -) \rightarrow \Hom(\mathcal{R}_0, -).$$
 induced by $g_n$ and by $e_{n+1,n} \circ g_{n+1}$.
This data induces a natural transformation of functors $\Art_k \rightarrow s\Sets$
 $$ \hocolim_n \Hom(\mathcal{C}_n, -) \rightarrow  \Hom(\mathcal{R}_0, -)$$
 which is then also an isomorphism in $\bigtangent^0$ and an
epimorphism on $\bigtangent^1$.

We claim that it is in fact an
isomorphism on $\bigtangent^i$ for all $i$, for which it suffices to
verify that both sides vanish for $i \notin \{1,0\}$, and that both
sides have the same ``Euler characteristic.''

The vanishing is true for $\tangent^* \mathcal{R}_0$ by assumption.
As for $\varinjlim_n \tangent^i \mathcal{C}_n$, there's an exact
triangle in the derived category of $k$-modules
$$ \mathfrak{t}  \mathcal{C}_n  \rightarrow \mathfrak{t}
(\usR_{\infty}/\mathfrak{a}_n) \oplus \mathfrak{t} (W_n) \rightarrow
\mathfrak{t} (\usrS_{\infty}/\mathfrak{a}_{n} )
\stackrel{[1]}{\rightarrow} $$ 
and taking cohomology and  taking direct limit as $n \rightarrow
\infty$,  we get 
$$ \tangent^i \mathcal{C}_n  \rightarrow \tangent^i \usR_{\infty}
\oplus \underbrace{ \mathfrak{t}^i  W(k) }_{0} \rightarrow  \tangent^i
\usrS_{\infty}  \stackrel{[1]}{\rightarrow}.$$
Here we have used
$\tangent^i \usR_{\infty} = \varinjlim \mathfrak{t}^i
(\usR_{\infty}/\mathfrak{a}_n)$ and similarly for $\usrS_{\infty}$,
because $\usR_{\infty}$ considered in pro-$\Art_k$ can be defined by
the projective system $\usR_{\infty}/\mathfrak{a}_n$ inside
$\Art_k$.\footnote{Details: we need to see that $\mathfrak{a}_n$
  defines the standard profinite topology on $\usR_{\infty}$.  Write
  $\mathfrak{m}_{\usR}$ for the maximal ideal of $\usR_{\infty}$. On
  the one hand, the rings $\usR_{\infty}/\mathfrak{a}_n$ are Artin
  because $\usR_{\infty}$ is module-finite over $\usrS_{\infty}$ by
  assumption; therefore there exists $n_1$ such that
  $\mathfrak{m}_{\usR}^{n_1} \subset \mathfrak{a}_n$.  On the other
  hand, $\mathfrak{a}_n \subset \mathfrak{m}_{\usR}^{n}$: clearly
  $p^n \in \mathfrak{m}_{\usR}^n$, and if $Y \in \mathfrak{m}_{\usR}$,
  then $(1+Y)^{p^n}-1 = \sum_{j \geq 1} {p^n \choose j} Y^j$ belongs
  to $\mathfrak{m}_{\usR}^n$ too.  } Now the claimed vanishing follows
from our computation of the tangent complex for power series rings
\eqref{numerical}, and the Euler characteristic claim follows from our
assumption \eqref{assump:dim}.

Hence the pro-objects $\mathcal{R}_0$ and $(n \mapsto \mathcal{C}_n)$
represent equivalent functors, and hence the induced map of homotopy
groups
\begin{equation*}
  \varprojlim \pi_* \mathcal{R}_0 \to \varprojlim \pi_* \mathcal{C}_n
\end{equation*}
is also (see \eqref{pistar iso}) an isomorphism for all $i$.  This concludes the proof of
Theorem~\ref{Patching_Theorem} since
$$\varprojlim \pi_i \mathcal{C}_n = \varprojlim
\Tor^i_{\usrS_{\infty}/\mathfrak{a}_{n}}(\usR_{\infty}/\mathfrak{a}_n,
W_n) \cong \Tor^i_{\usrS_{\infty}}(\usR_{\infty},W). $$ (for the last
step see Lemma \ref{Tor_limit}; for the computation of homotopy groups
of a tensor product, see \cite[Theorem 6]{QuillenHomotopicalAlgebra}).

\section{Background on the obstructed Taylor--Wiles method, after Calegari--Geraghty}\label{derivedTW}

  We present the ``obstructed'' Taylor--Wiles method, in the form given by Khare and Thorne \cite{KT};
 it is originally discovered by Calegari--Geraghty \cite{CG}, and another version  has been developed by D. Hansen \cite{H}. 
  This section has no original ideas  (although any errors are due to us).  
  It is largely independent of the rest of the paper; it just provides the input for us to apply our previous theorems. 
 It uses the conjecture on the existence of Galois representations, as formulated in 
 Conjecture \ref{GaloisRepConjecture}, and outputs a set
 of Taylor-Wiles primes $Q_n$ where one has very good control, if not on all of $\pi_0 \mathcal{R}_n$, 
 then at least on a ``very large'' Artinian quotient of it.

 \subsection{Notation and setup}  \label{TWnotationsetup} 
  
 Let us try to describe more carefully our setup. Thus far we have dealt with an abstract Galois representation $\rho$. We will  henceforth be
 considering the case where $\rho$ comes from a modular form, and we will specify some details about this modular form. 
\begin{enumerate}
 
  \item  As before, $S$ will be a finite set of primes containing $p$, and $T=S-\{p\}$.

 \item   As before (\S \ref{YKdef})  $\mathbf{G}$ is a split semisimple $\Q$-group dual to $G$. We suppose that $\mathbf{G}$
 admits a smooth reductive model over $\Z[\frac{1}{T}]$. 
 
  \item     We let $Y_0$ be the arithmetic manifold associated to a subgroup $K_0 =  \prod K_{0,v}$, i.e.
 \begin{equation} \label{Y0def} Y_0  =  \mathbf{G}(\Q) \backslash \mathbf{G}(\adele) / K_{\infty}^{\circ} K_0,\end{equation}
   as in \S \ref{YKdef}.

Here we suppose that  $K_{0,v}$ is obtained from the integral points of $\mathbf{G}$ for $v \notin T$,
 and for $v \in T$ we take $K_{0,v}$ to be an Iwahori subgroup.  (Informally,  
 this means one takes the preimage of a Borel subgroup over the residue field.  
For a formal definition see \cite[\S 3.7]{TitsCorvallis}). 

More generally, we define $Y(K)$  just as in \eqref{Y0def} for an open compact subgroup $K \subset \mathbf{G}(\adele_f)$; we will only deal with examples where $K$ decomposes
as a product $\prod_{v} K_v$. 
 
 \item
We follow \cite{KT} on Hecke algebras:  

   Consider the chain complex of $Y_0$ as an object in the derived category of $\Z$-modules; each
  Hecke operator gives an endomorphism of this object. 
   The Hecke algebra $\mathrm{T}_0$ for $Y_0$ then means  the ring   of endomorphisms 
 generated by all Hecke operators prime to the level.      
  (This has a few minor advantages, in particular,  $\mathrm{T}_0$ acts on cohomology with any coefficients.)

 For some larger level $K'$, the Hecke algebra $\mathrm{T}_{K'}$ is defined in the same way (again, Hecke operators relatively prime to the level). 
 
 Warning: we will use at certain points slight variations of this definition, in particular enlarging the Hecke operator by using certain operators at primes dividing the level,
  but if used without explanation, ``Hecke algebra'' should be taken in the sense just defined. 
 
 \item  The invariants $q,\delta$:
 
 It is known
(see 
  \cite[III, \S 5.1]{BW}),  \cite[VII, Theorem 6.1]{BW} and for the non-compact case \cite[5.5]{Borel2})  
  that the tempered cuspidal cohomology  $$H^*(Y(K), \C)_{\temp}$$  of $Y(K)$,
 i.e.\ that part of the cohomology associated to tempered cuspidal representations
 under the standard indexing of cohomology by representations, is concentrated  in degrees $[q, q+\delta]$. 
  Here $2q+\delta =\dim Y(K)$ and $\delta$ is the difference $\mathrm{rank} \  \G(\R) - \mathrm{rank}  \  K_{\infty}$. 

 \item 
 Fix a Hecke eigenclass $f \in H^q(Y_0, \overline{\Q})_{\temp} := H^q(Y_0, \overline{\Q}) \cap H^q(Y(K), \C)_{\temp}$; let $K_f$ be the field generated
 by all Hecke eigenvalues, and let $\mathcal{O}$ be the ring of integers of $K_f$. 
 Thus $f$ defines a homomorphism
 $$ \mathrm{T}_0 \longrightarrow \mathcal{O}$$
 from the Hecke algebra for $Y_0$ to $\mathcal{O}$.    
 
\item   Let $\wp$ be a  prime of $\mathcal{O}$, above the rational prime $p$, such that   \begin{itemize}
 \item[(a)]  $H^*(Y_0, \Z)$ is $p$-torsion free. \item[(b)] $p$ is ``large'' relative to $\GG$: larger than the order of the Weyl group. 
\item[(c)]  ``No congruences between $f$ and other forms:''  Writing $\mathfrak{m}$ for the maximal ideal of $\mathrm{T}_0$ 
given by the kernel of $\mathrm{T}_0 \rightarrow \mathcal{O}/\wp$, 
the induced map of completions $\left(\mathrm{T}_0\right)_{\mathfrak{m}} \rightarrow \mathcal{O}_{\wp}$ is an isomorphism. 

 \item[(d)] $\mathcal{O}$ over $\Z$ is unramified at $\wp$. 
\item[(e)] The localization $H_j(Y_0, \Z_p)_{\mathfrak{m}}$  vanishes in degrees $j \notin [q, q+\delta]$. 
   \end{itemize} 
 
 One expects that all the conditions are   satisfied for all but  finitely many $\wp$.   
 
 Note also that if $K' \supset K_0$ is some deeper level structure, with Hecke algebra $\mathrm{T}_{K'}$,
  there is still a morphism $\mathrm{T}_{K'} \rightarrow \mathcal{O}$; as an abuse of notation
  we will still use $\mathfrak{m}$ for the correspondingly defined maximal ideal,
  and $\mathrm{T}_{K', \mathfrak{m}}$ the completion of $\mathrm{T}_{K'}$. 
   
 \item  
Set $k=\mathcal{O}/\wp$, the residue field of $\wp$. 
Let  $$\rho:
 \pi_1 \Z[\frac{1}{S}] \rightarrow G(k)$$
 be the Galois representation (modulo $\wp$) associated to $f$  (see Conjecture
\ref{GaloisRepConjecture}).  Again, by Conjecture \ref{GaloisRepConjecture}, it is the reduction of a representation
\begin{equation} \label{tilderhodef} \rhoglob:  \pi_1 \Z[\frac{1}{S}] \rightarrow G(\mathcal{O}_{\wp}= W(k)).\end{equation}
 
 \item 
  We suppose 
that the image of $\rho$ satisfies the conditions of \S \ref{minimalevel}, i.e.\  (among others)  big image, trivial deformation theory at $T$, 
and crystalline with small weights at $p$.

 \end{enumerate}

 We prove the following result (we apologize that to precisely formulate the Theorem here, we have to make a few forward references to later in the text):
 
 \begin{theorem} \label{TWoutputtheorem}
 Let assumptions be as in \S  \ref{TWnotationsetup} and assume also Conjecture
 \ref{GaloisRepConjecture}. 
 Let $q$ be a large enough integer, and let $s = q \cdot \mathrm{rank}(G)$. 
 Let $\delta$ be the ``defect'' of $G$   as in (5) of \S \ref{TWnotationsetup}. Set 
 $$\usrS_{\infty} = W(k)[[X_1, \dots, X_s]], \usR_{\infty} = W(k)[[X_1, \dots , X_{s-\delta}]],$$
Let  $\mathfrak{a}_n $ be  the ideal  of $\usrS_{\infty}$ defined by 
\begin{equation} \label{ardef} \mathfrak{a}_n =  (p^n, (1+X_i)^{p^n}-1)\end{equation}  inside $\usrS_{\infty}$;
thus $\usrS_{\infty}/\mathfrak{a}_n$ is naturally identified with the $\Z/p^n$-group algebra of $(\Z/p^n)^s$. 
Then we can find the following data: 
\begin{itemize}
\item[(a)]  Homomorphisms $\usrS_{\infty} \stackrel{\iota}{\longrightarrow }\usR_{\infty} \twoheadrightarrow  \pi_0 \mathcal{R}_{\Z[\frac{1}{S}]}$ of complete local rings, 
whose composite is the natural augmentation $\usrS_{\infty} \rightarrow W \rightarrow \pi_0 \mathcal{R}_{\Z[\frac{1}{S}]}$.   
\item[(a)']  For each integer $n$, allowable Taylor--Wiles data $Q_n$ of level $n$,
with associated covering groups $\Delta_n$ (\eqref{DeltaQdef}), group rings
$\overline{\usrS_n} := W_n[\Delta_n]$, and deformation rings $\mathcal{R}_n$; 

\item [(b)']  An explicit function $K(n) \rightarrow \infty$ and for each integer $n$,   homomorphisms  $f_n, g_n$ rendering commutative the diagram 
   \begin{equation} \label{SinftySn0}
 \xymatrix{
\usrS_{\infty}  \ar[r]^{\iota} \ar[d]^{f_n}  &  \usR_{\infty}  \ar[d]^{g_n} \ar[r]  & \ar[d]  \pi_0 \mathcal{R}_{\Z[\frac{1}{S}]}     \\
\overline{\usrS_n} \ar[r]   & \usbR_n  \ar[r] &  \pi_0 \mathcal{R}_{\Z[\frac{1}{S}]}/(p^n, \mathfrak{m}^{K(n)})  
 }
\end{equation}

  where:
  \begin{itemize}
  \item[-]  we write $\usbR_n$ for the quotient of $\pi_0 \mathcal{R}_n/ (p^n, \mathfrak{m}^{K(n)})$
 that classifies deformations of inertial  level $\leq n$ (see \eqref{ILL def})  at all primes in $Q_n$,  
 \item[-] 
  the map $\overline{\usrS_n} = W_n[\Delta_n] \rightarrow \usbR_n$  is the natural one (explained in and before \eqref{quotient2}).    
  \item[-] All the maps commute with the natural augmentations to $k$, i.e.\ are local homomorphisms. 
  \end{itemize} 
 Also $f_n(\mathfrak{a}_n) = 0$, and the induced maps
\begin{equation} \label{Snid0}  \begin{cases} \usrS_{\infty}/\mathfrak{a}_n \longrightarrow  \overline{\usrS_n}, \\ \usR_{\infty}/\mathfrak{a}_n \cong \ \usR_{\infty} \otimes_{\usrS_{\infty}}   \overline{\usrS_n}   \longrightarrow \usbR_n  \end{cases} \end{equation}
are both isomorphisms.

 \item[(d)]  \label{GGB2}
 A  complex $\Dinf$ of finite free $\usrS_{\infty}$-modules, equipped with a compatible $\usR_{\infty}$-action in the derived category of $\usrS_{\infty}$-modules;
 and equipped with an isomorphism
 $$ H_*(\Dinf \otimes_{\usrS_{\infty}} W) \cong H_{*}(Y_0, W)_{\mathfrak{m}} ,$$
 which is compatible for the $\usR_{\infty}$-actions. Here $\usR_{\infty}$ acts on the right-hand side via the map $\usR_{\infty} \rightarrow \pi_0 \mathcal{R}_{\Z[\frac{1}{S}]}$
 from (a),  followed by the action of $\pi_0 \mathcal{R}_{\Z[\frac{1}{S}]}$ supplied by  Conjecture \ref{GaloisRepConjecture}. 
 Moreover, $H_*(\Dinf)$ is concentrated in degree $q$ and $H_q(\Dinf)$ is a finite free $\usR_{\infty}$-module. 
 \item[(e)]  [Easy consequence of (d):] 
 The natural surjections   \begin{equation} \label{RequalsT} \usR_{\infty}\otimes_{\usrS_{\infty}} W  \twoheadrightarrow \pi_0 \mathcal{R}_{\Z[\frac{1}{S}]} \rightarrow  \mbox{image of Hecke algebra on $H_{q}(Y_0, W)_{\mathfrak{m}} $}\end{equation}
are both isomorphisms.  The action of $\usR_{\infty} \otimes_{\usrS_{\infty}} W$ on $H_*(Y_0, W)_{\mathfrak{m}}$
extends to a free graded action of  $\Tor_*^{\usrS_{\infty}}(\usR_{\infty},  W) $ on $ H_{*} (Y_0, W)_{\mathfrak{m}} $, 
over which this homology is freely generated in degree $q$. 
 
\end{itemize}
\end{theorem}

This is exactly the setup needed to apply the results of the prior sections.
For later use, observe that  our assumptions (7(e) of \S \ref{TWnotationsetup})  imply that the ``image of the Hecke algebra'' appearing in \eqref{RequalsT} is simply $W$; 
then (e) implies that $\usrS_{\infty} \rightarrow \usR_{\infty}$  is surjective.

\subsection{  Outline of the argument}  
We briefly outline how the argument goes. 
The discussion is somewhat vague, but we hope it is better than saying nothing at all. 

The sets $Q_n$ will be allowable Taylor--Wiles  data (\S \ref{TWprimesdef}). The cohomological condition
in ``allowable'' means  that we know exactly the minimal number of generators for $$\usR_n := \pi_0 \mathcal{R}_n$$
(the usual deformation ring at level $S \coprod Q_n$); this number is $s-\delta$.  Any choice of elements $x_1, \dots, x_{s-\delta} \in \mathfrak{m}(\usR_n)$ projecting to a basis for $\mathfrak{m}/\mathfrak{m}^2$ results in surjections
 $$ W[[x_1, \dots, x_{s-\delta}]] \twoheadrightarrow \usR_n \twoheadrightarrow k[[x_1, \dots, x_{s-\delta}]]/(x_i x_j).$$ 
 
Informally speaking, the content of the Theorem is that $\usR_n$ actually becomes closer and closer to $W[[x_1, \dots, x_{s-\delta}]]$ as $n$ gets large.
To prove this, we construct a morphism
\begin{equation} \label{iota finite} \usrS_{\infty} \rightarrow \usR_n\end{equation} 
which comes, in the end, from studying the action of the inertia group at places in $Q_n$ (see \eqref{DeltaR}).
We try to show that the image of $\usrS_{\infty}$ is ``big''. The map $\iota$ of  Theorem  \ref{TWoutputtheorem}  also arises by taking a ``limit'' of the maps \eqref{iota finite}, using a compactness argument.

To  prove that the image of $\usrS_{\infty}$ is big,  we study the action of $\usR_n$ on a space of modular forms -- indeed, the homology of a certain arithmetic manifold obtained by adding level $Q_n$. 
We claim the pullback of this action    to $\usrS_{\infty}$ has kernel that is ``not too large.''  The action  of $\usR_n$ comes from Conjecture~\ref{GaloisRepConjecture}, and the ``local-to-global compatibility'', which we shall state more precisely in Section~\ref{mprimeconstruc}, asserts that the resulting action of $\usrS_\infty$  agrees 
with the action defined geometrically via
``diamond operators,'' i.e.\
out of the action of explicit groups $\Delta_n$ of automorphisms of the underlying arithmetic manifold. (Studying  such genuine automorphisms is much easier than studying Hecke operators.)

The key point here comes from the fact that $\Delta_n$ acts freely on the underlying arithmetic manifold, and so  the homology is computed from a complex of free modules whose length is tightly bounded from above. 
Then the Auslander-Buchsbaum theorem states, roughly, that if a module has a short projective resolution, then its annihilator cannot be too large. 
The relevance of this last point is  a key observation of Calegari and Geraghty.

 If the reader will forgive one more vagueness, the mechanism of this last argument is similar to a well-known mechanism that prohibits free group actions of a finite group $\Delta$
 on a homology sphere $M$: the quotient $M/\Delta$ has homology ``approximated'' by that of $\Delta$, and thus has difficulty looking like the homology of a manifold.

  \subsection{Construction of complexes}  \label{KTreduced}
  
  A key  point for the analysis is to construct complexes 
  which compute the cohomology of various $Y(K)$'s localized at $\mathfrak{m}$.
  These complexes should be of  reasonable size (much smaller than the full chain complex) so that one can make compactness arguments.
  These should also carry 
  actions of the usual deformation ring, when we consider them as objects in the derived category. 
 We summarize briefly here the properties of these complexes and how they are constructed. 
  \begin{Lemma}  \label{ConsequenceOfConjectureLemma} Assume Conjecture \ref{GaloisRepConjecture}. 
 Then for each  complete local $W$-algebra $E$ and 
 each pair $(K \vartriangleleft K')$ as in \S \ref{sec:galois-rep-conjecture}, with $\Delta =K'/K$ acting in the natural way on $Y(K)$, 
we may construct an object 
  $C_*^{\Delta}(Y(K); E)_{\mathfrak{m}}$   
  in the derived category of $E\Delta$-modules\footnote{Despite the notation,   $C_*^{\Delta}(Y(K); E)_{\mathfrak{m}}$    is not intended
  as the localization or completion of some complex at $\mathfrak{m}$!} , equipped with a $\Delta$-equivariant identification of its homology with  $H_*(Y(K); E)_{\mathfrak{m}}$,  
  in such a way that the following properties are satisfied:
  
 \begin{itemize}
 \item[a.]   The object $C_*^{\Delta}(Y(K), E )_{\mathfrak{m}}$ is  quasi-isomorphic to a bounded complex of finite free $E \Delta$-modules. 
 (In fact, we will just fix such a quasi-isomorphism and use  $C_*^{\Delta}(Y(K), E )_{\mathfrak{m}}$ to denote that complex.) 
 \item[b.]  Let $\mathrm{R}_{K}$ be the (usual, underived) deformation ring  classifying deformations of $\rho$  that are unramified at each place where $K$ is hyperspecial. 
  There is an action of $\mathrm{R}_K$ on $C_*^{\Delta}(Y(K); E)_{\mathfrak{m}}$ by endomorphisms in the derived category of $E\Delta$-modules, satisfying, at the level of homology,  
 the usual compatibility between Frobenius and Hecke operators.  
 
 \item[c.] 
 We have  descent from level $K$ to level $K'$, i.e.\ there is an isomorphism 
 $$C_*^{\Delta}(Y(K), E)_{\mathfrak{m}} \otimes_{E \Delta} E \simeq C_*(Y(K'), E)_{\mathfrak{m}}$$
 where $\otimes_{E \Delta} E$ refers to the derived tensor product, a functor
 from the derived category for $E\Delta$ to that for $E$.  This homomorphism is compatible with  the homomorphism $\mathrm{R}_{K} \rightarrow \mathrm{R}_{K'}$.

\end{itemize} 
\end{Lemma}

 \proof (Outline, summarizing the work of Khare--Thorne \cite{KT}):   

The main point is that we can do the localization at $\mathfrak{m}$ at the derived level:

 Take first of all $C_*^{\Delta}(Y(K), E)$ to be the chain complex of $Y(K)$ with $E$ coefficients, 
considered as an object in the derived category of $E \Delta$-modules.  Let $\mathrm{T}_K$ be the Hecke algebra at level $K$.
(For the current discussion this means the Hecke algebra generated by prime-to-$K$ operators, as in \S \ref{TWnotationsetup};  however, the same arguments will apply in the setting
of \S \ref{mprimeconstruc}, where will use a slightly larger Hecke algebra.) 

   As before we define
 $\mathrm{T}_{K,\mathfrak{m}}$  as the completion of $\mathrm{T}_K$ at $\mathfrak{m}$; it  is a direct factor of the $p$-completion of $\mathrm{T}_K$. Denote by $e$ the corresponding idempotent; it is an endomorphism
of $C_*^{\Delta}(Y(K), E )$ in the derived category, and thus \cite[Proposition 3.4]{BN} there is a splitting $C_* = C_* e  \oplus C_* e'$ corresponding to $1=e + e'$; the resulting
splitting is unique up to unique isomorphism (``unique isomorphism'' understood to be in the derived category).
Thus we put
$$C_*^{\Delta}(Y(K), E )_{\mathfrak{m}} := C_*^{\Delta}(Y(K), E) \cdot e.$$

 Its homology is finitely generated as $E\Delta$-module and we obtain (a) by passing to a minimal resolution.  
 Part (b) follows at once from the Conjecture enunciated above.

For descent, part (c), we need to be careful since the Hecke algebras acting at different levels are not quite the same;
for discussion of that issue in the case needed we refer to \cite[Lemma 6.6]{KT}.  \qed 

\subsection{Taylor-Wiles sets}  \label{moredetails}  
  
 Fix an allowable Taylor--Wiles datum $Q$ of level $n$ (see \S \ref{TWprimesdef});
 recall this comes equipped with a choice of element of $T(k)$ conjugate to Frobenius, for each prime in $Q$. 

 If $Q$ is an auxiliary set of primes, disjoint from $S$, we let $Y_0(Q)$
 be the manifold obtained by adding Iwahori level at $Q$, i.e.
 replacing $K_0 = \prod K_{0,v}$ by $$\prod_{v \notin Q} K_{0,v} \times \prod_{v \in Q} \mathrm{Iwahori}_v.$$

 As usual (generalizing the case of the covering of modular curves $X_1(q) \rightarrow X_0(q)$)
 we can produce a manifold $Y_1(Q)$ which covers $Y_0(Q)$ with Galois group $$\prod_{q \in Q}  \mathbf{T}(\F_q)   $$
 where $\mathbf{T}$ was, we recall, a maximal split torus in $\mathbf{G}$.

From the surjection $\mathbf{T}(\F_q) \twoheadrightarrow \mathbf{T}(\F_q)/p^n  \cong (\Z/p^n)^r$ (where $r$ is the rank of $G$)   
we get a unique subcovering $Y(Q, n)$ with Galois group \begin{equation} \label{DeltaQdef} \Delta_Q :=  \prod_{q \in Q}  \mathbf{T}(\F_q)/p^n  \cong (\Z/p^n)^{r \cdot \# Q} \end{equation}
over $Y_0(Q)$.  
In other words, this fits in the following diagram 
$$\overbrace{Y_1(Q) \rightarrow \underbrace{ Y(Q, n) \rightarrow Y_0(Q)}_{\Delta_Q \cong (\Z/p^n)^{r \cdot \# Q}}}^{\prod_{q \in Q}  \mathbf{T}(\F_q)}$$   
$\Delta_Q$ also enters into the deformation theory of Galois representations at level $Q$, as we now recall:

Let $\usR_{S \coprod Q}$ be the usual deformation ring for Galois representations   at level $S \coprod Q$,
which are, as always, crystalline at $p$.  Thus $\usR_{S \coprod Q} = \pi_0 \mathcal{R}_{S\coprod Q}$
in the notation of  \S \ref{TWintro}. 
Form  a universal deformation
$$ \widetilde{\rho}^{\univ}: \Gal(\overline{\Q}/\Q) \longrightarrow G(\usR_{S \coprod Q})$$
and restrict this to the Galois group  of a place $q \in Q$. 
 
This restricted representation  can be conjugated, by Remark \ref{factoring remark},  to take values in $T$.
Note that the resulting conjugated homomorphism  $\pi_1(\Q_v) \to T(R_{S \coprod Q})$ is uniquely determined by requiring
that the image of Frobenius reduce to the element of $T(k)$ prescribed as part of the Taylor--Wiles data    (see first paragraph of this subsection). 
We shall suppose that this has been done.

We therefore get a map of deformation rings
$$ \mbox{framed deformation ring of $ \rho_{\Q_q}^T$, with target group $T$} \longrightarrow \usR_{S \coprod Q}.$$  
 Now the left-hand side is identified with the group algebra of a certain quotient of $\mathbf{T}(\Q_q)$ (Remark \ref{canonicalS})
and in particular we obtain a map 
\begin{equation} \label{oldTmap} \mathbf{T}(\mathbb{F}_q)_p \rightarrow \usR_{S \coprod Q}^{*}.\end{equation}  
 where the subscript $p$ means ``$p$-part'', i.e.\ $A_p = A \otimes \Z_p$.

(More explicitly,  we can factorize
$ \widetilde{\rho}^{\univ}|_{\pi_1 \Q_q}:   (\pi_1 \Q_q)^{\ab} \rightarrow  T(\usR_{S \coprod Q})$, and if we compose this 
  with any character $\chi:  T \rightarrow \mathbb{G}_m$, 
we get a homomorphism $(\pi_1 \Q_q)^{\ab} \rightarrow \usR_{S \coprod Q}^*$; in particular, restricting to $\mathbb{F}_q^* \subset (\pi_1 \Q_q)^{\ab}$ and noting that the image is pro-$p$, we get a map
$\mathbb{F}_q^* \otimes \Z_p \rightarrow \usR_{S \coprod Q}^*$.  Thus, we have produced a map $\mathbb{F}_q^* \otimes \Z_p \rightarrow \usR_{S \coprod Q}^*$ for each character  $T \rightarrow \mathbb{G}_m$ --
which, said intrinsically, amounts to \eqref{oldTmap}). 

Taking a product over $q \in Q$, we get at last a morphism
\begin{equation} \label{DeltaR} \widetilde{ \Delta_Q } := \prod_{q \in Q} \mathbf{T}(\mathbb{F}_q)_p \rightarrow \usR_{S \coprod Q}^*\end{equation} 
Note that $\widetilde{\Delta}_Q /p^n = \Delta_Q$. Observe that  $\widetilde{\Delta}_Q$ is isomorphic to $\prod_{q \in Q} (\Z/p^{n'_q})^{r}$, where $n'_q$ is the highest power of  $p$ dividing $q-1$. 

\subsection{Enlarging $\mathfrak{m}$}  \label{mprimeconstruc}

 In the classical theory of modular forms,
 the space of modular oldforms at level $q$ is a sum of two copies of the space of modular forms at level $1$;
one can distinguish these two copies as the two eigenspaces for the $U_q$-operator. We carry out the analogue here, 
using the ``$U_q$-operator'' to slightly enlarge the maximal ideal $\mathfrak{m}$.

Let $\mathrm{I}_q$ be an Iwahori subgroup of $\G(\Q_q)$
and $\mathrm{I}_q' \leqslant \mathrm{I}_q$ the subgroup corresponding to $Y(Q, n)$,
so that $\mathrm{I}_q/\mathrm{I}_q' \simeq (\Z/p^n)$. 
Writing $X_*$ for the co-character group of $\mathbf{T}$ (more canonically
one replaces $\mathbf{T}$ by the torus quotient of a Borel subgroup)
and let $X_*^+ \subset X_*$ be the positive submonoid defined by the Borel subgroup $\mathbf{B}$ (the dual to the cone spanned by roots). 
 In the usual way each $\chi \in X_*^+$ gives rise to a Hecke operator at level $Y(Q, n)$,
 namely the operator defined by the double coset $\mathrm{I}_q'   \ \chi(q)  \ \mathrm{I}_q'$;
 for $\mathbf{G}=\PGL_2$ and $\chi$ a generator for $X_*^+$
 this gives the $U_q$-operator.

Let the extended Hecke algebra at level $Y(Q, n)$
be the subalgebra generated\footnote{As with the Hecke algebra, we regard these as acting
on the chain complex of $Y(Q, n)$, considered as an object of a suitable derived category, and ``generated''
is taken inside the endomorphism ring of that object.}  by prime-to-level Hecke operators
together with all the $\mathrm{I}_q' \chi(q) \mathrm{I}_q'$, as above.
We make an ideal $\mathfrak{m}'$
of this extended Hecke algebra thus:
$$ \mathfrak{m}' = \left( \mathfrak{m},   \mathrm{I}_q' \  \chi(q)  \  \mathrm{I}_q' -   \langle \chi, \Frob^T_q \rangle\right). $$ 
Here, we have usen the choice of $\Frob^T_q \in T(k)$
that comes with a Taylor--Wiles prime, 
and the identification $\chi \in X_*(\mathbf{T}) \simeq X^*(T)$ 
to form  $\langle \chi, \Frob^T_q \rangle \in k.$

We can now formulate: 
\begin{Assumption} \label{LGasssumption} {\em   (Local-global compatibility): } 
Let $Q$ be a Taylor-Wiles datum of  level $n$, and let $\mathfrak{m}'$ be as defined in the preceding two paragraphs. The action of $\widetilde{\Delta}_Q$ on 
$C_*^{\Delta_Q}(Y(Q,n); E)_{\mathfrak{m}'}$ 
via \eqref{DeltaR}  and Conjecture  \ref{GaloisRepConjecture}  
coincides with the natural action, that is to say, the action via $\widetilde{\Delta}_Q \rightarrow \Delta_Q$
and deck transformations.
\end{Assumption}

\subsection{Sequences of Taylor--Wiles sets }  \label{seq TW sec}
Now choose arbitrarily allowable Taylor--Wiles data $(Q_n, n)$, indexed by an increasing sequence of integers $n$,
and write for short $\Delta_n = \Delta_{Q_n}$ and $\widetilde{\Delta}_n = \widetilde{\Delta}_{Q_n}$ (see \eqref{DeltaR} for definition).   We choose all the $Q_n$   to be of the same size $\# Q$; recall that 
\begin{equation} \label{Deltaiso} \widetilde{\Delta}_n/p^n = \Delta_n \cong (\Z/p^n)^{r \# Q}\end{equation}
where $r$ is the rank of $G$.

Let $C_n$ be the reduced chain complex constructed in Assumption \ref{LGasssumption} for $Y_n = Y(Q_n,n)$ with $E=W_n$
and with maximal ideal $\mathfrak{m}'$ as in \S \ref{mprimeconstruc}, i.e.
\begin{equation} \label{Cndef} C_n = C_*^{\Delta_n}(Y(Q_n, n), W_n)_{\mathfrak{m}'},\end{equation}
so it computes the homology of $Y(Q_n, n)$ with $W_n$ coefficients, localized at $\mathfrak{m}'$,
and equivariantly for the action of the Galois group $\Delta_n$ of $Y(Q_n, n) \rightarrow Y_0(Q)$.

The complex $C_n$ is by definition a complex of   $\overline{\usrS_n}$-modules,  where 
   $$\usrS_{n} :=  W[\Delta_n], \qquad \qquad  \overline{\usrS_n} = \usrS_n/p^n.$$ \index{$\usrS_n$}     \index{$\overline{\usrS_n}$}

   We observe for later use that we have an identification  \begin{equation} \label{Fred0} C_n \otimes_{  \overline{\usrS_n}} W_n \simeq C(Y_0, W_n)_{\mathfrak{m}}, \end{equation}
inside the derived category of $W_n$-modules.  In fact, part (c) of Lemma
\ref{ConsequenceOfConjectureLemma} says that the left-hand side is identified with $C(Y_0(Q_n), W_n)_{\mathfrak{m}'}$,
where $\mathfrak{m}'$ is the analogue to $\mathfrak{m}'$ above but at level $Y_0(Q_n)$; and then 
  the natural projection $Y_0(Q_n) \rightarrow Y_0$ induces an isomorphism
$C(Y_0(Q_n), W_n)_{\mathfrak{m}'} \stackrel{\sim}{\rightarrow}  C(Y_0, W_n)_{\mathfrak{m}}$ (this corresponds to \cite[Lemma 6.25 (4)]{KT}). 
Note that \eqref{Fred0} implies that \begin{equation}
\label{vanishing0} H_j(C_n) = 0,  \ \ j \notin  [q,q+\delta].\end{equation}
because, by  5 and 7  from \S \ref{TWnotationsetup}, the homology $H_*(Y_0, k)_{\mathfrak{m}}$
is vanishing for $j \notin [q, q+\delta]$.  In particular, we can and do suppose
that $C_n$ itself is supported in degrees $[q, q+\delta]$.

We let $$\mathrm{R}_n = \mbox{usual deformation ring at level }S \coprod Q_n.$$
By virtue of our assumptions (Lemma  \ref{ConsequenceOfConjectureLemma})
 there is a natural action of $\mathrm{R}_n$  
on $C_n$, by endomorphisms in the derived category of $\overline{\usrS_n}$-modules. 
The assumption 
of local-global compatibility (Assumption
\ref{LGasssumption}) means that this action factors through the
quotient
\begin{equation} \label{quotient}
  \begin{aligned}
    \theta: \usR_n \twoheadrightarrow \underbrace{ \usR_n
      \otimes_{W[\widetilde{\Delta}_n]} W_n[\Delta_n]}_{\mbox{inertial
        level $\leq n$}}
  \end{aligned}
\end{equation}
We can think of the right hand side as classifying deformations of inertial level $\leq n$  \eqref{ILL def} at primes in $Q$ -- that is to say, deformations 
such that, at each prime $q\in Q$, the deformation factors not merely through the $(\Z/q)^*$ quotient of inertia,
 but actually through the maximal exponent $p^n$ quotient of that group.  
 In particular, we have a natural map
\begin{equation} \label{quotient2}  \Delta_n \rightarrow  \left(  \usR_n \otimes_{W[\widetilde{\Delta}_n]} W_n[\Delta_n] \right)^*\end{equation}
 
 To do the limit process we  replace $\mathrm{R}_n$ (which might, {\em a priori},  be a bit too large -- e.g., infinite) by an Artinian quotient: 
 
  \begin{Lemma} \label{nilpotence}  Suppose that $D_n$ is any complex of finite free $\overline{\usrS_n}$-modules.   Let $\mathbf{e} $ be an endomorphism of $D_n$, 
 considered as an object in the derived category, which acts nilpotently on homology. Then $\mathbf{e}^{K(n)} = 0$, where we can take
 $K(n) = \ell \cdot Q \cdot o$, with
 \begin{itemize}
 \item $\ell$ the length of $\overline{\usrS_n}$ over itself;
 \item $o$ the range of degrees in which $D_n$ is supported (i.e., if supported in $[a,b]$, we have $o=b-a+1$);
 \item $Q$ is the total rank of $D_n$ over $\overline{\usrS_n}$. 
 \end{itemize}
  Without loss of generality we can and do always suppose $K(n) \geq 2n$. 
 \end{Lemma}
 \proof  The total length of homology, as a $\overline{\usrS_n}$-module, is clearly at most
 $ Q  \cdot  \ell$. This shows that  $\mathbf{e}^{Q \ell} $ acts trivially on homology.
 Using \cite[Lemma 2.5]{KT} the claimed result follows. 
 \qed   
 
 With $K(n)$ as in the above lemma, applied to $C_n$ in place of $D_n$, 
  put
$$\mathfrak{d}_n = \left( \mathfrak{m}_{\usR_n}^{K(n)},  \ker (\theta ) \right)\subset R_n$$
where $\theta$ is as in \eqref{quotient}. 
Then $$\usbR_n =  \usR_n/\mathfrak{d}_n$$
is Artinian, and the action of $\usR_n$ on $C_n$ factors through $\usbR_n$. 
The map from $\widetilde{\Delta}_n$ to 
$\mathrm{R}_n$ yields a map  from the group algebra of $\Delta_n$ into $\usbR_n$, i.e.
\begin{equation} \label{martian} \overline{ \usrS_n} \rightarrow \usbR_n \end{equation}
and we also have descent
\begin{equation} \label{DescentEquation} \usbR_n \otimes_{\overline{\usrS_n}} W_n\cong \pi_0 \mathcal{R}_{S}/(p^n, \mathfrak{m}^{K(n)})\end{equation}
which expresses the fact that ``a Galois representation at level $Q_n$, which is trivial on $\Delta_n$,  actually descends to base level,''
i.e.\ the underived version of the discussion of \S \ref{sec:TWprimes}.   \footnote{More precisely, $ \usbR_n \otimes_{\overline{\usrS_n}} W_n$
computes the quotient of $\mathrm{R}_n$ by the ideal generated by elements $\delta-1$ with $\delta \in \widetilde{\Delta}_n$, 
by $p^n$, and by $\mathfrak{m}^{K(n)}$. The quotient of $\mathrm{R}_n$ by the ideal generated by $\delta-1$
recovers $\pi_0 \mathcal{R}_S$, whence \eqref{DescentEquation}. }

Finally, the assumption of local-global compatibility from the previous sections  still says that
 the action of $\usrS_n$ on $C_n$ via  $\overline{\usrS_n} \rightarrow \usbR_n$ (see \eqref{martian}) 
coincides with the natural action of $\overline{\usrS_n}$ on $C_n$ (i.e.\ the natural action by deck transformations.)  

\subsection{Numerology} \label{choice}  

Recall \eqref{vanishing_sequence} that, the $Q_n$ being an allowable Taylor--Wiles datum of level $n$,  we have with  $S' = S \coprod Q_n$, 
a surjection
$$A: H^1(\Z[\frac{1}{S'}], \Ad \rho) \twoheadrightarrow H^1(\Q_p, \Ad \rho)/H^1_f(\Q_p, \Ad \rho).$$  
 The kernel of $A$ describes the minimum number of generators for $\usR_n$ as a $W(k)$-module. 
An Euler characteristic computation shows that $\ker(A)$ has dimension  
$$  \underbrace{ \dim(B) - \delta }_{  \dim H^1(\Z[\frac{1}{S}], \Ad \rho)   - \dim H^2(\Z[\frac{1}{S}], \Ad \rho) }   - \underbrace{ \dim G}_{\dim H^1(\Q_p, \Ad \rho) } +  \underbrace{ \dim U}_{\dim H^1_f} +  \underbrace{  (\mathrm{rank} G) \cdot \# Q }_{\bigoplus_{Q} H^2}  $$
which equals $r \cdot \# Q - \delta$; recall that $r$ is the rank of $G$.  In other words, $\usR_n$ is a quotient of a power series ring over $W(k)$ in $r \# Q -\delta$ variables. 
We set $$s=r \cdot \# Q.$$ \index{$s$} 
and will choose $\# Q$ so large that $\pi_0 \mathcal{R}_S$ can be generated over $W(k)$ by $s-\delta$ elements.  
\subsection{Limit rings} 
   
 Now we pass to the limit in the following way.  
As we just proved in \S \ref{choice}, each $\usR_n$ and so also $\usbR_n$ is a quotient of $\usR_{\infty} :=W(k)[[x_1, \dots, x_{s-\delta}]]$. 
 Choose  once and for all a surjection
$\usR_{\infty} \rightarrow \pi_0 \mathcal{R}_{S}$ and choose  $\usR_{\infty} \rightarrow  
\usbR_n$ to lift
the resulting map $\usR_{\infty} \rightarrow \pi_0 \mathcal{R}_{S}/ (p^n,\mathfrak{m}^{K(n)})$, i.e.\ the  following diagram commutes 
   \begin{equation}  \label{factor_through0}
     \begin{aligned}
       \xymatrix{
         \usR_{\infty} \ar[r] \ar[d]^{=} &  \usbR_n   \ar[d]^{\pi_n} \\
         \usR_{\infty} \ar[r] & \pi_0 \mathcal{R}_{S}/(p^n,
         \mathfrak{m}^{K(n)}) .  }
     \end{aligned}
\end{equation}

-- this can be  done since the right vertical map $\pi_n$ from \eqref{DescentEquation} is surjective. In fact, $\usR_{\infty} \rightarrow  \usbR_n$ can and will be chosen to be {\em surjective}:\footnote{We can also avoid this point,  by not choosing $\usR_{\infty} \rightarrow \pi_0 \mathcal{R}_S$ 
at the start,  instead just choosing the map $\usR_{\infty} \rightarrow \usbR_n$ and then using a compactness argument
to pass to a subsequence where the resulting maps $\usR_{\infty} \rightarrow \pi_0 \mathcal{R}_{S}/ (p^n,\mathfrak{m}^{K(n)})$ converge. }
\begin{Lemma}Suppose $A \stackrel{e}{\twoheadrightarrow} B \rightarrow k$ are two Artin local rings over $k$, both abstractly
 quotients of $W[[x_1, \dots, x_k]]$. Then any surjection $\varphi: W[[x_1, \dots, x_k]] \twoheadrightarrow B$
can be lifted to a surjection $W[[x_1, \dots, x_k]] \twoheadrightarrow A$.  (All the maps here are assumed to commute with the augmentations to $k$). 
\end{Lemma}
\proof  
First note that a corresponding assertion for vector spaces is easy: if $V \twoheadrightarrow W$
is a surjection of vector spaces, both of dimension $\leq k$, then any spanning set of $W$ of size $k$
can be lifted to a spanning set for $V$. In terms of matrices, this says that a $(\dim W) \times k$
matrix of  maximal rank  can be extended to a $(\dim V) \times k$ matrix of maximal rank, which is clear. 

 We can replace $A, B$ by $A/p, B/p$: if we can lift from $B/p$ to $A/p$, we get a homomorphism 
  $W[[x_1, \dots, x_k]] \rightarrow  A' = A/p \times_{B/p} B$.     The map $A \rightarrow A'$ induces an isomorphism $A/p \simeq A'/p$,
  and is therefore surjective;
  so any  surjection $W[[x_1, \dots, x_k]] \rightarrow A'$ can be lifted to $A$, and the resulting homomorphism is still surjective.

Let $\mathfrak{m}_A, \mathfrak{m}_B$ be the respective maximal ideals and let $I$ be the kernel of $A \rightarrow B$. A homomorphism $W[[x_1, \dots, x_k]] \rightarrow A$
is now surjective if and only if the images of $x_i$ generate $\mathfrak{m}_A/\mathfrak{m}_A^2$ as a $k$-vector space, and the same is true for $B$. 
  We have a short exact sequence
$I \hookrightarrow \mathfrak{m}_A \twoheadrightarrow \mathfrak{m}_B$ and so, applying $\otimes_{A} k$, 
we get $$I \otimes_A k \rightarrow \mathfrak{m}_A/\mathfrak{m}_A^2 \twoheadrightarrow \mathfrak{m}_B/\mathfrak{m}_B^2,$$
exact at right and middle; thus
any element in the kernel of $\mathfrak{m}_A/\mathfrak{m}_A^2 \twoheadrightarrow \mathfrak{m}_B/\mathfrak{m}_B^2$
can be lifted to an element of $I$. 

  Now let $y_1, \dots, y_k \in \mathfrak{m}_B$ be the images of $x_i$ in $B$.  
Let $s_1, \dots, s_k$ be a generating set for $\mathfrak{m}_A/\mathfrak{m}_A^2$
so that each $s_i$ maps to the class $y_i + \mathfrak{m}_B^2$. 
Lift $y_i$ arbitrarily to some $y_i^* \in \mathfrak{m}_A$; since the class of $y_i^* - s_i$ dies in $\mathfrak{m}_B/\mathfrak{m}_B^2$,
we can, by the above, modify $y_i^*$ to  another lift $y_i' \in \mathfrak{m}_A$ of $y_i$, 
but such that $y_i' + \mathfrak{m}_A^2 = s_i$. Sending $x_i$ to $y_i'$
gives the desired lift $W[[x_1, \dots, x_k]] \twoheadrightarrow A$. 
  \qed 

Recall that  $\usrS_{\infty} = W(k)[[x_1, \dots, x_s]]$.
Also fix isomorphisms as in \eqref{Deltaiso} for each $n$, i.e.\ we fix a  $\Z/p^n$ basis for $\Delta_n$. 
Then there is a map
$$\usrS_{\infty} \twoheadrightarrow \usrS_n,$$ 
sending $x_i$ to $[\sigma_i]-[e]$, where  $\sigma_i$ is a generator for the $i$th factor of $(\Z/p^n)$ and $e$ is the identity element. 
This induces an isomorphism
$$\usrS_{\infty}/\mathfrak{a}_n \stackrel{\sim}{\longrightarrow} \overline{\usrS_n},$$
with $\mathfrak{a}_n$ as in \eqref{ardef}. 
Having fixed this, there are also unique maps $\usrS_n \rightarrow \usrS_m$ and $\overline{\usrS_n} \rightarrow \overline{\usrS_m}$ for $n > m$, compatible with the maps  from $\usrS_{\infty}$, and 
   the natural map 
   $$\usrS_{\infty} \stackrel{\sim}{\longrightarrow}  \varprojlim \overline{ \usrS_n}$$
is an isomorphism. 
   Notice that the maps of this paragraph are ``meaningless,'' i.e.\ do not correspond to any ``map of natural origin,'' as they depended on fixing a choice of isomorphisms in \eqref{Deltaiso}.

Next, for each $n$ we can choose a lift $\iota_n$ in the following diagram:
    \begin{equation} \label{SinftySn}
      \begin{aligned}
        \xymatrix{
          \usrS_{\infty}  \ar[r]^{\iota_n  \qquad } \ar[d] &  \usR_{\infty}/(p^n,\mathfrak{m}^{K(n)})  \ar[d] \\
          \overline{\usrS_n} \ar[r] & \usbR_n }
      \end{aligned}
\end{equation}
again, this is possible because $\usR_{\infty}$ surjects to $\usbR_n$. 
(Here, $\mathfrak{m}$ is the maximal ideal of $\usR_{\infty}$; we will allow ourselves to use $\mathfrak{m}$ for the maximal
ideals of different local rings when the correct ring is clear by context.) 
 
 \subsection{Passage to the limit}  \label{passage to limit}
 We now can ``pass to the limit'' by a compactness argument. 
Indeed, our discussion to date
 shows that we get, for each $n$, the following type of ``level $n$ datum'': 
 
 \begin{quote}
\begin{itemize}
\item[i.] A  complex $C_n$ of free $\usrS_{\infty}/\mathfrak{a}_n$-modules of   bounded rank  and bounded support (both bounds being independent of $n$).

\item[ii.]  An action of $\usR_{\infty}$ by endomorphisms of $C_n$, considered as an object in the derived category of $\usrS_{\infty}/\mathfrak{a}_n$-modules. 
By Lemma   \ref{nilpotence}, this action is automatically trivial on $(p^n, \mathfrak{m}^{K(n)})$, where $K(n)$ is as in the quoted Lemma.

 \item[iii.] A  homomorphism $\iota_n: \usrS_{\infty} \rightarrow R_{\infty}/(p^n,\mathfrak{m}^{K(n)})$  with the property that the two actions (``actions'' in the same sense as (ii))
 of  $\usrS_{\infty}$ on $C_n$ coincide:
 \begin{itemize} \item[-] via  $\usrS_{\infty} \rightarrow \usR_{\infty}/(p^n,\mathfrak{m}^{K(n)} )$ 
and the action specified in (ii), which automatically factors through $\usR_{\infty}/(p^n,\mathfrak{m}^{K(n)})$; 
\item[-] via  $\usrS_{\infty} \rightarrow  \usrS_{\infty}/\mathfrak{a}_n$ and the fact that $C_n$ is a complex of $ \usrS_{\infty}/\mathfrak{a}_n$-modules
\end{itemize}

 \item[iv.] 
   (Recovering the homology of the base manifold): An identification  
of  
\begin{equation} \label{Fred} C_n \otimes_{ \usrS_{\infty}/\mathfrak{a}_n} W_n \simeq C(Y_0, W_n)_{\mathfrak{m}}, \end{equation}
in the derived category of  $W_n$-modules, 
 compatibly for the actions of $\usR_{\infty}$. Here $C(Y_0, W_n)_{\mathfrak{m}}$
 is the complex constructed in
 \S  \ref{KTreduced}
 which computes the homology $H_*(Y_0, W_n)_{\mathfrak{m}}$, and 
 the action on the right being via   $\usR_{\infty} \twoheadrightarrow \pi_0 \mathcal{R}_{S}$, 
 the surjection chosen at the very beginning of the argument).

 \end{itemize}
\end{quote}

 Each such datum of level $n$ induces a datum of level $n'$ for each $n' \leq n$: 
  replace $C_n$ by $E_{n'} := C_n \otimes_{ \usrS_{\infty}/\mathfrak{a}_n}   \usrS_{\infty}/\mathfrak{a}_{n'}$. 
 The $\usR_{\infty}$-module structure is induced by functoriality, regarding
 $\otimes_{ \usrS_{\infty}/\mathfrak{a}_n}  \usrS_{\infty}/\mathfrak{a}_{n'}$ as a functor between the derived categories, i.e.
 $r \in \usR_{\infty}$ acts on $E_{n'}$ as ``$r \otimes 1$.'' 
   $\iota_{n'}$ is obtained by projecting $\iota_n$.  

 It is easy to see that there are only finitely many possible data of level $n$ up to isomorphism. 
  Therefore, by a compactness argument, there  is a subsequence $j_n$ of integers so that
 the data associated to $(Q_{j_n}, n)$ are all compatible -- i.e.:
 \begin{itemize}
 \item[-]
 $(Q_{j_n}, n)$ is the datum of level $n$ obtained from the level $j_n$ datum  $Q_{j_n}$, in the fashion just described;
 \item[-] 
The datum of level $n-1$ induced
 by $(Q_{j_n}, n)$ is actually isomorphic to the datum associated to $(Q_{j_{n-1}},  n-1)$.
 \end{itemize}
  
 \subsection{Reindexing} \label{reindexing}
  As we just saw, by passing to the subsequence $j_n$ we can achieve a compatible system of data, as above. 
 We will now change notation somewhat so we don't have to keep   writing $j_n$s. 
 
  Let us talk this through carefully now
 and then we will simplify the notation; when, later on, we need to consider more carefully the internals of the limit process,
 we will refer back to this section.  (The reader might want to skip this somewhat tedious discussion, and just pretend that $j_n=n$).  
 
For any $n \leq m$, we can regard $\mathfrak{a}_n$ as an ideal of    $\overline{\usrS_{m}}$ by 
means of the map $\usrS_{\infty} \rightarrow \overline{\usrS_{m}}$.  The resulting ideal doesn't depend on the choice of this map:
it is the ideal of $\overline{\usrS_m} = W_m[\Delta_m]$ 
given by the kernel of the map $W_m[\Delta_m] \rightarrow W_n[\Delta_m/p^n]$.

Set $$   D_n:= C_{j_n} \otimes_{  \usrS_{\infty}/\mathfrak{a}_{j_n}} \usrS_{\infty}/\mathfrak{a}_n,$$
a   complex of free $ \usrS_{\infty}/\mathfrak{a}_n$ modules
-- this is part of the level $n$ datum obtained from $Q_{j_n}$.			

It comes with compatible actions of 
$\usR_{\infty}$ on $D_n$, as an object in the derived category of $ \usrS_{\infty}/\mathfrak{a}_n$; also the   homomorphisms  obtained by composing
  $$\usrS_{\infty} \rightarrow \usR_{\infty}/(p^{j_n},\mathfrak{m}^{K(j_n)}) \rightarrow \usR_{\infty}/(p^n,\mathfrak{m}^{K(n)})$$
  are compatible with one another as $n$ increases. (The first map was the one chosen at step $j_n$ of the limit process, the second map is the obvious one).     
Therefore, passing to the $n \rightarrow \infty$ limit, we get  
\begin{equation} \label{limit iota} \iota: \usrS_{\infty} \rightarrow \usR_{\infty}\end{equation}
  whose reduction modulo $(p^n, \mathfrak{m}^{K(n)})$ recovers the maps above.

Next, set  
\begin{equation} \Delta_n^{\mathrm{reindexed}} = \Delta_{j_n}/p^n \end{equation}
\begin{equation} \overline{\usrS_n}^{\mathrm{reindexed}} =  \underbrace{\usrS_{j_n}/\mathfrak{a}_n }_{ W_{j_n}[\Delta_{j_n}]/\mathfrak{a}_n} =  W_n[\Delta_{n}^{\mathrm{reindexed}}]  \end{equation}
\begin{equation} \label{reindex3}  \usbR_n^{\mathrm{reindexed}} = \left( \usbR_{j_n} \otimes_{\overline{\usrS_{j_n}}}   \overline{\usrS_{j_n}}/\mathfrak{a}_n \right) /\mathfrak{m}^{K(n)}   \end{equation} Here, we regard $\mathfrak{a}_n$ as an ideal of $\overline{\usrS_{j_n}}$  as noted above. 
 We introduce the superscript ``reindexed'' to recall that this has been reindexed after the limit process. 
  Just as before, the choice of isomorphisms   in \eqref{Deltaiso}
 give an identification of $\overline{\usrS_n}^{\mathrm{reindexed}}$ with $\usrS_{\infty}/\mathfrak{a}_n$, 
 and thus make $\overline{\usrS_n}^{\mathrm{reindexed}}$ into an inverse system with limit $\usrS_{\infty}$.

Note that deformations classified by $\usbR_n^{\mathrm{reindexed}}$ 
are deformations of level $S \coprod Q_{j_n}$, but forcing the inertial level (\eqref{ILL def})  to be $\leq n$ at primes in $Q_{j_n}$. 
The morphism $\usR_{\infty} \rightarrow \usbR_{j_n}$ chosen in the limit process induces $\usR_{\infty} \twoheadrightarrow \usbR_n^{\mathrm{reindexed}}$
and the action of  $\usR_{\infty}$ on $D_n$ factors through $\usbR_n^{\mathrm{reindexed}}$ by the definitions.

{\bf Remark.} Although we don't use it in this paper,  we observe that $D_n$ has a ``physical'' interpretation: it is naturally identified with a complex that computes
the homology of $Y(Q_{j_n}, n)$. More precisely,  $D_n$ is naturally identified with $C_{j_n} \otimes_{ W_{j_n}[\Delta_{j_n}]} W_n [\Delta_n^{\mathrm{reindexed}}]$. 
The natural projection $Y(Q_{j_n}, j_n) \rightarrow Y(Q_{j_n}, n)$ induces a quasi-isomorphism of the latter with 
 the complex  $C_*^{\Delta_n^{\mathrm{reindexed}}} (Y(Q_{j_n}, n), W_n)_{\mathfrak{m}'}$,
where $\mathfrak{m}'$ is defined similarly to \eqref{Cndef}.     
Similarly, the identification $C_{j_n} \otimes_{  \overline{\usrS_{j_n}}} W_{j_n} \simeq C(Y_0, W_{j_n})_{\mathfrak{m}}$
from \eqref{Fred0} induces a similar identification
$D_n \otimes_{ \overline{\usrS_n}^{\mathrm{reindexed}}} W_n \simeq C(Y_0, W_{n})_{\mathfrak{m}}$.

We finally observe that the map $\usR_{\infty} \twoheadrightarrow \usbR_n^{\mathrm{reindexed}}$  can be extended to a map
\begin{equation} \label{goofball} \usR_{\infty} \otimes_{\usrS_{\infty}} \overline{\usrS_n}^{\mathrm{reindexed}} \twoheadrightarrow \usbR_n^{\mathrm{reindexed}}.\end{equation}
(Later we will see this is actually an isomorphism, under our assumptions.)
In fact, we have a commutative diagram 
    \begin{equation} \label{factor_through}
     \xymatrix{
\usrS_{\infty} \ar[r]^{\overline{\iota} \qquad } \ar[d] &    \usR_{\infty}/(p^n,\mathfrak{m}^{K(n)})  \ar[d] \ar[rd]    \\
\overline{\usrS_n}^{\mathrm{reindexed}} \ar[r]   &   \usbR_n^{\mathrm{reindexed}}   \ar[r] &  \pi_0 \mathcal{R}_S/  (p^n, \mathfrak{m}^{K(n)})  }
\end{equation}  
where the composite $\overline{\usrS_n}^{\mathrm{reindexed}} \rightarrow \pi_0 \mathcal{R}_S/  (p^n, \mathfrak{m}^{K(n)}) $ at the bottom factors
through the natural augmentation  $\overline{\usrS_n}^{\mathrm{reindexed}}  \rightarrow W_n$.
The commutativity of the left-hand square follows from 
\eqref{SinftySn} and the commutativity of the right hand triangle follows similarly from \eqref{factor_through0}.

{\em Henceforth we drop the superscript ``reindexed'' from the notation; when we need to be explicit, we refer back to this section.}

\subsection{Analysis of the limit} 
  After reindexing,  as just described, we have a sequence of   complexes $D_n$  of free $\overline{\usrS_n}$-modules, together with isomorphisms
\begin{equation} \label{coo}  D_n \otimes_{\overline{\usrS_n}} \overline{\usrS_m}  \stackrel{\sim}{\longrightarrow} D_m \end{equation}
    
       Set $D_{\infty} = \varprojlim D_n$ (where the transition maps are those from \eqref{coo}).   It  is now a complex of free        modules under  $\varprojlim \overline{ \usrS_n} = \usrS_{\infty}$.
    Because all the groups involved are finite, we have   $$H_* D_{\infty} \cong \varprojlim H_* {D}_n.$$
   
   Now $\usR_{\infty}$  acts on each $H_*(D_n)$ and 
these actions are compatible  with the maps $H_*(D_{n}) \rightarrow H_*(D_m)$ for $n  > m$;
therefore, we obtain a $\usR_{\infty}$ action on $H_* D_{\infty}$ too.  In fact, we can even get a  $\usR_{\infty}$-action on $D_{\infty}$ in the derived category of $\usrS_{\infty}$-modules using
\cite[Lemma 2.13]{KT}: The generators $x_1, \dots, x_{s-\delta}$ of $\usR_{\infty}$ are only homotopy compatible at each level, but this is nonetheless just enough
to lift them to $D_{\infty}$ in a fashion that is unique up to homotopy;  the uniqueness means that the resulting lifts also commute up to homotopy, thus giving an action of $\usR_{\infty}$. 
Thus we have a quasi-isomorphism  $D_{\infty} \otimes_{\usrS_{\infty}} \overline{\usrS_n} \stackrel{\sim}{\longrightarrow} D_n$,
compatible with $\usR_{\infty}$-actions.

Finally,  we have as in \eqref{limit iota} a limit map $\iota = \varprojlim \iota_n: \usrS_{\infty} \rightarrow \usR_{\infty}$. 
By assumption (iii) of \S 
\ref{passage to limit} the
 two actions of $\usrS_{\infty}$ on $H_* D_n$ -- one via $\iota_n$ and one via $\usrS_{\infty} \rightarrow \overline{ \usrS_n}$ -- coincide,
and by passage to the limit, the two actions of $\usrS_{\infty}$ on $H_* D_{\infty}$ -- one via $\iota$
and the other the natural one -- coincide;
in fact the two actions on $D_{\infty}$ coincide in the derived category, by the uniqueness in \cite[Lemma 2.13]{KT}.

The  key lemma in commutative algebra that follows shows that $D_{\infty}$ is a resolution of the edge homology, i.e.\ $H_i D_{\infty} = 0$
for $i > q$, and so $D_{\infty}$ 
is quasi-isomorphic to $M := H_q D_{\infty}$, shifted in degree $q$:

\begin{Lemma} (Calegari--Geraghty, Hansen)  
Suppose given a complex $D_{\infty}$ of finite free $\usrS_{\infty}$ modules, supported in degrees $[q,q+\delta]$ over $\usrS_{\infty} \cong W[[x_1, \dots, x_s]]$, with degree-decreasing differential.  
Suppose  also we are given a homomorphism $\usrS_{\infty} \rightarrow \usR_{\infty} \cong W(k)[[x_1, \dots, x_{s-\delta}]]$
and   the action of $\usrS_{\infty}$ on $H_* D_{\infty}$ factors through  this homomorphism. 
Then $H_i D_{\infty}$ is nonzero only for $i=q$, and $H_q D_{\infty}$ is free over $\usR_{\infty}$.
\end{Lemma}
\proof (\cite[Lemma 3.2]{CG}, \cite[Lemma 2.9]{KT}): 
 Let $j $ be the largest integer for which  $H_{q+j}  D_{\infty} \neq 0$. Say $j =\delta-\delta'$, so that  
$$D_{\infty,  q+\delta-\delta'} \leftarrow \dots \leftarrow D_{\infty,q+\delta}$$
is a $\usrS_{\infty}$-projective resolution of certain overmodule $\tilde{M} \supset  H_{q+j} D_{\infty}$.
 By the Auslander-Buchsbaum formula,  $\mathrm{depth}_{\usrS_{\infty}}(\tilde{M}) +  \mathrm{proj. dim}_{\usrS_{\infty}}(\tilde{M}) = 
 \dim(\usrS_{\infty})$, and so
 the $\usrS_{\infty}$-depth of $\tilde{M}$ is at least  $\dim \usrS_{\infty}-\delta+\delta'$. 
 
Therefore, the dimension of $H_{q+j} D_{\infty}$ as an $\usrS_{\infty}$-module is at least $\dim \usrS_{\infty}-\delta+\delta'$ (a  lemma from commutative algebra: depth bounds from below the dimension of any submodule).

But if $N$ is a  module over $\usR_{\infty}$ which is finitely generated when considered as an $\usrS_{\infty}$-module, then 
$\dim_{\usrS_{\infty}} N \leq  \dim \usR_{\infty}$.   This is Lemma 2.8 (3) of Khare-Thorne \cite{KT}; note that our usage of $S$ and $R$ is switched from theirs.

So we have a contradiction unless $\delta' = 0$, i.e.\ $D_{\infty}$ is a projective resolution of $H_q D_{\infty}$;
also  the dimension and depth of $H_q D_{\infty}$ as an $\usrS_{\infty}$-module was $\dim(\usR_{\infty})$.

In particular, the $\usR_{\infty}$-depth of $H_q D_{\infty}$ is $\dim(\usR_{\infty})$ (definition via regular sequences);
Auslander--Buchsbaum now forces $H_q D_{\infty}$ to be free over $\usR_{\infty}$.   \qed

Thus, if $M$ is the unique nonvanishing homology group of $D_{\infty}$, placed in degree $q$, 
there is a natural quasi-isomorphism $D_{\infty} \stackrel{\sim}{\rightarrow} M$,
compatible (by definition) with the $\usR_{\infty}$-actions and 
we have quasi-isomorphisms in the derived category of $\overline{\usrS_n}$-modules: 
  \begin{equation} \label{Cnhomology} \underbrace{M  \dotimes_{\usrS_{\infty}} \overline{\usrS_n} }_{\mbox{  derived, i.e.\ $\mathrm{Tor}$}}
\stackrel{\sim}{\longleftarrow} D_{\infty}  \otimes_{\usrS_{\infty}} \overline{\usrS_n}  
  \stackrel{\sim}{\longrightarrow}  D_n. \end{equation}  These
  are compatible with the action of $\usR_{\infty}$, where $r \in \usR_{\infty}$
  is acting as $r \otimes 1$ on the first and second terms in the sequence.   (This follows from
  the definition of ``compatible data.'') In particular, we get    \begin{equation} \label{solnit}  \underbrace{M  \otimes_{\usrS_{\infty}}  \overline{\usrS_n} }_{\mbox{usual tensor product (not Tor)}} \cong H_q(D_n),\end{equation}
  compatibly with the action of $\usR_{\infty}$. 
  
Finally return to \eqref{goofball}:
$$ \usR_{\infty} \otimes_{\usrS_{\infty}}\overline{\usrS_n} \twoheadrightarrow \usbR_n. $$ 
We claim it is an {\em isomorphism}. 
  Indeed,   writing as usual $\mathfrak{a}_n$ for the kernel of $\usS_{\infty} \rightarrow \overline{\usrS_n}$,
 we must prove that  \begin{equation} \label{finitelevelTW} \usR_{\infty}/\iota(\mathfrak{a}_{n}) \usR_{\infty} \stackrel{\sim}{\longrightarrow} \usbR_n. \end{equation} 
Clearly that map is surjective, so we just need to prove injectivity. Suppose $x \in \usR_{\infty}$ maps to zero in $\usbR_n$. 
Then $x$ acts trivially on $ H_0 D_n$ (discussion after \eqref{reindex3}). 
But in  \eqref{solnit} we see that $H_0 D_n$ is identified, as a module for $\usR_{\infty}$, with the quotient  $M/\mathfrak{a}_n M$,
and $M$ is free over $\usR_{\infty}$;  thus $x \in \iota(\mathfrak{a}_n) R_{\infty}$ as desired.

 Finally, note that we also can draw a conclusion about the homology of $Y_0$ itself.  We have an identification:  
\begin{eqnarray*} H_* (Y_0, W)_{\mathfrak{m}} = \varprojlim_n H_*(Y_0, W_n)_{\mathfrak{m}}   \stackrel{\eqref{Fred}}{\simeq} \varprojlim  H_* \left( D_n \otimes_{\overline{\usrS_n}} W_n \right) \\ 
= H_*(\varprojlim D_n \otimes_{\overline{\usrS_n}} W_n ) = H_*(D_{\infty} \otimes_{\usrS_{\infty}} W)  \stackrel{\sim}{\longrightarrow} \Tor_*^{\usrS_{\infty}}(M, W),\end{eqnarray*}
which carries a free action of $ \Tor_*^{\usrS_{\infty}}(\usR_{\infty},  W)$, as desired.    

In particular, in minimal nonvanishing degree, we obtain an identification 
\begin{equation} \label{fck} H_{\min}(Y_0, W)_{\mathfrak{m}} \cong M \otimes_{\usrS_{\infty}} W \end{equation}
 compatible with the action of $\usR_{\infty}$ (acting on the right as $r \otimes 1$). 

The surjection $\usR_{\infty} \rightarrow \pi_0 \mathcal{R}_S$ fixed at the start of the argument factors through 
      $\usR_{\infty} \otimes_{\usrS_{\infty}} W $:   if $a \in \usrS_{\infty}$
      lies in the kernel of the augmentation $\usrS_{\infty} \rightarrow W$, then $\iota(a)$ dies   in $\pi_0 \mathcal{R}_S /(p^n, \mathfrak{m}^{K(n)})$ for all $n$, by diagram
\eqref{factor_through}.  Comparing this with \eqref{fck} we see that 
 $$ \usR_{\infty}\otimes_{\usrS_{\infty}} W  \twoheadrightarrow \pi_0 \mathcal{R}_{S} \twoheadrightarrow  \mbox{image of Hecke algebra on $H_{\min}(Y_0, W)_{\mathfrak{m}} $}$$
are both isomorphisms.

This concludes the proof of Theorem \ref{TWoutputtheorem} -- taking for the $Q$s the sets $Q_{j_n}$, 
considered as allowable Taylor--Wiles data of level $n$,   for $\Dinf$ we take $D_{\infty}$. 
 For the $\usbR_n$s (as in the theorem statement) we take $\usbR_n$ as above,
  which were explicitly given during the re-indexing process of
\S \ref{reindexing}.

\section{Identification of $\pi_* \mathcal{R}_S$} \label{piRSid}
Continuing in the situation of \S \ref{TWnotationsetup} -- in particular, the Galois representation associated to a cohomology class on $\mathbf{G}$ -- we can identify $\pi_* \mathcal{R}_S$ as an abstract graded  ring.    Again we recall our standing assumption that when we speak of global Galois deformation rings, we are always imposing the crystalline condition.

\begin{theorem} \label{TW main theorem}  
Notations and assumptions as in \S \ref{TWnotationsetup}, and assume
Conjecture \ref{GaloisRepConjecture}.  Also use the notation of Theorem  \ref{TWoutputtheorem}. Then
 there is an isomorphism of graded rings
\begin{equation} \label{iso1}  \pi_* \mathcal{R}_S \cong  \Tor_{\usrS_{\infty}}^*( \usR_{\infty} , W ), \end{equation}
between the associated graded ring $\pi_* \mathcal{R}_S$ to the derived deformation ring $\mathcal{R}_S$ at base level,
and the $\Tor$-algebra on the right.   

In particular, using Theorem \ref{TWoutputtheorem} (e), 
the homology $H_*(Y_0, W)_{\mathfrak{m}}$  carries the structure of  a free graded module over the graded ring $\pi_* \mathcal{R}_S$.

 \end{theorem} 
 This result seems rather unsatisfying at first,  since it does not pin down any characterizing property of the isomorphism or the action.  It must be so, in a sense,
 because everything depends on the limit process in the Taylor--Wiles method. 
 
 However, even at this abstract level, the following consequence of the final sentence is interesting:
 \begin{quote} The homotopy groups  $\pi_j \mathcal{R}_S$ are vanishing unless $j \in [0,\delta]$.
 \end{quote}

 In the final section of this paper  we will show that the action of $\pi_*\mathcal{R}_S$ on $H_*(Y_0, W)_{\mathfrak{m}}$ from the statement -- more precisely,
 the explicit action that is constructed in the proof of the statement -- is
 closely related to the action of a derived Hecke algebra, and is in particular independent of all choices made during the proof.
 
 Also, note that the Theorem  {\em does not use} the full strength of the assumptions
 from  \S \ref{TWnotationsetup}; in particular,  it doesn't use assumption 7(c) that excludes congruences with other forms.

 \proof  (of the Theorem)
  
 Apply Theorem \ref{TWoutputtheorem}; it   
 gives  Taylor--Wiles sets $Q_n$, limit rings $\usR_{\infty}, \usS_{\infty}$,
 and isomorphisms from $\usR_{\infty}/\mathfrak{a}_n$
 to a quotient  $\usbR_n$ of the usual deformation ring  at level $S \coprod Q_n$ (plus more!) 

We  next proceed as in \S \ref{patchingdiscussion}, and use notation as there: $\mathcal{R}_n$
is the derived deformation ring at level $S \coprod Q_n$, $\mathcal{S}_n$  is the
  framed derived deformation ring for places in $Q_n$, and $\mathcal{S}_n^{\ur}$ is the unramified version of $\mathcal{S}_n$.

The diagram $\mathcal{R}_n \leftarrow \mathcal{S}_n \rightarrow \mathcal{S}_n^{\ur}$
maps to the diagram  of discrete rings
\begin{equation} \label{doog} \usbR_n \leftarrow    \underbrace{ \pi_0 \mathcal{S}_n \otimes_{W[\widetilde{\Delta}_n]} W_n[\Delta_n]  }_{:= \usbS_n}\rightarrow  \underbrace{\pi_0 \mathcal{S}_n^{\ur}/p^n}_{:= \usbS_n^{\ur}} \end{equation}
Here, we recall from \S \ref{seq TW sec} that $\tilde{\Delta}_n = \prod_{q \in Q_n} \mathbf{T}(\mathbb{F}_q)_p$ is the Galois group of the covering $Y_1(Q_n) \rightarrow Y_0(Q_n)$, 
  whereas $\Delta_n = \tilde{\Delta}_n/p^n$ is just the piece of it used in the limit process;
  the  map from $W[\tilde{\Delta}_n]$ to $\pi_0 \mathcal{S}_n$ 
  is that arising from the description of the latter as a group algebra (Remark \ref{canonicalS}).

Note that $\usbS_n$ is larger than the ring $\overline{\usrS_n}$ -- roughly the former measures all local deformations at $Q_n$ whereas the latter just measures deformations
on inertia.  For example, if $Q_n$ consisted only of the  single prime $q$,
refer to \eqref{have-you-not-heard-of-this?} to see the difference (the bars over $\usbS_n$ and $\usrS_n$ are because they are reduced modulo $p^n$). 
However, the diagram \eqref{doog} still admits a map from
\begin{equation} \label{doog2} \usbR_n \leftarrow \overline{\usrS_n} \rightarrow W_n.\end{equation}
This map of diagrams, from \eqref{doog2} to \eqref{doog}, induces a weak equivalence on derived tensor product (by this, we mean
as usual a weak equivalence of represented functors)
by \eqref{inertia_pullback2}, or its obvious analogue with more than one prime. 

 Thus, we get
    \begin{equation} \label{ringmap2} \mathcal{R}_{\Z[\frac{1}{S}]}  \stackrel{\eqref{ringmap}}{\rightarrow} \mathcal{R}_n \dotimes_{\mathcal{S}_n} \mathcal{S}_n^{\ur} 
\rightarrow \usbR_{n} \dotimes_{\usbS_{n}} \usbS_{n}^{\ur}  \stackrel{\sim}{\leftarrow}
 \usbR_n \dotimes_{\overline{\usrS_n}} W_n \stackrel{\eqref{Snid0}}{\cong} \usR_{\infty}/\mathfrak{a}_n \dotimes_{ \usrS_{\infty}/\mathfrak{a}_n} W_n . \end{equation}   
   By the same type of discussion  
as \eqref{invertingringmaps}, we can  invert the weak equivalence above in a homotopical sense, obtaining a map  
$$ \mathcal{R}_{\Z[\frac{1}{S}]} \longrightarrow
\usR_{\infty}/\mathfrak{a}_n \dotimes_{ \usrS_{\infty}/\mathfrak{a}_n} W_n.$$
 in pro-$\Art_k$ 
which behaves in the expected way on tangent complexes.  
 We will apply  Theorem \ref{Patching_Theorem} to these maps. Once we check the conditions of this Theorem, it proves  Theorem \ref{TW main theorem}.

To apply Theorem \ref{Patching_Theorem}, we must check two conditions:  a numerical condition \eqref{assump:dim} and a tangent space condition \eqref{input-to-theorem-2}. 
 \eqref{assump:dim}  follows from Tate's Euler characteristic formula (this is quite routine  -- for a closely related computation, see  
\S \ref{choice}).  So let us examine
 \eqref{input-to-theorem-2}; it will basically follow from Theorem \ref{TWcloseness}, but let us write out the details.
 
 We must study the composite
\begin{equation} \label{ringmap3} \mathcal{R}_{\Z[\frac{1}{S}]} \rightarrow   \usR_{\infty}/\mathfrak{a}_n \dotimes_{ \usrS_{\infty}/\mathfrak{a}_n} W_n   \rightarrow 
 \usR_{\infty}/\mathfrak{a}_m \dotimes_{ \usrS_{\infty}/\mathfrak{a}_m} W_m\end{equation} 
for $n \geq m$: we must show it is an isomorphism on $\mathfrak{t}^0$ and a surjection on $\mathfrak{t}^1$.

 To study \eqref{ringmap3} consider, for $n \geq m$,  the quotients 
  \begin{equation}  \label{gp} \pi_0 \mathcal{R}_n \twoheadrightarrow \usbR_{n,m}, \pi_0 \mathcal{S}_n \twoheadrightarrow  \usbS_{n,m}, \  \pi_0 \mathcal{S}_n^{\ur} \twoheadrightarrow  \usbS_{n,m}^{\ur} \end{equation}
defined thus: take $\usbR_{n,m} = \usbR_n/ \mathfrak{a}_m, \usbS_{n,m} = \usbS_n/\mathfrak{a}_m$
 and   $\usbS_{n,m}^{\ur}$ to be the quotient $\usbS_n^{\ur}/p^m$. 
 Here $\mathfrak{a}_m$ in each case means the image of the ideal $\mathfrak{a}_m \subset \usrS_{\infty}$ (see \eqref{ardef});
 note there is a map $\usrS_{\infty}  \stackrel{\eqref{SinftySn0}}{\longrightarrow}\overline{ \usrS_n} = W_n[\Delta_n] \stackrel{\eqref{have-you-not-heard-of-this?}}{\longrightarrow}  \usbS_n$.

 All these maps are isomorphisms on $\mathfrak{t}^0$, at least if $m \geq 2$ (they are automatically injective and then one can count dimensions;  and for $m \geq 2$ the ideals $\mathfrak{a}_m$ and $p^m$  are contained in the 
square of the maximal ideal  of $\usrS_{\infty}$ and thus do not affect tangent spaces).

Theorem \ref{TWcloseness} shows that  the  resulting map  
$$  \mathcal{R}_{\Z[\frac{1}{S}]} \rightarrow   \usbR_{n,m} \dotimes_{\usbS_{n,m}} \usbS_{n,m}^{\ur}$$ 
  is an isomorphism on $\mathfrak{t}^0$ and a surjection on $\mathfrak{t}^1$. But then the same assertion holds for \eqref{ringmap3}, as follows from the diagram
      \begin{equation} \label{factor_through23}
     \xymatrix{
 \mathcal{R}_{\Z[\frac{1}{S}]}  \ar[r]   \ar[d]^{=}  &   \usR_{\infty}/\mathfrak{a}_n \dotimes_{ \usrS_{\infty}/\mathfrak{a}_n} W_n    \ar[r] \ar[d]^{\eqref{Snid0}} &  \usR_{\infty}/\mathfrak{a}_m \dotimes_{ \usrS_{\infty}/\mathfrak{a}_m} W_m \ar[d]^{t} \\ 
  \mathcal{R}_{\Z[\frac{1}{S}]}  \ar[r]  &\usbR_n \dotimes_{\usbS_n} \usbS_n^{\ur}   \ar[r] &   \usbR_{n,m} \dotimes_{\usbS_{n,m}} \usbS_{n,m}^{\ur}
   }
\end{equation} 
Here, the right hand vertical arrow $t$ induces an isomorphism on tangent complexes,  for reasons similar to the validity of  \eqref{inertia_pullback2}. 
    \qed

\subsection{} \label{two_limits}
    
    During the proof of this Theorem,
we carried out two limit processes:  first the limit process of Theorem \ref{TWoutputtheorem} and then again the  limit process
 of Theorem \ref{Patching_Theorem}.  By a slight reindexing,  very similar to that already done in \S \ref{reindexing}, it is possible to
 set things up so these limit processes apply 
  simultaneously - this is a minor observation
 that will be helpful in the next section.  
 
Explicitly, it   is possible to choose the sets $Q_n$ in such a way that:
\begin{itemize}
\item[-]
the conclusion of Theorem \ref{TWoutputtheorem} remains true, {\em and}  
\item[-] 
the composite maps $f_n : \mathcal{R}_S \rightarrow \usR_{\infty}/\mathfrak{a}_n \dotimes_{\usrS_{\infty}/\mathfrak{a}_n}  W_n$
from \eqref{ringmap2}
  all satisfy the conditions of Theorem \ref{Patching_Theorem}, 
  {\em and}
  \item[-]  
These composite maps  are all homotopy compatible, in the sense of the proof of Theorem \ref{Patching_Theorem} -- in particular,
they fit together to give an isomorphism 
\begin{equation} \label{froudo} \pi_* \mathcal{R}_S \cong \Tor_{\usrS_{\infty}}(\usR_{\infty} ,W).\end{equation}
\end{itemize}
    
    This is just like \S \ref{reindexing}:  first choose sets $Q_n$ as in Theorem \ref{TWoutputtheorem}, 
then apply Theorem \ref{Patching_Theorem}. What the proof of Theorem \ref{Patching_Theorem} actually outputs is     
  a  subsequence $j_n$
 so that the composite of $f_{j_n}$ with the natural map $ \usR_{\infty}/\mathfrak{a}_{j_n} \dotimes_{\usrS_{\infty}/\mathfrak{a}_{j_n}  } W_{j_n} \rightarrow
  \usR_{\infty}/\mathfrak{a}_n \dotimes_{\usrS_{\infty}/\mathfrak{a}_n  } W_n$.
  Now reindex   just like \S \ref{reindexing}: set $Q_n^{\mathrm{reindexed}} = Q_{j_n}$,
  and let $\usbR_n^{\mathrm{reindexed}}$ be the quotient $\usR_{\infty}/(\mathfrak{a}_n,\mathfrak{m}^{K(n)})$ of $\usbR_{j_n}  \cong  \usR_{\infty}/\mathfrak{a}_{j_n}$.

\section{Comparison with the derived Hecke algebra. Conclusions}
 \label{DHAcompare}

The goal of this section is to compare the action of $\pi_* \mathcal{R}_{\Z[\frac{1}{S}]}$ that we have defined,
with the action of the derived Hecke algebra from \cite{DHA}.    This comparison  (Theorem \ref{refit}) will show, in particular,
that the action of $\pi_* \mathcal{R}_{\Z[\frac{1}{S}]}$ constructed in Theorem \ref{TW main theorem} do not depend on the choices involved
in that construction. 
 Unfortunately,  the constructions from \cite{DHA} are a little long to describe here.
Our current discussion is therefore not at all self-contained: we will freely quote what we need from \cite{DHA}. 

Let us note (as mentioned in Remark \ref{minimal level}) that our fairly strong local assumptions on $\rho$
force $\pi_* \mathcal{R}_{\Z[\frac{1}{S}]}$ to be an {\em integral} exterior algebra, and
the analysis of this section will really use this structure. However, we expect that the analogues of the results
of this section will remain valid without these local assumptions, so long as one tensors with $\Q$.

We continue with the notations that have been set up in prior sections, 
in particular in \S \ref{TWnotationsetup}.

\subsection{The derived Hecke algebra} 
In the paper \cite{DHA} a derived version of the Hecke algebra is introduced; it acts on $H^*(Y_0, W)$ by degree-increasing endomorphisms 
or (by a slight variant, replacing the role of cup product with cap  products) on $H_*(Y_0, W)$ by degree-decreasing endomorphisms. 
Thus it cannot be ``the same as'' the action of $\pi_* \mathcal{R}$ because the operations move degrees in different directions. 

Our statement is rather that the two actions are ``compatible,'' in a sense to be explained -- if $\Vinf$ is a vector space,
there is a natural action of $\wedge^* \Vinf$ on $\wedge^* \Vinf$ by multiplication, but there is {\em also}
an action of $\wedge^* \Vinf^*$ by contractions. The relationship between the two actions is a model for the relationship
between the derived Hecke algebra and $\pi_* \mathcal{R}$.

   \subsection{Some linear algebra}
   
   Let $\Vinf$ be a finite-dimensional vector space, with dual $\Vinf^*$, and let $M$ 
   be a finite-dimensional graded  vector space (over the same field) that carries a graded action of 
   $\wedge^* \Vinf$ and $\wedge^* \Vinf^*$;
   here $\Vinf$ increases degree by $1$ and $\Vinf^*$ decreases degree by $1$.  
   
   We say these actions are {\em compatible}  
   if for $v^* \in \Vinf^* $ and $v \in \Vinf$ we have
\begin{equation} \label{compat} v^* \cdot v \cdot m + v \cdot v^* \cdot m = \langle v, v^* \rangle \cdot m\end{equation}
Thus, the standard actions of $\wedge^* \Vinf, \wedge^* \Vinf^*$ on $\wedge^* \Vinf$
(the first by multiplication, the second by contractions) are compatible.

 Suppose now that $M$ is generated as a $\wedge^* \Vinf$-module by elements of minimal degree. 
 In that case,  \eqref{compat} uniquely specifies the $\wedge^* \Vinf^*$ action.  Similarly the other way around.

Now take $$\Vinf = H^1_f( \Ad^* \rhoglob \ (1))^{\vee},$$ where the $\vee$ means to take $W$-linear homomorphisms to $W(k)$,
$ \rhoglob$ is as in \eqref{tilderhodef}, and $\Ad^*$ is the dual of the adjoint representation. Under our assumptions, it is \cite[\S 8.8]{DHA} a finite free $W(k)$-module of rank $\delta$;
in \cite[Theorem 8.5]{DHA}   we construct, by means of the derived Hecke algebra and under the same conjecture and same local hypotheses on $\rho$, an action: 
\begin{equation} \label{DHAV}  \wedge^* \Vinf \acts H^*(Y_0, W)_{\mathfrak{m}}  \end{equation} 
By adjointness we obtain a homological version: 
 \begin{equation} \label{DHAVvar}  \wedge^* \Vinf \acts H_*(Y_0, W)_{\mathfrak{m}}  \end{equation} 
which is  adjoint to \eqref{DHAV} (in the case at hand, the natural pairing of homology and cohomology
is a perfect pairing with $W$-coefficients).  This can also be constructed directly from an action of the derived Hecke algebra:
in \cite{DHA} the derived Hecke algebra acts on cohomology; 
 but replacing the role of cup products by cap produts it also acts on homology.

We will exhibit below (see \eqref{go_seigen_2}) an isomorphism 
 \begin{equation} \label{DHAV2} \pi_* \mathcal{R}_S \cong \wedge^* \Vinf^* \end{equation} 
 and show that the following actions are compatible, in the sense of \eqref{compat}:
 \begin{itemize}
 \item[-] the  action of $\wedge^* \Vinf$ via \eqref{DHAVvar}
 \item[-] the
 action of $\wedge^* \Vinf^*$ via \eqref{DHAV2} and the action constructed in the course of proving Theorem \ref{TW main theorem}
 \end{itemize} 
 In particular this means
 that the action of $\pi_* \mathcal{R}_S$ on $H_*(Y_0, W)_{\mathfrak{m}}$ is independent of the choice of Taylor-Wiles sets.

   We begin with:
   
\begin{Lemma} \label{Relative tangent space}
As usual, let $\mathcal{R}_{S}$ be the crystalline deformation ring of $\rho$; 
let $\rmA \in \Art_k$ be homotopy discrete, and $M$ a discrete $\rmA$-module.
  Fix a lift $\rho_{\rmA} : \Gamma_S \rightarrow G(A)$, classified by a map  $\phi: \pi_0 \mathcal{R}_S \rightarrow A$. 
  
 The set of homotopy classes of maps 
 $\mathcal{R}_S \rightarrow \rmA \oplus M[1]$ lifting $\phi$
 are in natural bijection with  $H^2_f(\Ad \rho_{\rmA} \otimes M)$. 
\end{Lemma}
In this statement,  $\Ad \rho_{\rmA} \otimes M$ refers to $\Lie(G)_{W(k)} \otimes M$, endowed with a $\Gamma_S$-module structure
by identifying it with the kernel of $G(\rmA \oplus M) \rightarrow G(\rmA)$.

\proof  
For $\rmA = k$ and $M =k$ this is the computation of the tangent complex,  
i.e.\ Theorem  \ref{tangent complex with def conditions}.  The same formalism works replacing $k$ by $\rmA$; for example
if $\mathcal{F}: \Art_k \rightarrow s\Sets$ is a formally cohesive functor, and we fix a vertex $x_0$  of  set $\mathcal{F}(\rmA)$ (where $\rmA$ is homotopy discrete) then
we may form an $\Omega$-spectrum whose $n$th space is 
 $$X_n = \mbox{homotopy fiber of $\mathcal{F}(A \oplus M[n]) \rightarrow \mathcal{F}(A)$ above $x_0$}$$
 and it has similar formal properties to the tangent complex; the explicit computation in the case at hand proceeds just as in (the various inputs to)
 Theorem \ref{tangent complex with def conditions}.\qed

 Let us note  for future reference that we can also describe the $\pi_0$ described in the Lemma as the Andr{\'e}--Quillen cohomology group
\begin{equation} \label{needs a proof}  \Der^1_{\Z}(\mathcal{R}_S,  M), \end{equation}   which we understand, for $\mathcal{R}_{\alpha}$ a pro-system, to be the direct limit
 $\varinjlim \Der^1_{\Z}(\mathcal{R}_{\alpha}, M)$;
this description is valid with $\mathcal{R}_S$ replaced by any pro-simplicial ring.

Now, any map $\tilde{\phi}: \mathcal{R} _S\rightarrow \rmA \oplus \rmM[1]$ induces a map $\pi_1 \mathcal{R}_S \rightarrow \rmM$. This,
   and the statement of the Lemma, gives us 
a   pairing 
\begin{equation} \label{go_seigen}  \pi_1 \mathcal{R}_S  \times  H^2_f(\Ad \rho_{\rmA} \otimes \rmM)  \rightarrow \rmM.\end{equation} 
 In particular, taking $A=W_n$, $M=(W \otimes \Q)/W$,  and the representation $\rho_{\rmA}$ to be 
 the mod $\varpi^n$ reduction of the representation \eqref{tilderhodef}  defined by $\Pi$, 
  we get (after Tate global duality   and   passage to the limit) 
\begin{equation} \label{go_seigen_2} \pi_1 \mathcal{R}_S \rightarrow    H^1_f(\Ad^* \rhoglob(1) ) = \Vinf^*.  \end{equation}
We will see later that this map is an isomorphism, and its inverse  induces an isomorphism $\wedge^* \Vinf^* \rightarrow \pi_* \mathcal{R}_S$.

   \subsection{Background on the action \eqref{DHAV} } 
Now and in the remainder of this section, we put ourselves in the situation of \S \ref{two_limits}. In other words, 
we are given a sequence of allowable Taylor--Wiles data such that the limit process of 
 Theorem \ref{TWoutputtheorem} and the limit process
 of Theorem \ref{Patching_Theorem} can be carried out simultaneously (see \S \ref{two_limits}).
 In particular, as in Theorem \ref{TWoutputtheorem}, we have ``limit rings,'' augmented to $W(k)$: 
$$\usrS_{\infty} \stackrel{\iota}{\longrightarrow }\usR_{\infty} \twoheadrightarrow W(k).$$ 
 
 Let $\mathfrak{p}_S$ be the kernel of the augmentation $\usS_{\infty} \rightarrow W(k)$ (and similar for $\usR_{\infty}$);
write $\mathfrak{t}_S^*$ and $\mathfrak{t}_R^*$ for the quotient space $\mathfrak{p}_S/\mathfrak{p}_S^2$. Let $\mathfrak{t}_S, \mathfrak{t}_R$ be the $W$-linear duals.  To say differently,
$$ \mathfrak{t}_S = \Der^0_W(\usrS_{\infty}, W)$$
is the set of derivations of the $W$-algebra $\usrS_{\infty}$ with values  in $W$ and similarly for $\usR_{\infty}$. 

The map $\usrS_{\infty}  \twoheadrightarrow \usR_{\infty}$  (recall it is surjective, discussion after Theorem \ref{TWoutputtheorem}) induces
$\mathfrak{t}_R \hookrightarrow \mathfrak{t}_S$, and we write $  \mathfrak{t}_S/\mathfrak{t}_R$ be the cokernel.   
In suitable coordinates we have (\cite[\S 7.3]{DHA}): 
$$\usS_{\infty} \cong W[[x_1, \dots, x_s]], \usR_{\infty} \cong W[[y_1, \dots, y_{s-\delta}]]$$
and the map between them sends $x_i$ to $y_i$ for $i \leq s-\delta$ and kills $x_i$ for $i > s-\delta$. 
 
There are natural isomorphisms
\begin{equation} \label{first0} \Ext_{\usS_{\infty}}^1(W,W) \cong  \mathfrak{t}_S \end{equation}
\begin{equation} \label{second0} \Tor^{\usS_{\infty}}_1(\usR_{\infty}, W) \cong  (\mathfrak{t}_S/\mathfrak{t}_R)^* \end{equation}
Indeed the  sequence $\mathfrak{p}_S \rightarrow \usS_{\infty} \rightarrow W$ induces
 $\Ext^1_{\usS_{\infty}}(W, W) \cong  \Hom_{\usS_{\infty}-\mathrm{mod}}(\mathfrak{p}_S/\mathfrak{p}_S^2,W) $. Similarly  write $K$ for the kernel of $\usS_{\infty} \rightarrow \usR_{\infty}$; the sequence $K \rightarrow \usS_{\infty} \rightarrow \usR_{\infty}$ 
 of $\usS_{\infty}$-modules induces 
 $\Tor_1^{\usS_{\infty}}(\usR_{\infty}, W)  =   K \otimes_{\usS_{\infty}}  W = (K/\mathfrak{p}_S K)$;
the inclusion $K \hookrightarrow \mathfrak{p}_S$ maps  this isomorphically to the kernel of $\mathfrak{p}_S/\mathfrak{p}_S^2 \rightarrow \mathfrak{p}_R/\mathfrak{p}_R^2$, 
giving rise to an isomorphism $ \Tor_1^{\usS_{\infty}}(\usR_{\infty}, W)  \cong  (\mathfrak{t}_S/\mathfrak{t}_R)^*$.

The isomorphisms \eqref{first0} and \eqref{second0} induce:  
\begin{equation} \label{first} \Ext_{\usS_{\infty}}^*(W,W) \cong \wedge^* \mathfrak{t}_S \end{equation}
\begin{equation} \label{second} \Tor^{\usS_{\infty}}_*(\usR_{\infty}, W) \cong \wedge^*  (\mathfrak{t}_S/\mathfrak{t}_R)^* \end{equation}
More precisely, in both cases, one computes that the left-hand side with its natural algebra structure
is a free exterior algebra on its degree $1$ component.

Now Theorem \ref{TWoutputtheorem} part (d) gives an identification
\begin{equation} \label{id} H_*(Y_0,W)_{\mathfrak{m}} \cong \Tor^{\usS_{\infty}}_*(\Dinf, W) \cong \Tor^{\usS_{\infty}}_*(\underbrace{H_q(\Dinf)[-q]}_{ M},  W) \end{equation}
 where $\Dinf$ is homologically concentrated in degree $q$ and $M := H_q(\Dinf)$ is a free module over $\usR_{\infty}$. 
  $\Tor^{\usS_{\infty}}_*(M,  W) $ has both a natural 
  structure of a graded  $\Tor^{\usS_{\infty}}_*(\usR_{\infty},W) \cong \wedge^* (\mathfrak{t}_S/\mathfrak{t}_R)^* $-module
  and also a natural structure of  module under $\Ext_{\usS_{\infty}}^*(W,W)  \cong \wedge^* \mathfrak{t}_S$. The latter action 
  factors through $\wedge^* \mathfrak{t}_S \rightarrow \wedge^* (\mathfrak{t}_S/\mathfrak{t}_R)$.
  The resulting actions of $ \wedge^* (\mathfrak{t}_S/\mathfrak{t}_R)$ and $\wedge^* (\mathfrak{t}_S/\mathfrak{t}_R)^*$
  are compatible,    in the sense of \eqref{compat}. All these  assertions  are verified by explicit computation.

Now given a ``convergent'' sequence of Taylor--Wiles data as in the statement of Theorem \ref{TW main theorem},  
yields an identification  \cite[\S 8.25, proof of Theorem 8.5]{DHA}\begin{equation} \label{from another world} \mathfrak{t}_S/\mathfrak{t}_R \cong \Vinf\end{equation}
such that  the following diagram commutes: 
    \begin{equation} \label{SinftySn2}
 \xymatrix{
  \Ext_{\usS_{\infty}}(W, W) \ar[rr]  \ar[d]_{\eqref{first}}   &&    \End  \left( \Tor^{\usS_{\infty}}_*(M,  W) \right) \ar[d]^{\sim}_{\eqref{id}}  \\
  \wedge^* (\mathfrak{t}_S/\mathfrak{t}_R)  \ar[r]_{\eqref{from another world}}  &  \wedge^* \Vinf\ar[r]^{\eqref{DHAVvar}\qquad }  &    \End \left(  H_*(Y_0, W)_{\mathfrak{m}}  \right)  }
\end{equation}
(of course, in this diagram, the identification \eqref{id} also depended on the choices of Theorem \ref{TW main theorem}).

We are now ready for the  basic result relating the derived Hecke algebra with the action of the derived deformation ring: 
\begin{theorem} \label{refit}  
The actions of $ \wedge^* \Vinf$ on $H_*(Y_0, W)_{\mathfrak{m}}$
via \eqref{DHAVvar} and $\wedge^* \Vinf^*$ on  $H_*(Y_0, W)_{\mathfrak{m}}$
(via the isomorphism $\wedge^* \Vinf^* \cong \pi_* \mathcal{R}$ induced by the inverse of
\eqref{go_seigen_2}) are compatible with one another, in the sense of \eqref{compat}. 
\end{theorem}

 In view of  the discussion of the paragraph following \eqref{id},   this follows from the following Lemma. 
  
 \begin{Lemma}
Let notation be as above; in particular,  $Q_n$ and other data $\usrS_{\infty}, \usR_{\infty}$ etc.\ as in  \S \ref{two_limits}. The composite  of the isomorphisms 
\begin{equation} \label{tobedone} \pi_* \mathcal{R}  \stackrel{\eqref{froudo}}{\longrightarrow} \Tor^{\usrS_{\infty}}_*(\usR_{\infty}, W)  \stackrel{\eqref{second}}{\cong}  \wedge^*  (\mathfrak{t}_S/\mathfrak{t}_R)^*  \stackrel{\eqref{from another world}}{\longrightarrow} \wedge^* \Vinf^* \end{equation}  
coincides, in degree $1$, with the map $\pi_1 \mathcal{R} \rightarrow \Vinf^*$ constructed in
\eqref{go_seigen_2}. 
\end{Lemma}
The map $\pi_1 \mathcal{R} \rightarrow \Vinf^*$ is constructed in a ``natural'' fashion, whereas the 
   first map   and the third map  in sequence \eqref{tobedone}  depend on the various choices made to set up the situation of \S \ref{two_limits}. 
Nonetheless the Lemma asserts that the composite is independent of all choices. 
   
 \proof  (of Lemma): 
 
As above $\usrS_{\infty}, \usR_{\infty}$ are naturally augmented to $W$.
For $n \geq m$ consider the following diagram, where $\Hom$ simply means homomorphisms of abelian groups:  
  \begin{equation}
 \xymatrix{
\Hom( \pi_1(\usR_{\infty} \dotimes_{\usrS_{\infty}} W) , W_m)  & \ar[l]  \pi_0 \left( \mbox{ lifts to $\usR_{\infty} \dotimes_{\usrS_{\infty}} W \rightarrow W_m \oplus W_m[1]$} \right)  \\
 \Hom(\pi_1(\usbR_n \dotimes_{\overline{\usrS_n}} W_n) , W_m)   \ar[u] \ar[d]^{\eqref{ringmap2}} &  \ar[l]    \pi_0 \left( \mbox{ lifts to $\usbR_n \dotimes_{\overline{\usrS_n}} W_n \rightarrow W_m \oplus W_m[1]$} \right) \ar[u] \ar[d]      \\ 
\Hom(\pi_1 \mathcal{R}, W_m)  &    \ar[l]  \pi_0 \left( \mbox{ lifts to $\mathcal{R} \rightarrow W_m \oplus W_m[1]$}\right)  .%
 }
 \end{equation} 
  
On the right hand side of this diagram, ``lifts''  always refers to lifting the natural map to the discrete ring $W_m$ given by the augmentations;
and the horizontal maps result by applying $\pi_1$ to such a lift. 

Consider the middle row. Using 
\eqref{needs a proof}, and the long exact sequence for Andr{\'e}-Quillen cohomology of a derived tensor product\footnote{This can be deduced from Lemma \ref{lem:2.51} (iv).}
, we get a sequence:
 $$ \frac{\Der^0_W(\overline{\usrS_n}, W_m) }{ \Der^0_W(\usbR_n, W_m)  }  \rightarrow
\Der^1_W(\usbR_n \otimes_{\overline{\usrS_n}} W_n, W_m)  \stackrel{\sim}{\rightarrow}
  \pi_0 \left( \mbox{ lifts to $\usbR_n \dotimes_{\usbS_n} W \rightarrow W_m \oplus W_m[1]$} \right) .$$

Proceeding similarly for the top row, and using
  Lemma \ref{Relative tangent space}  for the bottom row,  we arrive at the following diagram:
  
 \begin{equation} \label{cow-2} 
 \xymatrix{
\Hom( \pi_1(\usR_{\infty} \dotimes_{\usS_{\infty}} W) , W_m)  & \ar[l]_{\ \qquad \ \  \sim}  ( \mathfrak{t}_{S}/\mathfrak{t}_{R}) \otimes W_m  \ar[l] \\ 
\Hom(\pi_1(\usbR_n \dotimes_{\overline{\usrS_n}} W_n) , W_m)   \ar[u]^f \ar[d]^g   &   \ar[l]   \frac{ \Der^0_W(\overline{\usrS_n}, W_m) }{ \Der^0_W(\usbR_n, W_m)}   \ar[u]^h \ar[d]^{\theta}  \\ 
\Hom(\pi_1 \mathcal{R}, W_m)  &    \ar[l]^{{\tiny \mbox{Lemma \ref{Relative tangent space}}}}  
H^2_f(\Z[\frac{1}{S}], \Ad \rho_m).%
 }
 \end{equation} 
  
The top horizontal isomorphism is the the dual of the isomorphism \eqref{second0} $$\underbrace{\Tor^{\usS_{\infty}}_1(\usR_{\infty}, W)}_{= \pi_1(\usR_{\infty} \dotimes_{\usS_{\infty}} W)} \cong  (\mathfrak{t}_S/\mathfrak{t}_R)^*.$$ 
We need to identify the map  
$
 \theta: \Der^0_W(\overline{\usrS_n}, W_m)    \rightarrow H^2_f(\Z[\frac{1}{S}], \Ad \rho_m) $ 
in the diagram above. This is  
of the deformation rings $\mathcal{R}_S$ and $\mathcal{R}_{SQ_n}$ -- that is to say, the map $\beta$
in \eqref{foot2}; in our case: 
 \begin{equation} \label{moored}  \theta:  \bigoplus_{v \in Q_n} \frac{ H^1(\Q_v, \Ad \rho_m)}{H^1_{\ur}(\Q_v, \Ad \rho_m)} \longrightarrow  H^2_f(\Z[\frac{1}{S}], \Ad \rho_m),\end{equation} 
where the identification of
$\Der^0_W(\overline{\usrS_n}, W_m)    $ with $ \bigoplus_{v \in Q_n} \frac{ H^1(\Q_v, \Ad \rho_m)}{H^1_{\ur}(\Q_v, \Ad \rho_m)}$
is as in   \cite [\S 8.14]{DHA}.   

%
%
%
%
%
%
%
%
%

We will now use  the following fact, which can be proved by tracing through the definitions:  for $\beta \in \bigoplus_{v \in Q_n} \frac{ H^1(\Q_v, \Ad \rho_m)}{H^1_{\ur}(\Q_v, \Ad \rho_m)}$
and $\alpha \in H^1_f(\Z[\frac{1}{S}], \Ad^* \rho_m(1))$, the pairing $\langle \theta(\beta), \alpha \rangle$
of Tate global duality
can be computed as the sum of local pairings $\sum_{v \in Q_n} \langle \beta_v, \alpha|_{\Q_v} \rangle$. 
This  shows that $\theta$ coincides with the composite
\begin{equation} \label{Cow7}  \bigoplus_{v \in Q_n} \frac{ H^1(\Q_v, \Ad \rho_m)}{H^1_{\ur}(\Q_v, \Ad \rho_m)} \rightarrow H^1_f(\Z[\frac{1}{S}], \Ad^* \rho_m(1))^*   \rightarrow H^2_f(\Z[\frac{1}{S}], \Ad \rho_m)  \end{equation}
where the former map   is induced by the 
sum of local pairings 
 $\sum_{v \in Q_n} \langle \beta_v, \alpha|_{\Q_v} \rangle$ (this coincides with a map constructed in \cite[\S 8.15]{DHA}) 
 and the latter map is induced by Tate global duality.\footnote{In the situation at hand, $H^1_f$ and so also $H^2_f$ is a free
module over $W_m$, and the pairing of Tate global duality is perfect.}%

Having identified $\theta$, let us return to diagram \eqref{cow-2}. 
The middle row forms an inverse system over $n$ (by means of the identifications of (b) of Theorem  \ref{TWoutputtheorem}).
We claim that all the maps in diagram \eqref{cow-2} 
  are compatible with this inverse system. For the maps $f,h$ this is clear by definition; for the map $g$ 
it is the homotopy compatibility discussed in 
\S \ref{two_limits}; for the map $\theta$ it is  discussed in \cite[\S 8.25, after (162)]{DHA}.  

Therefore, we may obtain a new diagram by passing to an inverse limit over $n$. 
After one passes to the inverse limit over $n$ the upper maps become isomorphisms;
then we can invert  the 
top layer of vertical maps in \eqref{cow-2}. The result is 
   \begin{equation} \label{cow-3} 
 \xymatrix{
\Hom( \pi_1(\usR_{\infty} \dotimes_{\usS_{\infty}} W) , W_m) \ar[d]^F   & \ar[l]^{\qquad \eqref{second0}}  ( \mathfrak{t}_{S}/\mathfrak{t}_{R}) \otimes W_m  \ar[l]  \ar[d]^G \\ 
\Hom(\pi_1 \mathcal{R}, W_m)  &    \ar[l]^{{\tiny \mbox{Lemma \ref{Relative tangent space}}}}   
H^2_f(\Z[\frac{1}{S}], \Ad \rho_m).%
 }
 \end{equation} 
where $F =  g \circ (\varprojlim_n f)^{-1}$ and $G = \theta \circ (\varprojlim_n h)^{-1}$. 

 The isomorphism $F$ is, by definition,  exactly that  of \eqref{froudo}. Thus the composite
 $ ( \mathfrak{t}_{S}/\mathfrak{t}_{R}) \otimes W_m  \rightarrow \Hom(\pi_1 \mathcal{R}, W_m) $ is 
 the map induced by the first two arrows in  \eqref{tobedone}. 
The equality of $\theta$ and the composite of \eqref{Cow7} shows that the map $G$
above
coincides with the (reduction mod $p^m$ of the) isomorphism induced by   \eqref{from another world}. 
(Unfortunately, to verify this requires wading into the maze of \cite{DHA} to unravel the origin of \eqref{from another world}: the relevant definitions, which match well with \eqref{Cow7}, are in \cite[\S 8.15]{DHA}). 
This concludes the proof of the Lemma, and so also of Theorem \ref{refit}. \qed

\printindex 
\appendix

\section{Homotopy colimits and homotopy limits}
\label{sec:homotopy-theory}

We shall use the \emph{Bousfield--Kan formula} for homotopy (co)limits %
of simplicial sets.  We recall the definition and some well known properties here.  Reference: Bousfield and
Kan's book, chapter XI and XII.

Notationwise, we shall write $C(x,y)$ for the set of morphisms from
$x$ to $y$ in a category $C$.  For \emph{simplicially enriched
  categories} we shall also write $C(x,y)$ for the simplicial set of
morphisms.  For example, if $X$ and $Y$ are simplicial sets we shall
write $s\Sets(X,Y)$ for the simplicial set of maps $X \to Y$.

\subsection{Nerves of categories}
\label{sec:nerve}

The set $[p] = \{0, \dots, p\}$ may be regarded as an ordered set
using $\leq$, and hence a category, with one morphism $i \to j$ if
$i \leq j$ and none otherwise.  This gives a functor
$\Delta \to \mathrm{Cat}$.

If $C$ is a small category, the \emph{nerve} of $C$ is the simplicial
set whose $p$-simplices are the functors $[p] \to C$, i.e.\ $p$-tuples
of composable morphisms.  We write $NC$ for the nerve and $N_pC$ for
its set of $p$-simplices.

For an object $c \in C$, the \emph{under category} $(c \downarrow C)$
has objects pairs $(d,f)$ with $f: c \to d$ and morphisms commutative
triangles.  It comes with a forgetful functor $(c \downarrow C) \to C$
and we have a canonical functor
\begin{align*}
  C^\mathrm{op} &\to s\Sets\\
  c & \mapsto N(c \downarrow C).
\end{align*}
All values $N(c \downarrow C)$ are contractible simplicial sets, so
the functor is naturally weakly equivalent to the terminal functor
which takes all objects to a point.

\subsection{Homotopy colimits}
\label{sec:hocolim}

For a small category $C$ and a simplicial set $Y$ we obtain a functor
\begin{align*}
  C &\to s\Sets\\
  c & \mapsto s\Sets(N(c\downarrow C), Y),
\end{align*}
which we shall denote $s\Sets(N(-\downarrow C),Y)$.  The homotopy
colimit of a functor $X: C \to s\Sets$ is a simplicial set and has the
universal property that the set of maps
\begin{equation*}
  f: \hocolim_{c \in C} X(c) \to Y
\end{equation*}
are in natural bijection with the set of natural transformations
\begin{equation*}
  f: X \to s\Sets(N(- \downarrow C),Y).
\end{equation*}
(And $p$-simplices in $s\Sets(\hocolim X,Y)$ are given by the same
formula with $Y$ replaced by $s\Sets(\Delta[p],Y)$.)  In other words,
to specify a map $f: \hocolim X \to Y$ amounts to specifying maps of
simplicial sets $f_c: X(c) \times N(c \downarrow C) \to Y$ for all
objects $c \in C$, in a way that is compatible with morphisms
$c_0 \to c_1$ in $C$.  A simplicial set $\hocolim X$ with this
universal property can be constructed as a quotient of
$\coprod_{c \in C} X(c) \times N(c \downarrow C)$: explicitly it is
the coequalizer of a diagram
\begin{equation*}
  \coprod_{(c_0 \to c_1) \in N_1 C} X(c_0) \times N(c_1 \downarrow C)
  \rightrightarrows \coprod_{c \in N_0 C} X(c) \times N(c \downarrow C).
\end{equation*}

For example, if $G$ is a group considered as a category with one
element then a functor $G \to s\Sets$ is a simplicial set with an
action and the homotopy colimit is just the ``Borel construction''.

The most important property of homotopy colimits is their homotopy
invariance, in the following sense.
\begin{Lemma}
  Let $F,G: C \to s\Sets$ be two functors, and let $T: F \to G$ be a
  natural transformation.  If $T: F(c) \to G(c)$ is a weak equivalence
  for all objects $c$, then the induced map
  \begin{equation*}
    \hocolim_C F \stackrel{T}\to \hocolim_C G
  \end{equation*}
  is a weak equivalence.\qed
\end{Lemma}

Recall that a category is \emph{filtered} if for any objects $c,c'$ of
$C$ there exists an object $c''$ and morphisms $c \to c'$ and
$c \to c''$, and that any two parallel arrows $c \rightrightarrows c'$
become equal after composing with some arrow $c' \to c'''$.  For such
categories the ordinary categorical colimit is already homotopy
invariant.
\begin{Lemma}
  If $C$ is filtered and $F: C \to s\Sets$ is a functor, then the
  natural map
  \begin{equation*}
    \hocolim_{c \in C} F(c) \to \colim_{c \in C} F(c)
  \end{equation*}
  is a weak equivalence.\qed
\end{Lemma}
Despite this lemma, the homotopy colimit over a filtered category can
still be quite useful: for example, it can be easier to define
explicit maps out of the homotopy colimit.

\begin{Remark}\label{rem:telescope}
  In the special case $C = \N$ there is a convenient smaller model for
  the homotopy colimit of $X: \N \to s\Sets$.  Namely for each
  $n \in \N$ we have the sub simplicial space
  \begin{equation*}
    \Delta[1] \cong N(n \downarrow \{n, n-1\}) \subset N(n \downarrow \N),
  \end{equation*}
  and the union of the sub simplicial sets
  $X(n) \times \Delta[1] \subset \hocolim X(n)$ is the \emph{mapping
    telescope} of $X(0) \to X(1) \to \dots$, obtained by gluing the
  spaces $X(n) \times \Delta[1]$ along the maps $X(n) \to X(n+1)$.
  The inclusion of the telescope into the homotopy colimit is a weak
  equivalence.
\end{Remark}

\subsection{Homotopy limits}
\label{sec:holim}

The homotopy limit of a functor $Y: C^\mathrm{op} \to s\Sets$ is  defined dually.  The under category $(c \downarrow C^\mathrm{op})$ is
opposite to the over category $(C\downarrow c)$, and gives a functor
\begin{align*}
  C & \to s\Sets\\
  c & \mapsto N(c \downarrow C^\mathrm{op}).
\end{align*}
The homotopy limit of a functor $Y: C^\mathrm{op} \to s\Sets$ is a
simplicial set with the universal property that the set of maps
\begin{equation*}
  X \to \holim_{c \in C^\mathrm{op}} Y(c)
\end{equation*}
are in natural bijection with the set of natural transformations
\begin{equation*}
  X \times N(- \downarrow C^\mathrm{op}) \to Y.
\end{equation*}
In fact, the homotopy limit can be regarded as the simplicial set of
natural transformations $N(-\downarrow C^\mathrm{op}) \Rightarrow Y$,
where the simplicialness comes from the fact that both functors take
values in the simplicially enriched category $s\Sets$.  Thus, $\holim
Y$ is the simplicial subspace
\begin{equation*}
  \holim Y \subset \prod_{c \in N_0 C} s\Sets(N(c \downarrow C^\mathrm{op}),Y(c))
\end{equation*}
consisting of elements satisfying that for each
$(C_0 \to C_1) \in N_1(C)$ the usual square (defining ``naturality'')
commutes, as a diagram of simplicial sets.
 
\begin{Example} \label{homotopy pullback square Example}
  When $C$ is the three-object category
  $C = (\bullet \leftarrow \bullet \rightarrow \bullet)$, a functor
  $Y: C^\mathrm{op} \to s\Sets$ is the same thing as a diagram of
  simplicial sets $Y_0 \rightarrow Y_{01} \leftarrow Y_1$.  In this
  case the homotopy limit is often called the \emph{homotopy
    pullback}, at least when all three spaces are Kan, and we shall
  denote it $Y_0 \times_{Y_{01}}^h Y_1$.  Spelling out the definition,
  there is a bijection between maps $X \to Y_0 \times_{Y_{01}}^h Y_1$
  and tuples $(f_0, f_1, f_{01},h_0,h_1)$ consisting of maps
  \begin{equation*}
    \xymatrix{
      X \ar[r]^-{f_1}\ar[dr]^-{f_{01}}\ar[d]_{f_{0}} & Y_1\ar[d]^-{g_1} \\
      Y_0 \ar[r]_{g_0} & Y_{01}
      }
  \end{equation*}
  and simplicial homotopies $h_0: \Delta[1] \times X \to Y_{01}$  from
  $g_0 \circ f_0$ to $f_{01}$ and
  $h_1: \Delta[1] \times X \to Y_{01}$ from $g_1 \circ f_1$ to $f_{01}$.
\end{Example}

We have the following nice adjunction formula
\begin{Lemma}
  Let $X: C \to s\Sets$ be a functor and $Y$ be a simplicial set.
  Then we have a natural isomorphism of simplicial sets
  \begin{equation*}
    s\Sets(\hocolim_{c \in C} X(c),Y) \cong \holim_{c \in C^\mathrm{op}}
    s\Sets(X(c),Y).
  \end{equation*}
\end{Lemma}

\begin{Lemma}
  Let $Y \to Y'$ be a natural transformation of functors
  $C^\mathrm{op} \to s\Sets$ be a functor such that $Y(c) \to Y'(c)$
  is a weak equivalence of Kan simplicial sets for all objects
  $c \in C$.  Then the induced map $\holim Y \to \holim Y'$ is a weak
  equivalence.
\end{Lemma}
However, we caution the reader that the Bousfield--Kan formula for the
homotopy limit of $Y \to s\Sets$ does \emph{not} have good homotopical
properties unless all values $Y(c)$ are Kan.  In fact it might not
even agree with what other authors (such as Quillen) call ``homotopy
limit'' in such cases.

\subsubsection{Products}

The following lemma follows from the definition.
\begin{Lemma}
  Let $X,Y: C \to s\Sets$ be two functors and let
  $X \times Y: C \to s\Sets$ denote the objectwise cartesian product.
  Then the natural map
  \begin{equation*}
    \holim(X \times Y) \to (\holim X) \times (\holim Y)
  \end{equation*}
  is an isomorphism of simplicial sets.\qed
\end{Lemma}
\begin{Corollary}
  Let $\SCR$ denote the category of simplicial commutative rings and
  let $X: C \to \SCR$ be a diagram of such.  Then $\holim X$ is
  naturally a simplicial commutative ring again.  Similarly for
  simplicial modules, simplicial abelian groups, etc.\qed
\end{Corollary}

 $\SCR$ has the structure of a 
\emph{simplicial model category}, and hence it makes sense to talk
about ``internal'' homotopy limits and colimits of functors
$C \to \SCR$.  The above corollary shows that the forgetful functor
$\SCR \to s\Sets$ commutes with homotopy limits, just as the forgetful
map from commutative rings to sets commutes with ordinary limits.

\subsection{Calculation of homotopy limits and colimits}
Let us briefly discuss a few tools for manipulating homotopy limits
and colimits.

\subsubsection{Cofinality}
\label{sec:cofinality}

If $D$ is a category and $X: D \to s\Sets$ and
$Y: D^\mathrm{op} \to s\Sets$ are diagrams, then any functor $F: C \to
D$ induce maps of simplicial sets
\begin{align*}
  \hocolim_C (X \circ F) &\to \hocolim_D X\\
  \holim_{D^\mathrm{op}} Y \to \holim_{C^\mathrm{op}} Y \circ F.
\end{align*}
These maps are both weak equivalences, provided $F$ is \emph{cofinal}
in the strong sense that for all objects $d \in D$ the under category
$(d \downarrow F)$, whose objects are pairs $(c,f)$ with $c \in C$ and
$f: d \to F(c)$, have \emph{contractible} nerve.  (Assuming only that
these under categories have \emph{connected} nerves is sufficient for
the corresponding maps for limits and colimits of sets to be
bijections.)

Let us also recall from [SGA 4, Expose 1, 8.1.6] or [Edwards--Hastings
2.1.6] that any filtered category $D$ admits a cofinal functor
$t: C \to D$ with $C$ a \emph{directed set}.  The standard proof is to
let $C$ be the set of finite subdiagrams $A \subset D$ which has a
unique final element.  Then $C$ is ordered by inclusion and
$t: C \to D$ sends a diagram to its final element.  Let us point out
that the resulting directed set $C$ comes with an order preserving map
$C \to \N$ given by sending $A$ to its cardinality.  Thus, when
calculating homotopy limits or colimits over filtered categories, we
may assume that the indexing category is a directed set equipped with
an order preserving map to $\N$.

\subsubsection{Homotopy groups of homotopy limits}

It can be difficult to understand the homotopy groups of a homotopy
limit, even when the indexing category is filtered and even for
$\pi_0$.  Under an additional assumption on either the indexing
category or on the values of the functor we can say more.
\begin{Lemma} \label{pi0 commutes holim 1}
  If $C$ is countable and filtered and $F: C^\mathrm{op} \to s\Sets$
  has Kan values, then
  \begin{equation*}
    \pi_0 \holim F(c) \to \lim \pi_0 F(c)
  \end{equation*}
  is surjective.  For any $k \geq 0$ and any point $x \in \holim F(c)$
  there is a short exact sequence of groups (pointed sets for $k = 0$)%
  \begin{equation*}
    {\lim_{c \in C^\mathrm{op}}}^1 \pi_{k+1} (F(c),x) \to \pi_k (\holim_{c \in C^\mathrm{op}} F(c),x)
    \to \lim_{c \in C^\mathrm{op}}
    \pi_k(F(c),x).
  \end{equation*}
\end{Lemma}
\begin{proof}[Proof sketch]
  The lemma is clear if $C$ has a final element.  If this is not the
  case, countability of $C$ implies that there is a cofinal functor
  $\N \to C$, so without loss of generality we may just assume
  $C = \N$.  An element in $\lim \pi_0 F(n)$ may then be represented
  by zero-simplices $x_n \in F(n)$ such that there exists 1-simplices
  $h_n: \Delta[1] \to F(n)$ from $x_n$ to the image of $x_{n+1}$.
  Then use the Kan-ness of the $F(n)$ to inductively extend the $h_n$
  to compatible maps $N(n \downarrow \N^\mathrm{op}) \to F(n)$.  (Or
  alternatively use a different model for the homotopy limit, dual to
  the telescope model for homotopy colimit, in which no further data
  than the $(x_n,h_n)$ is required.)
\end{proof}
Without the countability assumption on $C$ it can apparently happen
that for example $\pi_0 \holim F \to \lim \pi_0 F$ is not
surjective.  However, if we assume that the values of $F$ have
finite homotopy groups we may use Tychonoff's theorem to rule out this
kind of behavior.
\begin{Proposition} \label{homotopy limits in the finite case}
  Let $C$ be filtered and $F: C^\mathrm{op} \to s\Sets$ be a functor
  whose values are Kan and have $\pi_0(F(c))$ finite for all $c$ and
  $\pi_k(F(c),x)$ finite for all $k$ and all $x \in F(c)$.  Then the
  natural map
  \begin{equation*}
    \pi_0 \holim F(c) \to \lim \pi_0 F(c)
  \end{equation*}
  is a bijection and, for any point $x \in \holim F(c)$, so is the
  natural map
  \begin{equation*}
    \pi_k(\holim F(c),x) \to \lim \pi_k(F(c),x).
  \end{equation*}
\end{Proposition}
\begin{proof}[Proof sketch]
  It should be well known
     that this follows from the Bousfield--Kan spectral
  sequence and the vanishing of higher derived limits of finite groups
  over filtered categories.  Let us outline a ``manual'' argument.

  It suffices to prove the claim about $\pi_0$, since $\holim$
  commutes with based loop spaces.  We shall outline the argument for
  surjectivity of the map $\pi_0 \holim \to \lim \pi_0$.  An element
  in $\lim \pi_0 F(c)$ determines a sub-functor of $F$ with path
  connected values, so it suffice to prove that if $F(c)$ is path
  connected for all $c$ then $\holim F(c)$ is non-empty.

  Without loss of generality $C$ is a directed set such that
  $\{c \in C \mid c \leq d\}$ is finite for all $d$ and hence that any
  finite collection of objects is contained in a finite sub poset
  which has a maximal terminal element.  Let $x_c \in F(c)$ be a
  vertex, and choose for all pairs $c_0 < c_1$ a 1-simplex
  $h_{c_0 < c_1}: \Delta[1] \to F(c_0)$ from $x_{c_0}$ to the image of
  $x_{c_1} \in F(c_1) \to F(c_0)$.  These 1-simplices may be assembled
  to define compatible maps
  $N_{\leq 1}(c \downarrow C^\mathrm{op}) \to F(c)$, where
  $N_{\leq 1}$ denotes the 1-skeleton of the nerve.  The problem is
  that these maps may not admit extensions over the 2-skeleton, and in
  fact there is an obstruction in $\pi_1(F(c_0),x_{c_0})$ for each
  $c_0 < c_1 < c_2$, given by concatenating the (images in $F(c_0)$
  of) the paths $h_{c_0 < c_1}$, $h_{c_0 < c_2}$, and $h_{c_1 < c_2}$
  in the appropriate order.  These obstructions assemble to an element
  of
  \begin{equation}\label{eq:1}
    \prod_{c_0 < c_1 < c_2} \pi_1(F(c_0),x_{c_0})
  \end{equation}
  which vanishes if and only if the maps
  $N_{\leq 1}(c \downarrow C^\mathrm{op}) \to F(c)$ extend to
  compatible maps over the two-skeleton.  We decide to keep the $x_c$
  and rechoose the $h_{c_0 < c_1}$.  The homotopy class of
  $h_{c_0 < c_1}$ relative to its endpoints is a torsor for
  $\pi_1(F(c_0),x_{c_0})$, and acting simultaneously for all $(c_0 <
  c_1)$ gives a map
  \begin{equation}\label{eq:2}
    \prod_{c_0 < c_1} \pi_1(F(c_0),x_{c_0}) \to \prod_{c_0 < c_1 <
      c_2} \pi_1(F(c_0),x_{c_0}).
  \end{equation}
  If we can show that this map is surjective, the paths
  $h_{c_0 < c_1}$ may be rechosen to make the obstruction element
  in~\eqref{eq:1} vanish.  Such a re-choice is always possible on a
  sub poset of the form $\{c \in C \mid c \leq d\}$ so the image
  of~\eqref{eq:2} surjects onto any finite product of the factors.
  Another way to say this is that the image of~\eqref{eq:2} is dense
  in the product topology on
  $\prod_{c_0 < c_1 < c_2} \pi_1(F(c_0),x_{c_0})$ when each
  $\pi_1(F(c_0),x_{c_0})$ is given the discrete topology.  If each
  $\pi_1(F(c_0),x_{c_0})$ is finite, the source of~\eqref{eq:2} is
  compact and hence its image is too so density of the image implies
  surjectivity, and hence there is no obstruction to rechoosing the
  $h_{c_0 < c_1}$ and get compatible extensions to maps
  \begin{equation*}
    N_{\leq 2}(c \downarrow C^\mathrm{op}) \to F(c).
  \end{equation*}

  Extending this map over the 3-skeletons we encounter obstructions in
  the cokernel of a homomorphism
  \begin{equation*}
    \prod_{c_0 < c_1 < c_2} \pi_2(F(c_0),x_{c_0}) \to \prod_{c_0 < c_1 <
      c_2 < c_3} \pi_2(F(c_0),x_{c_0}),
  \end{equation*}
  which is surjective by a similar argument (in fact, its cokernel is
  precisely the derived limit $\lim^3 \pi_2(F(c),x_c)$).  Continuing
  this way we end up with maps
  $N(c \downarrow C^\mathrm{op}) \to F(c)$, compatible over varying
  $c \in C$, and hence a point in $\holim F$.  Injectivity is similar,
  with obstructions vanishing for the same reason as
  $\lim^k \pi_k(F(c))$ vanishes.
\end{proof}
As the proof shows, one can sometimes get by with slightly weaker
assumptions in the above Proposition, e.g.\ that $\pi_k(F(c))$ is a
compact group/set for some topology in which the functoriality is
continuous.

\begin{Corollary}
  Let $C$ is a filtered category, $F: C \to s\Sets$ a functor, and $Z$
  is a Kan simplicial set such that $[F(c),Z] = \pi_0 s\Sets(F(c),Z)$
  is finite for all $c$ and $\pi_k(s\Sets(F(c),Z),f)$ is finite for
  all $k$ and all $c \in C$ and all $f: F(c) \to Z$.  (This happens
  e.g.\ if $Z$ has finite homotopy groups and $F(c)$ is equivalent to
  a finite CW complex for all $c$.)  Then the natural map
  \begin{equation*}
    [\hocolim_{c \in C} F(c),Z] \to \lim_{c \in C} [F(c),Z]
  \end{equation*}
  is a bijection.  In other words, any collection of maps $F(c) \to Z$
  which are compatible up to some (unspecified) homotopies may in fact
  be glued together to a map out of the homotopy colimit, and uniquely
  so up to homotopy.\qed
\end{Corollary}

\begin{Corollary}
  Let $X$ be a simplicial set and $F: C^\mathrm{op} \to s\Sets$ be a
  functor with Kan values, such that $s\Sets(X,F(c))$ has finite
  homotopy groups for all $c \in C$.  Then the natural map
  \begin{equation*}
    [X,\holim_{c \in C} F(c)] \to \lim_{c \in C} [X,F(c)]
  \end{equation*}
  is a bijection.
\end{Corollary}

We shall not directly use these two results, but the analogous results
in the similar setting of functors into simplicial commutative rings
will be very useful.

  \section{Duality and local conditions in Galois cohomology} \label{GlobalDuality}
  We formulate a version of Poitou--Tate duality with local constraints, following a suggestion of Harris to work with cone constructions.
  This is implicit in \cite{CHT} and surely known to all experts.

 \subsection{Statement of the theorem} \label{A1}

  \subsection{}  
  
  We work in {\'e}tale cohomology of $\Z[\frac{1}{S}]$ where $p \in S$. Let $M$ be a $p$-torsion {\'e}tale  locally constant sheaf, which
  we may think of as a representation of the $S$-unramified quotient $\Gal(\Q_S/\Q)$ of the Galois group;
  we write simply $H^i(M)$ for the etale cohomology or $H^i(\Z_S, M)$ where the set $S$ must be made explicit.
  
    We fix once and for all algebraic closures $\overline{\Q_v}$ and embeddings $ \iota_v: \Q_S \hookrightarrow \overline{\Q_v}$ for each $v \in S$. 
  This induces a map on Galois groups
  $$\iota_v^* : \Gal(\overline{\Q_v}/\Q_v) \rightarrow \Gal(\Q_S/\Q).$$
  
We write $H^*(\Q_v, M)$ for the Galois cohomology of $H^*(\Gal(\overline{\Q_v}/\Q_v), M)$.

  If $G$ is any group and $M$ a $G$-module, we refer to the ``standard cochain complex'' of {\em inhomogeneous} cochains computing the cohomology $H^*(G, M)$;
  for example, a $2$-cochain is a function $G \times G \rightarrow M$, etc.  Note that $G$ acts on this complex by conjugation on the domain (e.g. $G \times G$ in the example just given),
  this action descends to the trivial action on cohomology.  Also the cup product lifts to the cochain level in the standard (back face, front face) way. 
  If $G$ is a profinite group and $M$ a discrete $G$-module we will always understand cochains to be continuous.

    Let $C^*(\Q_v, M)$ be the standard cochain complex computing Galois cohomology of $H^*(\Gal(\overline{\Q_v}/\Q_v), M)$,
    and similarly define $C^*(M)$  as the  the standard cochain complex computing Galois cohomology  $H^*(\Gal(\Q_S/\Q), M)$.
    For brevity, if the module $M$ is understood, we will sometimes refer (e.g.) to $C^2(M)$ as $C^2$ and $C^2(\Q_v, M)$ as $C^2_v$. 
  
  If $x$ is a cochain in the standard cochain complex computing  the cohomology of $\Gal(\Q_S/\Q)$ 
  we use the notation $x|_{\Q_v}$ -- or simply $x_v$ if there is no risk of confusion -- for the
  cochain obtained by pulling back $x$ under the map   $\iota_v^* : \Gal(\overline{\Q_v}/\Q_v) \rightarrow \Gal(\Q_S/\Q).$

  \subsection{Local conditions, lifted to the cochain level} 
  We want to impose ``local conditions'' $\mathcal{L}_v^i \subset H^i(\Q_v, M)$ at places in $S$; 
  we will usually just write for short $\mathcal{L}_v \subset H^*(\Q_v, M)$. 
There will be a long exact sequence
 $$H^*_{\mathcal{L}}(M) \rightarrow H^*(M) \rightarrow  \prod_{v \in S} H^*(\Q_v, M)/\mathcal{L} \stackrel{[1]}{\rightarrow} $$
  where $H^*_{\mathcal{L}}$ are ``cohomology classes that belong to $\mathcal{L}$.''

To achieve this precisely, we need lifts to the co-chain level. 
  More precisely, we will suppose that $\mathcal{L}_v$ comes equipped with a  subcomplex  $C^*_{\mathcal{L}, v} \subset  C^*(\Q_v, M)$ satisfying the following axioms: 
  
  \begin{itemize}
  \item[(i)] $C^*_{\mathcal{L},v}$ is closed under the differential, and 
  \item[(ii)] $C^*_{\mathcal{L},v}$ is invariant under conjugacy, and
\item[(iii)] The cohomology of $C^*(\Q_v, M)/C^*_{\mathcal{L}}(\Q_v, M)$ ``is'' $H^*/\mathcal{L}$, i.e. the natural map
from the cohomology of $C^*(\Q_v, M)$ to the cohomology of $C^*(\Q_v, M)/C^*_{\mathcal{L}}(\Q_v, M)$ is surjective
in each degree and its kernel is precisely $\mathcal{L}_v$. 
  \end{itemize}

   We define $H^i_{\mathcal{L}}$ to be the {\em derived} set of classes in $H^i(M)$ that lie inside $\mathcal{L}$ for each $v \in S$, that is to say 
  the cohomology of the cone
\begin{equation} \label{local cohomology definition} C^n(M) \oplus  \bigoplus_{v \in S}  \frac{C^{n-1}(\Q_v, M)}{C^{n-1}_{\mathcal{L}}(\Q_v, M)} \end{equation}
We denote an element of this group by $(x, y_v)$, where $y_v$ really denotes an element of
$C^{n-1}/C^{n-1}_{\mathcal{L}}$ for each $v \in S$. The differential is the ``cone'' differential, that is to say,
$d(x,y_v) = (-dx, dy_v+x|_v)$. 

To be explicit:  cocycle in this group is a pair
$$ \left( x \in C^n(M), y_v \in C^{n-1}(\Q_v, M)/C^{n-1}_{\mathcal{L}}(\Q_v, M)\right)$$
with the property that $x|_{\Q_v} + d(y_v)$ belongs to $C^{n}_{\mathcal{L}}$, i.e.
``$x$ equipped with a reason for its restriction to $\Q_v$ to belong to $\mathcal{L}$.''

   Write $M = \Hom(M^*, \mu_{p^{\infty}})$.
  In what follows, we will want to consider also a dual local condition $\mathcal{L}^{\perp}$.  By this, we mean
  that we take the orthogonal complement $\mathcal{L}^{\perp} \subset H^*(\Q_v, M^*)$ 
  and {\em assume} that it is equipped with   a similar enrichment to the cochain level,
  $C_{\mathcal{L}^{\perp}} \subset C(\Q_v, M)$, which satisfies (i)--(iii) above and additionally 
  
\begin{itemize}
\item[(iv)]  the cup product 
\begin{equation} \label{desid:cup} C^i_{\mathcal{L}}(\Q_v, M) \times C^{3-i}_{\mathcal{L}^{\perp}}(\Q_v , M^*) \rightarrow C^3(\Q_v, \mu_{p^{\infty}})\end{equation}
vanishes {\em at the chain level.}  
\end{itemize}
 
 We are ready to formulate the statement of duality:
    
\begin{theorem} Suppose  $p>2$, that $\mathcal{L}, \mathcal{L}^{\perp}$ are as above, and both equipped with lifts to the cochain level, where both lifts   satisfy (i)--(iv) above.  There is a  duality
  $$H^i_{\mathcal{L}}(M) \times H^{3-i}_{\mathcal{L}^{\perp}}(M^*) \longrightarrow \mathbf{Q}/\mathbf{Z},$$
  where as usual $M^*=\Hom(M, \mu_{p^{\infty}})$.
  \end{theorem}
 
\subsection{Examples of local conditions meeting our conditions}  \label{ExampleAB}

The most important example, besides the trivial example\footnote{ take $\mathcal{L} = 0$ in all degrees, with $C_{\mathcal{L}}(\Q_v, M) = 0$ also;
dually take $\mathcal{L}^{\perp} = H^i$ in all degrees, with $C_{\mathcal{L}^{\perp}}(\Q_v, M^*) = C(\Q_v, M^*)$}  
is the following generalization of ``unramified'' local conditions: 

Take an arbitrary subgroup $\mathfrak{l} \subset H^1$, and take 
  $$\mathcal{L} = \begin{cases}  H^0 \subset H^0  \\  \mathfrak{l} \subset H^1 \\ 0 \subset H^2\end{cases} , \ \ \ 
 \mathcal{L}^{\perp} = \begin{cases}  H^0 \subset H^0 \\   \mathfrak{l}^{\perp} \subset H^1 \\ 0  \subset H^2 \end{cases}.$$ 
 
 We take, both for $\mathcal{L}$ (respectively $\mathcal{L}^{\perp}$):
  \begin{itemize}
 \item $C_{\mathcal{L}}^{0} =  C_0$. 
 \item $C_{\mathcal{L}}^1 $ to consist of all classes $x \in C^1(\Q_v, M)$ where $dx=0$ and $[x] \in \mathfrak{l}$ (resp. $\mathfrak{l}^{\perp})$
 \item  $C_{\mathcal{L}}^j = 0$ for $j \geq 2$. 
  \end{itemize}
 
The complex $C/C_{\mathcal{L}}$ looks like 
 $$0  \rightarrow C^1(\Q_v, M)/C^1_{\mathcal{L}} \rightarrow C^2(\Q_v, M) \rightarrow \dots $$
 and visibly the cohomology groups are  $0, H^1/\mathcal{L}_1, H^2$ as desired. The vanishing of the desired cup product
 \eqref{desid:cup} is obvious.

\subsection{Proof of the theorem}

\subsubsection{} 
First we verify this for $i=0$.  Given
  $$ \alpha \in \ker(H^0(M) \rightarrow   \prod_{v \in S} H^0(\Q_v, M)/\mathcal{L}^0_v)$$
$$ \beta \in \mathrm{coker} \left( H^2(M^*) \rightarrow  \prod_{v \in S} H^2(\Q_v, M^*)/\mathcal{L}_{\perp, v}^2 \right)$$
we define $\langle \alpha, \beta \rangle$ as the sum of local reciprocity pairings
$$\sum_{v \in S} (\alpha_v, \beta_v).$$
This is  well-defined because in fact
$\alpha|_{\Q_v} \in \mathcal{L}_v^0$ for each $v \in S$.   We claim it is a perfect pairing. 

The perfection of this pairing amounts to the statement that the resulting map 
\begin{equation} \label{brianL}  \frac{ \prod_{v \in S} H^2(\Q_v, M^*) }{ \langle \mathcal{L}_{\perp,v}^2, H^2(M^*) \rangle } \rightarrow \ker(H^0(M) \rightarrow   \prod_{v \in S} H^0(\Q_v, M)/\mathcal{L}_v)^{\vee} \end{equation}
is an isomorphism.   Here $(\dots)^{\vee} $ on the right hand side means maps to $\Q/\Z$; in what follows, let us use the word ``functional'' for ``map to $\Q/\Z$.'' 

The map \eqref{brianL} is visibly surjective: the nine-term exact sequence asserts that $\prod_{v \in S} H^2(\Q_v, M^*) \rightarrow 
H^0(M)^{\vee}$ is surjective (which is obvious anyway). 
For injectivity,  suppose that $\beta \in \prod_{v \in S} H^2(\Q_v, M^*)$
induces the zero functional on the right-hand side.  In particular ``pairing with $\beta$'' descends to a  functional on 
the image of $H^0(M)$ inside $\prod_{v} H^0(\Q_v, M)/\mathcal{L}^0_v$. 
If we replace $\beta$ by $\beta + s_v$ for $s_v \in \mathcal{L}_{\perp,v}^2$, the induced functional on 
this image changes  the restriction of $\langle -, s_v \rangle$ on $\prod_{v} H^0(\Q_v, M)/\mathcal{L}_v$.  But the pairing $$\mathcal{L}_{2,v}^{\perp} \times H^0(\Q_v, M)/\mathcal{L}_v^0 \rightarrow \Q/\Z$$ is an isomorphism, 
and so we can modify $\beta$ by an element of $\prod \mathcal{L}_{\perp,v}^2$ in such a way that 
the functional ``pairing with $\beta$'' is actually zero on the image of $H^0(M)$. 
 Then the nine-term exact sequence means that $\beta$ actually lies in the image of $H^2(M^*)$, as desired.

\subsubsection{}  \label{explicit pairing formula}
Now we examine the trickier case $i=1$.  
The main issue is to construct the pairing. We then verify it is perfect by filtering the groups involved and looking at the pairing on graded pieces,
where it reduces to better-known pairings. 

Take classes $$(x \in C^1, \overline{y_v} \in C^0_v/C^0_{\mathcal{L},v})$$ and
$$(x' \in C^2,  \overline{y'_v } \in C^1_v/C^1_{\mathcal{L}^{\perp},v})$$ representing elements of $H^1_{\mathcal{L}}$
and $H^2_{\mathcal{L}^{\perp}}$.  Lift $\overline{y_v}, \overline{y'_v}$ to $y_v \in C^0_{\mathcal{L},v}, y'_v \in C^1_{L^{\perp},v}$. 
 Set $$\epsilon_v = dy_v + x_v \in C^1_{v, \mathcal{L}}, \ \ \epsilon'_v  = dy'_v +x'_v \in C^2_{\mathcal{L}^{\perp},v} $$ similarly. Note that $d\epsilon_v = 0$.   Take $z \in C^2$ with $dz = x \cup x'$.
Now form
\begin{equation} \label{moo} (y_v \cup x'_v) - (\epsilon_v \cup y'_v) + z_v \in C^2(\Q_v, \mu_{p^{\infty}})  \end{equation} The differential of this equals
 $$  \epsilon_v \cup x'_v   + \epsilon_v \cup dy'_v =  (\epsilon_v \cup \epsilon_v')  \stackrel{(iv)}{=} 0.$$
 where we used assumption (iv). 

In other words, we have a class in $H^2(\Q_v, \mu_{p^{\infty}})$ for each  $v \in S$.
Taking the sum of invariants - evidently independent of choice of $z$ -- gives an element of $\Q_p/\Z_p$
associated to the classes $(x, y_v)$ and $(x', y_v')$.  This is our pairing
\begin{equation} \label{12pairing} H^1_{\mathcal{L}}(M) \times H^{2}_{\mathcal{L}^{\perp}}(M^*) \longrightarrow \Q_p/\Z_p.\end{equation}

\subsubsection{Symmetry}

This construction is ``symmetric,'' or rather a similar constructio with roles reversed gives the same:
if we consider, instead of \eqref{moo}, the class 
\begin{equation} \label{moo2}  - x_v \cup y'_v  + y_v \cup \epsilon'_v  +  z_v \end{equation}
and form an invariant similarly, the result is the same. 

 In fact, the differential of \eqref{moo2} is $  x_v \cup \epsilon'_v  +dy_v \cup \epsilon'_v =  \epsilon_v \cup \epsilon_v' = 0$; 
and \eqref{moo} differs from \eqref{moo2} by 
 $$ y_v \cup x'_v + x_v \cup y'_v - \epsilon_v \cup y'_v - y_v \cup \epsilon'_v = - y_v \cup dy'_v - dy_v \cup y'_v =-d(y_v \cup y'_v)$$ 
 thus they have the same cohomology class.

 \subsubsection{Independence}
 
 Let's try to see this is independent of   the various choices we made. 
 
 \begin{itemize}
 \item[(a)] 
 If we modify $y_v$ by $w_v \in C^0_{v, \mathcal{L}}$  
 the class \eqref{moo} changes by
 $$w_v \cup x'_v - d w_v \cup y'_v   = w_v \cup (x'_v + dy'_v) -  \underbrace{  d(w_v \cup y'_v) }_{ d w_v \cup y'_v  + w_v \cup dy'_v  }$$
 which is cohomologically trivial. 
  
 If  we modify $y_v'$ by $w_v'$  the situation is similar.
  \item[(b)] Suppose we modify $x$ by a boundary in the fashion $(x  \mapsto x_v - da, y_v \mapsto y_v+a )$, for some $a \in C^0$. 
This does not change $\epsilon_v$. Replacing $z$ with
 $z  - a \cup x'$, we see that  the class \eqref{moo} is unchanged.

\end{itemize}

 \subsubsection{Perfect pairing}
 We now verify that \eqref{12pairing} is perfect, by comparing it to standard pairings.  
  We have filtrations on $H^i_{\mathcal{L}}$
 with successive  graded pieces as follows:
 $$ \underbrace{ \mathrm{cokernel}(H^{i-1} \rightarrow \prod_v H^{i-1}(\Q_v)/\mathcal{L})}_{x_v=\epsilon_v=0}, 
  \  \underbrace{ \Sha^{i}}_{\epsilon_v=0},  \mathrm{image}(H^i \rightarrow \prod_{v} H^i(\Q_v)/\mathcal{L}),$$
  e.g. the first piece comes from the image in cohomology of cycles satisfying $x_v=\epsilon_v=0$ for all $v \in S$. 
  As usual, we define $\Sha^i$ here to be global cohomology classes that are  everywhere locally trivial.
    
  It's easy to see that this filtration is its own dual under this pairing.  %
  We will explicate the pairing piece by piece, with respect to this filtration:

\begin{equation} \label{fro1}   \mathrm{cokernel}(H^{0} \rightarrow \prod_v H^{0}(\Q_v,M)/\mathcal{L}^0_v) \times   \mathrm{image}(H^2 \rightarrow \prod_{v} H^2(M^*)/\mathcal{L}_{\perp,v}^2) 
\end{equation}
\begin{equation} \label{fro2}  \Sha^1 \times \Sha^2 \end{equation}
\begin{equation} \label{fro3}   \mathrm{image}(H^1 \rightarrow \prod_{v} H^1(M^*)/\mathcal{L}_{v,\perp}^1)  
\times   \mathrm{cokernel}(H^{1} \rightarrow \prod_v H^{1}(M)/\mathcal{L}^1_v)    \end{equation}
  
For the first: we take $(x =0, y_v \in H^0(\Q_v))$ representing the left-hand class,  and on the right we can take 
any pair $(x', y'_v)$ with $x'$ representing the given class in the image.   We can take $z=0$. Then \eqref{moo} is given by 
$$ y_v \cup x'_v - dy_v \cup y'_v \sim y_v \cup (x'_v - dy'_v)$$ 
and thus realizes the standard local duality of $H^0$ and $H^2$.

For \eqref{fro2}: it's easy to see the pairing from above recovers the Tate pairing.

 For \eqref{fro3}:  We can choose a representative for the right-hand class with $x'= \epsilon'_v = 0$; 
 on the left we choose a representative $(x_v, y_v)$.  Again we can take $z=0$. Then \eqref{moo} is given by
 $-\epsilon_v \cup y_v'  =  -(x_v+ d y_v)  \cup y'_v$, which realizes up to sign the standard local pairing of $H^1$ and $H^1$.

\newpage

\bibliography{ddr-refs}
  \bibliographystyle{plain}

\end{document}